\documentclass[reqno]{amsart}
\pdfoutput=1 
\input epsf
\address{Simons Center for Geometry and Physics,
State University of New York, Stony Brook, NY 11794-3636 U.S.A.} 
\email{kfukaya@scgp.stonybrook.edu}
\address{Center for Geometry and Physics, Institute for Basic Sciences (IBS), Pohang, Korea \& Department of Mathematics,
POSTECH, Pohang, Korea} \email{yongoh1@postech.ac.kr}
\address{Graduate School of Mathematics,
Nagoya University, Nagoya, Japan} \email{ohta@math.nagoya-u.ac.jp}
\address{Research Institute for Mathematical Sciences, Kyoto University, Kyoto, Japan}
\email{ono@kurims.kyoto-u.ac.jp}

\usepackage{graphicx}
\usepackage{amsmath}
\usepackage{amscd}
\usepackage{amssymb}
\usepackage{amstext}
\usepackage{amsmath}
\usepackage[all]{xy}
\usepackage{yhmath}
\usepackage{mathrsfs}
\usepackage{bbding}
\usepackage{euscript}
\usepackage{comment}
\usepackage{color}
\usepackage{hyperref}
\hypersetup{
 setpagesize=false,
 bookmarksnumbered=true,%
 bookmarksopen=true,%
 colorlinks=true,%
 linkcolor=blue,
 citecolor=green,
}
\makeatletter
\def\l@section{\@tocline{1}{0pt}{3mm}{8mm}{}}
\def\l@subsection{\@tocline{2}{0pt}{6mm}{10mm}{}}
\def\l@subsubsection{\@tocline{3}{0pt}{9mm}{11mm}{}}
\makeatother

\def\E{\ifmmode{\mathbb E}\else{$\mathbb E$}\fi} %
\def\N{\ifmmode{\mathbb N}\else{$\mathbb N$}\fi} %natural numbers%
\def\R{\ifmmode{\mathbb R}\else{$\mathbb R$}\fi} %real numbers
\def\Q{\ifmmode{\mathbb Q}\else{$\mathbb Q$}\fi} %rational numbers
\def\C{\ifmmode{\mathbb C}\else{$\mathbb C$}\fi} %complex numbers
\def\H{\ifmmode{\mathbb H}\else{$\mathbb H$}\fi} %complex numbers
\def\Z{\ifmmode{\mathbb Z}\else{$\mathbb Z$}\fi} %integers
\def\P{\ifmmode{\mathbb P}\else{$\mathbb P$}\fi} %
\def\T{\ifmmode{\mathbb T}\else{$\mathbb T$}\fi} %
\def\SS{\ifmmode{\mathbb S}\else{$\mathbb S$}\fi} %
\def\DD{\ifmmode{\mathbb D}\else{$\mathbb D$}\fi} %
\def\K{\ifmmode{\mathbb K}\else{$\mathbb K$}\fi}
\def\F{\ifmmode{\mathbb F}\else{$\mathbb F$}\fi} %finite field 

\newcommand{\rank}{\operatorname{rank}}

\theoremstyle{theorem}
\newtheorem{thm}{Theorem}[section]
\newtheorem{cor}[thm]{Corollary}
\newtheorem{lem}[thm]{Lemma}

\newtheorem{prop}[thm]{Proposition}

\newtheorem{clm}[thm]{Claim}

\theoremstyle{definition}
\newtheorem{defn}[thm]{Definition}
\newtheorem{rem}[thm]{Remark}

\newtheorem{cons}[thm]{\rm\bfseries{Construction}}
\newtheorem{exm}[thm]{Example}
\newtheorem{conds}[thm]{Condition}

\newtheorem{shitu}[thm]{Situation}
\newtheorem{choi}[thm]{Choice}
\newtheorem{notation}[thm]{Notation}

\numberwithin{equation}{section}
\makeindex
\begin{document}

\title[Construction of a linear K-system]{Construction of a linear K-system in Hamiltonian Floer theory}
\dedicatory{To Claude Viterbo on the occasion of his 60th birthday}

\author{Kenji Fukaya, Yong-Geun Oh, Hiroshi Ohta, Kaoru Ono}

\thanks{Kenji Fukaya is supported partially by 
Simons Collaboration on homological mirror symmetry,
Yong-Geun Oh by the IBS project IBS-R003-D1, Hiroshi Ohta by JSPS Grant-in-Aid
for Scientific Research Nos. 
15H02054, 21H00983, 21K18576, 21H00985 and Kaoru Ono by JSPS Grant-in-Aid for
Scientific Research, Nos. 26247006, 19H00636.}

\begin{abstract}
The notion of linear K-system was introduced by the present authors as
an abstract model arising from the structure of compactified moduli
spaces of solutions to Floer's equation in the book \cite{fooonewbook}.
The purpose of the present article is to provide a geometric realization of
the linear K-system associated with solutions to Floer's equation 
in the Morse-Bott setting.
Immediate consequences (when combined with the abstract theory from \cite{fooonewbook})
are the construction of Floer cohomology for periodic Hamiltonian systems on general compact symplectic manifolds without any restriction,
and the construction of an isomorphism over the Novikov ring between the Floer cohomology and the singular cohomology of the underlying symplectic manifold. The present article utilizes various analytical results on pseudoholomorphic curves established in our earlier papers and books.
\par
However the paper itself is geometric in nature,
and does not presume much prior knowledge of Kuranishi structures
and their construction but assumes only the elementary part thereof, 
and results from \cite{fooo:const1} and \cite[Chapter 8]{foooanalysis} on their construction, 
and the standard knowledge on Hamiltonian Floer theory.
We explain the general procedure of the construction of a linear K-system by explaining
in detail the inductive steps of ensuring the compatibility conditions for the system of
Kuranishi structures leading to a linear K-system for the case of Hamiltonian Floer theory.
\end{abstract}
\maketitle
\date{}

\maketitle

\tableofcontents

\section{Introduction}
\label{sec:intro}
The technique of virtual fundamental cycles or chains now provides us with a general and powerful method for studying certain moduli spaces in symplectic geometry and gauge theory. 
In \cite{FO}, the first author and the fourth of this article 
introduced the notion of Kuranishi structure and constructed the 
virtual fundamental cycle of the moduli space of stable maps from a marked semi-stable curve to a general closed symplectic manifold based on the theory of Kuranishi structures, and applied it to Floer theory for periodic Hamiltonian systems. After that, 
in \cite{fooobook}, \cite{fooobook2} the authors of the present article
constructed the virtual fundamental chain of the moduli space of stable maps from a  marked bordered semi-stable curve to a closed symplectic manifold with Lagrangian boundary condition by using the theory of Kuranishi structures, and applied it to 
Lagrangian intersection Floer theory. 
In these two studies, we had to construct a certain compatible system of Kuranishi structures on these moduli spaces.

Recently in \cite{fooonewbook}, 
the authors developed a systematic foundation of the 
theory of Kuranishi structures and virtual fundamental chains in a 
general abstract setting.
In particular, we axiomatized the properties of two systems,
called a {\it linear K-system} and a {\it tree-like K-system}, consisting of abstract paracompact metrizable spaces equipped with Kuranishi structures satisfying certain 
compatibility conditions. 
\par
A tree-like K-system \cite[Definition 21.9]{fooonewbook} is a model    
arising from the moduli spaces of stable maps from a marked bordered semi-stable curve with Lagrangian boundary condition which are used in Lagrangian Floer theory.
In fact, in \cite{fooo:const1}, \cite{fooo:const2} 
we carried out a geometric realization of the tree-like K-system by using the moduli spaces
of stable maps from a marked bordered semi-stable curve.  
Namely, we constructed a system of Kuranishi structures 
on the moduli spaces which indeed satisfies the axioms of the tree-like K-system given in \cite{fooonewbook}.
In the procedure of constructing the required Kuranishi structures 
we introduced the notion of {\it obstruction bundle data} assigned to each point of the moduli space \cite[Definition 5.1]{fooo:const1} and showed that we can associate a Kuranishi structure on each moduli space to any obstruction bundle data in a canonical way \cite[Theorem 7.1]{fooo:const1}.
In addition, if the system of the obstruction bundle data satisfies certain compatibility conditions ({\it disk-component-wise system of obstruction bundle data} in the sense of \cite[Definition 5.1]{fooo:const2}), then the obtained system of Kuranishi structures associated to the system of obstruction bundle data defines a tree-like K-system \cite[Theorem 5.3]{fooo:const2}. 
\par
On the other hand, a linear K-system \cite[Definition 16.6]{fooonewbook} 
is 
the model that is used in Floer theory for periodic Hamiltonian systems
which arises from the moduli spaces of solutions to Floer's equation.
The purpose of the present article is to give a geometric realization of 
the linear K-system by using the moduli spaces of solutions to Floer's equation.
The strategy to achieve this realization lies on the line similar to that for the tree-like K-system.  
Namely, we introduce the notion of obstruction bundle data at each point for the moduli spaces of solutions to Floer's equation in Definition \ref{obbundeldata1}.
Since we explained the construction of a Kuranishi structure 
associated to obstruction bundle data in \cite{fooo:const1}, \cite{fooo:const2} in detail and 
since such a construction is similar for the case of a linear K-system, we mainly present the way to equip moduli spaces with suitable obstruction bundle data in this article.
In Section \ref{subsec;KuramodFloercor} we will construct a system of obstruction bundle data satisfying 
certain compatibility conditions which gives rise to a linear K-system.
Comparing to \cite{fooo:const2}, 
we take a slightly different way for constructing the compatible system of obstruction bundle data in order to demonstrate that both methods are applicable to both cases. 
Here we use {\it outer collars} of the moduli spaces to describe the compatibility.
This method is outlined in \cite[Remark 4.3.89]{foooAst}.
The notion of outer collars is introduced and used in 
\cite[Chapter 17]{fooonewbook} in the abstract setting. 
\par
Now our main theorem in this article is summarized as follows.
\begin{thm}[Theorem \ref{theorem266}]\label{thm:main}
Let $(X,\omega)$ be a closed symplectic manifold and 
$H : X \times S^1 \to \R$ a smooth function.
Suppose that the set ${\rm Per} (H)$ of all contractible 1-periodic orbits of the time dependent Hamiltonian vector field of $H$ is Morse-Bott non-degenerate 
(see Condition \ref{264form}).
Then we can construct a linear K-system $\mathcal F_X(H)$ 
such that the critical submanifolds are connected components of $\widetilde{{\rm Per}} (H)$ 
(which are copies of connected components of ${\rm Per} (H)$, see Definition \ref{def:Per}), and the spaces of connecting orbits are the 
outer collared spaces of the moduli spaces of solutions to Floer's equation \eqref{Fleq}.
\end{thm}
We also construct morphisms of linear K-systems. Combining these results 
with general properties concerning linear K-systems proved in \cite{fooonewbook},
we have the following theorems:
\begin{thm}[Theorem \ref{HFwelldefine}]
Under the assumption of Theorem \ref{thm:main}
we can define the Floer cohomology
$HF(X,H;\Lambda_{0,{\rm nov}})$ of a periodic Hamiltonian system,
also called the \emph{Hamiltonian Floer cohomology,}
which is independent of various choices involved in the definition.
\end{thm}
Here $\Lambda_{0,{\rm nov}}$ is the Novikov ring defined by
\begin{equation}\label{nov0}
\Lambda_{0,{\rm nov}}=
\left\{ \left. \sum_{i=0}^{\infty} a_i T^{\lambda_i} \right\vert~ a_i \in \C,  \lambda_i \in \R_{\ge 0}, \lim_{i\to \infty} \lambda_i =\infty \right\}.
\end{equation}
We define the Novikov field $\Lambda_{\rm nov}$ as its field of fractions 
by allowing $\lambda_i$ to be negative.
\begin{thm}[Theorem \ref{HFwelldefine2}, Corollary \ref{cor2333333}]
Suppose that two functions $H^1, H^2 : X \times S^1 \to \R$ satisfy the assumption in Theorem \ref{thm:main}.
Then the Floer cohomologies
$
HF(X,H^r;\Lambda_{{\rm nov}})
$
over the Novikov field $\Lambda_{{\rm nov}}$
associated to $H^r$ $(r=1,2)$ are isomorphic as $\Lambda_{{\rm nov}}$-modules:
$$
HF(X,H^1;\Lambda_{{\rm nov}})
\cong HF(X,H^2;\Lambda_{{\rm nov}}).
$$
Moreover they are isomorphic to $H(X;\Lambda_{{\rm nov}})$.
\end{thm}
More specifically, 
in Theorem \ref{theorem266} we associate an inductive system of linear K-systems in the sense of \cite[Section 16]{fooonewbook}
to a periodic Hamiltonian system.
Together with \cite[Theorem 16.39]{fooonewbook} we obtain the Floer cohomology of a periodic Hamiltonian system in Theorem \ref{HFwelldefine}.
In Section \ref{subsec;KuramodFloermor} we construct a morphism between two
linear K-systems associated to
different Hamiltonians in Theorem \ref{theorem266rev}.
Together with \cite[Theorem 16.39]{fooonewbook}
it implies that the Floer cohomology of a periodic
Hamiltonian system is independent of the Hamiltonian function
(Theorem \ref{HFwelldefine2}).
We then calculate the Floer cohomology of a periodic
Hamiltonian system on any compact symplectic manifold in Section \ref{sec;calc}.
See Corollary \ref{cor2333333}.
It gives a proof of the homological version of Arnold's conjecture \cite{Ar}
about the number of periodic orbits of a periodic
Hamiltonian system.
This proof 
is the same as those in the literature
\cite{FO}, \cite{LiuTi98}, \cite{Rua99}, modulo technical detail.
All the proofs in the literature as well as the one in this article,
first define the Floer cohomology of a periodic
Hamiltonian system.
In the generality we discuss here, we need to use
virtual fundamental chain techniques.
The references \cite{FO}, \cite{LiuTi98}, \cite{Rua99} use various versions of
virtual fundamental chain techniques for this purpose.
\par
A significant difference lies in the way we prove that 
the Floer cohomology of a periodic 
Hamiltonian system on a symplectic manifold $X$ is isomorphic to the cohomology of 
$X$ itself.
In the literature there appeared basically three different ways to prove it.
(We mention only the methods which are established to work for arbitrary compact symplectic manifolds.)
\begin{enumerate}
\item
Morse-Bott method:
\begin{enumerate}
\item
We first include the case when the set ${\rm Per}(H)$ of periodic
orbits of the periodic Hamiltonian $H :  X \times S^1\to \R$ 
does not necessarily consist of isolated points but is a submanifold.
Technically speaking we include the case
when Floer's functional ${\mathcal A}_H$ (see (\ref{form2688ene}))
is not necessarily a Morse
function but only a Morse-Bott function.
\item
We prove the independence of the Floer cohomology of the periodic Hamiltonian $H$.
\item
We next consider the case when $H \equiv 0$
and prove that the Floer cohomology of the periodic Hamiltonian
system is isomorphic to the cohomology of $X$.
\end{enumerate}
\item Method to reduce to a small Morse function:
\par
We consider the case when the Hamiltonian $H$ is sufficiently
small in the $C^2$ sense and is time independent. We
prove that the Floer cohomology of the periodic Hamiltonian
system is isomorphic to the (Morse) homology of $X$ in that case.
\item
Reduction to Lagrangian Floer cohomology of the diagonal:
\begin{enumerate}
\item
We consider the diagonal $\Delta \subset X \times X$ in the direct product
of $X$ with itself with symplectic form $-\omega \oplus \omega$.
\item
We prove that Lagrangian intersection Floer cohomology 
$$HF((\Delta,b),(\Delta,b);\Lambda_{0,{\rm nov}})$$
is well-defined in the sense of \cite{fooobook}.
\item
We prove that $HF((\Delta,b),(\Delta,b);\Lambda_{0,{\rm nov}})$ is isomorphic to
$H(X;\Lambda_{0,{\rm nov}})$ for a certain choice of a bounding cochain $b$.
\end{enumerate}
\end{enumerate}
The method of this article (Section \ref{sec;calc}) is a variation of 
the Morse-Bott method (1).
We do not use $S^1$ equivariant Kuranishi structures explicitly.
This is the point where our discussion is
slightly different from those in the literature but is the same as in \cite{Pa}.
The method (2) was used in \cite{FO}.
We provided its technical details in \cite[Part 5]{foootech}.
The method (3) is \cite[Theorem H]{fooobook}.
See also \cite[Subsection 6.3.3]{fooo092}.
\par
Any of those methods implies the following inequality
due to \cite{FO}, \cite{LiuTi98}, \cite{Rua99}:
\begin{equation}\label{statementarnold}
\# {\rm Per}(H) \ge \sum_{k} \rank H_k(X;\Q),
\end{equation}
which was proved by several people for some special cases of $X$,
for example, \cite{CZ}, \cite{Flo89I}, \cite{HS}, \cite{On}.
The case when the coefficient field $\Q$ is replaced by a finite field
is studied in 
a recent paper by Abouzaid-Blumberg \cite{AB}.
\par
In the present paper, we presume that readers have only basic knowledge concerning the theory of Kuranishi structures, for example, some basic definitions and terminology 
contained in the quick survey \cite[Part 7]{fooospectr} or in the much shorter summary 
\cite[Section 6]{fooo:const1}. 
While we use some part of \cite{fooonewbook}, we refer to it in a pinpointed way, so
readers do not have to understand the details of the proof therein. 
We also use the notion of outer collaring introduced in \cite[Chapter 12]{fooonewbook}, but do not assume readers' knowledge of the contents thereof.
As for the analytic arguments, 
we employ the results from \cite{foooanalysis}.
We also use the arguments in \cite[Chapter 8]{foooanalysis} to prove smoothness of the coordinate change.
In other words, we use the exponential decay estimates but do not use its proof.
In this way \cite{foooanalysis} (except Chapter 8) and most parts of \cite{fooonewbook}
are used as a `black box' in this article.
\par
Then based on the arguments and results in \cite{fooo:const1} we will construct 
the desired Kuranishi structures. 
Although there are some parts, especially Part 5 in \cite{foootech} related to this article, 
we do not assume 
the contents of \cite{foootech} 
for reading this article. Indeed we will repeat and describe the arguments here if necessary.
\par
Throughout this paper, a {\it K-space} means 
a paracompact metrizable space with a Kuranishi structure.
A {\it K-system} is an abbreviation for a system of K-spaces.
We assume that $X$ is any closed symplectic manifold, 
unless otherwise mentioned.
\par
The authors would like to thank a referee for careful reading and useful comments.

\section{Floer's equation and moduli space of solutions}
\label{subsec;feq}
We review Floer's moduli space in this section.
Our discussion is mostly the same as \cite[Section 29]{foootech},
except that we include the case when the space of contractible periodic orbits
has positive dimension.
We repeat several parts for the reader's convenience.
\par
Let $H : X \times S^1 \to \R$ be a smooth function on a symplectic manifold $(X,\omega)$.
We put $H_t(x) = H(x,t)$ where $t \in S^1$ and $x \in X$.
The function $H_t$ generates the Hamiltonian vector field $\frak X_{H_t}$ defined by
$$
i_{\frak X_{H_t}}\omega = d{H_t}.
$$
It defines a one parameter family of diffeomorphisms
$
{\rm exp}_t^{H} : X \to X
$
by
\begin{equation}\label{fporm2622}
\aligned
{\rm exp}_0^{H}(x) & = x, \\
\left(\frac{d\,{\rm exp}_t^{H}}{dt}\right)(x,t_0) & = \frak X_{H_{t_0}}({\rm exp}_{t_0}^{H}(x)).
\endaligned
\end{equation}
We denote by ${\rm Per}(H)$ the set of  all 1-periodic orbits of the time dependent vector field
$\frak X_{H_t}$.
We can identify
\begin{equation}\label{264form}
{\rm Per}(H) 
\cong 
{\rm Fix}({\rm exp}_1^{H}) = \{x \in X \mid {\rm exp}_1^{H}(x) = x\}.
\end{equation}
From now on, ${\rm Per}(H)$ denotes the set of contractible 1-periodic orbits.
Our assumption that  ${\rm Per}(H)$  is non-degenerate in the Morse-Bott sense
is as follows:
\begin{conds}\label{weaknondeg}
We say that ${\rm Per}(H)$ is {\it Morse-Bott non-degenerate} if the following holds.
\begin{enumerate}
\item ${\rm Fix}({\rm exp}_1^{H})$ is a smooth submanifold of $X$.
\item
Let $x \in {\rm Fix}({\rm exp}_1^{H})$ and consider the linear map:
$
d_x{\rm exp}_1^{H} : T_x X \to T_x X.
$
We require
\begin{equation}
T_x ({\rm Fix}({\rm exp}_1^{H}))
=
\{V \in  T_x X\mid (d_x{\rm exp}_1^{H})(V) = V \}.
\end{equation}
\end{enumerate}
\end{conds}
\begin{rem}\label{rem2622}
The typical examples where Condition \ref{weaknondeg} is satisfied
are the following two cases.
\begin{enumerate}
\item
${\rm Per}(H)$ is discrete.
In this case Condition  \ref{weaknondeg} is equivalent to
the condition that
the graph of ${\rm exp}_1^{H}$ is transversal to the diagonal in
$X \times X$.
\item
The case $H=0$.
\end{enumerate}
To prove (\ref{statementarnold}) it suffices to study these two cases only.
\end{rem}
\begin{defn}\label{def:Per}
We put
$$
\widetilde{\rm Per}(H) = \{(\gamma,w) \mid \gamma \in {\rm Per}(H),\,\, w : D^2 \to X,\,\, w(e^{2\pi it}) = \gamma(t)\}/\sim,
$$
where $(\gamma,w) \sim (\gamma',w')$ if and only if $\gamma = \gamma'$ and
$$
\omega([w] - [w']) = 0,\quad
c_1(TX)([w] - [w']) = 0.
$$
We have a natural surjection $\varpi : \widetilde{\rm Per}(H) \ni [(\gamma, w)]  \to 
\gamma \in {\rm Per}(H)$. 
We put $\frak G = \pi_2(X)/\sim$,
where $\alpha \sim \alpha'$ if and only if $\omega[\alpha] = \omega[\alpha']$
and $c_1(TX)[\alpha] = c_1(TX)[\alpha']$.
Then the group $\frak G$ acts on $\widetilde{\rm Per}(H)$ by changing the bounding disk 
$w:D^2 \to X$ so that the action on the fiber $\varpi^{-1}(\ast)$ is simply transitive.
\end{defn}
Following \cite{Flo89I}, we consider maps $u : \R \times S^1 \to X$ satisfying the equation
\begin{equation}\label{Fleq}
\frac{\partial u}{\partial \tau} +  J \left( \frac{\partial u}{\partial t} - \frak X_{H_t}
\circ u \right) = 0
\end{equation}
which we call {\it Floer's equation}.
Here $\tau$ and $t$ are the coordinates of $\R$ and $S^1 = \R/\Z$, respectively.
For $\tilde{\gamma}_{\pm} = ({\gamma}_{\pm},w_{\pm}) \in \widetilde{\rm Per}(H)$
we consider the boundary condition
\begin{equation}\label{bdlatinf}
\lim_{\tau \to \pm \infty}u(\tau,t) = \gamma_{\pm}(t).
\end{equation}
\begin{prop}\label{prof263}
We assume Condition \ref{weaknondeg}. Then for any solution $u$
of (\ref{Fleq}) with
$$
\int_{\R \times S^1}
\left\Vert\frac{\partial u}{\partial \tau}\right\Vert^2 d\tau dt < \infty
$$
there exists ${\gamma}_{\pm} \in {\rm Per}(H)$ such that (\ref{bdlatinf})
is satisfied.
\end{prop}
\begin{proof}
In the case of Remark \ref{rem2622} (1) this is proved by Floer
\cite[Proposition 3b]{Flo89I}.
In the case Remark \ref{rem2622} (2) this follows
from the removable singularity
theorem for pseudo-holomorphic curves.
The general case can be proved in the same way as in 
\cite[Lemma 11.2]{FO}.
We omit the details
of the proof of the general
case since to prove (\ref{statementarnold}) it suffices to
consider the two cases in Remark \ref{rem2622}.
\end{proof}
\begin{rem}
The convergence (\ref{bdlatinf}) is of exponential order.
Namely we have
$$
\Vert u(\tau,t) - \gamma_{\pm}(t)\Vert_{C^k} \le C_k e^{-c_k \vert\tau\vert}
$$
for some $c_k >0, C_k>0$.
\end{rem}
We decompose
$\widetilde{\rm Per}(H)$ into connected components and write
\begin{equation}\label{267form}
\widetilde{\rm Per}(H)
= \coprod_{\alpha \in \frak A}R_{\alpha}.
\end{equation}
Here $\frak A$ is the index set for connected components of 
$\widetilde{\rm Per}(H)$.
We denote by $\overline{R_{\alpha}}\subset X$ the image of $R_{\alpha}$ under the projection 
$\varpi : [(\gamma, w)] \in \widetilde{\rm Per}(H) \to 
\gamma(0) \in 
{\rm Fix}({\rm exp}_1^{H})
(\cong {\rm Per}(H))$, see \eqref{264form}. 
In the Morse-Bott situation\footnote{In the non-degenerate case, the issue discussed here was written in \cite{Flo89I, FO}. Here we extend the argument 
to the Morse-Bott case.}, we need to use certain local systems on the space of $1$-periodic orbits in order to equip the moduli spaces of 
solutions to Floer's equation for defining Floer complexes and chain homomorphisms with the orientation isomorphisms in the sense of linear K-systems 
\cite[Condition 16.1 (VII)]{fooonewbook}.  
Gluing $D^2$ and $[0, \infty) \times S^1$ by identifying $e^{2\pi \sqrt{-1} t} \in \partial D^2$ and $(0, [t]) \in \{0\} \times S^1$, we 
obtain a capped half cylinder $Z$.
\par
For each $1$-periodic orbit $\gamma \in \overline{R_\alpha}$, let $\mathcal{P}_{\gamma}(\overline{R_\alpha)}$ be the set of  trivializations of $\gamma^*TX$ as a symplectic vector bundle 
and write $\mathcal{P}(\overline{R_\alpha})= \cup_{\gamma \in \overline{R_\alpha}} \mathcal{P}_{\gamma}(\overline{R_\alpha})$.  
Pick $\mathfrak{t} \in \mathcal{P}_{\gamma}(\overline{R_\alpha})$.  
Using $\mathfrak{t}$, we extend $\gamma^*TX$ to a symplectic vector bundle over $D^2$.  Gluing it with $(\gamma \circ {\rm pr}_2)^* TX$ on 
$[0,\infty) \times S^1$, we obtain a symplectic vector bundle $E(\mathfrak{t})$ on $Z$.  Here ${\rm pr}_2: [0,\infty)\times S^1 \to S^1$ is the second factor projection.  
We extend $\gamma^*J_{\partial D^2=\{0\} \times S^1}$ to a complex structure, which is also denoted by $J$  of the vector bundle $E(\mathfrak{t})$.  
The pull-back of the unitary connection induces a holomorphic vector bundle structure on $E(\mathfrak{t}) \to Z$.  Pick a cut-off function $\chi:[0, \infty) \to \mathbb{R}$ 
such that $\chi = 0$ near $0$ and $\chi (\tau) =1$ for sufficiently large $\tau$.  
\par
Based on \cite[section 2e]{Flo89I} and \cite[section 21]{FO}, 
we define the map 
$P(\gamma; \mathfrak{t}):\Gamma(Z;E(\mathfrak{t})) \to 
\Gamma(Z; E(\mathfrak{t})\otimes \Lambda^{0,1})$ by 
\begin{equation}\label{operatorforori}
\aligned
& P(\gamma;\mathfrak{t}) \xi \\
= & 
\begin{cases} 
\overline{\partial}_{E(\mathfrak{t})} \xi & \text{on $D^2$}, \\
\left(1 - \chi(\tau)\right) \overline{\partial}_{E(\mathfrak{t})} \xi 
(d\tau - \sqrt{-1} dt) + \chi(\tau) D_{\gamma}\overline{\partial}_{J,H} \xi  & \text{on $[0,\infty) \times S^1$}.
\end{cases}
\endaligned
\end{equation}
Here $D_{\gamma}\overline{\partial}_{J,H}$ is the linearization operator for Floer's equation of connecting orbits at the stationary solution 
$u(\tau, t)= \gamma(t)$, i.e., 
$$
D_{\gamma}\overline{\partial}_{J,H} \xi =\left(\nabla_{\frac{\partial}{\partial \tau}} \xi +  \nabla_{\frac{\partial}{\partial t}} (J\xi)- \nabla_{\xi} (J{\mathfrak{X}_H})\right)
(d\tau - \sqrt{-1} dt). 
$$
Pick a positive number $\delta$ such that $\delta$ is less than the absolute values of the non-zero eigenvalues of  
the self adjoint differential operator 
$\zeta \mapsto A \zeta =  \nabla_{\frac{\partial}{\partial t}} (J\zeta )-  \nabla_{\zeta} (J{\mathfrak{X}_H})$ on $S^1$.  
We use $e^{\delta \tau}$ for defining weighted Sobolev spaces and regard the operator 
$P(\gamma;\mathfrak{t})$ as an operator 
$$P(\gamma;\mathfrak{t}): W^{1,p}_\delta (Z; E(\mathfrak{t})) \to L^p_{\delta} (Z;E(\mathfrak{t})\otimes\Lambda^{0,1}).$$ 
For two trivializations $\mathfrak{t}_1, \mathfrak{t}_2$, the operators $P(\gamma; \mathfrak{t}_1)$ and $P(\gamma;\mathfrak{t}_2)$ may have different 
Fredholm indices, in general.  However, the real determinant bundles as $O(1)$-bundles are canonically identified.  
These operators coincide on $[0,\infty) \times S^1$, where we do not use trivializations.  They may differ on $D^2$.  
Note that there exists a complex vector bundle $F$ on $\mathbb{C}P^1$ such that $E(\mathfrak{t}_2)$ is isomorphic to the gluing and smoothing of $E(\mathfrak{t}_1)$ and $F$ 
along the fibers on $0 \in D^2$ and $\infty \in \mathbb{C}P^1$.  
Since the Dolbeault operator $\overline{\partial}_F$ 
on $\mathbb{C}P^1$ with coefficients in $F$
is a complex Fredholm operator, its index is a virtual complex vector space 
and the coincidence condition for sections of $E({\mathfrak t}_1)$ and $F$ at $0 \in D^2$ and $\infty \in {\mathbb C}P^1$ 
is required in the complex vector space $E({\mathfrak t}_1)|_0 = F |_{\infty}$.
Hence its real determinant is canonically oriented (complex orientation).  
Thus the real determinant lines for the operator $P(\gamma; \mathfrak{t}_1)$ and $P(\gamma; \mathfrak{t}_2)$ are canonically isomorphic.  
The determinant line bundle for the family $P(\gamma; \mathfrak{t})$ on $\mathcal{P}(\overline{R_\alpha})$ descends to an $O(1)$-bundle on $\overline{R_\alpha}$, which we denote by 
$o_{R_\alpha}$.  (In the case of Floer theory for Lagrangian intersections,  see \cite[Proposition 8.8.1]{fooobook2}.)  
\begin{defn}\label{oricricial}
We call $o_{R_\alpha}$ on $R_{\alpha}$ 
the {\it orientation system of the critical submanifold 
$R_{\alpha}$} 
(see \cite[Condition 16.1 (VII) (i)]{fooonewbook}).  
\end{defn}

\begin{defn}\label{tildeMreg}
Let $\alpha_1,\alpha_2 \in\frak A$.
We denote by $\widetilde{\mathcal M}^{\text{\rm reg}}(X,H;\alpha_-,\alpha_+)$
the set of all maps $u : \R \times S^1 \to X$
with the following properties:
\begin{enumerate}
\item $u$ satisfies (\ref{Fleq}).
\item There exist $\tilde{\gamma}_{\pm}=({\gamma}_{\pm},w_{\pm}) \in R_{\alpha_{\pm}}$
such that
 (\ref{bdlatinf}) and
$$
w_- \# u \sim w_+
$$
are satisfied.
Here $\#$ is the obvious concatenation.
\end{enumerate}
The translation along $\tau \in \R$ defines an $\R$ action on 
$\widetilde{\mathcal M}^{\text{\rm reg}}(X,H;\alpha_-,\alpha_+)$.
This $\R$ action is free unless $\alpha_- = \alpha_+$.
We denote by ${\mathcal M}^{\text{\rm reg}}(X,H;\alpha_-,\alpha_+)$
the quotient space of this action.
For the case $\alpha_-=\alpha_+$, the set
${\mathcal M}^{\text{\rm reg}}(X,H;\alpha_-,\alpha_+)$ is the
empty set by definition.
\par
We define the evaluation map
\begin{equation}\label{evatinfqaaa}
{\rm ev} = ({\rm ev}_{-}, {\rm ev}_{+}): {\mathcal M}^{\text{\rm reg}}(X,H;\alpha_-,\alpha_+)
\to R_{\alpha_-} \times R_{\alpha_+}
\end{equation}
by
$$
{\rm ev}(u) = (({\gamma}_{-},w_{-}),({\gamma}_{+},w_{+})).
$$
\end{defn}
\begin{rem}
Note that $w_{-}$ can not be determined by the map $u$ only.
In fact if $[v] \in \pi_2(X)$ then $u$ may be also
regarded as an element of
${\mathcal M}^{\text{\rm reg}}(X,H;[v]\#\alpha_-,[v]\#\alpha_+)$.
More precisely, an element of
$\widetilde{\mathcal M}^{\text{\rm reg}}(X,H;\alpha_-,\alpha_+)$
should be regarded as a pair $(u,\alpha_-)$.
We write $u$ instead of $(u,\alpha_-)$ in the case no confusion can occur.
\end{rem}
The main result we will prove in Sections
\ref{subsec;compactFloer}-\ref{subsec;KuramodFloercor} is the following.
\begin{thm}\label{theorem266}
We assume Condition \ref{weaknondeg}.
\begin{enumerate}
\item
The space
${\mathcal M}^{\text{\rm reg}}(X,H;\alpha_-,\alpha_+)$ has a compactification
${\mathcal M}(X,H;\alpha_-,\alpha_+)$.
\item
The compact space ${\mathcal M}(X,H;\alpha_-,\alpha_+)$ has a Kuranishi structure with
corners.
The evaluation map ${\rm ev}$ is extended to it as a strongly smooth map
in the sense of \cite[Definition 3.40]{fooonewbook}.
\item
There exists a linear K-system $\mathcal F_X(H)$,
whose critical submanifold is $R_{\alpha}$ ($\alpha \in \frak A$) and
whose space of connecting orbits
is ${\mathcal M}(X,H;\alpha_-,\alpha_+)^{\boxplus 1}$.
Here ${\mathcal M}(X,H;\alpha_-,\alpha_+)^{\boxplus 1}$ is the 
outer collaring of ${\mathcal M}(X,H;\alpha_-,\alpha_+)$.
See Definition \ref{def:outcollar} for the notion of the outer collaring.
(We use $o_{R_{\alpha}}$ in Definition \ref{oricricial} as the orientation on $R_{\alpha}$.)
\item
The Kuranishi structure on ${\mathcal M}(X,H;\alpha_-,\alpha_+)^{\boxplus 1}$
in (3) coincides with one in (2) on
 ${\mathcal M}(X,H;\alpha_-,\alpha_+)
 \subset {\mathcal M}(X,H;\alpha_-,\alpha_+)^{\boxplus 1}$.
\end{enumerate}
\end{thm}
In later sections, we consider a slightly general 
case when we include additional interior marked points on the domain.
We will denote such a moduli space by
${\mathcal M}_{\ell}(X,H;\alpha_-,\alpha_+)$
where $\ell$ is a number of interior marked points.
See Section \ref{subsec;compactFloer} for the precise definition.

\begin{defn}
Theorems \ref{theorem266} and \cite[Theorem 16.39]{fooonewbook}
define a $\Lambda_{0,{\rm nov}}$-module.
We call the resulting $\Lambda_{0,{\rm nov}}$-module 
the {\it Floer cohomology of a periodic Hamiltonian
system} and write $HF(X,H;\Lambda_{0,{\rm nov}})$.
\end{defn}
See Theorem \ref{HFwelldefine} for the well-definedness of the
Floer cohomology group $HF(X,H;\Lambda_{0,{\rm nov}})$.
\par
The proof of Theorem \ref{theorem266}
occupies Sections \ref{subsec;compactFloer}-\ref{subsec;KuramodFloercor}.
In the rest of this section, we
discuss an easy part of the construction.
\begin{defn}\label{def:210}
We define group homomorphisms 
$E : \pi_2(X) \to \R$ and $\mu : \pi_2(X) \to \Z$
by
$E(\beta) = \omega[\beta]$, 
$\mu(\beta) = 2c_1(TX)[\beta]$.
We recall that 
the image of $(E,\mu) : \pi_2(X) \to \R \times \Z$
is isomorphic to the group $\frak G$ in Definition \ref{def:Per}.
We also denote by 
$E : \frak G \to \R$ and
$\mu : \frak G \to \Z$ the induced homomorphisms respectively.
\par
The map $E : \frak A \to \R$ is defined by
$E(\alpha) = \omega[w]$ where $(\gamma,w) \in R_{\alpha}$.
\end{defn}
\begin{rem}
It is easy to see that $E : \frak A \to \R$ is well-defined.
Namely $\omega[w]$ is independent of
the element $(\gamma,w)$ but depends only on $R_{\alpha}$.
In the case of Remark \ref{rem2622} (1)
this is trivial. In the case of Remark \ref{rem2622} (2)
this follows from the fact that $\frak G = \frak A$ in this case.
In the general case, it follows from the fact that
$R_{\alpha}$ is a critical submanifold
of Floer's functional $\mathcal A_H$ defined by
\begin{equation}\label{form2688ene}
\mathcal A_H(\gamma,w)
= \int_{D^2}w^*\omega + \int_{S^1} H(\gamma(t),t) dt.
\end{equation}
\end{rem}
To define $\mu : \frak A \to \Z$ we recall that the linearized operator
of the equation (\ref{Fleq}) is given by
\begin{equation}\label{form2688}
\aligned
D_u\overline{\partial} - 
\nabla_{\cdot}(J\frak X_H) \otimes
(d\tau -{\sqrt{-1}}dt)
~:~ 
& L^2_{m+1,\delta}(\R \times S^1 ;u^*TX) \\
\to
& L^2_{m,\delta}(\R \times S^1 ;u^*TX \otimes \Lambda^{0,1}).
\endaligned
\end{equation}
Here $V \mapsto \nabla_V (J\frak X_H)$ is the covariant derivative of 
$J\frak X_H$ with respect to the tangential vector $V$.
Here $L^2_{m,\delta}$ is the completion of the space of
smooth sections with respect to the weighted $L^2_m$ norm
with weight $e^{\delta \vert\tau\vert}$.
(Here $\delta >0$ and $\tau$ is the coordinate of $\R$.
We assume that $m$ is sufficiently large, say $>100$.)
\begin{defn}\label{defn2610}
We define the {\it virtual dimension} of
${\mathcal M}^{\text{\rm reg}}(X,H;\alpha_-,\alpha_+)$
by
$$
\dim
{\mathcal M}^{\text{\rm reg}}(X,H;\alpha_-,\alpha_+) = {\rm Index}(\ref{form2688}) +\dim R_{\alpha_-}
+ \dim R_{\alpha_+} - 1.
$$
Here $u$ is an element of ${\mathcal M}^{\text{\rm reg}}(X,H;\alpha_-,\alpha_+)$.
\end{defn}
Let $(\gamma,w) \in R_{\frak a}$.
Recall $w : D^2 \to X$.
We take a smooth map $\R \times S^1 \to D^2$
such that it is the projection to $\partial D^2 = S^1$ on $\R_{\ge 0} \times S^1$,
is constantly equal to $0$ on $\R_{\le -T} \times S^1$ and  defines a diffeomorphism from $(-T,0) \times S^1$ onto
${\rm Int}\,{D^2} \setminus \{0\}$.
We compose it with $w$ to obtain
$u_w : \R \times S^1 \to X$.
Let $\chi : \R \to [0,1]$ be a smooth function such that
$\chi(\tau) = 0$ for $\tau < -1$ and $\chi(\tau) = 1$ for $\tau > 1$.
We modify (\ref{form2688}) to define
\begin{equation}\label{form26882}
\aligned
D_u\overline{\partial} - \chi(\tau ) \nabla_{\cdot}(J\frak X_H)
\otimes
(d\tau -{\sqrt{-1}}dt)
~:~ 
& L^2_{m+1,\delta}(\R \times S^1;u_w^*TX)\\
\to &  L^2_{m,\delta}(\R \times S^1;u_w^*TX \otimes \Lambda^{0,1}).
\endaligned
\end{equation}
Let $(\gamma,w) \in R_{\frak a}$.
Note that $w : D^2 \to X$ induces a trivialization of $w^*TX \to D^2$, hence  
the one of $\gamma^*TX$, which is denoted by ${\mathfrak t}_w$.  
\begin{defn}
We define $\mu : \frak A \to \Z$ by
$$
\mu(\alpha)
= {\rm Index} P(\gamma, {\mathfrak t}_w).  
$$
\end{defn}
\begin{lem}\label{lem2612}
\begin{enumerate}
\item
For $\beta \in \pi_2(X)$ and $\alpha \in R_{\alpha}$.
We have
$$
\mu(\beta  \alpha) = \mu(\alpha)  + 2c_1(TX)[\beta].
$$
\item
In the case $H \equiv 0$, $\frak A = \frak G$ and
$\mu(\beta) = 2 c_1(TX)[\beta]$.
\item
We have
$$
\dim{\mathcal M}^{\text{\rm reg}}(X,H;\alpha_-,\alpha_+)
= \mu(\alpha_+) - \mu(\alpha_-) - 1 + \dim R_{\alpha_+}.
$$
\end{enumerate}
\end{lem}
\begin{proof}
(1) follows from the excision property of the index.  
\par
Let $(\gamma_{\pm},w_{\pm}) \in R_{\alpha_{\pm}}$.
To prove (3) it suffices to show that
\begin{equation}\label{indexsum}
{\rm Index}(\ref{form2688}) + \dim R_{\alpha_-}+
{\rm Index} P(\gamma_-, {\mathfrak t}_{w_-})
= {\rm Index} P(\gamma_+, {\mathfrak t}_{w^+}).
\end{equation}
We can prove (\ref{indexsum}) as follows.
We remark that our operator is of the form $\frac{d}{d\tau} + Q_{\tau}$
where $Q_{\tau}$ is a family of elliptic differential operators on $S^1$.
If we write $P(\gamma_-, {\mathfrak t}_{w^-})$ in this form, then
in the limit $\tau \to \infty$, the
multiplicity of zero eigenvalue of the operator $Q_{\tau}$ is exactly
equal to $\dim R_{\alpha_-}$.
If we write (\ref{form2688}) in this form, then in the limit
$\tau \to -\infty$,
the operator $Q_{\tau}$ has exactly
$\dim R_{\alpha_-}$ as the multiplicity of zero eigenvalue.
(\ref{indexsum}) is a consequence of this observation and
a well established result on spectral flow \cite[section 7]{APS-III}.
\par
We finally prove (2). In view of (1)(3), it suffices to show
the case $\beta = 0$.  Namely, $\gamma_p$ is the constant periodic orbit at $p$ for $H=0$ 
and the constant bounding disk $w_p: D^2 \to X$ induces the canonical trivialization 
${\mathfrak t}_{\rm tri}$, i.e., $(\gamma_p)^*TX \cong S^1 \times T_p X$.  
In that case we can see directly that
the index of $P(\gamma_p, {\mathfrak t}_{\rm tri})$ is $0$ 
by e.g., the Atiyah-Patodi-Singer index formula \cite{APS-I}.   
\end{proof}

\section{Stable map compactification of Floer's moduli space}
\label{subsec;compactFloer}
\subsection{Definition of ${\mathcal M}_{\ell}(X,H;\alpha_-,\alpha_+)$}
We define a compactification of the space 
${\mathcal M}^{\text{\rm reg}}(X,H;\alpha_-,\alpha_+)$
in this section.
This compactification is classical.
The description of this section is mostly the same as that of 
\cite[Section 30]{foootech}.
We repeat the argument for reader's convenience and
to fix the notation.
We need to stabilize the domain sometimes
in later constructions. For this purpose we treat 
the case when we include interior  marked points on the domain.
We will denote it by
${\mathcal M}_{\ell}(X,H;\alpha_-,\alpha_+)$
where $\ell$ is a number of interior marked points.
\begin{rem}
Including marked points also has applications to symplectic topology.
Especially it is used in the construction of the spectral invariant with bulk,
which is now known to contain more informations than
the version without bulk. (See \cite{fooospectr} \cite{usherspe}.)
\end{rem}
We consider $(\Sigma,(z_-,z_+,\vec z))$, a genus zero semistable curve $\Sigma$ 
with $2+\ell$ marked points
$z_-, z_+$ and $\vec z=(z_1,\dots ,z_{\ell})$.
Two marked points $z_-,z_+$ will play different roles.
\begin{defn}\label{defn41}
$\bullet$ Let $\Sigma_0$ be the union of the irreducible components of $\Sigma$ such that
\begin{enumerate}
\item $z_-,z_+ \in \Sigma_0$.
\item $\Sigma_0$  is connected.
\item $\Sigma_0$ is the smallest among those satisfying (1),(2) above.
\end{enumerate}
We call $\Sigma_0$ the {\it mainstream} of $(\Sigma,(z_-,z_+,\vec z))$, or simply, of $\Sigma$.
An irreducible component of the mainstream $\Sigma_0$ is called a {\it mainstream component}.
Other irreducible components of $\Sigma$ 
are called {\it bubble components}.
\par\noindent
$\bullet$ Let $\Sigma_a \subset \Sigma$ be a mainstream component.
If $z_- \notin \Sigma_a$, then there exists a unique singular point $z_{a,-}$ of $\Sigma$ contained in $\Sigma_a$
such that
\begin{enumerate}
\item
$z_-$ and $\Sigma_a \setminus\{z_{a,-}\}$ belong to the different connected component of
$\Sigma \setminus\{z_{a,-}\}$.
\item
$z_+$ and $\Sigma_a \setminus\{z_{a,-}\}$ belong to the same connected component of
$\Sigma \setminus\{z_{a,-}\}$.
\end{enumerate}
In case $z_- \in \Sigma_a$ we set $z_- = z_{a,-}$.
We define $z_{a,+}$ in the same way.
\par
We call $z_{a,\pm}$ the {\it transit points}.
We call other singular points  {\it non-transit singular points}.
\par\noindent
$\bullet$ A {\it parametrization of the mainstream} of $(\Sigma,(z_-,z_+,\vec z))$ is
$\varphi = \{\varphi_a\}$, where $\varphi_a : \R \times S^1 \to \Sigma_a$
for each irreducible component $\Sigma_a$ of the mainstream such that:
\begin{enumerate}
\item
$\varphi_a$ is a biholomorphic map $\varphi_a : \R \times S^1 \cong \Sigma_a \setminus \{z_{a,-},z_{a,+}\}$.
\item
$\lim_{\tau \to \pm \infty}\varphi_a(\tau,t) = z_{a,\pm}$.
\par
\end{enumerate}
$\bullet$ 
A union of one mainstream component and all the trees of the
bubble components rooted on it is called an {\it extended mainstream
component}.
We sometimes denote by $\widehat{\Sigma}_a$ the
extended main stream component containing
the mainstream component ${\Sigma}_a$.
Here and hereafter we call each of
$\widehat{\Sigma}_a \setminus\{z_{a,-},z_{a,+}\}$
an {\it interior of extended mainstream component}.
See Figure \ref{Figpage9}.
\begin{figure}[htbp]
\centering
\includegraphics[scale=0.5]{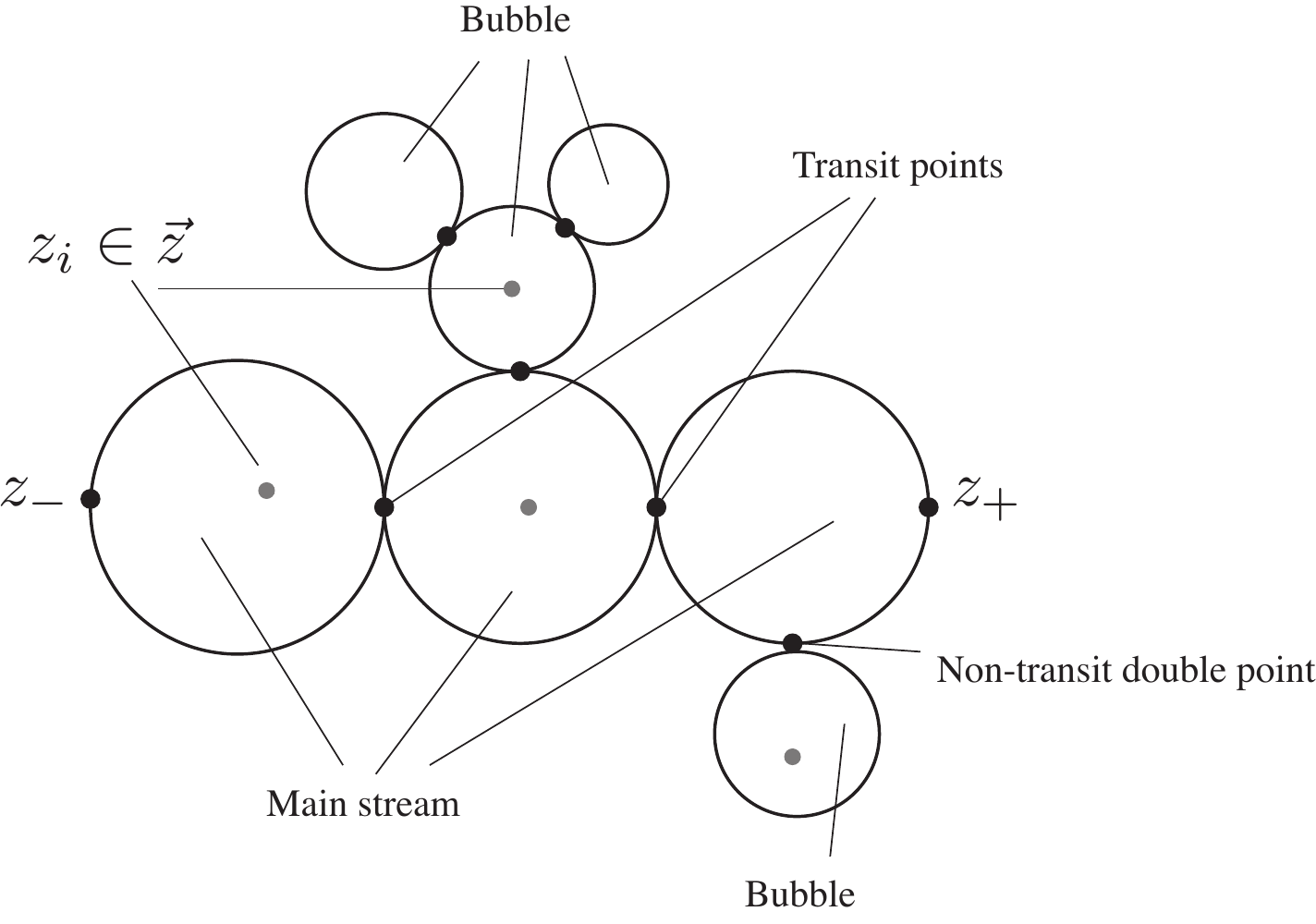}
\caption{$(\Sigma,(z_-,z_+,\vec z))$}
\label{Figpage9}
\end{figure}
\end{defn}
Let $\frak A$ be as in (\ref{267form})
and $\alpha_{\pm} \in \frak A$.
\begin{defn}\label{defn210}
The set $\widehat{\mathcal M}_{\ell}(X,H;\alpha_-,\alpha_+)$
consists of triples $((\Sigma,(z_-,z_+,\vec z)),u,\varphi)$ satisfying
the following conditions: Here $\ell =\# \vec z$.
\begin{enumerate}
\item
$(\Sigma,(z_-,z_+,\vec z))$ is a genus zero semistable curve with $\ell + 2$ marked points.
\item
$\varphi$ is a parametrization of the mainstream.
\item
For each extended main stream component $\widehat{\Sigma}_a$, the map
$u$ induces
$u_a : \widehat{\Sigma}_a \setminus\{z_{a,-},z_{a,+}\} \to X$
which is a continuous map.\footnote{In other words
$u$ is a continuous map from the complement of the set of the transit points.}
\item
If $\Sigma_a$ is a mainstream component and
$\varphi_a : \R \times S^1 \to \Sigma_a$ is as above, then
the composition $u_a \circ \varphi_a$ satisfies the equation (\ref{Fleq}).
\item
$$
\int_{\R \times S^1}
\left\Vert\frac{\partial (u \circ \varphi_a)}{\partial \tau}\right\Vert^2 d\tau dt < \infty.
\footnote{Condition (5) follows from the rest of the conditions in Definition 
\ref{defn210}.}
$$
\item
If $\Sigma_{b}$  is a bubble component,
then $u$ is pseudo-holomorphic on it.
\item
If
$\widehat{\Sigma}_{a_1}$ and $\widehat{\Sigma}_{a_2}$ are
extended mainstream components and
if  $z_{a_1,+} = z_{a_2,-}$, then
$$
\lim_{\tau\to+\infty} (u_{a_1} \circ \varphi_{a_1})(\tau,t)
=
\lim_{\tau\to-\infty} (u_{a_2} \circ \varphi_{a_2})(\tau,t)
$$
holds for each $t \in S^1$.
((5) and Proposition \ref{prof263} imply that
both of the left and right hand sides converge.)
\item
If
$\widehat{\Sigma}_{a}$, $\widehat{\Sigma}_{a'}$
are
extended mainstream components and $z_{a,-} = z_-$, $z_{a',+}
= z_+$, then there exist $(\gamma_{\pm},w_{\pm})
\in R_{\alpha_{\pm}}$ such that
$$
\aligned
\lim_{\tau\to-\infty} (u_{a} \circ \varphi_{a})(\tau,t)
&= \gamma_-(t), \\
\lim_{\tau\to +\infty} (u_{a'} \circ \varphi_{a'})(\tau,t)
&= \gamma_+(t).
\endaligned
$$
Moreover
$$
[u_*[\Sigma]] \# w_- = w_+
$$
where $\#$ is the obvious concatenation.
\item
We assume $((\Sigma,(z_-,z_+,\vec z)),u,\varphi)$
is stable in the sense of Definition \ref{stable26} below.
\end{enumerate}
\end{defn}

To define stability we first define the group of automorphisms.
\begin{defn}\label{defn2615}
We assume that $((\Sigma,(z_-,z_+,\vec z)),u,\varphi)$
satisfies (1)-(8) above.
The {\it extended automorphism group}
${\rm Aut}^+((\Sigma,(z_-,z_+,\vec z)),u,\varphi)$
of $((\Sigma,(z_-,z_+,\vec z)),u,\varphi)$
consists of maps $v : \Sigma \to \Sigma$
with the following properties:
\begin{enumerate}
\item
$v(z_{-}) = z_{-}$ and $v(z_{+}) = z_{+}$.
In particular, $v$ preserves each of the mainstream component $\Sigma_a$
of $\Sigma$.
Moreover $v$ fixes each of the transit points.
\item
$u = u\circ v$ holds outside the set of the transit points.
\item
If $\Sigma_a$ is a mainstream component of $\Sigma$, then
there exists $\tau_a \in \R$ such that
\begin{equation}\label{preserveparame2}
(v \circ \varphi_a)(\tau,t) = \varphi_a(\tau+\tau_a,t).
\end{equation}
on $\R \times S^1$.
\item
There exists $\sigma \in {\rm Perm}(\ell)$ such that
$v(z_i) = z_{\sigma(i)}$. Here ${\rm Perm}(\ell)$ denotes the group of permutations of $\ell$ elements.
\end{enumerate}
\par
The {\it automorphism group}
${\rm Aut}((\Sigma,(z_-,z_+,\vec z)),u,\varphi)$
of $((\Sigma,(z_-,z_+,\vec z)),u,\varphi)$
consists of the elements of ${\rm Aut}^+((\Sigma,(z_-,z_+,\vec z)),u,\varphi)$
such that $\sigma$ in Item (4) is the identity.
\end{defn}
\begin{defn}\label{stable26}
We say $((\Sigma,(z_-,z_+,\vec z)),u,\varphi)$ is {\it stable} 
if ${\rm Aut}((\Sigma,(z_-,z_+,\vec z)),u,\varphi)$ is a finite group.
(This is equivalent to the finiteness of
${\rm Aut}^+((\Sigma,(z_-,z_+,\vec z)),u,\varphi)$.)
\end{defn}
\begin{lem}
An element $((\Sigma,(z_-,z_+,\vec z)),u,\varphi)$
satisfying (1)-(8) of Definition \ref{defn210} is stable
if and only if,
for each irreducible component $\Sigma_i$ of $\Sigma$,
one of the following holds.
\begin{enumerate}
\item
$\Sigma_i$ is a bubble component, and $u$ is nonconstant on $\Sigma_i$.
\item
$\Sigma_i$ is a bubble or a mainstream component, and 
$$
\# ((\vec z \cup\{ z_{-},z_{+}\}) \cap \Sigma_i) + \# (\text{Singular points on $\Sigma_i$})
\ge 3.
$$
\item
$\Sigma_i$ is a mainstream component $\Sigma_a$, and 
$$
\frac{d}{d\tau} (u\circ \varphi_a)
$$
is nonzero at some point. Here $\tau$ is the coordinate of the
$\R$ factor in $\R \times S^1$.
\end{enumerate}
\end{lem}
The proof is left to the reader.

\begin{defn}\label{3equivrel}
On the set $\widehat{\mathcal M}_{\ell}(X,H;\alpha_-,\alpha_+)$  we define two
equivalence relations $\sim_1$, $\sim_2$ as follows.
\par
$((\Sigma,(z_-,z_+,\vec z)),u,\varphi) \sim_1 ((\Sigma',(z'_-,z'_+,\vec z')),u',\varphi')$ if and only if there exists
a biholomorphic map $v : \Sigma \to \Sigma'$ with the following properties:
\begin{enumerate}
\item
$v(z_{-}) = z'_{-}$ and $v(z_{+}) = z'_{+}$.
In particular $v$ sends the mainstream of $\Sigma$ to the mainstream of $\Sigma'$.
Moreover $v$ sends transit points to transit points.
\item
$u' = u\circ v$ holds outside the transit points.
\item
If $\Sigma_a$ is a mainstream component of $\Sigma$
and $v(\Sigma_a) = \Sigma'_{a'}$, then we have
\begin{equation}\label{preserveparame}
v \circ \varphi_a = \varphi'_{a'}
\end{equation}
on $\Sigma_a \setminus\{z_{a,-},z_{a,+}\}$.
\item
$v(z_i) = z'_i$.
\end{enumerate}
\par
The equivalence relation $\sim_2$ is defined  replacing (\ref{preserveparame}) by
existence of $\tau_a \in \R$ such that
\begin{equation}\label{preserveparame2}
(v \circ \varphi_a)(\tau,t) = \varphi'_{a'}(\tau+\tau_a,t).
\end{equation}
We put
$$\aligned
\widetilde{\mathcal M}_{\ell}(X,H;\alpha_-,\alpha_+) &=
\widehat{\mathcal M}_{\ell}(X,H;\alpha_-,\alpha_+)/\sim_1, \\
{\mathcal M}_{\ell}(X,H;\alpha_-,\alpha_+) &=
\widehat{\mathcal M}_{\ell}(X,H;\alpha_-,\alpha_+)/\sim_2.
\endaligned$$
In case $\ell = 0$ we write
${\mathcal M}(X,H;\alpha_-,\alpha_+)$ etc.
\end{defn}
\begin{defn}\label{not2620}
For the case $X = \text{one point}$ and $H\equiv 0$, 
we obtain the space ${\mathcal M}_{\ell}(\text{one point},0;\alpha_0,\alpha_0)$.
Here $\alpha_0$ is the unique point in ${\rm Per}(0)$.
We denote this space by ${\mathcal M}_{\ell}(\text{source})$.
\end{defn}
\begin{rem}
The parametrization $\varphi_a
: \R \times S^1 \to \Sigma_a \setminus \{z_{a,-},z_{a,+}\}$ of each
mainstream component $\Sigma_a$ is automatically determined
by the marked  curve $(\Sigma,z_-,z_+,\vec z)$ up to the
ambiguity $\varphi_a(\tau,t) \mapsto \varphi_a(\tau+\tau_0,t+t_0)$
where $(\tau_0,t_0) \in \R \times S^1$.
In other words, the choice of parametrization of the mainstream 
one to one corresponds
to the choice of one additional marked point on each 
mainstream component. We can use this fact to define the structure of
an orbifold with corner on  ${\mathcal M}_{\ell}(\text{source})$.
\end{rem}
\begin{exm} 
\begin{enumerate}
\item
${\mathcal M}_{0}(\text{source}) = \emptyset$.
${\mathcal M}_{1}(\text{source}) \cong S^1$.
In fact, its point is determined by the $S^1$ factor of the marked point.
\item
${\mathcal M}_{2}(\text{source}) \cong S^1 \times S^1 \times [0,1]$.
To see this let us first consider the case when there is only one mainstream component.
In that case let $\varphi(\tau_i,t_i)$ $(i=1,2)$ be the marked points.
Because of translation symmetry $t_1,t_2$ and $\tau_2 - \tau_1$ parametrize
such elements of ${\mathcal M}_{2}(\text{source})$.
(Note in case these two marked points coincide there occurs bubble.)
The boundary $S^1 \times S^1 \times \partial[0,1]$ then corresponds
to the situation where there are two mainstream components.
\end{enumerate}
\end{exm}
\begin{defn}
An element $[(\Sigma,(z_-,z_+,\vec z)),u,\varphi]$ of
${\mathcal M}_{\ell}(X,H;\alpha_-,\alpha_+)$
is said to be {\it source stable} if
$(\Sigma,(z_-,z_+,\vec z))$ is stable as genus zero marked curve.
\end{defn}
\subsection{Topology on ${\mathcal M}_{\ell}(X,H;\alpha_-,\alpha_+)$}
We next define a topology on the moduli space 
${\mathcal M}_{\ell}(X,H;\alpha_-,\alpha_+)$.
It is mostly the same as the topology of the moduli space
of stable maps which was introduced in \cite[Definitions 10.2 and 10.3]{FO}.
However since the notion of equivalence relation
$\sim_2$ is slightly
different (namely the condition (\ref{preserveparame2})
is included) it is slightly different.
\par
We here recall the following well-known fact.
Suppose that a sequence of genus $0$ stable curves
$(\Sigma^j,(z^j_-,z^j_+,\vec z^j))$
converges to $(\Sigma,(z_-,z_+,\vec z))$ in
the moduli space of stable curves.
Then for a sufficiently small $\epsilon >0$ and a
large $j$ we have a biholomorphic embedding
\begin{equation}\label{form2613}
\psi_{j,\epsilon}
: \Sigma \setminus \text{($\epsilon$ neighborhood of singular points)}
\to \Sigma^j ,
\end{equation}
$\vec z^{j,\epsilon} \subset \Sigma$ and $R(\epsilon)>0$
such that the following holds.
\begin{conds}\label{conds2617}
\begin{enumerate}
\item
$\lim_{j\to \infty, \epsilon \to 0} z^{j,\epsilon}_i = z_i$ for each $i
\in \{1,\dots,\ell,+,-\}$.
\item
$\lim_{\epsilon \to 0} R(\epsilon) = \infty$.
\item
$\psi_{j,\epsilon}(z^{j,\epsilon}_i) = z^{j}_i$.
\item
Any connected component of the complement of the image of
$\psi_{j,\epsilon}$ is biholomorphic to one of 
$[0,R] \times S^1$ with $R > R(\epsilon)$,
$(-\infty, 0] \times S^1$, or $[0,\infty) \times S^1$.
\end{enumerate}
\end{conds}
We can use Margulis' lemma to prove it.
(See \cite[Chapter 4]{B}, \cite[Chapter 11]{Ra}, \cite[Chapter IV]{H}, for example.)
In fact, we can use this condition as the definition of the topology
of the moduli space of marked stable curves of genus $0$.
\begin{defn}\label{stableconvergence}
Suppose that elements $[(\Sigma,(z_-,z_+,\vec z)),u,\varphi]$
and $[(\Sigma^j,(z^j_-,z^j_+,\vec z^j)),u^j,\varphi^j]$
of ${\mathcal M}_{\ell}(X,H;\alpha_-,\alpha_+)$
are source stable. We say
$[(\Sigma^j,(z^j_-,z^j_+,\vec z^j)),u^j,\varphi^j]$
{\it converges} to $[(\Sigma,(z_-,z_+,\vec z)),u,\varphi]$
and write
$$
\underset{j\to \infty}{\rm lims}
[(\Sigma^j,(z^j_-,z^j_+,\vec z^j)),u^j,\varphi^j]
=
[(\Sigma,(z_-,z_+,\vec z)),u,\varphi]
$$
if there exist $\psi_{j,\epsilon}$,
$\vec z^{j,\epsilon} \subset \Sigma$ and $R(\epsilon)>0$ as in
(\ref{form2613}) with the following properties.
\begin{enumerate}
\item
Condition \ref{conds2617} (1)-(4) are satisfied.
\item 
For each $\epsilon >0$
\begin{equation}
\lim_{j\to \infty} \sup \{d(u^j(\psi_{j,\epsilon}(z)),u(z)) \mid z \in
{\rm Dom}(\psi_{j,\epsilon})\} = 0.
\end{equation}
Here
${\rm Dom}(\psi_{j,\epsilon})
= \Sigma \setminus \text{(Union of $\epsilon$ neighborhoods of singular points)}$
is the domain of $\psi_{j,\epsilon}$ as in
(\ref{form2613}) and $d$ is a metric 
on $X$.\footnote{This condition is independent of $d$ since $X$ is compact.}
\item
For each mainstream component $\Sigma_a$ of $\Sigma$ there
exist a mainstream component $\Sigma^j_{a(j)}$ of $\Sigma^j$
and $T_j \to \infty$, $\tau_j \in \R$ such that
\begin{equation}
\lim_{j\to \infty} \sup
\{
d((\varphi_{a_j}^{-1}\circ \psi_{j,\epsilon}
\circ \varphi_{a}) (\tau-\tau_j,t), (\tau,t))
\mid (\tau-\tau_j,t)  \in [-T_j,T_j] \times S^1
\}
= 0.
\end{equation}
Here $d$ is the standard metric on $\R \times S^1$.
\item
The `diameter' of 
the image of $u^j$ of each connected component
$\frak W$ of
$\Sigma^j \setminus \psi_{j,\epsilon}({\rm Dom}(\psi_{j,\epsilon}))$
converges to $0$ in the sense of Condition \ref{conds2621} below.
\end{enumerate}
\end{defn}
To state Condition \ref{conds2621} below we need some notation.
Let 
$$
((\Sigma,(z_-,z_+,\vec z)),u,\varphi)
\in \widehat{\mathcal M}_{\ell}(X,H;\alpha_-,\alpha_+)
$$
and define
$$
\hat u : \Sigma \to X
$$ 
as follows.
We use the flow map
${\rm exp}_t^{H} : X \to X$ defined in (\ref{fporm2622}).
We identify $t \in [0,1) \subset S^1 \cong \R /\Z$.
\begin{enumerate}
\item
If $z \in \Sigma$ is in the mainstream component $\Sigma_a$ and is not
a transit point, then we take $(\tau,t)$ with $\varphi_a(\tau,t) = z$ and put 
$$
\hat u(z) = ({\rm exp}_t^{H})^{-1}(u(z)).
$$
\item
Suppose $z \in \Sigma$ is in a bubble component.
Let $z_0$ be the root of the bubble tree containing $z$.
We take $(\tau,t)$ with $\varphi_a(\tau,t) = z_0$.
Then we put 
$$
\hat u(z) = ({\rm exp}_t^{H})^{-1}(u(z)).
$$
\item
(1) and (2) define $\hat u$ outside the set of the singular points.
It is easy to see that it extends to a map $\hat u : \Sigma \to X$.
\end{enumerate}

\begin{rem}\label{rem:314}
\begin{enumerate}
\item
The map $u$ does {\it not} extend continuously to the
transit point. In fact, the image of the neighborhood of the
transit point contains a periodic orbit which may not consist of a point.
After composing $({\rm exp}_t^{H})^{-1}$ it can be extended to
a continuous map because of (\ref{bdlatinf}).
\item
The map $\hat u$ is {\it not} continuous at $t = [0] = [1] \in S^1$.
This is because $\varphi^H_1 \ne \varphi^H_0 = {\rm identity}$. However it is
continuous at the transit point.
\end{enumerate}
\end{rem}

\begin{defn}\label{redefconnecting}
We say $\hat u$ the {\it redefined connecting orbit map}.
\end{defn}

\begin{conds}\label{conds2621}
In the situation of Definition \ref{stableconvergence}
we require the following.
There exists $o(\epsilon,j)>0$ with $\lim_{j\to \infty, \epsilon \to 0} o(\epsilon,j)
\to 0$ such that
for each connected component $\frak W$ of
$\Sigma^j \setminus \psi_{j,\epsilon}({\rm Dom}(\psi_{j,\epsilon}))$
we have either
\begin{enumerate}
\item[(i)] ${\rm Diam} (\hat u^j(\frak W))  <  o(\epsilon,j)$ for the case when $\frak W$ corresponds to a
transit point of $(\Sigma,(z_-,z_+,\vec z)),u,\varphi)$, 
or 
\item[(ii)] ${\rm Diam} (u^j(\frak W))  <  o(\epsilon,j)$ for the case when
$\frak W$ corresponds to a
non-transit singular point of $(\Sigma,(z_-,z_+,\vec z)),u,\varphi)$.
\end{enumerate}
Here $\hat u^j$ is the redefined connecting orbit map of
$((\Sigma^j,(z^j_-,z^j_+,\vec z^j)),u^j,\varphi^j)$.
\end{conds}
\begin{defn}\label{def2626}
We say a sequence $[(\Sigma^j,(z^j_-,z^j_+,\vec z^j)),u^j,\varphi^j]$
in ${\mathcal M}_{\ell}(X,H;\alpha_-,\alpha_+)$
 {\it converges} to
$[(\Sigma,(z_-,z_+,\vec z)),u,\varphi]
\in {\mathcal M}_{\ell}(X,H;\alpha_-,\alpha_+)$
and write
$$
\lim_{j\to \infty}
[(\Sigma^j,(z^j_-,z^j_+,\vec z^j)),u^j,\varphi^j]
=
[(\Sigma,(z_-,z_+,\vec z)),u,\varphi],
$$
if there exist $\vec w^j \subset \Sigma^j$
and $\vec w \subset  \Sigma$ such that
$[(\Sigma^j,(z^j_-,z^j_+,\vec z^j\cup \vec w^j)),u^j,\varphi^j]$
and
$[(\Sigma,(z_-,z_+,\vec z \cup \vec w)),u,\varphi]$
are source stable and
$$
\underset{j\to \infty}{\rm lims}\,\,
[(\Sigma^j,(z^j_-,z^j_+,\vec z^j\cup \vec w^j)),u^j,\varphi^j]
=
[(\Sigma,(z_-,z_+,\vec z \cup \vec w)),u,\varphi]
$$
in the sense of Definition \ref{stableconvergence}. 
\end{defn}
\begin{rem}
We define the topology by defining the notion of convergence of sequences.
We can do so by the following reason.
We can define the notion of combinatorial or topological type of an
element of ${\mathcal M}_{\ell}(X,H;\alpha_-,\alpha_+)$.
(See \cite[Section 19]{FO} for example.)
If we fix a combinatorial type, we can embed
the set of elements ${\mathcal M}_{\ell}(X,H;\alpha_-,\alpha_+)$
to a space obtained as a pair of stable curve and a map from it
with the fixed combinatorial type.
The topology in Definition \ref{def2626} coincides with one which is defined
by the standard topology of the moduli space of marked curves and,
say, the $C^{\infty}$ topology of the maps. Obviously, 
the latter topology is metrizable. Also there exist only countably many
combinatorial types.
\end{rem}
\begin{thm} The space 
${\mathcal M}_{\ell}(X,H;\alpha_-,\alpha_+)$ is compact and
Hausdorff with respect to the topology of Definition \ref{def2626}.
\end{thm}
The proof is the same as that of
\cite[Lemma 10.4 and Theorem 11.1]{FO}, which relies on Proposition \ref{prof263}, and so is omitted.
\par
We close this section with a few small remarks.
\par
The evaluation map (\ref{evatinfqaaa}) is extended to
\begin{equation}\label{evatinfqbbb}
 ({\rm ev}_{-}, {\rm ev}_{+}): {\mathcal M}_{\ell}(X,H;\alpha_-,\alpha_+)
\to R_{\alpha_-} \times R_{\alpha_+}.
\end{equation}
Moreover we define evaluation maps 
\begin{equation}\label{evatinfqccc}
 ({\rm ev}_{1}, \dots,{\rm ev}_{\ell}): {\mathcal M}_{\ell}(X,H;\alpha_-,\alpha_+)
\to X^{\ell}
\end{equation}
by $$
{\rm ev}_{i}([(\Sigma,(z_-,z_+,\vec z \cup \vec w)),u,\varphi])
=
u(z_i).
$$
They are continuous.
\par
Let $\alpha_-=\alpha_0,\alpha_1,\dots,\alpha_{m-1},\alpha_m = \alpha_+
\in \frak A$.
We consider the fiber product
\begin{equation}\label{2622}
\aligned
&{\mathcal M}_{\ell_1}(X,H;\alpha_0,\alpha_1)
\,{}_{{\rm ev}_{+}}\times_{{\rm ev}_-} {\mathcal M}_{\ell_2}(X,H;\alpha_1,\alpha_2)
\,{}_{{\rm ev}_{+}}\times_{{\rm ev}_-} \dots _{{\rm ev}_{+}}\times_{{\rm ev}_-}\\
& {\mathcal M}_{\ell_{i+1}}(X,H;\alpha_i,\alpha_{i+1})
\,{}_{{\rm ev}_{+}}\times_{{\rm ev}_-}
\dots
{}_{{\rm ev}_+}\times_{{\rm ev}_-} {\mathcal M}_{\ell_m}(X,H;\alpha_{m-1},\alpha_{m})
\endaligned
\end{equation}
and denote it by
$$
{\mathcal M}_{(\ell_1,\dots,\ell_m)}(X,H;\alpha_0,\alpha_1.,\dots,\alpha_m).
$$
See Figure \ref{Fig1}.
In case $\ell_i$ are all $0$ we write
$
{\mathcal M}(X,H;\alpha_0,\alpha_1.,\dots,\alpha_m)
$. 
\begin{figure}[htbp]
\centering
\includegraphics[scale=0.5]{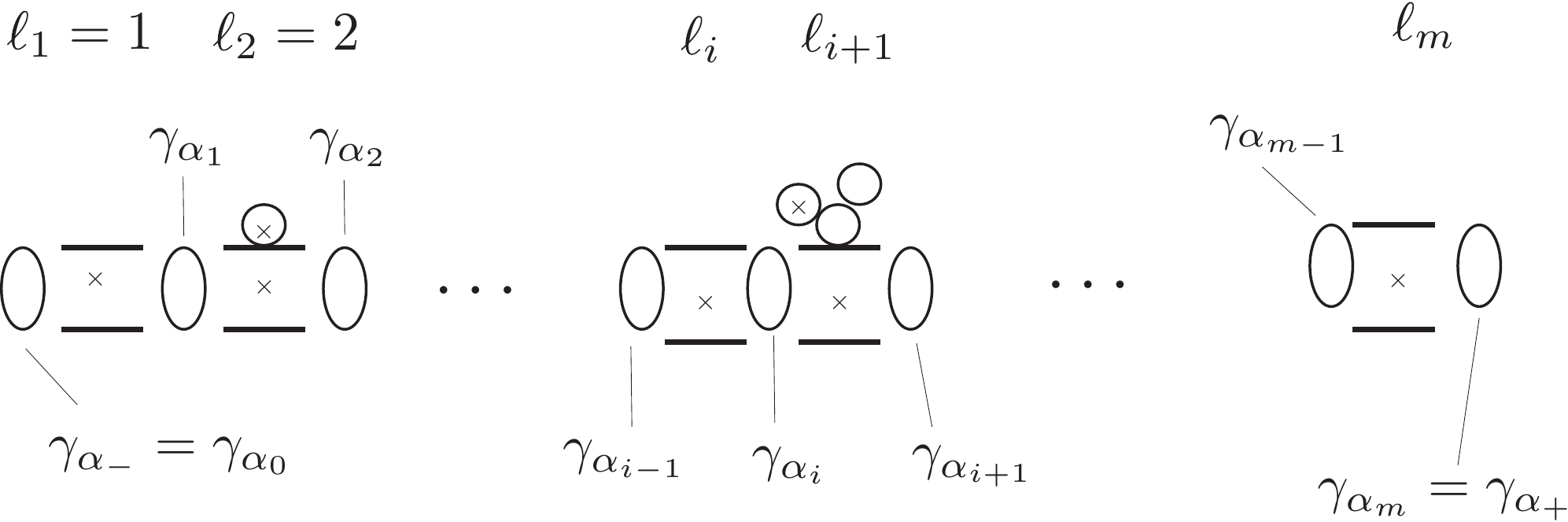}
\caption{$
{\mathcal M}_{(\ell_1,\dots,\ell_m)}(X,H;\alpha_0,\alpha_1.,\dots,\alpha_m)
$}
\label{Fig1}
\end{figure}
\par
We also define a map
\begin{equation}\label{mapform2621}
{\mathcal M}_{(\ell_1,\dots,\ell_m)}(X,H;\alpha_0,\alpha_1.,\dots,\alpha_m)
\to
{\mathcal M}_{\ell_1+\dots+\ell_m}(X,H;\alpha_-,\alpha_+)
\end{equation}
as follows.
Let
$[(\Sigma^a,(z^a_-,z^a_+,\vec z^a)),u^a,\varphi^a]
\in {\mathcal M}_{\ell_a}(X,H;\alpha_{a-1},\alpha_a)$
for $a=1,\dots,m$.
We assume
\begin{equation}\label{form2622}
\aligned
&{\rm ev}_{+}([(\Sigma^a,(z^a_-,z^a_+,\vec z^a)),u^a,\varphi^a])\\
&=
{\rm ev}_{-}([(\Sigma^{a+1},(z^{a+1}_-,z^{a+1}_+,\vec z^{a+1})),u^{a+1},\varphi^{a+1}]).
\endaligned
\end{equation}
On the disjoint union $\bigsqcup_{a=1}^m \Sigma_a$
we identify $z^a_+$ and $z^{a+1}_-$ to obtain $\Sigma$.
We put $z_- = z_-^0$ and $z_+ = z^m_+$.
We also put $\vec z = \bigcup_{a=1}^m \vec z^a$, and 
$u$ and $\varphi$ to be the union of $u^a$ and $\varphi^a$ respectively.
They obviously satisfy Definition \ref{defn210}
(1)-(6). Definition \ref{defn210} (7) is a consequence of
(\ref{form2622}).
We thus obtain the map (\ref{mapform2621}).
This map is obviously continuous, injective and is a homeomorphism
onto its image.

\section{Construction of Kuranishi structure}
\label{subsec;KuraFloer}

\subsection{Statement}
In this section we present technical details of the proof of the next theorem.
\begin{thm}\label{kuraexists}
\begin{enumerate}
\item
The space ${\mathcal M}_{\ell}(X,H;\alpha_{-},\alpha_{+})$ has a Kuranishi structure with corners
together with an isomorphism 
$$
{\rm ev}_{+}^{\ast} (\det TR_{\alpha_+}) \otimes {\rm ev}_{+}^{\ast}(o_{R_{\alpha_+}})
\cong
o_{{\mathcal M}_{\ell}(X,H;\alpha_{-},\alpha_{+})} \otimes {\rm ev}_{-}^{\ast}(o_{R_{\alpha_-}})
$$
of principal $O(1)$-bundles. Here 
$o_{{\mathcal M}_{\ell}(X,H;\alpha_{-},\alpha_{+})}$ is the orientation bundle defined by the 
Kuranishi structure as in \cite[Definition 3.10 (1)]{fooonewbook} and 
$o_{R_{\alpha_{\pm}}}$ are defined in Definition \ref{oricricial}.
\item
Its codimension $k$ normalized corner\footnote{See \cite[Definition 24.18]{fooonewbook} for the definition of normalized corner.} is the union of the images of the map
(\ref{mapform2621}) with a certain Kuranishi structure on
${\mathcal M}_{(\ell_1,\dots,\ell_m)}(X,H;\alpha_0,\alpha_1.,\dots,\alpha_k)$.
\item
The evaluation maps (\ref{evatinfqbbb}) and (\ref{evatinfqccc})
are induced from 
a strongly smooth map\footnote{See \cite[Definition 3.40]{fooonewbook}
for the definition of strongly smooth map.}
of Kuranishi structures.
The map ${\rm ev}_+$ is weakly submersive.
\item
The dimension (as a K-space) is given by
$$
\dim {\mathcal M}_{\ell}(X,H;\alpha_{-},\alpha_{+})
= \mu(\alpha_+) - \mu(\alpha_-) - 1 + \dim R_{\alpha_+} + 2\ell.
$$
\end{enumerate}
\end{thm}
\begin{rem}\label{rem628}
The Kuranishi structure on
${\mathcal M}_{(\ell_1,\dots,\ell_m)}(X,H;\alpha_0,\alpha_1.,\dots,\alpha_m)$
mentioned in Theorem \ref{kuraexists} (2) is
not in general the fiber product Kuranishi structure.
The method we give in this section does not provide
such a system of Kuranishi structures compatible with the fiber product yet.
In Section \ref{subsec;KuramodFloercor}
we will modify the Kuranishi structures on 
${\mathcal M}_{\ell_i}(X,H;\alpha_{i-1},\alpha_{i})$
at their outer collars so that 
they are consistent with the fiber product. 
Then we will obtain a required K-system so that
${\mathcal M}_{(\ell_1,\dots,\ell_m)}(X,H;\alpha_0,\alpha_1.,\dots,\alpha_m)$ has indeed 
a fiber product Kuranishi structure.
We will carry out this procedure in Section \ref{subsec;KuramodFloercor}.
\end{rem}
\begin{rem}\label{rem43}
The group $\frak G$ in Definition \ref{def:Per} acts on
$\widetilde{\rm Per}(H)$ 
so that it induces a simply transitive action on the index set 
$\frak A$ in \eqref{267form}.
Then we have a natural identification between 
${\mathcal M}_{\ell}(X,H;\alpha_{-},\alpha_{+})$
and ${\mathcal M}_{\ell}(X,H;g(\alpha_{-}),g(\alpha_{+}))$
for any $g\in \frak G$ and $\alpha_{\pm} \in \frak A$. 
Since our construction of Kuranishi structures given below 
is independent of the choice of the bounding disk $w : D^2 \to X$ 
in the definition of $\widetilde{\rm Per}(H)$, 
the resulting Kuranishi structures for other choices of $w$'s
are isomorphic under this identification.
\end{rem}
\subsection{Obstruction bundle data}
The detail of the proof of Theorem \ref{kuraexists} (1)(2)
given in this section is mostly the same as the one given in 
\cite{fooo:const1} (see also \cite[Parts 4 and 5]{foootech} if necessary).
We repeat the proof for the completeness
and also to prepare notations for the discussion
in the next section.
More specifically, in \cite[Definition 5.1]{fooo:const1} we introduced the notion of {\it obstruction bundle data} 
for the moduli space of bordered stable maps of genus $0$ and 
showed the existence \cite[Theorem 11.1]{fooo:const1} and 
that we can associate a Kuranishi structure for any obstruction bundle data
in a canonical way \cite[Theorem 7.1]{fooo:const1}. 
The strategy of this article is the same.
So we first discuss a version of the obstruction bundle data
in our situation with Hamiltonian perturbation.
\par
\begin{defn}\label{symstab}
A {\it symmetric stabilization} of $((\Sigma,(z_-,z_+,\vec z)),u,\varphi)$ is 
an ordered set $\vec w$ of points in $\Sigma$ with the following properties.
\begin{enumerate}
\item
The set $\vec w$ is contained in the union of bubble components.
\item
$\vec w \cap \vec z = \emptyset$.
None of the points in $\vec w$ is a singular point.
\item
For each  bubble component $\Sigma_b$ we have the
inequality
$$
\# (\vec w \cap \vec z) + \# (\text{Singular point on $\Sigma_b$})
\ge 3.
$$
\item
For any $v \in {\rm Aut}^+((\Sigma,(z_-,z_+,\vec z)),u,\varphi)$
we have
$\sigma \in {\rm Perm}(\# \vec w)$
such that $v(w_i) = w_{\sigma(i)}$.
\item If $v \in {\rm Aut}((\Sigma,(z_-,z_+,\vec z)),u,\varphi)$ and 
$\sigma =$ id in (4), then $v = $ id also.
\item
For each $w_i$ the map $u$ is an immersion at $w_i$.
\end{enumerate}
\end{defn}
\begin{rem}
Condition (5) is assumed only to simplify the notations in later discussion.
By this condition we have an embedding
${\rm Aut}((\Sigma,(z_-,z_+,\vec z)),u,\varphi)\to  {\rm Perm}(m)$ where $m = \#\vec w$.
\end{rem}
\begin{lem}
There exists a symmetric stabilization
for any $((\Sigma,(z_-,z_+,\vec z)),u,\varphi)$.
\end{lem}
The proof is similar to that for the case of the moduli space of stable maps, which is now standard. So we omit the proof.
\par
Note that $(\Sigma,\{z_-,z_+\} \cup \vec z \cup \vec w)$ may not yet be a
stable marked curve because there can be a mainstream component $\Sigma_a$
which contains only transit points or $z_{\pm}$. (Namely $\Sigma_a$ may
not contain any of non-transit singular point or a point of $\vec z$.)
In Definition \ref{symstab} we do not put additional
marked points $\vec w$ on such a component.
Actually it is in general impossible
to find $w_i$ on such a component
satisfying Definition \ref{symstab} (6).
However, we can find a `canonical' marked point
on such a component.
Namely we
use a function $f_{H,u,\Sigma_a}$ 
to find the canonical position in the $\R$ coordinate
and the parametrization $\varphi_a$ to find the $S^1$ coordinate.
This is the way taken in
\cite[page 204]{foootech}, which we repeat below.
\par
Let $\Sigma_a$ be a mainstream component of $\Sigma$ which does not
contain non-transit singular point or marked point $\in \vec z$.
We define a function $f_{H,u,\Sigma_a} : \R \to \R$
as follows.
We first define an element
$(\gamma_{a,-},w_{a,-}) \in \widetilde{{\rm Per}}(H)$.
Here $\gamma_{a,-}$ is a closed orbit
which corresponds to the transit point $z_{a,-}$ by
$$
\gamma_{a,-}(t) = \lim_{\tau\to -\infty} u(\varphi_a(\tau,t)).
$$
We consider the connected component of $\Sigma \setminus \{z_{a,-}\}$
which contains $z_-$. The restriction of $u$ to this connected component
together with 
$w_{-}$
(which is a disk that bounds a curve, an element in
$R_{\alpha_-}$)
determines a homotopy class of maps $w_{a,-} : D^2 \to X$
which bounds $\gamma_{a,-}$. We thus obtain
$(\gamma_{a,-},w_{a,-}) \in \widetilde{{\rm Per}}(H)$.
\par
We define a function
\begin{equation}\label{eq:41}
f_{H,u,\Sigma_a}(\tau_0)
=
\int_{D^2}w_{a,-}^* \omega +
\int_{\tau=-\infty}^{\tau=\tau_0}\int_{t \in S^1} (u\circ \varphi_a)^*\omega
+
\int_{t\in S^1} H(u(\tau_0,t),t) dt.
\end{equation}
Proposition \ref{prof263} implies that this integral converges and 
both limits
$
\lim_{\tau \to \pm\infty}f_{H,u,\Sigma_a}(\tau)
$
exist.
Furthermore, using the fact that $u\circ \varphi_a$ satisfies (\ref{Fleq})
we can show that $f_{H,u,\Sigma_a}$ is nondecreasing.
Since $\Sigma_a$ does not
contain non-transit singular points or marked points $\in \vec z$,
the stability implies that $f_{H,u,\Sigma_a}$ is {\it strictly} increasing.
In fact, the first derivative
$df_{H,u,\Sigma_a}(\tau)/d\tau$ is strictly positive
unless $(\gamma_{a,-},w_{a,-}) = (\gamma_{a,+},w_{a,+})$.
Thus we have proved:
\begin{lem}\label{lemma46}
There exists a unique $\tau_a \in \R$ such that
\begin{equation}
f_{H,u,\Sigma_a}(\tau_a) = \frac{1}{2}
\left( \lim_{\tau \to -\infty}f_{H,u,\Sigma_a}(\tau) + \lim_{\tau \to +\infty}f_{H,u,\Sigma_a}(\tau)
\right).
\end{equation}
\end{lem}
\begin{defn}\label{defncanmark}
We call the point 
$w_{a,{\rm can}} = \varphi_a(\tau_a,0)$ the {\it canonical marked point on $\Sigma_a$}.
(Note that by our assumption that $\Sigma_a$ has no marked or singular points
other than transit points or $z_{\pm}$, the point $w_{a,{\rm can}}$ is neither a marked
nor a singular point.)
\par
We denote by $\vec w_{\rm can}$ the totality of $w_{a,{\rm can}}$
for each mainstream component $\Sigma_a$
which does
not contain any non-transit singular point or a point of $\vec z$.
\end{defn}
The above discussion also proves the next lemma.
\begin{lem}\label{lemma2639}
If $((\Sigma,(z_-,z_+,\vec z)),u,\varphi)
\in {\mathcal M}_{\ell}(X,H;\alpha_{-},\alpha_{+})$
and $\vec w$ is its symmetric stabilization, then
$(\Sigma,\{z_-,z_+\} \cup \vec z \cup \vec w \cup \vec w_{\rm can}))$
is stable.
\end{lem}
Now we define the notion of obstruction bundle 
data.\footnote{This notion was originally introduced in \cite[Definition 31.1]{foootech}.
The difference is that we include the Morse-Bott case in this article,
while in \cite[Definition 31.1]{foootech} we assumed \cite[Definition 29.4]{foootech}
which implies that all the periodic orbits are isolated.}
\begin{notation}\label{nota2642}
\begin{enumerate}
\item[$\bullet$]
We denote by $\mathcal M_{\ell}^{\rm cl}$ the moduli space
of stable curves of genus zero without boundary and
with $\ell$ marked points, and 
by $\overset{\circ}{\mathcal M^{\rm cl}_{\ell}}$ its subset consisting of
elements with only one irreducible component.
\item[$\bullet$]
We denote by $\overset{\circ}{\mathcal M}_{\ell}({\rm source})$ the set of points of
${\mathcal M}_{\ell}({\rm source})$ which have only one irreducible component,
(which is necessarily a mainstream component).
\item[$\bullet$]
We denote by $\pi : \mathcal C_{\ell}^{\rm univ} \to \mathcal M_{\ell}^{\rm cl}$ 
the universal family of stable curves of genus zero
with $\ell$ marked points. See \cite[Theorem 2.2]{fooo:const1} for example.
\end{enumerate}
\end{notation}
\begin{defn}\label{obbundeldata1}
{\it Obstruction bundle data $\frak E_{\bf p}$ centered at}
$$
{\bf p} = [(\Sigma,(z_-,z_+,\vec z)),u,\varphi] \in
{\mathcal M}_{\ell}(X,H;{\alpha_-},{\alpha_+})
$$
are the data 

\begin{equation}\label{eq:obbunddata}
\Big( \vec w, 
\{\mathcal V(\frak x_{\rm v} \cup \vec w_{\rm v}\cup \vec w_{{\rm can},{\rm v}})\}_{\rm v},
\{(\psi_{\rm v}, \phi_{\rm v})\}_{\rm v},
\{ K^{\rm obst}_{\rm v}\}_{\rm v},
\{ E_{{\bf p},{\rm v}}(\frak y_{\rm v})\}_{\rm v},
\{\mathcal D_i\}_{w_i \in \vec w}
\Big)
\end{equation}
satisfying the conditions described below.
We put 
$$
\frak x = (\Sigma,\{z_-,z_+\} \cup \vec z).
$$
Let ${\rm Irr}(\Sigma)$ be the set of the irreducible components of $\Sigma$ and,
for ${\rm v} \in {\rm Irr}(\Sigma)$, we denote by $\Sigma_{\rm v}$
the corresponding irreducible component.
\begin{enumerate}
\item
A symmetric stabilization $\vec w$ of $((\Sigma,z_-,z_+,\vec z),u,\varphi)$.
\item
A neighborhood $\mathcal V(\frak x_{\rm v} \cup \vec w_{\rm v}\cup \vec w_{{\rm can},{\rm v}})$ of 
$$
\frak x_{\rm v} \cup \vec w_{\rm v}\cup \vec w_{{\rm can},{\rm v}} = 
(\Sigma_{{\rm v}}, \{ z_{{\rm v},-}, z_{{\rm v},+} \}\cup\vec z_{\rm v} \cup \vec w_{\rm v}\cup \vec w_{{\rm can},{\rm v}})
$$
for each ${\rm v} \in {\rm Irr}(\Sigma)$:
Here 
\begin{enumerate}
\item[$\bullet$]
$\frak x_{\rm v} \cup \vec w_{\rm v}\cup \vec w_{{\rm can},{\rm v}}$ is an irreducible component of 
$\frak x$.
\item[$\bullet$] 
$\{ z_{{\rm v},-}, z_{{\rm v},+} \}
= \{{\text{\rm transit points on $\Sigma$}}\}\footnote{Containing $z_{\pm}$, see Definition \ref{defn41}.}  \cap \Sigma_{\rm v}$.
\item[$\bullet$]
$\vec{z}_{\rm v}=(\vec{z} \cup \{{\text{\rm non-transit singular points on $\Sigma$}}\})
\cap \Sigma_{\rm v}$.
\item[$\bullet$]
$\vec w_{\rm v}=\vec w \cap \Sigma_{\rm v}$ and $\vec w_{{\rm can},{\rm v}}
=\vec w_{\rm can} \cap \Sigma_{\rm v}$.
\end{enumerate}
Namely:
\begin{enumerate}
\item
In the case $\Sigma_{\rm v}$ is a mainstream component $\Sigma_a$, we include
the parametrization $\varphi_a$ to $\frak x_{\rm v}$ and
$\mathcal V(\frak x_{\rm v} \cup \vec w_{\rm v}\cup \vec w_{{\rm can},{\rm v}})$
is an open subset of
$\overset{\circ}{\mathcal M}_{\ell_{\rm v}+\ell_{\rm v}'+\ell_{\rm v}''}({\rm source})
$
where $\ell_{\rm v} = \# \vec z_{\rm v},
\ell_{\rm v}'= \# \vec w_{\rm v}$ and $\ell_{\rm v}'' = \# \vec w_{{\rm can},{\rm v}}$.
\item
In the case $\Sigma_{\rm v}$ is a bubble component, $\mathcal V(\frak x_{\rm v} \cup \vec w_{\rm v}\cup \vec w_{{\rm can},{\rm v}})$
is an open subset of
$\overset{\circ}{\mathcal M^{\rm cl}}
_{\ell_{\rm v}+\ell_{\rm v}'+\ell_{\rm v}''}$ where $\ell_{\rm v} = \# \vec z_{\rm v},
\ell_{\rm v}'= \# \vec w_{\rm v}$ and $\ell_{\rm v}'' = \# \vec w_{{\rm can},{\rm v}}=0$.
\end{enumerate}
\item
{\it Local trivialization data} $\{(\psi_{\rm v}, \phi_{\rm v})\}_{\rm v}$ at 
$\frak x \cup \vec w\cup \vec w_{{\rm can}}$ in the sense of 
\cite[Definition 3.8]{fooo:const1}:
Namely, denoting by $\mathcal V_{\rm v}$
the neighborhood $\mathcal V(\frak x_{\rm v} \cup \vec w_{\rm v}\cup \vec w_{{\rm can},{\rm v}})$ taken in (2) above,
$\psi_{{\rm v},z} : \mathcal V_{\rm v} \times {\rm Int}\,{D^2} \to 
\mathcal C_{\ell}^{\rm univ}$ is an {\it analytic family of coordinates} at each 
non-transit singular point $z$ 
of $\Sigma$ which is contained in $\Sigma_{\rm v}$ 
(see \cite[Definition 3.1]{fooo:const1}), 
and 
$\phi_{\rm v} : \mathcal V_{\rm v} \times (\frak x_{\rm v} \cup \vec w_{\rm v}\cup \vec w_{{\rm can},{\rm v}}) \to \pi^{-1}(\mathcal V_{\rm v})$ is a
$C^{\infty}$ trivialization over $\mathcal V_{\rm v}$ 
of the universal family 
$\pi : \mathcal C_{\ell}^{\rm univ} \to \mathcal M_{\ell}^{\rm cl}$ of marked stable curves 
(see \cite[Definition 3.6]{fooo:const1}), which is 
compatible with the analytic family of coordinates in the sense of \cite[Definition 3.7 (1)]{fooo:const1}.
We also require the following additional conditions on a part of the local trivialization data.
\begin{enumerate}
\item
Let $z_{a}$ be a transit point contained in $\Sigma_a$ and $\Sigma_{a+1}$.
Then the coordinate near $z_{a}$ given by the local trivialization data 
is
the parametrization $\varphi_a$ or $\varphi_{a+1}$
up to the $\R$ action. Namely it is  $(\tau,t) \mapsto \varphi_{a}(\tau+\tau_0,t)$
(resp. $\varphi_{a+1}(\tau+\tau_0,t))$ for some $\tau_0 \in \R$. 
\item
Let $z$ be a non-transit singular point contained in $\Sigma_{\rm v}$.
Since $\Sigma_{\rm v}$ is a sphere, there exists a biholomorphic map
$$
\phi : \Sigma_{\rm v} \cong \C \cup \{\infty\}
$$
such that $\phi(z) = 0$. 
Then the coordinate around $z$ given by the local trivialization data is
$$
(\tau,t) \mapsto \phi^{-1}(e^{- 2\pi(\tau+\sqrt{-1}t)}),
$$
for some choice of $\phi$. 
\end{enumerate}
\item
A compact subset $K^{\rm obst}_{\rm v}$ of $\Sigma_{\rm v}$
such that 
$\mathcal V(\frak x_{\rm v}\cup \vec w_{\rm v}\cup \vec w_{{\rm can},{\rm v}})
\times K^{\rm obst}_{\rm v}$ is contained in 
a compact subset of $\pi^{-1}(\mathcal V_{\rm v})$ under the $C^{\infty}$ trivialization $\phi_{\rm v}$:
We assume that 
$\cup_{\rm v \in {\rm Irr}(\Sigma)} K^{\rm obst}_{\rm v}$ is invariant under the 
${\rm Aut}^+((\Sigma,(z_-,z_+,\vec z)),u,\varphi)$ action.
We call $K^{\rm obst}_{\rm v}$ the {\it support of the obstruction bundle}.
\item
A $\frak y_{\rm v} \in \mathcal V(\frak x_{\rm v} \cup \vec w_{\rm v}\cup \vec w_{{\rm can},{\rm v}})$-parametrized smooth family
of finite dimensional complex linear subspaces 
$E_{{\bf p},{\rm v}}(\frak y_{\rm v}) \subset 
C^{\infty}_0(\text{\rm Int}\,K^{\rm obst}_{\rm v}; u^*TX \otimes 
\Lambda^{0,1}\Sigma_{\frak y_{\rm v}})
$:
Here $C^{\infty}_0$ denotes the set of smooth sections with compact support in
$\text{\rm Int}\,K^{\rm obst}_{\rm v}$.
We also regard $K^{\rm obst}_{\rm v}$ as a subset of $\Sigma_{\frak y_{\rm v}}$ 
by using the $C^{\infty}$ trivialization of the analytic family of coordinates given as a part of the local trivialization data. 
(Here $\Sigma_{\frak y_{\rm v}}$ is the source curve of $\frak y_{\rm v}$.)
\par
We assume that the direct sum $\bigoplus_{{\rm v} \in {\rm Irr}(\Sigma)} E_{{\bf p},{\rm v}}(\frak y_{\rm v})$ is invariant under the
${\rm Aut}^+((\Sigma,(z_-,z_+,\vec z)),u,\varphi)$ action
in the sense of \cite[Definition 5.5]{fooo:const1}.
\par
When 
$\frak y_{\rm v}=\frak x_{\rm v} \cup \vec w_{\rm v}\cup \vec w_{{\rm can},{\rm v}}$
which is the center of the neighborhood $\mathcal V(\frak x_{\rm v} \cup \vec w_{\rm v}\cup \vec w_{{\rm can},{\rm v}})$, we denote $E_{{\bf p},{\rm v}}(\frak y_{\rm v})$ by $E_{{\bf p},{\rm v}}$.
\item
For each ${\bf p}$ we consider a linear differential operator 
\begin{equation}\label{eq:whole}
D_{\bf p} \overline{\partial}_{J,H}^{\rm whole} ~:~
\bigoplus_{\rm v\in {\rm Irr}(\Sigma)} L^2_{m+1,\delta}(\Sigma_{{\rm v}};
u^*TX) 
\to 
\bigoplus_{\rm v\in {\rm Irr}(\Sigma)} L^2_{m,\delta}(\Sigma_{{\rm v}}; u^*TX \otimes \Lambda^{0,1}
\Sigma_{{\rm v}})
\end{equation}
defined as follows. Here we use the above weighted Sobolev spaces in the same way as in \cite[Definition 3.4]{foooanalysis} for the case $\partial \Sigma_i =\emptyset$ there.
\begin{enumerate}
\item
On $ L^2_{m+1,\delta}(\Sigma_{{\rm v}};
u^*TX)$ for 
$\text{\rm v} \in {\rm Irr}(\Sigma)$ 
being a mainstream component
and ${\rm v} \in \mathcal V(\frak x_{\rm v}\cup \vec w_{\rm v}
\cup \vec w_{{\rm can},{\rm v}})$,
it is the linearized operator of Floer's equation 
(see \eqref{form2688} for this operator)
\begin{equation}\label{ducomponentssphererevrev}
\aligned
D^{{\rm Floer}}_{u}:=D_{u} \overline\partial -  \nabla_{\cdot}(J\frak X_H)\otimes 
(\varphi_{\rm v}^{\ast})^{-1}(d\tau -\sqrt{-1}dt) ~:~
L^2_{m+1,\delta}(\Sigma_{{\rm v}};u^*TX & ) \\
\to
L^2_{m,\delta}(\Sigma_{{\rm v}}; u^*TX \otimes \Lambda^{0,1}
\Sigma_{{\rm v}} & ). 
\endaligned
\end{equation}
\item
On $ L^2_{m+1,\delta}(\Sigma_{{\rm v}};
u^*TX)$ for 
$\text{\rm v} \in {\rm Irr}(\Sigma)$ 
being a bubble component and ${\rm v} \in \mathcal V(\frak x_{\rm v}\cup \vec w_{\rm v}
\cup \vec w_{{\rm can},{\rm v}})$, it is the linearization 
\begin{equation}\label{linearizationJholo}
D_{u} \overline\partial ~:~L^2_{m+1,\delta}(\Sigma_{{\rm v}};
u^*TX) \to
L^2_{m,\delta}(\Sigma_{{\rm v}}; u^*TX \otimes \Lambda^{0,1}
\Sigma_{{\rm v}})
\end{equation}
of the nonlinear Cauchy-Riemann operator.
\end{enumerate}
We consider the following subspace 
$$
L^2_{m+1,\delta}({\bf p}) := 
\{ s=(s_{\rm v}) \in  \bigoplus_{\rm v} L^2_{m+1,\delta}(\Sigma_{{\rm v}};
u^*TX) ~\vert~ {\text{ $s$ satisfies the condition $(\heartsuit)$}}
\}
$$
of the domain of the operator in \eqref{eq:whole},
where $(\heartsuit)$ is the coincidence condition at each singular point 
$z\in \Sigma_{{\rm v}^{+}}
\cap \Sigma_{{\rm v}^{-}}$ on $\Sigma$ such that 
$$
({\rm ev}_{+}^{{\rm v}^+})_{\ast}(s_{{\rm v}^{+}})=
({\rm ev}_{-}^{{\rm v}^-})_{\ast}(s_{{\rm v}^{-}})
$$ 
for a transit point $z=z_{{\rm v}^{+},+}=z_{{\rm v}^{-},-}$ and 
$$
({\rm ev}_{z}^{{\rm v}^+})_{\ast}(s_{\rm v^{+}})=({\rm ev}_{z}^{{\rm v}^-})_{\ast}(s_{{\rm v}^{-}})
$$ 
for a non-transit point $z$. 
Here at a transit point the maps 
$$
({\rm ev}_{\pm}^{{\rm v}^{\pm}})_{\ast} ~:~ L^2_{m+1,\delta}(\Sigma_{{\rm v}};u^*TX)
\to 
T_{\gamma_{{\rm v}^{\pm}}^{\pm}} {\rm Per} (H)
$$
are differentials of the evaluation maps as in \eqref{evatinfqbbb} and 
$$
\gamma_{{\rm v}^{\pm}}^{\pm} = \lim_{\tau \to \pm \infty}
u(\varphi_{{\rm v}^{\pm}}(\tau,t)).
$$
Restricting the domain to the subspace 
$L^2_{m+1,\delta}({\bf p})$, we have a Fredholm operator
\begin{equation}\label{eq:linearized}
D_{\bf p} \overline{\partial}_{J,H} ~:~
L^2_{m+1,\delta}({\bf p}) \to 
\bigoplus_{\rm v} L^2_{m,\delta}(\Sigma_{{\rm v}}; u^*TX \otimes \Lambda^{0,1}
\Sigma_{{\rm v}})
\end{equation}
and call it the {\it linearization operator at ${\bf p}$}.
Then in this item we require that the linearization operator 
$D_{\bf p} \overline{\partial}_{J,H}$  
is surjective $\mod \oplus_{\rm v} E_{{\bf p},{\rm v}}$, and 
${\rm Aut}^+((\Sigma,(z_-,z_+,\vec z)),u,\varphi)$ acts on 
$(D_{\bf p} \overline{\partial}_{J,H})^{-1}(\oplus_{\rm v}E_{{\bf p},{\rm v}})$ effectively.
Here $E_{{\bf p},{\rm v}}$ is introduced at the end of (5) above.
\item
For each $w_i \in \vec w \in \Sigma$ we take a codimension $2$ submanifold $\mathcal D_{i}$ of $X$
such that $u(w_i) \in \mathcal D_i$ and
$$
u_*T_{w_i}\Sigma + T_{u(w_i)}\mathcal D_i = T_{w_i}X.
$$
Moreover the set $\{\mathcal D_i\}$ is invariant under the ${\rm Aut}^+((\Sigma,(z_-,z_+,\vec z)),u,\varphi)$ action in the following sense.
Let $v \in {\rm Aut}^+((\Sigma,(z_-,z_+,\vec z)),u,\varphi)$
and $v(w_i) =w_{\sigma(i)}$ then
$$
\mathcal D_i = \mathcal D_{\sigma(i)}.
$$
(Note $u(w_i) = u(w_{\sigma(i)})$ since $u\circ v = u$.)
We note that we do {\it not} take such submanifolds for the canonical marked points $\in \vec w_{{\rm can}}$.
(In fact, since $u$ is not necessarily an immersion at the canonical marked points, we can not choose such submanifolds.)
\end{enumerate}
\end{defn}
\begin{rem}\label{rem:412}
\begin{enumerate}
\item
The condition (6) and the coincidence condition $(\heartsuit)$ therein imply that 
the restrictions of differential of the evaluation map
$\oplus_{z,\rm v} ({\rm ev}^{\rm v}_z)_{\ast}$, where $z$ is any singular point,
to $(D^{{\rm Floer}}_{u})^{-1}(\oplus_{\rm v}E_{{\bf p},{\rm v}})$
and $(D\overline{\partial}_{u})^{-1}(\oplus_{\rm v}E_{{\bf p},{\rm v}})$
are surjective. Here 
$D^{{\rm Floer}}_{u}$ and $D\overline{\partial}_{u}$ are the linearized operators 
in \eqref{ducomponentssphererevrev} and \eqref{linearizationJholo} respectively.
\item
In Item (5), we consider the family of finite dimensional complex subsapaces 
$E_{\mathbf p, {\rm v}}(\frak y_{\rm v})$ over 
$\frak y_{\rm v} \in \mathcal V(\frak x_{\rm v} \cup \vec w_{\rm v}\cup \vec w_{{\rm can},{\rm v}})$, while we impose the condition in (6) only at
$\frak x_{\rm v} \cup \vec w_{\rm v}\cup \vec w_{{\rm can},{\rm v}}$.
However, 
if we start with the finite dimensional complex subspace 
$E_{\mathbf p, {\rm v}}$ 
in (6),  
we can obtain 
a smooth family
of finite dimensional complex subspaces $E_{{\bf p},{\rm v}}(\frak y_{\rm v})$
of $C^{\infty}_0(\text{\rm Int}\,K^{\rm obst}_{\rm v}; u^*TX \otimes \Lambda^{0,1}\Sigma_{\mathfrak y_{\rm v}}) 
$
at $\frak y_{\rm v}$ in the neighborhood 
$\mathcal V(\frak x_{\rm v} \cup \vec w_{\rm v}\cup \vec w_{{\rm can},{\rm v}})$
satisfying the property in (6) as well 
by using the composition of the inclusion map $\Lambda^{0,1}_x (\Sigma_{\rm v}) \to \Lambda^1_x(\Sigma_{\rm v}) \otimes {\mathbb C}$ 
and the projection $\Lambda^1_x (\Sigma_{\rm v}) \to \Lambda^{0,1}_x (\Sigma_{\mathfrak y_{\rm v}})$ for $x \in K_{\rm v}^{\rm obst}$.   
\end{enumerate}
\end{rem}

\par
Note in \cite[page 204 line 6 to 4 from the bottom]{foootech} we wrote
\par\medskip
$<<$
We require that the data $K^{\rm obst}_{\rm v}$, $E_{\frak p,{\rm v}}(\frak y,u)$ depend only on
the mainstream component
$\frak p_i = [(\Sigma_i,z_{i-1},z_i),u,\varphi]$
(where $z_i$ is the $i$-th transit point) that contains the $\rm v$-th irreducible component.
We call this condition {\it mainstream-component-wise}.
$>>$
\par\medskip
This point is very much related to the main theme of the study 
in Section \ref{subsec;KuramodFloercor}, 
which is a construction of Kuranishi structure 
compatible with the fiber product description of
the boundary and corners.
As we mentioned in Remark \ref{rem628},
the Kuranishi structure we construct in this section is
{\it not} compatible with fiber product description of
the boundary and corners. 
The same proof given in \cite[Theorem 11.1]{fooo:const1}  (see also \cite[Lemma 17.11]{foootech}) 
yields the following.
\begin{lem}\label{lem2635}
For each 
${\bf p} \in {\mathcal M}_{\ell}(X,H;{\alpha_-},{\alpha_+})$
there exist obstruction bundle data 
$\frak E_{\bf p}$ centered at ${\bf p}$
in the sense of Definition \ref{obbundeldata1}.
\end{lem}

\subsection{Smoothing singularities and $\epsilon$-closeness}
We put 
$$
\frak x = [\Sigma,(z_-,z_+,\vec z),\varphi] \in {\mathcal M}_{\ell}(\text{source}).
$$
Let $\frak x_{\rm v}$ be an irreducible component $\Sigma_{\rm v}$ of $\Sigma$ together with
marked and singular points on it.
It is an element of an appropriate  moduli space of marked curves of genus zero if $\Sigma_{\rm v}$
is a bubble component. It is in  ${\mathcal M}_{m}(\text{source})$ if it is a mainstream
component. (More precisely, they may not be stable.
They become stable after we add $\vec w_{\rm v}\cup \vec w_{{\rm can},{\rm v}}$
that are parts of $\vec w\cup \vec w_{{\rm can}}$ on this irreducible component.)
\par
We recall from 
\cite[Lemma 3.9]{fooo:const1}\footnote{\cite[Lemma 3.9]{fooo:const1} treats the case of bordered stable curves with interior and boundary marked points.
In our case, smoothing at a non-transit point corresponds to that at an interior marked point which has a two dimensional parameter space.
On the other hand, at a transit point the parameter space of smoothing is only one dimensional, because the equation on the main stream component is not $S^1$ invariant of the rotational action. In this way smoothing at a transit point can be treated similarly to the case of smoothing at a boundary marked point of a bordered stable curve described in \cite{fooo:const1}.}
the way how to 
smooth the singularity of the curve $\Sigma$ and fix the local trivialization of the
universal family (outside the node).
Namely it determines a map:
\begin{equation}\label{form418}
\aligned
{\Phi}_{\bf p} : \prod_{\rm v} \mathcal V(\frak x_{\rm v} \cup \vec w_{\rm v}\cup \vec w_{{\rm can},{\rm v}}) \times {D}(k;\vec T_0)
\times &\prod_{j=1}^m\left(((T_{0,j},\infty] \times S^1)/\sim\right)
\\
&\to {\mathcal M}_{\ell+\ell'+\ell''}(\text{source})
\endaligned\end{equation}
that is a homeomorphism onto an open neighborhood of
$[\frak x \cup \vec w\cup \vec w_{{\rm can}}]$ in ${\mathcal M}_{\ell+\ell'+\ell''}(\text{source})$.
(See Definition \ref{not2620} for this notation.)
Here $\ell' = \# \vec w$ and $\ell'' = \# \vec w_{{\rm can}}$.
The map ${\Phi}_{\bf p}$ is defined in \cite[(31.4)]{foootech}.
We recall its definition 
together with the other notations appearing in (\ref{form418})
from \cite[Section 30]{foootech} below: 
\par\noindent
\begin{enumerate}
\item[$\bullet$] $k$ is the number of transit points except $\{ z_-, z_+ \}$ of $\Sigma$. 
\item[$\bullet$]
A manifold with corner $\widetilde D(k;\vec T_0)$ is defined as follows.
For any $\vec T_0 =(T_{0,(1)},\dots ,T_{0,(k)}) \in \R_{>1}^{k}$ we put
\begin{equation}\label{216}
\widetilde {\overset{\circ}D}(k;\vec T_0)\\
=
\{(T_1,\dots,T_{k+1}) \in \R^{k+1} \mid T_{a+1} - T_a > T_{0,(a)}\}
\end{equation}
and partially compactify it to $\widetilde {D}(k;\vec T_0)$ by admitting $T_{a+1} - T_a = \infty$
as follows.
We put $s'_a = 1/\log(T_{a+1}-T_a)$. Then
$T_1$ and $s'_1,\dots,s'_{k-1}$ define another parameters.
So (\ref{216}) is identified with $\R \times \prod_{i=1}^k(0,1/\log 
T_{0,(a)})$. 
We partially compactify it to $\R \times \prod_{a=1}^k[0,1/\log 
T_{0,(a)})$.
By taking the quotients of
$\widetilde {\overset{\circ}D}(k;\vec T_0)$ and
$\widetilde {D}(k;\vec T_0)$ by
the $\R$ action 
$T(T_1,\dots,T_{k+1}) = (T_1+T,\dots,T_{k+1}+T)$,
we obtain ${\overset{\circ}D}(k;\vec T_0)$ and $D(k;\vec T_0)$
respectively.
\end{enumerate}
\begin{rem}\label{rem2247}
We take the logarithm of $T_{a+1}-T_a$ to define our coordinate $s'_a$.  By doing so we can stay in the category of
admissible orbifolds or admissible Kuranishi structures 
in the sense of \cite[Chapter 25]{fooonewbook}.
We also take $e^{2\pi t\sqrt{-1}}/\log T$ for the coordinate of the $((T_0,\infty] \times S^1)/\sim$ factor.
This is a slightly different choice from \cite[Appendix A1.4]{fooobook2}.
\end{rem}
\begin{enumerate}
\item[$\bullet$]
The space ${D}(k;\vec T_0)$ is used to parametrize
the ways of smoothing the transit point singularities
as follows.
We consider the case when the parameter $\vec T $ is in
$\overset{\circ}D(k;\vec T_0)$.
Taking a section of the projection 
$\widehat{\mathcal M}_{\ast}(\text{source}) \to {\mathcal M}_{\ast}(\text{source})$,
we have a parametrization $\varphi_a : \R \times S^1 \to \Sigma_a$ for each 
mainstream component $\Sigma_a$ with $a=1,\dots ,k+1$ as in Definition \ref{defn210} (4).
Let us consider 
$$
[-5T_{0,(a-1)},5T_{0,(a)}]_a \times S^1_a
$$
where $T_{0,(0)}=T_{0,(k+1)}=+\infty$ as convention, 
and regard it as a subset of the domain of 
the parametrization $\varphi_a : \R \times S^1 \to \Sigma_a$.
We define
$$
\varphi_0 : \bigcup_a ([-5T_{0,(a-1)},5T_{0,(a)}]_a \times S^1_a)  \to \R \times S^1
$$
as follows. If $(\tau,t) \in [-5T_{0,(a-1)},5T_{0,(a)}]_a \times S^1_a$, then
$$
\varphi_0(\tau,t) = (\tau+10 T_a,t).
$$
We use $\varphi_0\circ\varphi_a^{-1}$ to identify (a part of) $\Sigma_a$ with a subset of
$\R \times S^1$.
Under this identification marked points on $\Sigma_a$ can be moved to $\R \times S^1$. 
Adding $z_{-}, z_{+}$, we have a marked Riemann surface. 
Taking the equivalence class by  
$\sim_2$ in Definition \ref{3equivrel}, we obtain an element of 
${\mathcal M}_{\ast}(\text{source})$ so the map $\Phi_{\bf p}$ in the case $m=0$.
See Figure \ref{Fig23.1}.
\begin{figure}[htbp]
\centering
\includegraphics[scale=0.5]{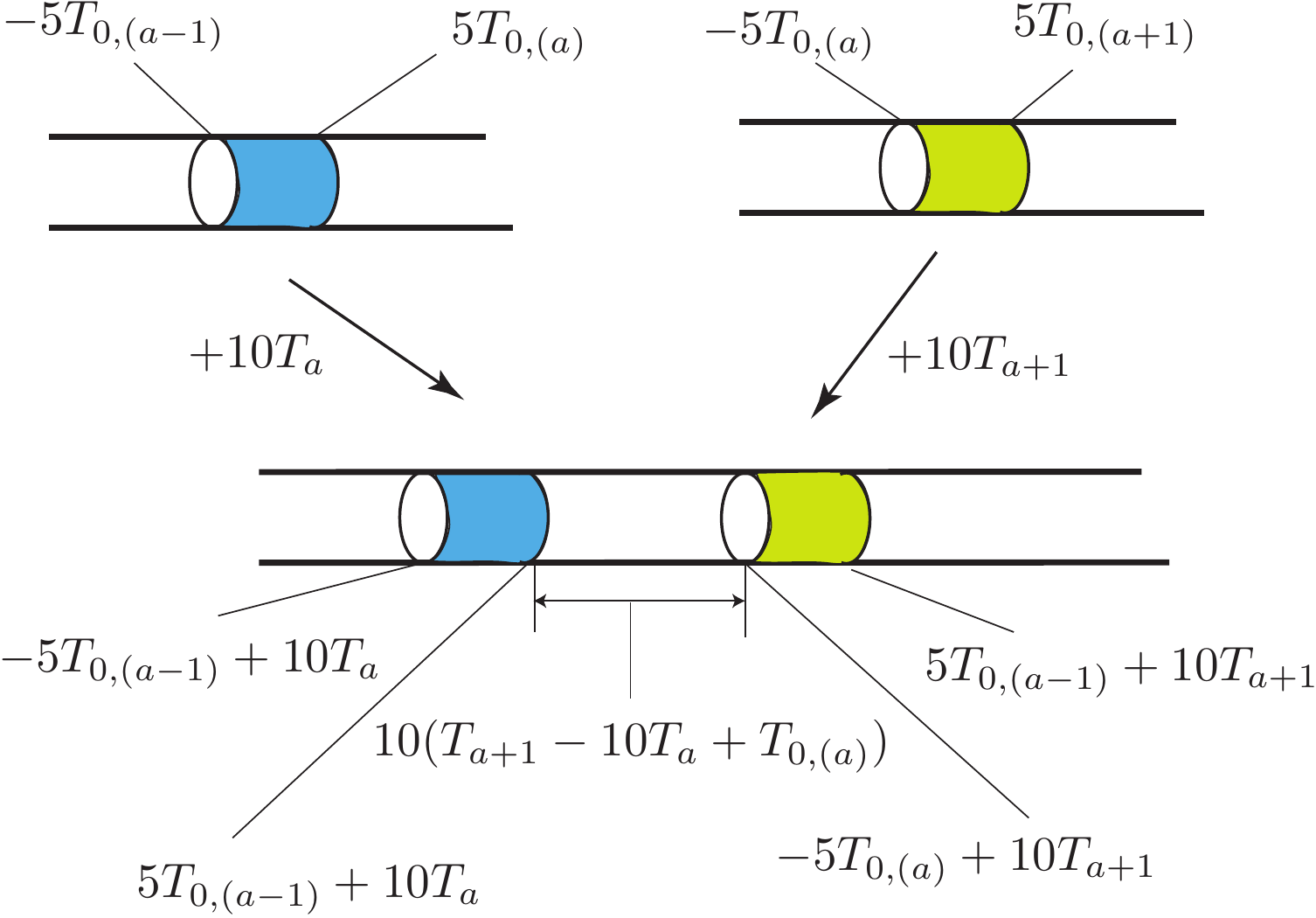}
\caption{The map $\varphi_0$}
\label{Fig23.1}
\end{figure}
\par
\item[$\bullet$]
$m$ in \eqref{form418} is the number of non-transit singular points.
The factor $((T_{0,j},\infty] \times S^1)/\sim$
$(j=1,\dots,m)$ is the space
to parametrize the way to smooth these singular points.
Here $\sim$ is the equivalence relation such that
$(T,t) \sim (T',t')$ if and only if $T=T' =\infty$ or
$(T,t) = (T',t')$.
The way to use this parameter to smooth
non-transit singular points is written in 
\cite[Lemma 3.9]{fooo:const1}
\end{enumerate}
We have thus defined all notations appearing in \eqref{form418}.
\begin{notation}\label{nota:415}
Suppose
\begin{equation}\label{eq:notation}
({\frak Y} \cup \vec w',\varphi') = {\Phi}_{\bf p}(\frak y,\vec T,\vec \theta) \in {\mathcal M}_{\ell+\ell'+\ell''}(\text{source}).
\end{equation}
Here
\begin{enumerate}
\item[$\bullet$] $\frak y = (\frak y_{\rm v})$ where $\rm v$ is 
the index in the set 
of irreducible components of
$\Sigma$ in ${\bf p}$ as in Definition \ref{obbundeldata1}.
\item[$\bullet$] $\vec w'$ is the set of the additional marked points corresponding to
$\vec w$ and $\vec w_{\rm can}$.
\item[$\bullet$] The notation $\frak Y$ includes the marked points corresponding to $\vec z$ and $z_{\pm}$.
\item[$\bullet$] The pair of parameters $(\vec T,\vec \theta) \in {D}(k;\vec T_0)
\times \prod_{j=1}^m\left(((T_{0,j},\infty] \times S^1)/\sim\right)$ and the map 
$\varphi'$ is a parametrization of the
mainstream of $\Sigma'$. Here $\Sigma'$ is the source curve of $\frak Y$.
\end{enumerate}
Of course, the left hand side depends on ${\bf p}$ and 
$\frak y,\vec T,\vec \theta$ in the right hand side also depend on ${\bf p}$ as well as the left hand side.
\end{notation}
\par
In the situation of Notation \ref{nota:415}, since $\Sigma'$ is obtained from the source curves 
$\Sigma_{\frak y_{\rm v}}$
by first removing neighborhoods of singular points and then gluing them,
there exists an embedding
\begin{equation}\label{form2631}
\frak v_{\frak Y,\frak y;{\rm v}} : K^{\rm obst}_{\rm v} \to  \Sigma'.
\end{equation}
Actually the embedding to $\Sigma'$ is defined in a larger region called
the {\it core} of the source curve $\Sigma_{\frak y}$.
(It is the complement of the {\it neck region}. 
See \cite[Definition 4.12]{fooo:const1} for the definition of neck region.)
\begin{defn}\label{orbitecloseness}
Let $(\frak Y\cup \vec w',\varphi') = {\Phi}_{\bf p}(\frak y,\vec T,\vec \theta) \in {\mathcal M}_{\ell+\ell'+\ell''}(\text{source})$ and
$u' : \Sigma' \setminus \{\text{transit points}\} \to X$. We assume
that $(\frak Y,u',\varphi')$ satisfies
Definition \ref{defn210} (1)(2)(3)(7) and (8).
We say that $(\frak Y,u',\varphi') \cup \vec w'$ is {\it $\epsilon$-close} to ${\bf p} \cup \vec w\cup \vec w_{{\rm can}}$  if the following holds.
\begin{enumerate}
\item
$\Vert u' \circ \frak v_{\frak Y,\frak y;{\rm v}} - u \Vert < \epsilon$ on the core of 
$\Sigma_{\frak y}$.
Here $\Vert \cdot\Vert$ is the $C^{10}$ norm.
\item
The map $u' \circ \varphi'$ satisfies the equation \eqref{Fleq} in the neck regions corresponding to transit points.  
For a bubble component $\Sigma'_{b'}$, there is a non-transit point on a mainstream component $\Sigma'_{a'}$ 
such that $\Sigma'_{b'}$ is joined to $\Sigma'_{a'}$ by a tree of bubble components.  
Then the map $u'\vert_{\Sigma'_{b'}}$ is $J$-holomorphic  on some neighborhoods of nodes.  
When a non-transit point on a mainstream component $\Sigma'_{a'}$ is smoothed, 
$u' \circ  \varphi'$  satisfies the equation \eqref{Fleq} on the corresponding neck region on the mainstream component.  
See Figure \ref{Figurep24}.
\item

Let $\hat u'$ be the redefined connecting orbit map of $u'$ 
(see Definition \ref{redefconnecting}).  
For a non-transit point $\varphi_{a'}(\tau_0,t_0)$ on a mainstream component $\Sigma'_{a'}$,  set 
$u^{\prime \#}(z) = ({\rm exp}^H_{t})^{-1}(u(z))$ for $z=\varphi_a(\tau, t)$, $t_0 -1/2 < t < t_0 + 1/2$. 
Then for each connected component $\mathfrak W$ of the complement of the core, we have either
\begin{equation}\label{form262020rever}
{\rm Diam} (\hat u'(\mathfrak W))  <  \epsilon, \enskip
 \text{ for } {\mathfrak W} \text{ around transit points}, 
 \end{equation}
or 
\begin{equation}\label{form262121revrev}
{\rm Diam} (u^{\prime\#}(\mathfrak W))  <  \epsilon, \enskip 
 \text{ for } {\mathfrak W} \text{ around non-transit points}.
\end{equation}
\item For each component $T_{0,(a)}$ of $\vec T_0$ we have $T_{0,(a)}> \epsilon^{-1}$ 
and $T_{0,j} > \epsilon^{-1}$ for any $j=1,\dots ,m$.
\end{enumerate}

\begin{figure}[htbp]
\centering
\includegraphics[scale=0.5]{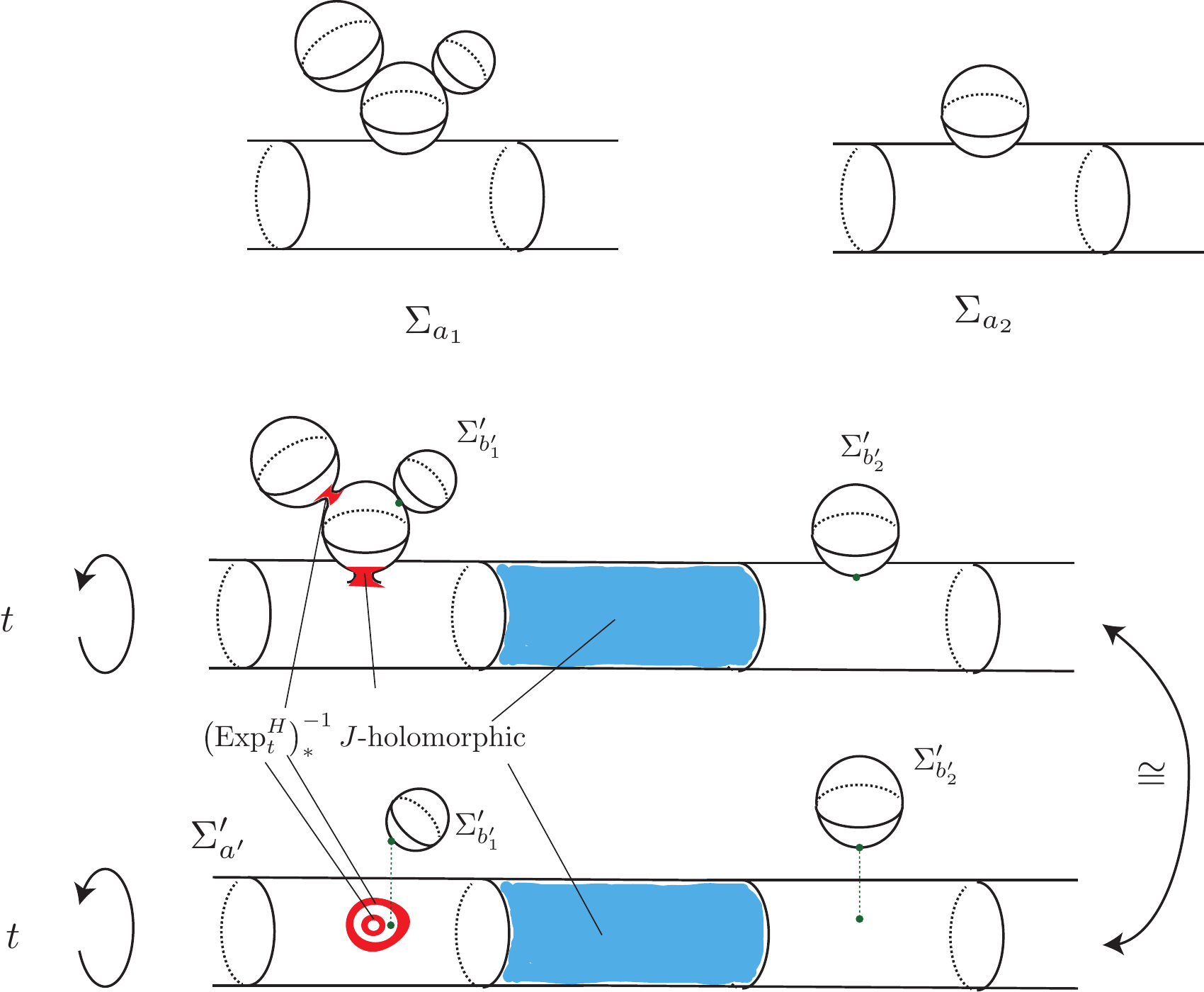}
\caption{Neck regions on mainstream components}
\label{Figurep24}
\end{figure}
\end{defn}
\begin{rem}
Definition \ref{orbitecloseness} is the same as the definition of
$\epsilon$-closeness appearing in \cite[page 215]{foootech}, 
\cite[Definition 4.12]{fooo:const1}.
Note that we do not assume Condition (5)' in \cite[page 215]{foootech},  
since it is a consequence of Definition \ref{orbitecloseness} (1) above, and so is unnecessary
to be assumed.
\end{rem}
\begin{shitu}\label{shitu2639}
Let ${\bf p} \in {\mathcal M}_{\ell+\ell'+\ell''}(X,H;{\alpha_-},{\alpha_+})$.
We fix obstruction bundle data $\frak C_{\bf p}$ centered at ${\bf p}$.
Let $({\frak Y} \cup \vec w',\varphi')  = {\Phi}_{\bf p}(\frak y,\vec T,\vec \theta) \in {\mathcal M}_{\ell+\ell'+\ell''}(\text{source})$
and
$u' : \Sigma' \to X$,
where $\Sigma'$ is the source curve of $\frak Y$ as in 
Notation \ref{nota:415}.
We assume that
$(\frak Y,u',\varphi') \cup \vec w'$ is $\epsilon$-close to  ${\bf p} \cup \vec w\cup \vec w_{{\rm can}}$. $\blacksquare$
\end{shitu}
\begin{defn}\label{constrainttt2}
Suppose we are in Situation \ref{shitu2639}.
We say that $(\frak Y,u',\varphi') \cup \vec w'$ satisfies the {\it transversal constraint} if the following holds.
\begin{enumerate}
\item
If $w'_i$ corresponds to $w_i \in \vec w$, then $u'(w'_i) \in \mathcal D_i$.
\item
Let $w'_j = \varphi'_{a'}(\tau'_j,t'_j) \in \vec w' \cap \Sigma'_{a'}$ be the marked point corresponding to the
canonical marked point $w_{a,{\rm can}} = \varphi_a(\tau_a,0) \in \vec w_{\rm can}$.
Then we require:
\begin{equation}\label{form2636}
f_{H,u',\Sigma'_{a'}}(\tau'_{j})
=
f_{H,u,\Sigma_a}(\tau_{a})
=
\frac{1}{2}
\left( \lim_{\tau \to -\infty}f_{H,u,\Sigma_a}(\tau) + \lim_{\tau \to +\infty}
f_{H,u,\Sigma_a}(\tau)
\right).
\end{equation}
Here $f_{H,u,\Sigma_a}$ 
is the function \eqref{eq:41}
for the mainstream component $\Sigma_a$ of $\Sigma$.
\item
In the situation of (2) we require also
$
t'_j = [0].
$
\end{enumerate}
\end{defn}
\begin{rem}
\begin{enumerate}
\item
The second equality of (\ref{form2636}) is the definition of $\tau_a$ (see Lemma \ref{lemma46}).
So the actual condition is the first equality.
\item
Note that $\Sigma'_{a'}$ may have a sphere bubble (or may contain one of the marked points
of $\frak Y$) even in the case when
$\Sigma_a$ has no sphere bubble (or does not contain one of the marked points
of $\vec z$).
In fact, $\Sigma_a$ may be glued with other mainstream component
that has a sphere bubble when we obtain $\Sigma'$ form $\Sigma$.
Therefore $\Sigma'_{a'}$ may not have a canonical marked point.
\item
Even in the case when $\Sigma'_{a'}$ has a canonical marked point,
it may be different from $w'_j$.
In fact $\Sigma'_{a'}$ may be obtained by gluing several
mainstream components which have no sphere bubbles or points of $\vec z$.
\end{enumerate}
\end{rem}
Suppose we are in Situation \ref{shitu2639}.
Let $z \in K^{\rm obst}_{\rm v}$. 
Recall from Definition \ref{obbundeldata1} (5) that 
$K^{\rm obst}_{\rm v}$ contains a support of elements of
$E_{{\bf p},{\rm v}}(\frak y)$. 
Let
$\frak v_{\frak Y,\frak y;{\rm v}} : K^{\rm obst}_{\rm v} \to \Sigma'$ be an embedding as in (\ref{form2631}).
By $\epsilon$-closeness we have
$$
d(u'(\frak v_{\frak Y,\frak y;{\rm v}}(z)),u(z)) < \epsilon.
$$
We may choose $\epsilon >0$ smaller than the injectivity radius
of $X$. Therefore there exists a unique minimal
geodesic $\ell_z$ in $X$ joining
$u(z)$ with
$u'(\frak v_{\frak Y,\frak y;{\rm v}}(z))$.
The complex linear part of the parallel transport along $\ell_z$
defines a complex linear map
\begin{equation}\label{form2633}
{\rm Pal}_z : T_{u(z)}X \to T_{u'(\frak v_{\frak Y,\frak y;{\rm v}}(z))}X.
\end{equation}
\begin{defn}\label{defn2642}
Suppose we are in Situation \ref{shitu2639}.
Using the map in \eqref{form2633}, we have 
a complex linear embedding
$$
I_{{\bf p},{\rm v};\Sigma',u',\varphi'} : E_{{\bf p},{\rm v}}(\frak y) \to
C^{\infty}(\Sigma';(u')^*TX \otimes \Lambda^{0,1}).
$$
\end{defn}
In Definition \ref{defn2646} we will define 
an obstruction space 
$E((\frak Y\cup \bigcup_{c\in \EuScript B}\vec w'_c,u',\varphi');{\bf q};\EuScript B)$
(see also Definition \ref{defn2646} for the notations used in this notation)
for any ${\bf q} \in {\mathcal M}_{\ell}(X,H;{\alpha_-},{\alpha_+})$
as a sum of images of the maps in Definition \ref{defn2642} for suitable choices of 
${\bf p}$'s.
To carry out this argument
we first observe the following.
\begin{lem}\label{lem2446}
Suppose we are in Situation \ref{shitu2639}.
Then for any ${\bf p} \in {\mathcal M}_{\ell}(X,H;{\alpha_-},{\alpha_+})$ 
there exist $\epsilon_{\bf p} >0$ and a closed neighborhood $W({\bf p})$ of
$\bf p$ in $ {\mathcal M}_{\ell}(X,H;{\alpha_-},{\alpha_+})$  
such that for any 
${\bf q} \in W({\bf p})$
there exists $\vec w^{\bf q}_{{\bf p}}$ {\rm uniquely} with the following
properties:
\begin{enumerate}
\item
${\bf q} \cup \vec w^{\bf q}_{{\bf p}}$ is $\epsilon_{\bf p}$-close to
${\bf p} \cup w_{{\bf p}} \cup \vec w_{{\rm can}}$.
\item
${\bf q} \cup \vec w^{\bf q}_{{\bf p}}$ satisfies the
transversal constraint. (Definition \ref{constrainttt2}.)
\item
The linearization operator 
$D_{\bf q} \overline{\partial}_{J,H}$ at $\bf q$ in \eqref{eq:linearized} 
is surjective $\mod \oplus_{\rm v}\operatorname{Im}~ I_{{\bf p},{\rm v};{\bf q}}$, where $I_{{\bf p},{\rm v};{\bf q}}$ is the map in Definition \ref{defn2642} for the case 
${\bf q}=(\Sigma',u',\varphi')$.
\end{enumerate}
\end{lem}
\begin{proof}
If $w_{{\bf p},i} \in \vec w_{\bf p}$, the unique existence of
$w_{{\bf p},i}^{\bf q}$ satisfying Definition \ref{constrainttt2} (1)
can be proved in the same way as in 
\cite[Lemma 9.9]{fooo:const1}.
\par
In case $w_{{\bf p},i}^{\bf q}$ corresponds to a canonical marked point, 
the unique existence of
$w_{{\bf p},i}^{\bf q}$ satisfying Definition \ref{constrainttt2} (2)(3)
is a consequence of the following two facts:
\begin{enumerate}
\item[(i)]
$u_{\bf q}$ is $C^1$ close to $u_{\bf p}$.
\item[(ii)]
The first derivative of the function $f_{H,u,\Sigma_a}$ in \eqref{eq:41} is strictly positive at $\tau_i$.
Here $\varphi_a(\tau_i,0)$ is the canonical marked point on $\bf p$
which corresponds to $w_{{\bf p},i}^{\bf q}$.
\end{enumerate}
Furthermore by taking $W(\bf p)$ small enough, the property (3) is satisfied 
because the surjectivity is an open condition.
\end{proof}
Then we make the following choices.
\begin{choi}\label{choice2650}
We fix $\ell$, $\alpha_-$, $\alpha_+$.
\par\noindent
$\bullet$
We take a finite set
$${\EuScript A}_{\ell}(\alpha_-,\alpha_+)
= \{{\bf p}_c \mid c \in \EuScript C_{\ell}(\alpha_-,\alpha_+) \}
\subset {\mathcal M}_{\ell}(X,H;{\alpha_-},{\alpha_+}).
$$
Here $\EuScript C_{\ell}(\alpha_-,\alpha_+)$ is an index set which will be taken as in 
Remark \ref{rem:order}.
\par\noindent
$\bullet$ For each $c \in \EuScript C_{\ell}(\alpha_-,\alpha_+)$
we take obstruction bundle data $\frak E_{{\bf p}_{c}}$
centered at ${\bf p}_c$.
\par\noindent
$\bullet$
For each $c \in \EuScript C_{\ell}(\alpha_-,\alpha_+)$
we take a closed neighborhood $W({\bf p}_c)$ of ${\bf p}_c$
in ${\mathcal M}_{\ell}(X,H;{\alpha_-},{\alpha_+})$
with the following property. For any ${\bf q} \in W({\bf p}_c)$
there exists $\vec w^{\bf q}_{{\bf p}_c}$ such that
${\bf q} \cup \vec w^{\bf q}_{{\bf p}_c}$ is $\epsilon_c$-close
to ${\bf p}_c \cup w_{{\bf p}_c} \cup \vec w_{{\rm can}}$.
Here $\epsilon_c >0$ depends on $c$,
which will be determined later.
Moreover the linearization operator 
$D_{\bf q} \overline{\partial}_{J,H}$ in \eqref{eq:linearized} is surjective 
$\mod \oplus_{\rm v} {\rm Im} \,\, I_{{\bf p}_c,{\rm v};{\bf q}}$
where ${\rm Im} \,\, I_{{\bf p}_c,{\rm v};{\bf q}}$
is the map in Definition \ref{defn2642} for the case 
${\bf p}={\bf p}_c$, ${\bf q}=(\Sigma',u',\varphi')$.
\par\noindent
$\bullet$
We require
\begin{equation}\label{coversuru}
\bigcup_{c\in \EuScript C_{\ell}(\alpha_-,\alpha_+) } 
{\rm Int}\,\,W({\bf p}_c)
= {\mathcal M}_{\ell}(X,H;{\alpha_-},{\alpha_+}).
\end{equation}
\end{choi}
\begin{rem}\label{rem:order}
The logical order to make such choices is as follows.
First for each ${\bf p} \in {\mathcal M}_{\ell}(X,H;{\alpha_-},{\alpha_+})$
we take obstruction bundle data $\frak E_{{\bf p}}$ by Lemma \ref{lem2635}.
We take $\epsilon_{\bf p}>0$ and a closed neighborhood $W({\bf p})$ as in 
Lemma \ref{lem2446}.
Then we have 
$$
\bigcup_{{\bf p}} 
{\rm Int}\,\,W({\bf p})
= {\mathcal M}_{\ell}(X,H;{\alpha_-},{\alpha_+}).
$$
Finally, by compactness of the moduli space, we can take a finite set 
$\EuScript C_{\ell}(\alpha_-,\alpha_+)$
such that 
$c \in \EuScript C_{\ell}(\alpha_-,\alpha_+)$ satisfies the properties in Choice \ref{choice2650}.
\end{rem}
\begin{defn}\label{defn2646}
\begin{enumerate}
\item
For each ${\bf q} \in {\mathcal M}_{\ell}(X,H;{\alpha_-},{\alpha_+})$
we put
$$
\EuScript E({\bf q}) = \{ c \in \EuScript C_{\ell}(\alpha_-,\alpha_+) \mid {\bf q} \in W({\bf p}_c)\}.
$$
\item
Let $\EuScript B \subset \EuScript E({\bf q})$ be a nonempty subset. 
\item
We consider $(\frak Y\cup\bigcup_{c\in \frak B}\vec w'_c, u', \varphi')$
such that for each $c$,
$(\frak Y\cup \vec w'_c,u',\varphi')$ is $\epsilon$-close to ${\bf q} \cup \vec w_{c}^{\bf q}$.
If $\epsilon >0$ is small, then $(\frak Y\cup \vec w'_c,u',\varphi')$ is $\epsilon$-close to
${\bf p}_c \cup \vec w_{c}$
and we can define the map
$$
I_{{\bf p}_c,{\rm v};\Sigma',u',\varphi'} : E_{{\bf p}_c,{\rm v}}(\frak y_{{\bf p}_c}) \to
C^{\infty}(\Sigma';(u')^*TX \otimes \Lambda^{0,1})
$$
in Definition \ref{defn2642} for each irreducible component ${\rm v}$ of ${\bf p}_c$.
Here
$(\frak Y\cup \vec w'_c,\varphi')=  {\Phi}_{{\bf p}_{c}}(\frak y_{{\bf p}_c},\vec T_{{\bf p}_c},\vec \theta_{{\bf p}_c})$
and $\Sigma'$ is the source curve of $\frak Y$ as in Notation \ref{nota:415}.\footnote{As we noticed in Notation \ref{nota:415}, $\frak y_{{\bf p}_c}$ etc depend on the choice of ${\bf p}_{c}$.
Here we put the suffix $c \in \EuScript B$ in the notations in order to remember the dependence.} 
We now put
\begin{equation}\label{obspacedefHFHF}
E((\frak Y\cup \bigcup_{c\in \EuScript B}\vec w'_c,u',\varphi');{\bf q};\EuScript B)
=
\bigoplus_{c \in \EuScript B}\bigoplus_{\rm v} {\rm Im} \,\, I_{{\bf p}_c,{\rm v};\Sigma',u',\varphi'}.
\end{equation}
\end{enumerate}
\end{defn}
We can perturb $E_{{\bf p}_c,{\rm v}}(\frak y_{{\bf p}_c})$ slightly
so that the right hand side of (\ref{obspacedefHFHF}) is a direct sum.
(See \cite[Lemma 11.7]{fooo:const1}, \cite[Lemma 18.8]{foootech}.)
\subsection{Kuranishi chart}
In this subsection we construct a Kuranishi chart 
for any ${\bf q} \in {\mathcal M}_{\ell}(X,H;{\alpha_-},{\alpha_+})$.
We refer \cite[Definition 3.1]{fooonewbook} for the definition of Kuranishi chart.
\begin{defn}\label{stabilizationdefdata}
{\it Stabilization data centered at ${\bf q} \in {\mathcal M}_{\ell}(X,H;{\alpha_-},{\alpha_+})$} are 
$$
\Big( \vec w, 
\{\mathcal V(\frak x_{\rm v} \cup \vec w_{\rm v}\cup \vec w_{{\rm can},{\rm v}})\}_{\rm v},
\{(\psi_{\rm v}, \phi_{\rm v})\}_{\rm v},
\{\mathcal D_i\}_{w_i \in \vec w}
\Big)
$$
in Definition \ref{obbundeldata1} (1)(2)(3) and (7), 
which are sub-data of the obstruction bundle data at ${\bf q}$.
\end{defn}
\begin{shitu}\label{situ2648}
Suppose we are in the situation of Definition \ref{defn2646}.
We also take stabilization data at ${\bf q}$.
Let $(\frak Y\cup\bigcup_{c\in \EuScript E({\bf q})}\vec w'_c, u',\varphi')$ be as in Definition \ref{defn2646} (3)
and $\vec w'_{\bf q}$ be additional marked points on $\frak Y$ such that
$(\frak Y \cup \vec w'_{\bf q},\varphi')$ is $\epsilon$-close to ${\bf q} \cup \vec w_{\bf q} \cup \vec w_{{\bf q},{\rm can}}$.
Here $\vec w_{\bf q}$ is the additional marked points on $\frak Y$ taken as a part of the stabilization data centered at ${\bf q}$
and $\vec w_{{\bf q},{\rm can}}$ are canonical marked points
we put on the mainstream component without sphere bubble or marked points. 
$\blacksquare$
\end{shitu}
\begin{defn}\label{defn2650}
In Situation \ref{situ2648} we consider the following conditions on
an object
$(\frak Y\cup\bigcup_{c\in \EuScript E({\bf q})}\vec w'_c \cup \vec w'_{\bf q},u', \varphi')$:
\begin{enumerate}
\item
If $\Sigma_a$ is the mainstream component and $\varphi'_a$ is a parametrization
of this  mainstream component (which is a part of given $\varphi'$), the following equation is satisfied on
$\R \times S^1$.
\begin{equation}\label{Fleqobstincl}
\aligned
&\frac{\partial(u'\circ \varphi'_a)}{\partial \tau} 
+  J \left( \frac{\partial(u'\circ \varphi'_a)}{\partial t} - \frak X_{H_t}
\circ u'  \circ \varphi'_a\right) \\
& \equiv 0 \mod 
E((\frak Y\cup \bigcup_{c\in \EuScript B} \vec w'_c,u',\varphi');{\bf q};
\EuScript B).
\endaligned
\end{equation}
\item
If ${\rm v}$ is a bubble component, the following equation is satisfied on $\Sigma'_{\rm v}$.
\begin{equation}
\overline{\partial} u' \equiv 0  
\mod E((\frak Y \cup \bigcup_{c\in \EuScript B}\vec w'_c,u',\varphi');{\bf q};\EuScript B).
\end{equation}
\item For each $c \in \EuScript E({\bf q})$ the additional marked points $\vec w'_c$
satisfy the transversal constraint in Definition \ref{constrainttt2} with respect to 
the obstruction bundle data $\frak E_{{\bf p}_c}$ centered at ${\bf p}_c$.
(Namely, for each $c \in \EuScript E({\bf q})$ 
$(\frak Y, u',\varphi') \cup \vec w'_c$ satisfies the transversal constraint in Definition \ref{constrainttt2}.)
\item The additional marked points $\vec w'_{\bf q}$
satisfy the transversal constraint in Definition \ref{constrainttt2} with respect to 
the stabilization data centered at ${\bf q}$
in Situation \ref{situ2648}.
(Namely, $(\frak Y, u',\varphi') \cup \vec w'_{\bf q}$ also satisfies the transversal constraint in Definition \ref{constrainttt2}.)
\footnote
{We take the stabilization data for ${\bf q}$ 
in Situation \ref{situ2648}.
This is enough to define the transversal constraint.}
\item
$(\frak Y\cup \bigcup_{c\in \EuScript E({\bf q})}\vec w'_c \cup \vec w'_{\bf q}, u',\varphi')$ is
$\epsilon_1$-close to
${\bf q} \cup \bigcup_{c\in \EuScript E({\bf q})}\vec w^{\bf q}_c \cup \vec w_{\bf q}$.
\end{enumerate}
We define 
$$
V({\bf q},\epsilon_1,\EuScript B)
$$
to be the set of isomorphism classes of
$(\frak Y\cup\bigcup_{c\in \EuScript E({\bf q})}\vec w'_c \cup \vec w'_{\bf q},u', \varphi')$ satisfying the conditions
(1)--(5) above. 
Here $(\frak Y\cup\bigcup_{c\in \EuScript E({\bf q})}\vec w'_c \cup \vec w'_{\bf q},u',\varphi')$
is said to be {\it isomorphic} to $(\frak Y''\cup\bigcup_{c\in \EuScript E({\bf q})}\vec w''_c \cup \vec w''_{\bf q},u'',
\varphi'')$
if there exists a biholomorphic map $v : \Sigma' \to \Sigma''$ with the following properties.
\begin{enumerate}
\item[(a)]
$u'' = u'\circ v$ holds outside the set of the transit points.
\item[(b)]
If $\Sigma'_a$ is a mainstream component of $\Sigma'$
and $v(\Sigma'_a) = \Sigma''_{a'}$, then we have
$
(v \circ \varphi'_a)(\tau,t) = \varphi''_{a'}(\tau+\tau_a,t)
$
on $\R \times S^1$ where $\tau_a \in \R$ is independent of $(\tau,t)$.
\item[(c)]
$v(z'_i) = z''_i$ and $v(w'_i) = w''_i$.
\end{enumerate}
\end{defn}
\begin{lem}\label{lem2653}
If $\epsilon_1 >0$ and $\epsilon_c >0$ are small enough, then $V({\bf q},\epsilon_1,\EuScript B)$
is a smooth manifold with boundary. Its dimension is
\begin{equation}
\dim {\mathcal M}_{\ell}(X,H;{\alpha_-},{\alpha_+})
+
\sum_{c\in \EuScript B} \sum_{{\rm v} \in {\rm Irr}({\bf p}_c)} \rank E_{{\bf p}_c,{\rm v}}(\frak y_{{\bf p}_c}).
\end{equation}
Here $ {\rm Irr}({\bf p}_c)$ is the set of irreducible components of ${\bf p}_c$.
\par
If ${\bf q}$ has $k+1$ mainstream components, then the element
$[{\bf q} \cup \bigcup_{c\in {\EuScript E}({\bf q})}\vec w^{\bf q}_c \cup \vec w_{\bf q}]$ of  $V({\bf q},\epsilon_1,\EuScript B)$
is in a codimension $k$ corner.
\end{lem}
\begin{proof}
We first consider the set of isomorphism classes of
$(\frak Y\cup\bigcup_{c\in \EuScript E({\bf q})}\vec w'_c \cup \vec w'_{\bf q},u', \varphi')$
satisfying Definition \ref{defn2650} (1)(2)(5) and denote it by
$\widehat V({\bf q},\epsilon_1,\EuScript B)$.\footnote{A similar moduli space appeared in \cite[Definition 18.15]{foootech}
and was called the thickened moduli space.
The Hamiltonian term
$\frak X_{H_t} \circ u'  \circ \varphi'_{\rm v}$ did not appear there.}
\par
We can prove that $\widehat V({\bf q},\epsilon_1,\EuScript B)$ is a smooth manifold with boundary and corner
in the same way as in \cite[Section 8]{foooanalysis}.
We use the map
\begin{equation}\label{form418rev0}
\aligned
{\Phi}_{\bf q} : \prod \mathcal V((\frak x_{\bf q})_{\rm v} \cup \vec w_{\bf q,\rm v}\cup \vec w_{{\bf q},{\rm can},{\rm v}}) \times& {D}(k;\vec T_0)
\times \prod_{j=1}^m \left( ((T_{0,j},\infty] \times S^1)/\sim \right)
\\
&\to {\mathcal M}_{\ell+\ell'+\ell''}(\text{source})
\endaligned\end{equation}
to work out the gluing analysis in \cite[Sections 5,6]{foooanalysis}. 
(The map (\ref{form418rev0}) is the same map as (\ref{form418}) except we use
the stabilization data at ${\bf q}$.)
(\ref{form418rev0}) parametrizes the source curve (plus marked points)
of elements of $\widehat V({\bf q},\epsilon_1,\EuScript B)$.
For each fixed source curve we can perform the gluing construction
as in \cite[Sections 5,6]{foooanalysis}\footnote{See \cite[Part 3]{foootech} (simple case), \cite[Section 19]{foootech} (the general case) for more detailed explanation if necessary.}
to find that $\widehat V({\bf q},\epsilon_1,\EuScript B)$ is a smooth manifold
strata-wise.
To obtain a smooth structure on the union of the strata, we use the exponential decay estimate
which we can prove in the same way as in 
\cite[Section 8]{foooanalysis}.\footnote{See 
\cite[Theorem 13.2]{foootech} (simple case) or \cite[Theorem 19.5]{foootech}
(general case) for more details.}
The way to use this exponential decay estimate is the same as in 
\cite[Sections 9 and 10]{fooo:const1}, \cite[Section 8]{foooanalysis}. 
\par
We note that the only difference for
the gluing analysis in the current situation is 
the presence of the Hamiltonian term 
$\frak X_{H_t}\circ u'  \circ \varphi'_{\rm v}$.
This term is also nonzero on the neck region where we glue two
solutions. 
However the Hamiltonian term is small in the exponential 
order on the neck region.
We can prove it easily by looking at the coordinate change 
from  $S^1 \times [0,\infty)$ to $D^2 \setminus \{0\}$.
Namely the derivatives of this map decays 
exponentially as the second component 
of the domain goes to infinity. 
(See for example \cite[Lemma 30.24]{foootech} for the detail.)
So it does not affect the argument here.
\par
Once we proved that $\widehat V({\bf q},\epsilon_1,\EuScript B)$
is a smooth manifold, we can
prove that
$V({\bf q},\epsilon_1,\EuScript B)$ is a smooth manifold
by the implicit function theorem.
In fact, Definition \ref{constrainttt2} (1) cuts out a smooth submanifold
transversally because of the implicit function theorem
\cite[Lemma 25.32]{fooonewbook}. (See also 
\cite[Section 20]{foootech}, especially Lemma 20.7.)
We can use the facts (i)(ii) appearing in the proof of Lemma \ref{lem2446}
to show that Definition \ref{constrainttt2} (2)(3) cut out a smooth submanifold
transversally.
We have thus proved that $V({\bf q},\epsilon_1,\EuScript B)$ is a 
smooth manifold.\footnote{In \cite[Sections 19,20,21]{foootech} we first cut out by the transversality
constraint strata-wise and then show that those strata-wise smooth structure
gives the smooth structure on the whole space. So the order
of the proof there is slightly different
from one we mention above. There is no mathematical issue on this point and we can do either way.
The order is changed only for the convenience of the exposition.}
\par
We assume that the source curve of ${\bf q}$ has exactly $k$
mainstream components.
Note that the space ${D}(k;\vec T_0)$ in (\ref{form418})
is a manifold with boundary. The point corresponding to the source curve
of $[{\bf q} \cup \bigcup_{c\in {\EuScript E}({\bf q})}\vec w^{\bf q}_c \cup \vec w_{\bf q}]$
corresponds to the codimension $k$ boundary point of ${D}(k;\vec T_0)$.
In fact, it corresponds to the point $(T,\infty,\dots,\infty)$  on the compactification of the image of the map
$(T_0,\dots,T_{k}) \mapsto (T_0,T_1-T_0,\dots,T_k-T_{k-1})$.
See the discussion right before Figure \ref{Fig23.1}.
Therefore $[{\bf q} \cup \bigcup_{c\in {\EuScript E}({\bf q})}\vec w^{\bf q}_c \cup \vec w_{\bf q}]$
is on the codimension $k$ boundary of  $V({\bf q},\epsilon_1,\EuScript B)$.
\par
The dimension formula follows from Lemma \ref{lem2612}.
\end{proof}
We note that the group ${\rm Aut}^+({\bf q})$ acts
on $V({\bf q},\epsilon_1,\EuScript B)$ since
the stabilization data are assumed to be preserved.
In particular, ${\rm Aut}({\bf q})$ acts on it.
We also note that by the condition in Definition \ref{obbundeldata1} (6)
this action is effective.
Therefore the quotient space
$V({\bf q},\epsilon_1,\frak B)/{\rm Aut}({\bf q})$
is an effective orbifold,
which we denote by $U({\bf q},\epsilon_1,\EuScript B)$.
\par
We define a vector bundle 
$$
E({\bf q},\epsilon_1,\EuScript B) \to U({\bf q},\epsilon_1,\EuScript B)
$$ 
whose fiber at
$(\frak Y\cup \bigcup_{c\in \EuScript E({\bf q})}\vec w'_c \cup \vec w'_{\bf q},u', \varphi')$ is
$E((\frak Y\cup \bigcup_{c\in \EuScript B}\vec w'_c,u',\varphi');{\bf q};\EuScript B)$.
We define its section $s_{({\bf q},\epsilon_1,\EuScript B)}$ by
$$
\aligned
& s_{({\bf q},\epsilon_1,\EuScript B)}(\frak Y\cup\bigcup_{c\in \EuScript E({\bf q})}\vec w'_c \cup \vec w'_{\bf q},u',
\varphi') \\
= &
\begin{cases}
\overline{\partial} u'
&\text{on a bubble component $\Sigma_{\rm v}$,} \\
\frac{\partial(u'\circ \varphi'_{\rm v})}{\partial \tau}
+  J \left( \frac{\partial(u'\circ \varphi'_{\rm v})}{\partial t} - \frak X_{H_t}
\circ u'  \circ \varphi'_{\rm v}\right)
&\text{on a mainstream component $\Sigma_{\rm v}$}.
\end{cases}
\endaligned
$$
Note that the right hand side is an element of
$E({\bf q},\epsilon_1,\EuScript B)$ by
Definition \ref{defn2650}.
\par
By definition if
$s_{({\bf q},\epsilon_1,\EuScript B)}(\frak Y\cup\bigcup_{c\in \EuScript E({\bf q})}\vec w'_c 
\cup \vec w'_{\bf q}, u', \varphi') = 0$, 
then $(\frak Y,u',\varphi')$ represents
an element of ${\mathcal M}_{\ell}(X,H;{\alpha_-},{\alpha_+})$.
We thus obtain a map
$$
\psi_{({\bf q},\epsilon_1,\EuScript B)}
: s_{({\bf q},\epsilon_1,\EuScript B)}^{-1}(0)
\to  {\mathcal M}_{\ell}(X,H;{\alpha_-},{\alpha_+}).
$$
\begin{prop}\label{lem2654}
If $\epsilon_1>0$ is small, then
$(U({\bf q},\epsilon_1,\EuScript B),E({\bf q},\epsilon_1,\EuScript B),
s_{({\bf q},\epsilon_1,\EuScript B)},\psi_{({\bf q},\epsilon_1,\EuScript B)})$
is a Kuranishi chart of ${\mathcal M}_{\ell}(X,H;{\alpha_-},{\alpha_+})$
at ${\bf q}$.
\end{prop}
\begin{proof}
Taking into account of the point mentioned in the proof of
Lemma \ref{lem2653},
the proof is the same as in \cite[Section 8]{foooanalysis}.
\end{proof}
\begin{lem}\label{lem2655}
We assume that ${\bf q} \in S_k({\mathcal M}_{\ell}(X,H;{\alpha_-},{\alpha_+}))$.
Then $S_k(V({\bf q},\epsilon_1,\EuScript B))$ is the set of
$(\frak Y\cup\bigcup_{c\in \EuScript E({\bf q})}\vec w'_c \cup \vec w'_{\bf q},u', \varphi')
\in V({\bf q},\epsilon_1,\EuScript B, u)$
such that $\frak Y$ has at least $k+1$ mainstream components.
\end{lem}
\begin{proof}
As explained in the proof of Lemma \ref{lem2653}, 
the codimension $k$ corner of
${\mathcal M}_{\ell}(X,H;{\alpha_-},{\alpha_+})$
corresponds to the
codimension $k$ corner of the ${D}(k;\vec T_0)$ factor of the left hand side of
(\ref{form418}).
Lemma \ref{lem2655} immediately follows from this fact.
\end{proof}
We also observe the following fact.
\begin{lem}\label{lem2656}
If ${\bf q} \in S_k({\mathcal M}_{\ell}(X,H;{\alpha_-},{\alpha_+}))$
and ${c} \in \EuScript E({\bf q})$,
then we have
${\bf p}_c \in  S_k({\mathcal M}_{\ell}(X,H;{\alpha_-},{\alpha_+}))$.
\end{lem}
\begin{proof}
This follows from the following fact.
If ${\bf q}$ is $\epsilon$-close to ${\bf p}$, then
the number of mainstream components of ${\bf q}$ is not
greater than
the number of mainstream components of ${\bf p}$.
\end{proof}
\subsection{Coordinate change}
Next we discuss coordinate changes of the Kuranishi charts.
We refer \cite[Definition 3.6]{fooonewbook} for the definition of coordinate changes.
\begin{lem}\label{lem2657}
For each ${\bf q}_1 \in {\mathcal M}_{\ell}(X,H;{\alpha_-},{\alpha_+})$
there exists $\epsilon_1 >0$ such that the following holds.
Suppose ${\bf q}_2 \in {\rm Im} (\psi_{({\bf q}_1,\epsilon_1,\EuScript B)})$, then 
\begin{enumerate}
\item
$\EuScript E({\bf q}_2) \subseteq \EuScript E({\bf q}_1)$.
\item
Let $\EuScript B_2 \subseteq  \EuScript E({\bf q}_2)$,
$\EuScript B_1 \subseteq  \EuScript E({\bf q}_1)$
and $\EuScript B_2 \subseteq \EuScript B_1$.
Then there exists $\epsilon_2 >0$
such that there exists a coordinate change
from
$$
(U({\bf q}_2,\epsilon_2,\EuScript B_2),E({\bf q}_2,\epsilon_2,\EuScript B_2),
s_{({\bf q}_2,\epsilon_2,\EuScript B_2)},\psi_{({\bf q},\epsilon_2,\EuScript B_2)})
$$
to
$$
(U({\bf q}_1,\epsilon_1,\EuScript B_1),E({\bf q}_1,\epsilon_1,\EuScript B_1),
s_{({\bf q}_1,\epsilon_1,\EuScript B_1)},\psi_{({\bf q},\epsilon_1,\EuScript B_1)}).
$$
\end{enumerate}
\end{lem}
\begin{proof}
The proof is the same as \cite[Sections 9 and 10]{fooo:const1}
\cite[Subsection 8.3]{foooanalysis}.
\end{proof}
\begin{lem}\label{lem2658}
We may choose $\epsilon_1>0$ and $\epsilon_2>0$ in
Lemma \ref{lem2657} such that the following holds.
\begin{enumerate}
\item
When we replace ${\bf q}_1$, $\epsilon_1$ by ${\bf q}_2$,
$\epsilon_2$ in Lemma \ref{lem2657},
the same conclusion as Lemma \ref{lem2657} holds.
\item
If ${\bf q}_3 \in {\rm Im} (\psi_{({\bf q}_2,\epsilon_2,\EuScript B_2)})$,
then there exists $\epsilon_3 >0$ such that
we have the conclusion of
Lemma \ref{lem2657}  with
${\bf q}_1$, $\epsilon_1$, ${\bf q}_2$
and $\epsilon_2$
replaced by
${\bf q}_2$, $\epsilon_2$, ${\bf q}_3$
and $\epsilon_3$, respectively.
\par
We also
have the conclusion of
Lemma \ref{lem2657}  with
${\bf q}_1$, $\epsilon_1$, ${\bf q}_2$
and $\epsilon_2$
replaced by
${\bf q}_1$, $\epsilon_1$, ${\bf q}_3$
and $\epsilon_3$, respectively.
\item
Let $(i,j)$ be one of $(i,j) = (1,2), (2,3), (1,3)$ and
$\Phi_{{\bf q}_i,{\bf q}_j}$ be the coordinate change
from
$$
(U({\bf q}_j,\epsilon_j,\EuScript B_j),E({\bf q}_j,\epsilon_j,\EuScript B_j),
s_{({\bf q}_j,\epsilon_j,\EuScript B_j)},\psi_{({\bf q},_j\epsilon_j,\EuScript B_j)})
$$
to
$$
(U({\bf q}_i,\epsilon_i,\EuScript B_i),E({\bf q}_i,\epsilon_i,\EuScript B_i),
s_{({\bf q}_i,\epsilon_i,\EuScript B_i)},\psi_{({\bf q},_i\epsilon_i,\EuScript B_i)})
$$
obtained by Lemma \ref{lem2657}.
Then we have
$$
\Phi_{{\bf q}_1,{\bf q}_3} = \Phi_{{\bf q}_1,{\bf q}_2} \circ \Phi_{{\bf q}_2,{\bf q}_3}.
$$
\end{enumerate}
\end{lem}
\begin{proof}
The proof is the same as in \cite[Section 7]{fooo:const1}.
\end{proof}

We can use Lemmas \ref{lem2654}, \ref{lem2657}
and \ref{lem2658} to construct
a required Kuranishi structure on
${\mathcal M}_{\ell}(X,H;{\alpha_-},{\alpha_+})$
in exactly the same way as in \cite{fooo:const1}.
\par 
The isomorphism on the orientation bundle stated in 
Theorem \ref{kuraexists} (1) can be proved as follows.
Since the marked points are parametrized by complex coordinates, hence even dimensional and canonically oriented,  
it is enough to consider the case that $\ell = 0$.  
The  fiber of the orientation bundle of $\mathcal{M}(X, H; \alpha_-, \alpha_+)$ 
at ${\bf p} \in \mathcal{M}(X, H; \alpha_-, \alpha_+)$
is defined by 
\begin{equation} \label{orionM}
o_{\mathcal{M}(X, H; \alpha_-, \alpha_+)}|_{\bf p} \otimes {\mathbb R}_{\alpha_-, \alpha_+} 
= \det T_{\gamma_+} R_{\alpha_+}  \otimes \det D_{\bf p} \overline{\partial}_{J,H} \otimes \det T_{\gamma_-}R_{\alpha_-},  
\end{equation} 
where $D_{\bf p} \overline{\partial}_{J,H} $ is the linearized operator at $\mathbf p$ in \eqref{eq:linearized} with fixed asymptotics $\gamma_{\pm} \in \overline{R}_{\alpha_{\pm}}$ and 
${\mathbb R}_{\alpha_-, \alpha_+}$ stands for the translation action on $\widetilde{\mathcal M}(X, H; \alpha_-, \alpha_+)$.  
Since the linearization operator \eqref{linearizationJholo} 
for bubble components is a complex linear Fredholm operator and the matching condition $(\heartsuit)$ is taken in 
complex vector spaces, we concentrate on the linearization operator \eqref{ducomponentssphererevrev}
for mainstream components with fixed asymptotics $\gamma_{\pm}$ here.    

For $D_{\bf p} \overline{\partial}_{J,H}$, we have 
\begin{equation} \label{relorisystem}
\det D_{\bf p} \overline{\partial}_{J,H} \otimes \det TR_{\alpha_-} \otimes \det P(\gamma_-;{\mathfrak t}_{w_-}) \cong \det P(\gamma_+; t_{w_+}).  
\end{equation}
Fixing\footnote{See \cite[Convention 8.3.1]{fooobook2}.} an orientation on ${\mathbb R}$ once and for all, we can drop 
the factor ${\mathbb R}_{\alpha_-, \alpha_+}$ from the left hand side of \eqref{orionM}.   
(This factor is important for (16.8) in \cite[Condition X]{fooonewbook} as we will see below.)
Combining \eqref{orionM} and \eqref{relorisystem}, we obtain 
\begin{equation}
\det T_{\gamma_+} R_{\alpha_+} \otimes \det P(\gamma_+; t_{w_+}) \cong o_{\mathcal{M}(X, H; \alpha_-, \alpha_+)}|_{\bf p} \otimes \det P(\gamma_-;{\mathfrak t}_{w_-}).  
\end{equation}
Recalling the definition of $o_{R_{\alpha}}$ from Definition \ref{oricricial}, it is the desired orientation isomorphism.  
\par
Next we prove  (16.8) in \cite[Condition X]{fooonewbook}.  
Suppose that $\mathbf p \in {\mathcal M}(X,H;\alpha_-, \alpha_+)$ decomposes into $({\mathbf p}_1, {\mathbf p}_2) \in  {\mathcal M}(X,H;\alpha_-, \alpha_+) \times_{R_{\alpha}}  {\mathcal M}(X,H;\alpha_-, \alpha)$.  Then we find that 
\begin{equation}\label{oribreak}
\det D_{\mathbf p} \overline{\partial}_{J,H} \cong \det D_{{\mathbf p}_2} \overline{\partial}_{J,H}  \otimes \det T_{\gamma}R_{\alpha} \otimes D_{{\mathbf p}_1} \overline{\partial}_{J,H},   
\end{equation}
where $\gamma$ is the positive asymptotic limit of ${\mathbf p}_1$ and the negative asymptotic limit of ${\mathbf p}_2$.  
Then we have the following 
\begin{equation}
\aligned
& o_{{\mathcal M}(X,H;\alpha_-, \alpha_+)} |_{\mathbf p} \otimes {\mathbb R}_{\alpha_-, \alpha_+} \\
\cong & \det T_{\gamma_+}R_{\alpha_+} \otimes \det D_{{\mathbf p}}  \overline{\partial}_{J, H} 
\otimes \det T_{\gamma_-}R_{\alpha_-} \\
\cong & \det T_{\gamma_+}R_{\alpha_+} \otimes \det D_{{\mathbf p}_2} \overline{\partial}_{J,H}  \otimes T_{\gamma}R_{\alpha} \otimes \det D_{{\mathbf p}_1} \overline{\partial}_{J,H}  \otimes \det T_{\gamma_-}R_{\alpha_-} \\
\cong & o_{{\mathcal M}(X,H;\alpha_-, \alpha_+)} |_{{\mathbf p}_2} \otimes {\mathbb R}_{\alpha, \alpha_+} \otimes \det D_{{\mathbf p}_1}  \overline{\partial}_{J,H} \otimes \det T_{\gamma_-}R_{\alpha_-} \\
\cong & o_{{\mathcal M}(X,H;\alpha, \alpha_+)} |_{{\mathbf p}_2} \otimes {\mathbb R}_{\alpha, \alpha_+}  \otimes (\det T_{\gamma}R_{\alpha})^* \otimes o_{{\mathcal M}(X,H;\alpha_-, \alpha)} |_{{\mathbf p}_1} \otimes {\mathbb R}_{\alpha_-, \alpha} \\
\cong &  (-1)^{\dim {\mathcal M}(X,H;\alpha, \alpha_+)} {\mathbb R}_{\alpha, \alpha_+} \otimes o_{{\mathcal M}(X,H;\alpha, \alpha_+) \times_{R_{\alpha}} {\mathcal M}(X,H;\alpha_-, \alpha)} \otimes {\mathbb R}_{\alpha_-, \alpha} \\
\cong &  (-1)^{\dim {\mathcal M}(X,H;\alpha, \alpha_+)} {\mathbb R}_{\rm out} \otimes o_{{\mathcal M}(X,H;\alpha, \alpha_+) \times_{R_{\alpha}} {\mathcal M}(X,H;\alpha_-, \alpha)} \otimes {\mathbb R}_{\alpha_-, \alpha_+}.
\endaligned
\end{equation}
Here ${\mathbb R}_{\rm out}$ is the outward normal direction of $\partial {\mathcal M}(X,H;\alpha_-, \alpha_+)$. 
The first, third and fourth isomorphisms follows from \eqref{orionM}.  The second follows from \eqref{oribreak}.  
The fifth is due to the definition of the fiber product orientation on K-spaces, 
see \cite[Convention 8.2.1 (4)]{fooobook2}.  
For the sixth isomorphism, see the proof of \cite[Proposition 8.3.3]{fooobook2}. 
Thus we have proved Theorem \ref{kuraexists} (1).
\par
To prove Theorem \ref{kuraexists} (2)
we can use Lemmas \ref{lem2655} and \ref{lem2656}
to show that 
the Kuranishi structure constructed above 
induces a Kuranishi structure on the codimension $k-1$ normalized corner
$\widehat S_{k-1}({\mathcal M}_{\ell}(X,H;{\alpha_-},{\alpha_+}))$
which is the disjoint union of the spaces
${\mathcal M}_{(\ell_1,\dots,\ell_k)}(X,H;\alpha_0,\alpha_1.,\dots,\alpha_k)$
for various
 $\alpha_-=\alpha_0,\alpha_1,\dots,\alpha_{k-1},\alpha_k = \alpha_+
\in \frak A$.
Thus we have Theorem \ref{kuraexists} (2).
\par
Theorem \ref{kuraexists} (3) follows from
Definition \ref{obbundeldata1} (6) and Remark \ref{rem:412} (1).
Theorem \ref{kuraexists} (4) is a consequence of Lemma \ref{lem2612}.
\par
Therefore the proof of Theorem \ref{kuraexists} is complete.
\qed
\begin{rem}
In \cite{fooo:const1} we introduced ambient `set'  
and used it to prove cocycle condition for the coordinate 
changes between Kuranishi charts. We can adopt the method 
in our situation as follows.
\par
We first define the set ${\mathcal X}_{\ell}(X,H;{\alpha_-},{\alpha_+})$
as the set of equivalence classes of 
$(\Sigma,(z_-,z_+,\vec z),u,\varphi)$ 
which satisfies the conditions in Definition \ref{defn210} 
except (4)(5)(6)(9). In other words, we do not require 
the condition that the equations 
(pseudo-holomorphic curve equation or Floer's equation) are 
satisfied or the stability.
The definition of two such objects being equivalent is the same as 
$\sim_2$ in Definition \ref{3equivrel}.
By definition ${\mathcal M}_{\ell}(X,H;{\alpha_-},{\alpha_+})$
is a subset of ${\mathcal X}_{\ell}(X,H;{\alpha_-},{\alpha_+})$.
For any ${\bf p}\in {\mathcal M}_{\ell}(X,H;{\alpha_-},{\alpha_+})$
we define its $\epsilon$-neighborhood in ${\mathcal X}_{\ell}(X,H;{\alpha_-},{\alpha_+})$
as the subset consisting of element which is $\epsilon$-close to ${\bf p}$
in the sense of Definition \ref{orbitecloseness}.
It then defines a partial topology of the pair 
$({\mathcal X}_{\ell}(X,H;{\alpha_-},{\alpha_+}),{\mathcal M}_{\ell}(X,H;{\alpha_-},{\alpha_+}))$
in the sense of \cite[Definition 4.1]{fooo:const1}.
(This is a consequence of Lemma \ref{lem2658}.)
Thus in the same way as in \cite[Section 7]{fooo:const1}
we obtain a Kuranishi structure.
\end{rem}

\section{Compatibility of Kuranishi structures}
\label{subsec;KuramodFloercor}

In this section we complete the proof of Theorem \ref{theorem266}.
Namely, we modify the Kuranishi structures in 
Theorem \ref{kuraexists} for the Morse-Bott case so that Kuranishi structures are compatible on boundary and corners under the fiber product. 
The argument presented here has not been given in detail for the moduli space 
of solutions to Floer's equation \eqref{Fleq} 
in the previous literature.
For the case of the moduli spaces of pseudo-holomorphic disks it is written in detail in \cite{fooo:const2}.
The method of \cite{fooo:const2} is different from that of this section.
We take a different route to illustrate two different methods. Both methods in \cite{fooo:const2} and 
in this section can be applied to both situations. 
\subsection{Outer collar}\label{subsec:outer}
The statements (1) (2) of Theorem \ref{theorem266} are a part of Theorem \ref{kuraexists}
and already proved in Section \ref{subsec;KuraFloer}.
To prove (3) we will introduce 
an enhanced space ${\mathcal M}_{\ell}(X,H;{\alpha_-},{\alpha_+})^{\boxplus 1}$
of ${\mathcal M}_{\ell}(X,H;{\alpha_-},{\alpha_+})$ by putting 
a collar `outside' of ${\mathcal M}_{\ell}(X,H;{\alpha_-},{\alpha_+})$ 
and modify the Kuranishi structure on the outer collar 
${\mathcal M}_{\ell}(X,H;{\alpha_-},{\alpha_+})^{\boxplus 1} \setminus {\mathcal M}_{\ell}(X,H;{\alpha_-},{\alpha_+})$ in 
the sense of \cite[Chapter 17]{fooonewbook}.
As a topological space 
we define the space ${\mathcal M}_{\ell}(X,H;{\alpha_-},{\alpha_+})^{\boxplus 1}$ 
as follows.\footnote{See \cite[Lemma-Definition 17.29]{fooonewbook}
for the definition of the outer collar for a general K-space.}
\begin{defn}\label{def:outcollar}
We consider $(\Sigma,(z_-,z_+,\vec z),u,\varphi;\vec t)$
where 
\begin{enumerate}
\item[$\bullet$]
$((\Sigma,(z_-,z_+,\vec z),u,\varphi) \in {\mathcal M}_{\ell}(X,H;{\alpha_-},{\alpha_+})$,
\item[$\bullet$]
$\vec t$ assigns a number $t_z \in [-1,0]$ to each transit point $z$ of $\Sigma$.
\end{enumerate}
We denote by ${\mathcal M}_{\ell}(X,H;{\alpha_-},{\alpha_+})^{\boxplus 1}$ 
the set of isomorphism classes of such objects 
$(\Sigma,(z_-,z_+,\vec z),u,\varphi;\vec t)$. 
We call it the {\it outer collaring} of 
${\mathcal M}_{\ell}(X,H;{\alpha_-},{\alpha_+})$.
\par
We say a sequence $(\Sigma^j,(z^j_-,z^j_+,\vec z^{\, j}),u^j,\varphi^j;\vec t^{\, j})
\in {\mathcal M}_{\ell}(X,H;{\alpha_-},{\alpha_+})^{\boxplus 1}$
converges to $(\Sigma,(z_-,z_+,\vec z),u,\varphi;\vec t)$
if the following holds.
\begin{enumerate}
\item
$(\Sigma^j,(z^j_-,z^j_+,\vec z^{\, j}),u^j,\varphi^j)$ 
converges to 
$(\Sigma,(z_-,z_+,\vec z),u,\varphi)$ 
in the sense of Definition \ref{def2626}. 
\item
Let $z$ be a transit point of $\Sigma$.
\begin{enumerate}
\item
Suppose $j_n \to \infty$ and there exists a sequence of transit points 
$z_n \in \Sigma^{j_n}$ which converges to $z$ in an obvious sense, 
then $t^{j_n}_{z_{j_n}}$ converges to $t_{z}$.
\item
If there is no such a sequence then $t_{z} = 0$.
\end{enumerate}
\end{enumerate}
It is easy to see that we can define a topology in this way.
\end{defn}
We can show that ${\mathcal M}_{\ell}(X,H;{\alpha_-},{\alpha_+})^{\boxplus 1}$ is 
compact and Hausdorff. 
The evaluation map also extends to the outer collar so that  
it does not depend on $\vec{t}$.
Moreover by inspecting the discussion in \cite[Chapter 17]{fooonewbook}
we can show that ${\mathcal M}_{\ell}(X,H;{\alpha_-},{\alpha_+})^{\boxplus 1}$ is the 
underlying topological space 
of the outer collared space of ${\mathcal M}_{\ell}(X,H;{\alpha_-},{\alpha_+})$
with respect to the Kuranishi structure defined in the last section.
(However, we do not use this fact in this paper.)
\par
Now the main part of this construction is Proposition \ref{prop2661}.
To state it we prepare some notations.
Let 
$$
\alpha_-=\alpha_0,\alpha_1,\dots,\alpha_{m-1},\alpha_m = \alpha_+
\in \frak A.
$$
In Section \ref{subsec;compactFloer}
we denoted by 
$
{\mathcal M}_{(\ell_1,\dots,\ell_m)}(X,H;\alpha_0,\alpha_1.,\dots,\alpha_m)
$
the fiber product (\ref{2622})
\begin{equation}\label{form2641}
\aligned
&{\mathcal M}_{\ell_1}(X,H;\alpha_0,\alpha_1)
\,{}_{{\rm ev}_{+}}\times_{{\rm ev}_-} {\mathcal M}_{\ell_2}(X,H;\alpha_1,\alpha_2)
\,{}_{{\rm ev}_{+}}\times_{{\rm ev}_-} \dots \\
&{}_{{\rm ev}_+}\times_{{\rm ev}_-} {\mathcal M}_{\ell_{i+1}}(X,H;\alpha_i,\alpha_{i+1})
\,{}_{{\rm ev}_{+}}\times_{{\rm ev}_-}
\dots
\times_{{\rm ev}_-} {\mathcal M}_{\ell_m}(X,H;\alpha_{m-1},\alpha_{m}).
\endaligned
\end{equation}
Hereafter we write (\ref{form2641}) as
$$
{\mathcal M}_{\vec \ell}(X,H; \vec \alpha).
$$
We will construct a certain Kuranishi structure on
$
{\mathcal M}_{\vec \ell}(X,H; \vec \alpha) \times [-1,0]^{m-1}.
$
\begin{defn}
We put $\underline{m-1} = \{1,\dots,m-1\}$.
Let 
$$
A \sqcup B  \sqcup C = \underline{m-1}
$$ 
be a decomposition into a disjoint union.
We put $B = \{j_1,\dots,j_{b}\}$.
We define an embedding
$$
\mathcal I_{A,B,C} : [-1,0]^{b} \to [-1,0]^{m-1}
$$
by
$\mathcal I_{A,B,C}(t_1,\dots,t_b) = (s_1,\dots,s_{m-1})$ where
\begin{equation}\label{form2642}
s_i =
\begin{cases}
-1 &\text{if $i \in A$,}\\
t_k &\text{if $i = j_k$, $(k=1,\dots,b)$,\,} \\
0  &\text{if $i \in C$}.
\end{cases}
\end{equation}
We put
$$
{\rm Part }_3(m-1) = \{(A,B,C) \mid A \sqcup B  \sqcup C = \underline{m-1}\}.
$$
\end{defn}
We note that the set of images of $\mathcal I_{A,B,C}$ for various $(A,B,C) \in {\rm Part}_3(m-1)$
coincides with the set of cells of the standard cell decomposition of $[-1,0]^{m-1}$.
\par
The compatibility condition we formulate below (Condition \ref{conds2660}) describes the restriction of the
Kuranishi structure
$\widehat{\mathcal U}_{\vec \ell}(X,H; \vec \alpha)$ on
the product space
$
{\mathcal M}_{\vec \ell}(X,H; \vec \alpha) \times [-1,0]^{m-1}
$
to the image of the embedding:
\begin{equation}\label{2642form}
{\rm id} \times \mathcal I_{A,B,C} :
{\mathcal M}_{\vec \ell}(X,H; \vec \alpha) \times [-1,0]^{b}
\to
{\mathcal M}_{\vec \ell}(X,H; \vec \alpha) \times [-1,0]^{m-1}.
\end{equation}
We need more notations.
We write elements of the set $A$ as 
$A = \{ i(A,1),\dots, i(A,a)\}$ with
$i(A,1)<i(A,2) < \dots< i(A,a-1) < i(A,a)$ and consider the fiber product
\begin{equation}\label{264222}
\aligned
&{\mathcal M}_{\vec \ell_{A,1}}(X,H;\alpha_0,\dots,\alpha_{ i(A,1)})
 \\
&
\,{}_{{\rm ev}_{+}}\times_{{\rm ev}_-} {\mathcal M}_{\vec \ell_{A,2}}(X,H;\alpha_{ i(A,1)},\dots,\alpha_{ i(A,2)})
\,{}_{{\rm ev}_{+}}\times_{{\rm ev}_-} \dots \\
&\,{}_{{\rm ev}_{+}}\times_{{\rm ev}_-} {\mathcal M}_{\vec \ell_{A,j+1}}(X,H;\alpha_{ i(A,j)},\dots,\alpha_{ i(A,j+1)})
\,{}_{{\rm ev}_{+}}\times_{{\rm ev}_-}
\dots \\
&
\,{}_{{\rm ev}_{+}}\times_{{\rm ev}_-} {\mathcal M}_{\vec\ell_{A,a+1}}(X,H;\alpha_{ i(A,a)},\dots,\alpha_{m}).
\endaligned
\end{equation}
Here
\begin{equation}\label{eq:54}
\vec\ell_{A,j}
=(\ell_{i(A,j-1)+1},\dots,\ell_{i(A,j)}),
\end{equation}
and
$i(A,0) = 0$ and $i(A,a+1) = m$ by convention.
Actually the fiber product (\ref{264222}) is nothing but ${\mathcal M}_{\vec \ell}(X,H; \vec \alpha)$.
Therefore we can use \eqref{264222} to define a fiber product Kuranishi structure
on ${\mathcal M}_{\vec \ell}(X,H; \vec \alpha)$.
\begin{notation}\label{not2659}
Let $(A,B,C) \in {\rm Part}_3(m-1)$ and $j=1,\dots,a+1$.
\begin{enumerate}
\item
We put
$$
{\mathcal M}_{\vec \ell}(X,H; \vec \alpha)^+
=
{\mathcal M}_{\vec \ell}(X,H; \vec \alpha) \times [-1,0]^{m-1},
$$
where $m+1$ is the number of components of $\vec\alpha$.
Note that ${\mathcal M}_{\vec \ell}(X,H; \vec \alpha)^{\boxplus 1}$ 
in Definition \ref{def:outcollar} is a union of ${\mathcal M}_{\vec \ell}(X,H; \vec \alpha)^+$ for various $\vec\alpha$.
\item For $a,b \in \Z$ we put $[a,b]_{\Z} = [a,b] \cap \Z$
and $(a,b)_{\Z} = [a,b]_{\Z} \setminus \{a,b\}$.
\item
We decompose $C$ into
$C'_j(A) = 
[i(A,j-1),i(A,j)]_{\Z} \cap C$
and put $C_j(A) = \{ i - i(A,j-1) \mid i \in C'_j(A)\}$,
$c_j(A) = \#C_j(A)$.
(Recall $i(A,0) = 0$ and $i(A,a+1) = m$ as above.)
\item
We also put 
$$
\vec\alpha_{A,j}
= (\alpha_{i(A,j-1)},\dots,\alpha_{i(A,j)}).
$$ 
We put $\alpha_0 = \alpha_-$ and $\alpha_m = \alpha_+$
by convention. 
We include them as elements of
$\vec\alpha_{A,0}$ and $\vec\alpha_{A,a+1}$, respectively.
Note that
\begin{equation}\label{formula26505050}
{\mathcal M}_{\vec\ell_{A,j}}(X,H;\alpha_{ i(A,j-1)},\dots,\alpha_{i(A,j)})
= {\mathcal M}_{\vec\ell_{A,j}}(X,H;\vec\alpha_{A,j}).
\end{equation}
See Figure \ref{Fig2}.
\begin{figure}[htbp]
\centering
\includegraphics[scale=0.5]{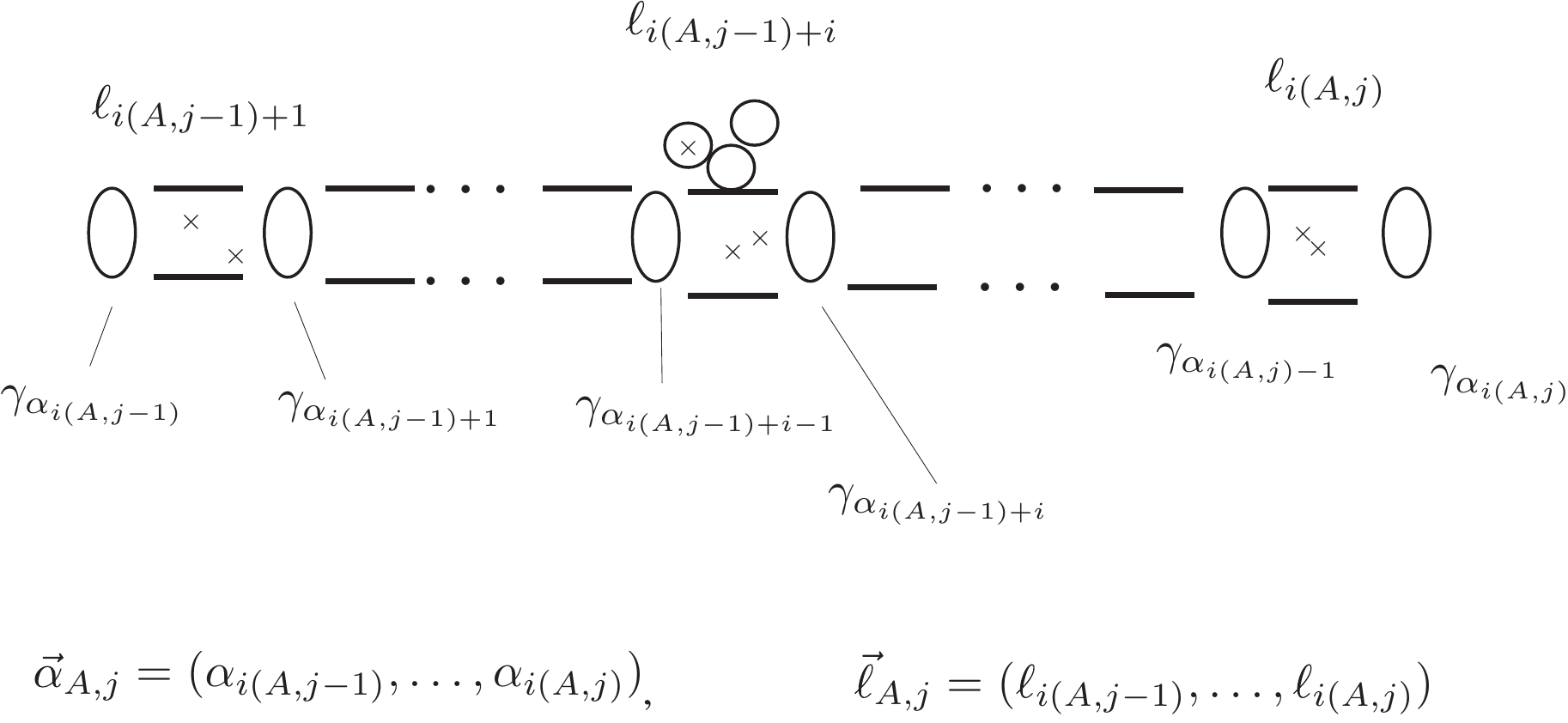}
\caption{${\mathcal M}_{\vec\ell_{A,j}}(X,H;\alpha_{ i(A,j-1)},\dots,\alpha_{ i(A,j)})
$}
\label{Fig2}
\end{figure}
\item
We remove $\{\alpha_i \mid i \in C'_j(A)\}$ from $\vec\alpha_{A,j}$ to
obtain $\vec\alpha_{A,j,C}$.
\item
We put $m_j(A) = i(A,j) - i(A,j-1)$ and 
$$
m_j(A,C) =  \# (B\cap  (i(A,j-1),i(A,j))_{\Z}) + 1.
$$
Then we have 
\begin{equation}\label{form264555}
\sum_{j=0}^{a+1} (m_j(A,C)-1)
= \#B=b.
\end{equation}
\item
We define $\vec \ell_{A,j,C}$ as follows.
Let
$\vec{\alpha}_{A,j,C}
= \{\alpha_{i(A,j-1) + k_s} \mid s =0, \dots, m_j(A,C)\}$.
Here $k_0 = 0 < k_1 < \dots < 
k_{m_j(A,C)}= i(A,j) - i(A,j-1)$.
Note if
$i \in ( i(A,j-1) + k_s,i(A,j-1) + k_{s+1})_{\Z}$, then $i \in C'_j(A)$.
We put
$$
\ell_{A,j,C,s} = \ell_{i(A,j-1)+k_{s-1} +1} + \dots
+ \ell_{i(A,j-1)+k_{s}}
$$
for $s=1,\dots , m_j(A,C)$ 
and
$\vec \ell_{A,j,C}
= (\ell_{A,j,C,1},\dots,\ell_{A,j,C,m_j(A,C)})$.
See Figure \ref{Fig3}.
\begin{figure}[htbp]
\centering
\includegraphics[scale=0.4]{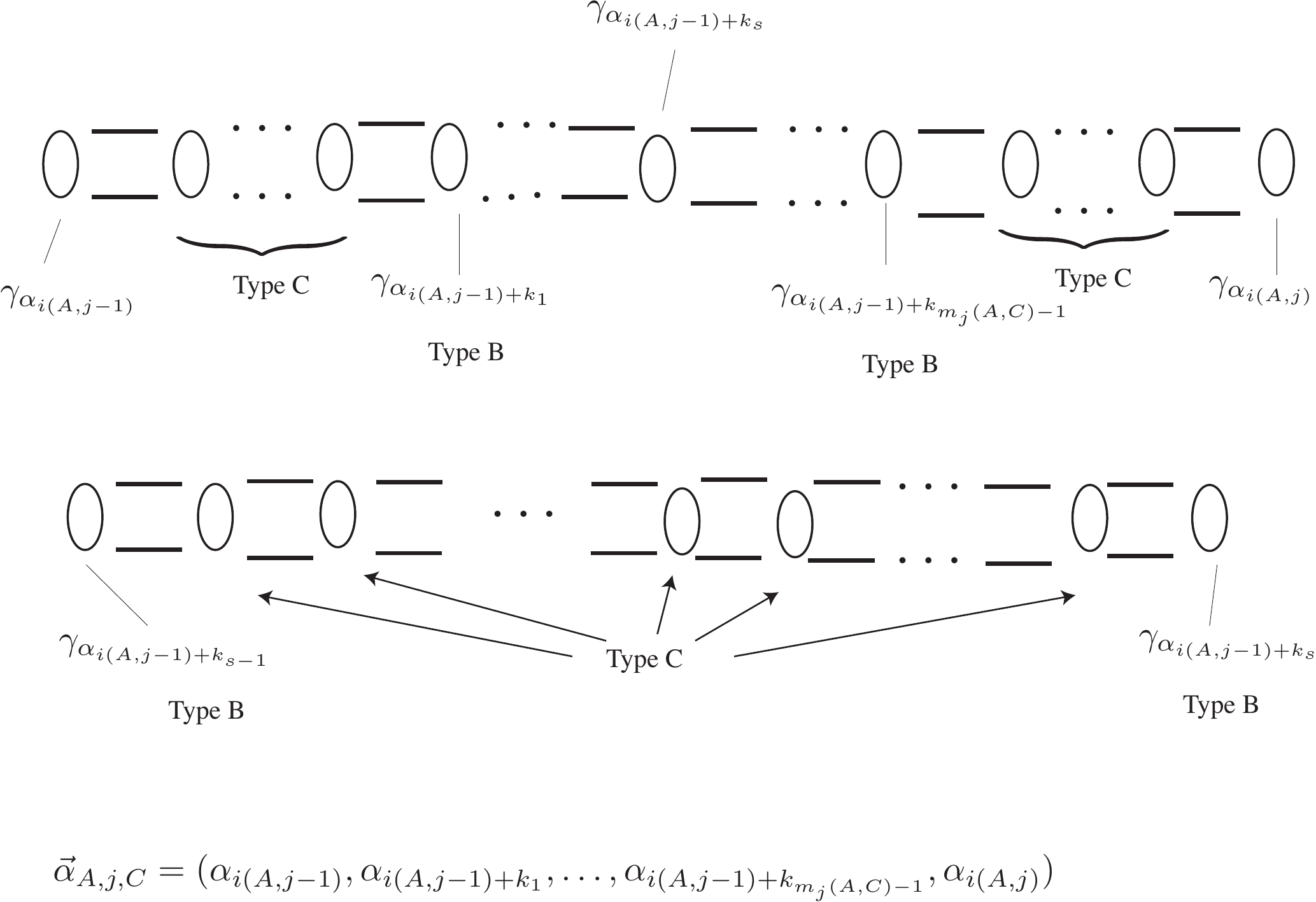}
\caption{$\vec{\alpha}_{A,j,C}$ and $\vec \ell_{A,j,C}$}
\label{Fig3}
\end{figure}
\end{enumerate}
\end{notation}
\par
Now the compatibility condition we require is described as follows.
\begin{conds}\label{conds2660}
The restriction of the Kuranishi structure
$\widehat{\mathcal U}_{\vec \ell}(X,H; \vec \alpha)$ on the space
$
{\mathcal M}_{\vec \ell}(X,H; \vec \alpha)^+
$
to the image of the embedding
(\ref{2642form})
coincides with the fiber product of the restrictions of the
Kuranishi structures $\widehat{\mathcal U}_{\vec \ell_{A,j,C}}(X,H; \vec \alpha_{A,j,C})$
to ${\mathcal M}_{\vec \ell_{A,j}}(X,H; \vec \alpha_{A,j})
\times [-1,0]^{m_j(A,C)-1}$.
\end{conds}
We note that 
$\widehat{\mathcal U}_{\vec \ell_{A,j,C}}(X,H; \vec \alpha_{A,j,C})$
is a Kuranishi structure on the direct product space
${\mathcal M}_{\vec \ell_{A,j,C}}(X,H; \vec \alpha_{A,j,C})^+ =
{\mathcal M}_{\vec \ell_{A,j,C}}(X,H; \vec \alpha_{A,j,C}) \times 
[-1,0]^{m_{j}(A,C)-1}$ $(j=1,\dots,a+1)$
and
${\mathcal M}_{\vec\ell_{A,j}}(X,H;\vec\alpha_{A,j})$ is a
component of the normalized corner\footnote{See \cite[Definition 24.18]{fooonewbook} for the normalized corner.} 
$$
\widehat{S}_{c_j(A)}({\mathcal M}_{\vec\ell_{A,j,C}}(X,H;\vec\alpha_{A,j,C})).
$$
\par
Therefore we can restrict the Kuranishi structure
$\widehat{\mathcal U}_{\vec \ell_{A,j,C}}(X,H; \vec \alpha_{A,j,C})$
to a Kuranishi structure on
$ {\mathcal M}_{\vec\ell_{A,j}}(X,H;\vec\alpha_{A,j}) \times  [-1,0]^{m_j(A,C)-1}$.
We take the fiber product of them for various $j$  using
(\ref{264222})
and obtain a Kuranishi structure on
$
{\mathcal M}_{\vec \ell}(X,H; \vec \alpha) \times [-1,0]^{b}
$.
(Note we use (\ref{form264555}) here.)
\par
Condition \ref{conds2660} requires
that this Kuranishi structure
coincides with the restriction of
$\widehat{\mathcal U}_{\vec \ell}(X,H; \vec \alpha)$
to the image of (\ref{2642form}).
Let us elaborate on Condition \ref{conds2660}.
Formula
(\ref{form2642}) shows that at the image of
(\ref{2642form}) we have $s_i = -1$  for $i \in A$
and $s_i = 0$ for $i \in C$.
\par
We are taking a fiber product over $R_{\alpha_{i(A,j)}}$.
Therefore we take the fiber product corresponding to
the singular points for which $s_i = -1$.
This is related to the compatibility of the Kuranishi structures at the boundary and corners.
(Theorem \ref{theorem266} (3) and \cite[Condition 16.1 (X)]{fooonewbook}.)
\par
Let us consider the part $s_i = 0$.
For simplicity of notation, we explain the case $m=2$ and $C = \{1\}$.
We consider
${\mathcal M}(X,H; \alpha_-,\alpha) \times_{R_{\alpha}} {\mathcal M}(X,H; \alpha,\alpha_+) \times \{0\}$.
Condition \ref{conds2660} in this case means that the Kuranishi structure there is the restriction of the Kuranishi structure on 
${\mathcal M}(X,H; \alpha_-,\alpha_+)$.
This condition is used to glue
${\mathcal M}(X,H; \alpha_-,\alpha) \times_{R_{\alpha}} {\mathcal M}(X,H; \alpha,\alpha_+) \times [-1,0]$ with
${\mathcal M}(X,H; \alpha_-,\alpha_+)$
there.

\begin{prop}\label{prop2661}
There exists a K-system  
$$
\{({\mathcal M}_{\vec \ell}(X,H; \vec \alpha)^+, ~\widehat{\mathcal U}_{\vec \ell}(X,H; \vec \alpha))\}
$$
for various $\vec \ell, \vec \alpha$
with the following properties.
\begin{enumerate}
\item
They satisfy Condition \ref{conds2660}.
\item
Let $\frak C$ be\footnote{In \cite[Situation 18.1]{fooonewbook} the symbol 
$\frak C$ is used to indicate a certain specific set of components among boundary components and to define the notion of `$\frak C$-partial outer collars' (\cite[Definition 18.9]{fooonewbook}). In the current situation this $\frak C$ stands for the boundary components arising from the 
$[-1,0]^{m-1}$ factors.} the union of the components of
$
\partial{\mathcal M}_{\vec \ell}(X,H; \vec \alpha)^+
$
which are in ${\mathcal M}_{\vec \ell}(X,H; \vec \alpha)
\times \partial ([-1,0]^{m-1})$.
Then
$\widehat{\mathcal U}_{\vec \ell}(X,H; \vec \alpha)$
is $\frak C$-collared in the sense of Remark \ref{defn1531revrev} below.
\item
For the case $\vec \alpha = (\alpha_-,\alpha_+)$, 
$\widehat{\mathcal U}_{\ell}(X,H; (\alpha_-,\alpha_+))$
coincides with the Kuranishi structure we produced in
Theorem \ref{kuraexists}.
\end{enumerate}
\end{prop}
We remark that
$
{\mathcal M}_{\ell}(X,H; (\alpha_-,\alpha_+))^+
=
{\mathcal M}_{\ell}(X,H; (\alpha_-,\alpha_+))
$.
So the statement (3) above makes sense.
\begin{rem}\label{defn1531revrev}
\begin{enumerate}
\item
The definition of $\frak C$-collared-ness of a K-space 
can be seen in \cite[Definition 18.9]{fooonewbook} in the abstract setting.
However, in our current situation we can describe 
the $\frak C$-collared-ness of the Kuranishi structure more explicitly as follows: 
\par
For a decomposition of the set $\underline{m-1}$ into three disjoint union $(S,M,L) \in {\rm Part}_3(m-1)$,
we define $I(S,M,L)$ to be the set of $(t_1,\dots,t_{m-1}) \in [-1,0]^{m-1}$
such that
$t_i \in [-1,-1+\tau]$ if $i \in S$, $t_i \in [-\tau,0]$ if $i \in L$,
and $t_i \in [-1+\tau,-\tau]$ if $i \in M$.
We consider the retraction $I(S,M,L) \to [-1+\tau,-\tau]^{\#M}$ 
by forgetting $t_i$ for $i\notin M$.
We embed  $[-1+\tau,-\tau]^{\#M}$ into $I(S,M,L)$ 
by putting $t_i =-1$ if $i\in S$, $t_i = 0$ if $i \in L$.
Let $\pi: \partial{\mathcal M}_{\vec \ell}(X,H; \vec \alpha)^+ 
\to [-1,0]^{m-1}$ be the obvious projection.
Now we require that the restriction of our Kuranishi structure 
to $\pi^{-1}(I(S,M,L))$ is obtained by a pull back 
of one on the image of $\pi^{-1}([-1+\tau,-\tau]^{\#M})$.
See Figure \ref{FigRemark56}.
\begin{figure}[htbp]
\centering
\includegraphics[scale=0.5]{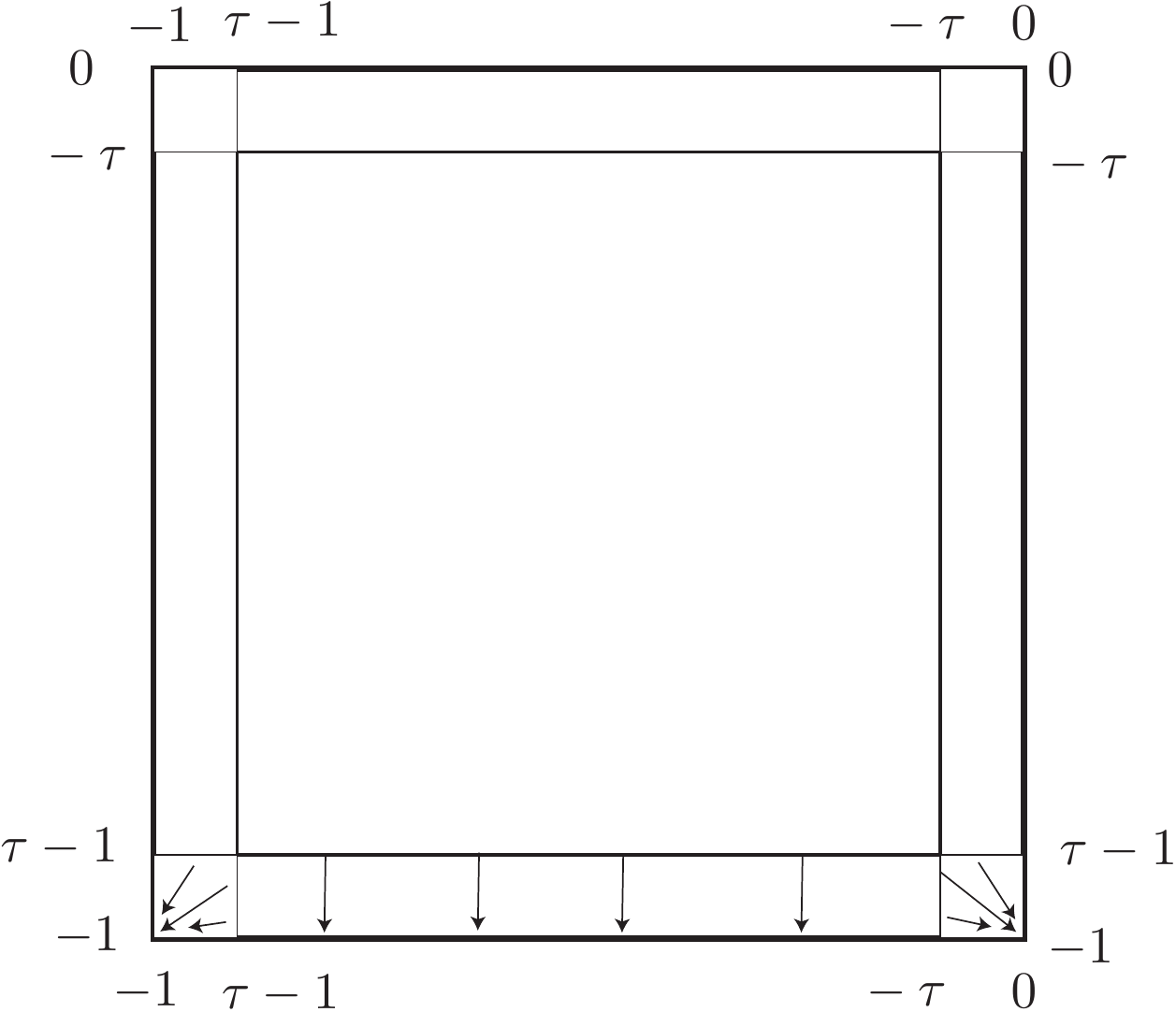}
\caption{$\frak C$-collared-ness and the retraction $I(S,M,L) \to [-1+\tau,-\tau]^{\#M}$}
\label{FigRemark56}
\end{figure}
\item
Indeed we use the $\frak C$-collared-ness in the proof of Theorem \ref{theorem266} below.
\end{enumerate}
\end{rem}
\begin{proof}
[Proof of
Theorem \ref{theorem266} assuming Proposition \ref{prop2661}]
We note that there is a retraction:
\begin{equation}\label{retraction}
{\mathcal R} : {\mathcal M}_{\ell}(X,H; (\alpha_-,\alpha_+))^{\boxplus 1} \to 
{\mathcal M}_{\ell}(X,H; (\alpha_-,\alpha_+)).
\end{equation}
This is a map which sends 
$(\Sigma,(z_-,z_+,\vec z),u,\varphi,\vec t)$
to $(\Sigma,(z_-,z_+,\vec z),u,\varphi)$. 
It is easy to see that the inverse image 
${\mathcal R}^{-1}(z)$ is $[-1,0]^m$ if and only if
$
z \in \overset{\circ}S_m({\mathcal M}_{\ell}(X,H; (\alpha_-,\alpha_+)))
$
Here $\overset{\circ}S_m({\mathcal M}_{\ell}(X,H; (\alpha_-,\alpha_+)))$
is the set of points in codimension $m$ corners 
of ${\mathcal M}_{\ell}(X,H; (\alpha_-,\alpha_+))$ which do not lie in higher
codimension corners. 
In fact, $\overset{\circ}S_m({\mathcal M}_{\ell}(X,H; (\alpha_-,\alpha_+)))$
consists of $(\Sigma,(z_-,z_+,\vec z),u,\varphi)$ 
such that $\Sigma$ has exactly $m$ transit points.
\par
By definition the original Kuranishi structure 
on $
{\mathcal M}_{\vec \ell}(X,H; \vec \alpha)^+
$
is the direct product of the Kuranishi structure on
$
{\mathcal M}_{\vec \ell}(X,H;\vec \alpha)
$
(which is the restriction of the Kuranishi structure on
${\mathcal M}_{\vec \ell}(X,H; \alpha_-,\alpha_+)$ given in
Theorem  \ref{kuraexists}) and the trivial
Kuranishi structure on $[-1,0]^{m-1}$.
(See \cite[Lemma-Definition 17.38]{fooonewbook}.)
\par
We replace the direct product Kuranishi structure on 
$
{\mathcal M}_{\vec \ell}(X,H; \vec \alpha) \times [-1,0]^{m-1}
$ by $\widehat{\mathcal U}_{\vec \ell}(X,H; \vec \alpha)$
given in Proposition \ref{prop2661}.
\par
We first check that those Kuranishi structures $\widehat{\mathcal U}_{\vec \ell}(X,H; \vec \alpha)$
can be glued to give a Kuranishi structure on ${\mathcal M}_{\ell}(X,H; (\alpha_-,\alpha_+))^{\boxplus 1}$.
\par
Because of Proposition \ref{prop2661} (2),
it suffices to check that they are
compatible for various
$
{\mathcal M}_{\vec \ell}(X,H; \vec \alpha) \times [-1,0]^{m-1}
$ at their intersection points. (Here the $\frak C$-collared-ness is used as we  mentioned in Remark \ref{defn1531revrev} (2).)
\par
The compatibility of two different members of the set $\{\widehat{\mathcal U}_{\vec \ell}(X,H; \vec \alpha)\}$
at the overlapped part 
is a consequence of
Condition \ref{conds2660}.
Indeed it follows from the case when $A = \emptyset$ of Condition \ref{conds2660}.
(Note that they are glued at the part where a certain coordinate 
of the $[-1,0]^{m-1}$ factor is $0$.)
The compatibility of $\widehat{\mathcal U}_{\vec \ell}(X,H; \vec \alpha)$
with the Kuranishi structure on 
$
{\mathcal M}_{\ell}(X,H;\alpha_-, \alpha_+)
$
given by Theorem \ref{kuraexists}
follows from Condition \ref{conds2660} and Proposition \ref{prop2661} (3).
\par
We thus obtain a Kuranishi structure on ${\mathcal M}_{\ell}(X,H; (\alpha_-,\alpha_+))^{\boxplus 1}$.
It is immediate from construction that it satisfies Theorem \ref{theorem266} (4).
\par
It remains to prove Theorem \ref{theorem266} (3).
The main point we need to prove is the compatibility of our
Kuranishi structures at the boundary and corners
\cite[Conditions 16.1 (X) (IX)]{fooonewbook}. 
This is the condition on the Kuranishi structure at the point where
a certain coordinate of the $[-1,0]^{m-1}$ factor is $-1$.
This is a consequence of Condition \ref{conds2660}.
More precisely, the case $C = \emptyset$ of Condition \ref{conds2660}
implies 
\cite[Conditions 16.1 (X) (IX)]{fooonewbook}.
\par
The other defining conditions for our Kuranishi structures to form 
a linear K-system are easy to check.
(The periodicity of linear K-system \cite[Conditions 16.1 (VIII)]{fooonewbook}
follows from Remark \ref{rem43}.)
The proof of Theorem \ref{theorem266} is now complete.
\end{proof}

\subsection{Proof of Proposition \ref{prop2661} I: Obstruction space with outer collar}\label{subsec:outerObst}
This subsection and the next are occupied with the proof of Proposition \ref{prop2661}.
In this subsection we define an obstruction space of Kuranishi chart at each point in 
$\mathcal M_{\vec\ell}(X,H;\alpha)^+$ 
(Definition \ref{def:513}). 
Then we will construct 
a desired Kuranishi structure and complete the proof of Proposition \ref{prop2661} in the next subsection.
\begin{proof}[Proof of Proposition \ref{prop2661}]
Let $(A,B,C) \in {\rm Part}_3(m-1)$, $\#B = b$ and $(D,E,F) \in {\rm Part}_3(b)$.
We define $(A,B(D,E,F),C) \in {\rm Part}_3(m-1)$ as follows.
We put $B = \{j_1,\dots,j_{b}\}$.
Then $(A,B(D,E,F),C) = (A',B',C')$ such that
\begin{enumerate}
\item
$A' \supseteq A$, $C' \supseteq C$.
\item
$j_i \in A'$ if $i \in D$.
\item
$j_i \in B'$ if $i \in E$.

\item $j_i \in C'$ if $i \in F$.
\end{enumerate}
Note (1)-(4) above is equivalent to $\mathcal I_{A,B,C} \circ \mathcal I_{D,E,F} = \mathcal I_{A',B',C'}$.
\begin{defn}\label{defn2659}
We consider a system of closed subsets ${\mathcal V}(A,B,C)$ of $[-1,0]^{m-1}$
over $m=1,2, \dots$ and  $(A,B,C) \in {\rm Part}_3(m-1)$ with the following properties.
\begin{enumerate}
\item
$$
\bigcup_{(A,B,C) \in {\rm Part}_3(m-1) } {\rm Int}\,{\mathcal V}(A,B,C)
= [-1,0]^{m-1}.
$$
\item
If $(A,B(D,E,F),C) = (A',B',C')$, then
$$
(\mathcal I_{A,B,C})^{-1}({\mathcal V}(A',B',C')) = {\mathcal V}(D,E,F).
$$
\item
If $\sigma \in {\rm Perm}(m-1)$, then
$$
\sigma({\mathcal V}(A,B,C)) = {\mathcal V}(\sigma A,\sigma B,\sigma C).
$$
Here $\sigma$ acts on $[-1,0]^{m-1}$ by permutation of the factors
and to ${\rm Part}_3(m-1)$ in an obvious way.
\item
If $(A,B,C), (A',B',C') \in {\rm Part}_3(m-1)$ and
${\mathcal V}(A,B,C) \cap {\mathcal V}(A',B',C')$ is nonempty, then
either ${\rm Im}(\mathcal I_{A,B,C})$ is a face of ${\rm Im}(\mathcal I_{A',B',C'})$
or  ${\rm Im}(\mathcal I_{A',B',C'})$  is a face of ${\rm Im}(\mathcal I_{A,B,C})$.
\item
Let $\vec t=(t_1,\dots,t_{m-1}) \in {\mathcal V}(A,B,C)$.
If $t_i =-1$ then $i \in A$. If $t_i=0$ then $i \in C$.
\item
${\mathcal V}(A,B,C)$ is a direct product in each sufficiently small neighborhood of
the strata of $[-1,0]^{m-1}$ in the following sense.
Let $\vec t \in S_{\ell}([-1,0]^{m-1})$ and
take its neighborhood which is isometric to 
$\sigma(U \times (-\epsilon,0]^p \times [-1,-1+\epsilon)^{q})$
for some $\sigma \in {\rm Perm}(m-1)$, $\epsilon >0$ with $p+q=\ell$.
(Here $U$ is isometric to an open set of $\R^{m-1-\ell}$.)
We then require 
$$
{\mathcal V}(A,B,C) \cap (\sigma(U \times (-\epsilon,0]^p \times [-1,-1+\epsilon)^{q}))
$$
is isometric to the direct product $({\mathcal V}(A,B,C) \cap \sigma(U \times \{0\}\times \{0\})) \times [0,\epsilon)^{\ell}$.
\end{enumerate}
\end{defn}
\begin{lem}
There exists a system of closed subsets ${\mathcal V}(A,B,C)$ satisfying
(1)-(6) in Definition \ref{defn2659}.
\end{lem}
\begin{proof}
The proof is by induction on $m$.
There is nothing to prove in the case $m=1$.
In the case $m=2$ we consider the interval $[-1,0] = [-1,0]^{2-1}$.
${\rm Part}_3(1)$ consists of exactly 3 elements
such that the barycenter of the image of $\mathcal I_{A,B,C}$ are $-1, -1/2, 0$, respectively.
We take ${\mathcal V}(A,B,C) =
[-1,-3/5]$, $[-4/5,-1/5]$, $[-2/5,0]$, respectively, for example.
\par
Suppose we defined  ${\mathcal V}(A,B,C)$ for $(A,B,C)  \in {\rm Part}_3(m'-1)$ with $m' < m$.
We consider ${\mathcal V}(A,B,C)$ for $(A,B,C) \in {\rm Part}_3(m-1)$.
The induction hypothesis and Definition \ref{defn2659} (2)(5)
determine ${\mathcal V}(A,B,C) \cap \partial [-1,0]^{m-1}$.
For example, in the case  $\mathcal A_0 = (A,B,C)$ where $B = \underline{m-1}$, 
$A= C = \emptyset$,
we have
${\mathcal V}(\mathcal A_0) \cap \partial [-1,0]^{m-1} = \emptyset$.
We can define ${\mathcal V}(A,B,C)$
for the case $(A,B,C) \ne \mathcal A_0$ by
taking a small neighborhood of ${\mathcal V}(A,B,C) \cap \partial [-1,0]^{m-1}$.
Properties (2)(4)(5) hold by construction and we can choose 
${\mathcal V}(A,B,C)$
so that (3)(6) hold also.
The union of the interiors of such ${\mathcal V}(A,B,C)$'s contains $\partial [-1,0]^{m-1}$.
Now we can choose ${\mathcal V}(\mathcal A_0)$ so that all of (1)-(6) are satisfied.
See Figure \ref{Fig23.2}.
\end{proof}
\begin{exm}\label{ex:V}
Here is an example of $\mathcal V(A,B,C)$.
Recall $b=\# B$.
If $i \in A$, $-1 \leq t_i \leq (1-3^{b+1})/3^{b+1}$.
If $i \in B$, $(2-3^{b+1})/3^{b+1} \leq t_i \leq -2/3^{b+1}$. 
If $i \in C$, $-1/3^{b+1} \leq t_i \leq 0$.
\end{exm}
\begin{figure}[htbp]
\centering
\includegraphics[scale=0.5]{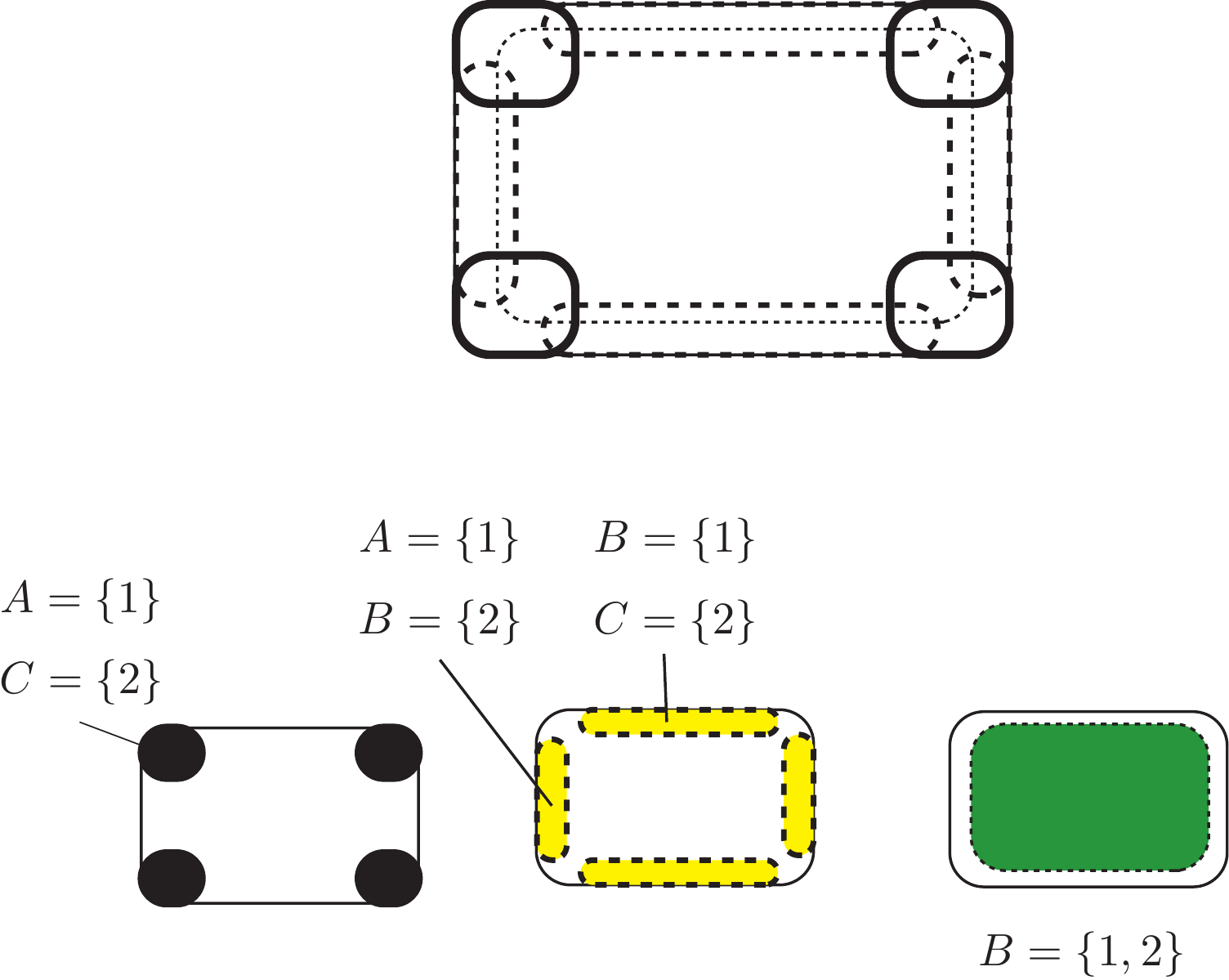}
\caption{${\mathcal V}(A,B,C)$}
\label{Fig23.2}
\end{figure}
\par\smallskip
Now we take and fix a system $\{{\mathcal V}(A,B,C)\}$ of closed subsets 
in Definition \ref{defn2659}.
Let $({\bf q},\vec t) \in {\mathcal M}_{\vec \ell}(X,H; \vec \alpha)^+$.
\begin{defn}\label{def:B}
For each given point $\vec t \in [-1,0]^{m-1}$, 
we put
$$
\mathcal B(\vec t)
=
\{(A,B,C) \in {\rm Part}_3(m-1) \mid \vec t \in {\mathcal V}(A,B,C)\}.
$$
\end{defn}
To each  $({\bf q},\vec t) \in {\mathcal M}_{\vec \ell}(X,H; \vec \alpha)^+$ and $\mathcal A \in \mathcal B(\vec t)$,
we are going to define a finite dimensional linear subspace
$$
E_{({\bf q},\vec t)}({\bf q};\mathcal A)
\subset
C^{\infty}(\Sigma_{\bf q},u_{\bf q}^*TX \otimes \Lambda^{0,1}).
$$
Their direct sum will be the obstruction space of the Kuranishi chart
at $({\bf q},\vec t)$. 
See Definition \ref{def:513}.
The way we do so is similar to the argument in Section \ref{subsec;KuraFloer}
but is slightly more complicated
because of describing 
combinatorial patterns of corners of the moduli spaces in terms of outer collars.
\par
We recall 
$$
\alpha_-= \alpha_0,\alpha_1,\dots,\alpha_{m-1},\alpha_m = \alpha_+
\in \frak A
$$
and put 
$$
\vec{\alpha}  = (\alpha_0,\dots,\alpha_{m}).
$$
Let $B \subseteq \underline{m-1}$.
We will construct a Kuranishi structure
on $\mathcal M_{\ell}(X,H;\vec{\alpha})$.
The way we do so is the same as the proof of Theorem  \ref{kuraexists}.
The construction in the proof of Theorem  \ref{kuraexists} involves various
choices. We take different choices for different  $B$.
The choice in the case when $A = \emptyset$ and
$B = \emptyset$ is exactly the same as
one taken during the proof of Theorem  \ref{kuraexists}.
The detail follows.
First we consider the case $A=\emptyset$.\footnote{When $A \ne \emptyset$, we will apply Choice \ref{choice2665} 
for each moduli space 
$\mathcal M(X,H,\alpha_{i(A,j-1)}, \alpha_{i(A,j)})$
by restricting $B$ to the subset $B_j'(A):=B \cap (i(A,j-1), i(A,j))_{\Z}$ with $j=1,\dots , a+1$.}
\begin{choi}\label{choice2665}
\begin{enumerate}
\item
We  take a finite set.
$$
\EuScript A_{\vec{\ell}}(\vec{\alpha};B)
= \{ {\bf p}_c \mid c \in \EuScript C_{\vec{\ell}}(\vec{\alpha},B) \}
\subset \mathcal M_{\vec\ell}(X,H;\vec{\alpha}).
$$
Here $\EuScript C_{\vec{\ell}}(\vec{\alpha},B)$ is a certain index set
which will be taken as in Condition \ref{conds2626}.
\item
For each ${\bf p}_c \in \EuScript A_{\vec{\ell}}(\vec{\alpha};B)$
we take its closed neighborhood $W({\bf p}_c;\vec{\ell},\vec{\alpha},B)$ in $\mathcal M_{\vec\ell}(X,H;\vec{\alpha})$
which is sufficiently small so that Lemma \ref{lem2446} holds for ${\bf p}={\bf p}_c$.
We also take obstruction bundle data $\frak E_{{\bf p}_c}(\vec{\ell},\vec{\alpha},B)$
centered at ${\bf p}_c$ for each $c \in \EuScript C_{\vec{\ell}}(\vec{\alpha},B)$.
\end{enumerate}
\end{choi}
Let $B = \{i(B,1),\dots,i(B,b)\}$ with 
$1\le i(B,1) < \dots < i(B,b)\le m-1$ 
and we put
$$
\aligned
\vec{\alpha}(B) & = (\alpha_-,\alpha_{i(B,1)},\dots,\alpha_{i(B,b)},\alpha_+), \\
\ell_j(B) & = \ell_{i(B,j-1)+1}
+ \dots + \ell_{i(B,j)}, \\
\vec{\ell}(B) & = (\ell_1(B),\dots,\ell_{b+1}(B)).
\endaligned
$$
Here we put $i(B,0) = 0$ and $i(B,b+1) = m$ by convention.
We note
$$
\mathcal M_{\vec\ell}(X,H;\vec{\alpha})
\subseteq
\mathcal M_{\vec{\ell}(B)}(X,H;\vec{\alpha}(B)).
$$
\begin{conds}\label{conds2626}
We require that the objects taken in Choice \ref{choice2665}
have the following properties.
\begin{enumerate}
\item If $ \vec{\alpha} = (\alpha_-,\alpha_+)$,  the choices are exactly
the same as we took during the proof of Theorem \ref{kuraexists}.
Namely the choices of ${\EuScript A}_{\ell}(\alpha_-,\alpha_+), \EuScript C_{\ell}(\alpha_-,\alpha_+)$ and $W({\bf p}_c)$ are made as in Choice \ref{choice2650}.
\item
\begin{equation}\label{form2645}
\EuScript A_{\vec{\ell}}(\vec{\alpha};B)
= \EuScript A_{\vec{\ell}(B)}(\vec{\alpha}(B);\underline b) \cap \mathcal M_{\vec\ell}(X,H;\vec{\alpha}).
\end{equation}
\item
If ${\bf p}_c$ is an element of (\ref{form2645}), we have 
$$
W({\bf p}_c;\vec{\ell},\vec{\alpha},B)
=
W({\bf p}_c;\vec{\ell}(B),\vec{\alpha}(B),\underline b) \cap \mathcal M_{\vec\ell}(X,H;\vec{\alpha}).
$$
Moreover
$$
\frak E_{{\bf p}_c}(\vec{\ell},\vec{\alpha},B)
=
\frak E_{{\bf p}_c}(\vec{\ell}(B),\vec{\alpha}(B),\underline b).
$$
\item
For each $B$
\begin{equation}
\bigcup_{{\bf p}_c \in \EuScript A_{\vec{\ell}}(\vec{\alpha},B)} {\rm Int}\,\,W({\bf p}_c;\vec{\ell},\vec{\alpha},B)
= \mathcal M_{\vec\ell}(X,H;\vec{\alpha}).
\end{equation}
\end{enumerate}
\end{conds}
We need to
require one more condition
(Lemma \ref{lem26689} Condition $(*)$) which will be given later.
\par
Now we describe the procedure of associating an obstruction space
to an element $({\bf q},\vec t)$ when Choice \ref{choice2665} is given.
We first review the procedure we have taken in Section \ref{subsec;KuraFloer}.
Let ${\bf q} \in  \mathcal M_{\vec\ell}(X,H;\vec{\alpha})$.
\begin{enumerate}
\item[(i)]
We put
\begin{equation}
G({\bf q};\vec{\ell},\vec{\alpha},B)
= \{{\bf p}_c \mid {\bf q} \in W({\bf p}_c;\vec{\ell},\vec{\alpha},B)\}.
\end{equation}
This is the same as in Definition \ref{defn2646} (1).
\item[(ii)]
For each ${\bf p}_c \in G({\bf q};\vec{\ell},\vec{\alpha},B)$
we take $\vec w_{{\bf p}_c}^{\bf q} \subset \Sigma_{\bf q}$ such that
${\bf q} \cup \vec w_{{\bf p}_c}^{\bf q}$ is $\epsilon_c$-close to
${\bf p}_c \cup \vec w_{{\bf p}_c} \cup \vec w_{\rm can}$
and satisfies the transversal constraint.
(Definition \ref{constrainttt2}.)
We note that such $\vec w_{{\bf p}_c}^{\bf q}$ uniquely exists if we take
$W({\bf p}_c;\vec{\ell},\vec{\alpha},B)$ sufficiently small.
(Lemma \ref{lem2446}.)
\item[(iii)]
If $(\frak Y,u',\varphi') \cup \vec w'_c$ is $\epsilon$-close to ${\bf q} \cup \vec w_{{\bf p}_c}^{\bf q}$,
then by Definition \ref{defn2642} we have a complex linear embedding
\begin{equation}\label{form2648}
I_{{\bf p}_c,{\rm v};\Sigma',u',\varphi'}^{\vec{\ell},\vec{\alpha},B} : 
E_{{\bf p}_c,{\rm v}}(\frak y_{{\bf p}_c};\vec{\ell},\vec{\alpha},B) \to
C^{\infty}(\Sigma';(u')^*TX \otimes \Lambda^{0,1}).
\end{equation}
See Definition \ref{defn2646} (3).
Here 
$(\frak Y\cup \vec w'_c,\varphi')=\Phi_{{\bf p}_c}(\frak y_{{\bf p}_c}, 
\vec T_{{\bf p}_c}, \vec \theta_{{\bf p}_c})$ as in Notation \ref{nota:415}.
We include $\vec{\ell},\vec{\alpha},B$ in the notation since the map depends on them.
\end{enumerate}
Now we go back to our situation.
We will sum up the images of the maps (\ref{form2648}) not only
for various ${\bf p}_c$ and ${\rm v}$ but also for various
$(A,B,C) \in \mathcal B(\vec t)$.
We describe this process now.
\par
Let $({\bf q},\vec t) \in \mathcal M_{\vec\ell}(X,H;\vec{\alpha})^+$
and $\vec\alpha = (\alpha_0,\dots,\alpha_m)$ as before.
Since  $\mathcal M_{\vec\ell}(X,H;\vec{\alpha})$ is a fiber product with
$m$-factors, ${\bf q}$ is decomposed into  factors
${\bf q}_i$, $i=1,\dots,m$. 
Let $(A,B,C) \in \mathcal B(\vec t)$.
We put $A = \{i(A,1),\dots,i(A,a)\}$ with $i(A,1) < \dots <i(A,a)$.
We define
\begin{equation}\label{formula2657567}
{\bf q}_{(A,j)}
=
({\bf q}_{i(A,j-1) +1},\dots,{\bf q}_{i(A,j)})
\in \mathcal M_{\vec{\ell}_{A,j}}(X,H;\vec{\alpha}_{A,j}).
\end{equation}
Recall $\vec{\alpha}_{A,j} = ({\alpha}_{i(A,j-1)},\dots,{\alpha}_{i(A,j)})$ 
and $\vec{\ell}_{A,j} = (\ell_{i(A,j-1)+1},\dots,\ell_{i(A,j)})$.
We put $m_j(A)= i(A,j) - i(A,j-1)$.
We also put $B'_j(A) = B \cap (i(A,j-1),i(A,j))_{\Z}$ and 
$$
B_j(A) = \{ i - i(A,j-1) \mid i \in B'_j(A)\}.
$$
Note that we made choices 
for $\vec{\ell}_{A,j},\vec{\alpha}_{A,j}$ and $B_j(A)$
as in Choice \ref{choice2665}.
\begin{defn}\label{defn27677}
For $({\bf q},\vec t) \in \mathcal M_{\vec \ell}(X,H,\vec \alpha)^+$
we define a set $\widetilde{\frak F}({\bf q},\vec t)$ by
$$
\widetilde{\frak F}({\bf q},\vec t)
=
\{((A,B,C),j) \mid (A,B,C) \in \mathcal B(\vec t),\,\,
j=1, \dots ,a+1
=\#A+1
\}.
$$
We define an equivalence relation $\sim$ on it as follows.
Let $((A(k),B(k),C(k)),j(k))$ be elements of this set for $k=1,2$.
Then
$$
((A(1),B(1),C(1)),j(1)) \sim ((A(2),B(2),C(2)),j(2))
$$
if and only if the following (1) (2) hold.
We put
$A(k) = \{i(A(k),1),\dots, i(A(k),a(k))\}$ with
$i(A(k),1)< \dots < i(A(k),a(k))$.
\begin{enumerate}
\item
$i(A(1),j(1)-1) = i(A(2),j(2)-1)$. $i(A(1),j(1)) = i(A(2),j(2))$.
\item
$B(1) \cap (i(A(1),j(1)-1),i(A(1),j(1)))_{\Z}  = B(2) \cap (i(A(2),j(2)-1),i(A(2),j(2)))_{\Z}$.
\end{enumerate}
(Note that it automatically implies
$\vec{\ell}_{A(1),j(1)} = \vec{\ell}_{A(2),j(2)}$.)
The conditions (1)(2) imply that the map \eqref{form2648rev} below
is independent of the $\sim$ equivalence class.
\par
Now we put\footnote{The sets $\mathcal V(A,B,C)$ are used to define $ \mathcal B(\vec t)$.
Then  $\mathcal B(\vec t)$ is used to define $ \widetilde{\frak F}({\bf q},\vec t)$.}
$$
\frak F({\bf q},\vec t) = \widetilde{\frak F}({\bf q},\vec t)/\sim.
$$
\par
For $\frak z \in \frak F({\bf q},\vec t)$ the three objects $\vec{\alpha}_{A,j}$, $B_j(A)$, 
$\vec{\ell}_{A,j}$ and
${\bf q}_{(A,j)}$
are determined in a way independent of the representatives.
We write them as  $\vec{\alpha}_{\frak z}$, $B(\frak z)$, $\vec{\ell}_{\frak z}$ and ${\bf q}_{\frak z}$.
\end{defn}
Let  $\frak z \in \frak F({\bf q},\vec t)$ and
${\bf p}_c \in G({\bf q}_{\frak z};\vec{\ell}_{\frak z},\vec{\alpha}_{\frak z},B({\frak z}))$.
Then we obtain the linear map (\ref{form2648}) which is
\begin{equation}\label{form2648rev}
I_{{\bf p}_c,{\rm v};\Sigma_{{\bf q}_{\frak z}},u_{{\bf q}_{\frak z}},\varphi_{\bf q_{\frak z}}}^{\vec{\ell}_{\frak z},\vec{\alpha}_{\frak z},B({\frak z})} : 
E_{{\bf p}_c,{\rm v}}(\frak y_{{\bf p}_c};\vec{\ell},\vec{\alpha},B) \to
C^{\infty}(\Sigma_{{\bf q}_{\frak z}};u_{{\frak z}}^*TX \otimes \Lambda^{0,1}).
\end{equation}
Here 
$(\Sigma_{{\bf q}_{\frak z}}\cup \vec z_{{\bf q}_{\frak z}} \cup 
\vec w_{{\bf p}_c}^{{\bf q}_{\frak z}})
=\Phi_{{\bf p}_c}(\frak y_{{\bf p}_c},
\vec T_{{\bf p}_c}, \vec\theta_{{\bf p}_c})$\footnote{As we noticed in 
Notation \ref{nota:415}, $\frak y_{{\bf p}_c}$ etc on the right hand side also depend on ${\bf q}_{\frak z}$ 
in this case, but we omit ${\bf q}_{\frak z}$ 
from the notations when no confusion can occur.} and
$\Sigma_{{\bf q}_{\frak z}}$, $u_{{\bf q}_{\frak z}}$,
$\varphi_{\bf q_{\frak z}}$ are the source curve, the map to $X$, and the parametrization of the mainstream,
which are parts of ${\bf q}_{\frak z}$, respectively.
Note that the target space of the map (\ref{form2648rev}) is a subset of
$C^{\infty}(\Sigma_{{\bf q}};u^*TX \otimes \Lambda^{0,1})$
and is the sum of the set of smooth sections of the irreducible components.
\begin{rem}
The notion of stabilization data is similar to the obstruction bundle data (Definition \ref{obbundeldata1}),
except we do not include obstruction spaces (Definition \ref{obbundeldata1} (5)(6)).
\end{rem}
\begin{lem}\label{lem26689}
We can achieve the choices laid out in 
Choice \ref{choice2665} so that Condition \ref{conds2626} and the following
condition $(*)$ are satisfied.
\begin{enumerate}
\item[$(*)$]
The sum of the images of the map (\ref{form2648rev}) for various $\frak z  \in \frak F({\bf q},\vec t)$,
${\bf p}_c \in G({\bf q}_{\frak z};\vec{\ell}_{\frak z},\vec{\alpha}_{\frak z},B({\frak z}))$ and irreducible components
${\rm v}$ of ${\bf p}_c$ is a direct sum in $C^{\infty}(\Sigma_{{\bf q}};u^*TX \otimes \Lambda^{0,1})$.
\end{enumerate}
\end{lem}
We will prove Lemma \ref{lem26689} in the next subsection.
Assuming it for the moment, 
we continue the proof of Proposition \ref{prop2661}.
\par
Note that for each $\frak z \in \frak F({\bf q},\vec t)$ and
${\bf p}_c \in G({\bf q}_{\frak z};\vec{\ell}_{\frak z},\vec{\alpha}_{\frak z},B({\frak z}))$
we took additional marked points $\vec w_{{\bf p}_c}^{{\bf q}_{\frak z}}$
on ${{\bf q}_{\frak z}}$.
These marked points can be regarded as marked points on ${\bf q}$.
We also take stabilization data centered at ${\bf q}$. In particular, we take
$\vec w_{\bf q}$.
\begin{shitu}\label{situ2669dd}
$\bullet$ 
We have $(\frak Y,u',\varphi')$ and marked points $\vec w'_{{\bf p}_c}$ on the source curve
of $\frak Y'$ for each $\frak z \in \frak F({\bf q},\vec t)$ and
${\bf p}_c \in G({\bf q}_{\frak z};\vec{\ell}_{\frak z},\vec{\alpha}_{\frak z},B({\frak z}))$.
We also take $\vec w'_{\bf q}$.
\par\noindent
$\bullet$ 
For each
$\frak z \in \frak F({\bf q},\vec t)$ and
${\bf p}_c \in G({\bf q}_{\frak z};\vec{\ell}_{\frak z},\vec{\alpha}_{\frak z},B({\frak z}))$,
we assume that $(\frak Y,u',\varphi') \cup \vec w'_{{\bf p}_c}$
is $\epsilon_1$-close to  ${\bf q} \cup \vec w_{{\bf p}_c}^{{\bf q}_{\frak z}}$.
\par\noindent
$\bullet$ 
We assume that $(\frak Y,u',\varphi') \cup \vec w'_{\bf q}$ is $\epsilon_1$-close to
${\bf q} \cup \vec w_{\bf q}$.
\par\noindent
$\bullet$ 
We assume that $\frak Y$ is decomposed into $m$ extended mainstream components
$\frak Y_i$ ($i=1,\dots ,m$),
and $(\frak Y_i,u',\varphi') \cup (\vec w'_{\bf q} \cap \frak Y_i)$ is $\epsilon_1$-close to
${\bf q}_i \cup (\vec w_{\bf q} \cap {\bf q}_i)$. 
$\blacksquare$
\end{shitu}
In Situation \ref{situ2669dd} we have
\begin{equation}\label{form2648revrev}
I_{{\bf p}_c,{\rm v};\Sigma',u',\varphi'}^{\vec{\ell}_{\frak z},\vec{\alpha}_{\frak z},B({\frak z})} : E_{{\bf p}_c,{\rm v}}(\frak y'_{{\bf p}_c};\vec{\ell},\vec{\alpha},B) \to
C^{\infty}(\Sigma';(u')^*TX \otimes \Lambda^{0,1})
\end{equation}
in the same way as in (\ref{form2648}). 
Here ${\mathfrak Y} \cup \vec{w}^{\prime}_{{\bf p}_c} = 
\Phi_{{\bf p}_c} ({\mathfrak y}^{\prime}_{{\bf p}_c}, \vec{T}^{\prime}_{{\bf p}_c}, 
\theta^{\prime}_{{\bf p}_c})$ so
this map depends also on $\vec w'_{{\bf p}_c}$.
\begin{defn}\label{def:513}
For each $({\bf q}, \vec{t}) \in \mathcal M_{\vec\ell}(X,H;\vec{\alpha})^+$
we define
a linear subspace $E((\frak Y,u',\varphi')\cup \bigcup \vec w'_{{\bf p}_c};{\bf q})$ of 
$C^{\infty}(\Sigma';(u')^*TX \otimes \Lambda^{0,1})$
by  the sum of all the images of the map
(\ref{form2648revrev}) for various ${\rm v}$,
$\frak z \in \frak F({\bf q},\vec t)$ and 
${\bf p}_c \in G({\bf q}_{\frak z};\vec{\ell}_{\frak z},\vec{\alpha}_{\frak z},B({\frak z}))$.
We call it the {\it obstruction space} of 
Kuranishi chart at $({\bf q}, \vec{t})$.
\end{defn}
\subsection{Proof of Proposition \ref{prop2661} II: Kuranishi chart and coordinate change}
We now define Kuranishi charts of our Kuranishi structure.
\begin{defn}\label{defn2650rev}
Let $({\bf q}, \vec{t}) \in \mathcal M_{\vec\ell}(X,H;\vec{\alpha})^+$.
In Situation \ref{situ2669dd} we consider the following conditions on
an object
$(((\frak Y,u',\varphi') \cup\bigcup_{{\frak z}, c}\vec w'_{{\bf p}_c} \cup \vec w'_{\bf q}),\vec t')$:
\begin{enumerate}
\item
If $\Sigma_a$ is the mainstream component of $\frak Y$ and $\varphi'_a$ is a parametrization
of this  mainstream component
(which is a part of given $\varphi'$), the following equation is satisfied on
$\R \times S^1$.
\begin{equation}\label{Fleqobstinclrev}
\aligned
&\frac{\partial(u'\circ \varphi'_a)}{\partial \tau} \\
&+  J \left( \frac{\partial(u'\circ \varphi'_a)}{\partial t} - \frak X_{H_t}
\circ u'  \circ \varphi'_a\right) \equiv 0 \mod E((\frak Y,u',\varphi')\cup \bigcup \vec w'_{{\bf p}_c};{\bf q}).
\endaligned
\end{equation}
\item
If $\Sigma'_{\rm v}$  is a bubble component of $\frak Y$, the following equation is satisfied on $\Sigma'_{\rm v}$.
\begin{equation}\label{form2653}
\overline{\partial} u' \equiv 0  \mod  E((\frak Y,u',\varphi')\cup \bigcup \vec w'_{{\bf p}_c};{\bf q}).
\end{equation}
\item For each $\frak z \in \frak F({\bf q},\vec t)$ and
${\bf p}_c \in G({\bf q}_{\frak z};\vec{\ell}_{\frak z},\vec{\alpha}_{\frak z},B({\frak z}))$ the additional marked points $\vec w'_{{\bf p}_c}$
satisfy the transversal constraint in Definition \ref{constrainttt2} with respect to ${\bf p}_c$.
\item The additional marked points $\vec w'_{\bf q}$
satisfy the transversal constraint in Definition \ref{constrainttt2} with respect to ${\bf q}$.
\item
For each ${\frak z}$ and $c \in \EuScript E({\bf q})$, 
$(\frak Y,u',\varphi')\cup\vec w'_{{\bf p}_c} \cup \vec w'_{\bf q}$ is
$\epsilon_1$-close to
${\bf q} \cup \vec w_{{\bf p}_c}^{{\bf q}_{\frak z}} \cup \vec w_{\bf q}$.
\item $\vert \vec t' - \vec t\vert < \epsilon_1$.
\end{enumerate}
\par
The set of isomorphism classes of
$(((\frak Y,u',\varphi')\cup\bigcup_{{\frak z}, c}\vec w'_{{\bf p}_c} \cup \vec w'_{\bf q}),\vec t')$ 
satisfying  Conditions
(1) - (6) above is denoted by
$$
V(({\bf q},\vec t),\epsilon_1)
$$ 
where the isomorphism is defined in the same way as in Definition \ref{defn2650}.
\end{defn}
We can prove that $V(({\bf q},\vec t),\epsilon_1)$ is a smooth manifold in the way similar to Lemma \ref{lem2653}.
It is easy to see that it is ${\rm Aut}({\bf q})$ invariant and so we obtain an orbifold
$U(({\bf q},\vec t),\epsilon_1) = V(({\bf q},\vec t),\epsilon_1)/{\rm Aut}({\bf q})$.
We define a vector bundle $E(({\bf q},\vec t),\epsilon_1)$ on it by taking
$E((\frak Y,u',\varphi')\cup \bigcup \vec w'_{{\bf p}_c};{\bf q})$ as the fiber.
Then the left hand side of (\ref{Fleqobstinclrev}) and (\ref{form2653})
define its smooth section, which we denote by $s_{({\bf q},\vec t),\epsilon_1}$.
An element of its zero set determines an element of
$\mathcal M_{\vec\ell}(X,H;\vec{\alpha})^+$.
We thus obtain $\psi_{({\bf q},\vec t),\epsilon_1}$.
\par
We can prove that $(U(({\bf q},\vec t),\epsilon_1),E(({\bf q},\vec t),\epsilon_1),s_{({\bf q},\vec t),\epsilon_1},
\psi_{({\bf q},\vec t),\epsilon_1})$ is a Kuranishi chart of $({\bf q},\vec t)$ in the same way as in Proposition \ref{lem2654}.
\par
We can define a coordinate change among them and show the compatibility among them
in the same way as in Lemmas \ref{lem2657} and \ref{lem2658}.
We finally adjust the size of the Kuranishi neighborhoods 
$\{V(({\bf q},\vec t),\epsilon_1)\}_{({\bf q},\vec t)}$ to obtain a Kuranishi structure on $\mathcal M_{\vec\ell}(X,H;\vec{\alpha})^+$
in the same way again as in 
Lemmas \ref{lem2657} and \ref{lem2658}.
\par
Condition \ref{conds2660} is a consequence of the Property (5) in Definition \ref{defn2659} and
the way we used $\mathcal V(A,B,C)$ to define $\mathcal B(\vec t)$ and ${\frak F}({\bf q},\vec t)$.
Let us elaborate on this point.
\par
If $t_i = 0$, then $i \in C$. Therefore we can apply Condition \ref{conds2626} (2)(3) to see that
the obstruction spaces are  restrictions on those which we defined on the moduli space
obtained by performing the gluing at the corresponding transit points.
\par
If $t_i = -1$, then $i \in A$. Therefore we are taking fiber product
Kuranishi structure at the corresponding transit points.
(See (\ref{formula2657567}) and Definition \ref{defn27677}.)
\par
Proposition  \ref{prop2661} (2) is a consequence of Definition \ref{defn2659} (6) and the
construction.
\par
Proposition  \ref{prop2661} (3) is a consequence of Condition \ref{conds2626} (1).
\par
Therefore to complete the proof of Proposition \ref{prop2661} and of Theorem \ref{theorem266}
it remains to prove Lemma \ref{lem26689}.
\begin{proof}[Proof of Lemma \ref{lem26689}]
The proof is by induction on $m$, where we recall $\#\vec\alpha =m-1$.
If $m=1$ then Choice \ref{choice2665}  is given during the
proof of Theorem  \ref{kuraexists}. 
We suppose that we made the choice for $m'$ that is smaller than $m$ and we will
prove the case of $m$.
We also assume that, as a part of the induction hypothesis,
the conclusion of  Lemma \ref{lem26689} holds
if the number of components of $\vec{\alpha}$ is strictly smaller than $m+1$.
\par
We now consider $\vec{\alpha}$ with $(m+1)$ components
and $\vec{\ell}$.
Let $\vec t \in [-1,0]^{m-1}$ and ${\bf q} \in \mathcal M_{\vec\ell}(X,H;\vec{\alpha})$.
We define a relation $<$ on the set 
$\mathcal B(\vec t)$
in Definition \ref{def:B}
such that $(A',B',C') < (A,B,C)$ if and only
if $\partial({\rm Im}\,\mathcal I_{A,B,C}) \supseteq
{\rm Im}\,\mathcal I_{A',B',C'}$.
\par
By Definition \ref{defn2659} (4) the set 
$\mathcal B(\vec t)$
is linearly ordered by this relation $<$. Therefore there exists a
maximal element, which we denote by $(A_0,B_0,C_0)$.
\par\smallskip
\noindent {\bf Case 1:} $B_0 \ne \underline{m-1}$.
\par\smallskip
\noindent {\bf Case 1-1:} $A_0 \ne \emptyset$.
Let $i \in A_0$.
We note that if $(A,B,C) \in \mathcal B(\vec t)
$
then ${\rm Im}\,\mathcal I_{A,B,C}$ is a face of
${\rm Im}\,\mathcal I_{A_0,B_0,C_0}$.
Since ${\rm Im}\,\mathcal I_{A_0,B_0,C_0} \subseteq \{ t_i=-1\}$,
we have ${\rm Im}\,\mathcal I_{A,B,C} \subseteq \{ t_i=-1\}$.
Therefore $i \in A$ by Definition \ref{defn2659} (4).
\par
We consider $\vec{\alpha}_1
= (\alpha_0,\dots,\alpha_i)$,
$\vec{\alpha}_2
= (\alpha_i,\dots,\alpha_m)$
and
$\vec{\ell}_1
= (\ell_0,\dots,\ell_i)$,
$\vec{\ell}_2
= (\ell_i,\dots,\ell_m)$.
We have
$$
({\bf q},\vec t) \in \mathcal M_{\vec\ell_1}(X,H;\vec{\alpha}_1)^+
\times_{R_{\alpha_i}}
\mathcal M_{\vec\ell_2}(X,H;\vec{\alpha}_2)^+.
$$
We decompose $({\bf q},\vec t)$ into $({\bf q}_1,\vec t_1)$
and $({\bf q}_2,\vec t_2)$.
Since $i \in A$ for all the elements $(A,B,C) \in \mathcal B(\vec t)
$, we find easily that
$$
\frak F({\bf q},\vec t)
\cong \frak F({\bf q}_1,\vec t_1) 
\sqcup
\frak F({\bf q}_2,\vec t_2).
$$
Therefore by induction hypothesis we can show that
Lemma \ref{lem26689} Condition $(*)$ holds at this element $({\bf q},\vec t)$.
\par\smallskip
\noindent {\bf Case 1-2:} $C_0 \ne \emptyset$.
Let $i \in C_0$. In the same way as in Case 1-1
we can show $i \in C$ for all $(A,B,C) \in 
\mathcal B(\vec t)$.
Then we put $\vec{\alpha}'
= (\alpha_0,\dots,\alpha_{i-1},\alpha_{i+1},\dots,\alpha_m)$,
 $\vec{\ell}'
= (\ell_1,\dots,\ell_{i-1},\ell_i + \ell_{i+1},\ell_{i+2},\dots,\ell_m)$.
We may identify $({\bf q},\vec t)$
as an element $({\bf q}',\vec t')$ of
$\mathcal M_{\vec\ell'}(X,H;\vec{\alpha}')^+$.
Then we have
$
\frak F({\bf q}',\vec t')
\cong \frak F({\bf q},\vec t)
$.
Therefore we can check Lemma \ref{lem26689} Condition $(*)$ by
induction hypothesis.
\par\smallskip
\noindent {\bf Case 2:} $B_0 = \underline{m-1}$.
In this case $A_0 = C_0 = \emptyset$.
Let $(A_1,B_1,C_1)$ be the maximal element of
$\mathcal B(\vec t)
\setminus \{(A_0,B_0,C_0)\}$.
Using the argument of Case 1, we can show that the
sum of the images of
$I_{{\bf p}_c,{\rm v};\Sigma_{{\bf q}_{\frak z}},u_{{\bf q}_{\frak z}},\varphi_{{\bf q}_{\frak z}}}^{\vec{\ell}_{\frak z},\vec{\alpha}_{\frak z},B({\frak z})}$
for $\frak z \in 
\mathcal B(\vec t)
\setminus \{(A_0,B_0,C_0)\}$,
${\bf p}_c \in G({\bf q};\vec{\ell},\vec{\alpha},B)$
and ${\rm v}$ (irreducible components of ${\bf p}_c$),
is a direct sum.
\par
Now we can make our choice for $(A_0,B_0,C_0)$
so that Lemma \ref{lem26689} Condition $(*)$ also holds in this case.
(See \cite[Lemma 11.7]{fooonewbook} for example.)
\par
The proof of Lemma \ref{lem26689} is now complete.
\end{proof}
The proof of Proposition \ref{prop2661} is now complete.
\end{proof}

\section{Construction of morphism}
\label{subsec;KuramodFloermor}
\subsection{Statement}
\label{subsec:statement}
Let $H^1, H^2$ be two functions $X \times S^1 \to \R$ which
are Morse-Bott non-degenerate in the sense of 
Condition \ref{weaknondeg}.
We put
\begin{equation}\label{267formrev}
\widetilde{\rm Per}(H^j)
= \coprod_{\alpha \in \frak A^j}R^j_{\alpha}
\end{equation}
as in (\ref{267form}) for $j=1,2$.
We define a local system $o_{R_{\alpha}^j}$
on each $R^j_{\alpha}$ as in Definition \ref{oricricial}.
Let $J_1, J_2$ be two almost complex structures tamed by $\omega$.
Using them we obtain linear K-systems $\mathcal F_j :=\mathcal F_{X}(H^j,J_j)$ ($j=1,2$)
by Theorem \ref{theorem266} whose spaces of connecting orbits are
$\mathcal M(X,J_j,H^j;\alpha_-,\alpha_+)^{\boxplus 1}$.
In this section we will construct a morphism
$\mathfrak{N}_{21} : \mathcal F_1 \to \mathcal F_2$ of linear K-systems
\cite[Definition 16.19]{fooonewbook}.
\begin{shitu}\label{situ2676}
We consider a smooth function\footnote{The reason we put $2,1$ in this order 
in the notations $H^{21}$ and $\mathfrak{N}_{21}$ is to be 
consistent with the order of compositions.}
$H^{21} : X \times \R \times S^1  \to \R$ such that:
\begin{enumerate}
\item
If $\tau < -1$, then
$H^{21}(x,\tau,t) = H^1(x,t)$.
\item
If $\tau > 1$, then
$H^{21}(x,\tau,t) = H^2(x,t)$.
\end{enumerate}
We put $H^{21}_{\tau,t}(x) = H(x,\tau,t)$ and denote by 
$\frak X_{H^{21}_{\tau,t}}$ the Hamiltonian vector field
associated to $H_{\tau,t}$.
\par
We also consider a one parameter
family $\mathcal J^{21} = \{J_{\tau}^{21}\}$
of almost complex structures tamed by $\omega$ such that:
\begin{enumerate}
\item[(i)]
If $\tau < -1$, then
$J_{\tau}^{21} = J_1$.
\item[(ii)]
If $\tau > 1$, then
$J_{\tau}^{21} = J_2$.
\end{enumerate} 
When no confusion can occur, we simply write $\mathcal J=\mathcal J^{21}$ in this section.
$\blacksquare$
\end{shitu}
\begin{defn}\label{def2677}
Suppose we are in Situation \ref{situ2676}.
Let $\alpha_- \in \frak A_1$ and $\alpha_+\in \frak A_2$.
We consider the set of smooth maps
$u : \R \times S^1 \to X$ with the following properties.
\begin{enumerate}
\item It satisfies
\begin{equation}\label{Fleqref}
\frac{\partial u}{\partial \tau} + J_{\tau}^{21} \left( \frac{\partial u}{\partial t} - \frak X_{H^{21}_{\tau,t}}
\circ u \right) = 0.
\end{equation}
Here $\tau$ and $t$ are the coordinates of $\R$ and $S^1 = \R/\Z$, respectively.
\item
There exist $\tilde{\gamma}^{-} = ({\gamma}^{-},w^{-}) \in R^1_{\alpha_-}$ and
$\tilde{\gamma}^{+} = ({\gamma}^{+},w^{+}) \in R^2_{\alpha_+}$
such that
\begin{equation}\label{bdlatinfrev}
\lim_{\tau \to \pm \infty}u(\tau,t) = \gamma^{\pm}(t)
\end{equation}
and $w^{-} \# u \sim w^{+}$.
\end{enumerate}
\par
We denote by $\mathcal N^{\rm reg}(X,\mathcal J,H^{21};\alpha_-,\alpha_+)$ the totality of such maps $u$.
We define ${\rm ev}_{\pm} : \mathcal N^{\rm reg}(X,\mathcal J,H^{21};\alpha_-,\alpha_+)
\to R^1_{\alpha_-}$, $R^2_{\alpha_+}$ by
${\rm ev}_{\pm}(u) = \tilde{\gamma}^{\pm}$.
\end{defn}
\begin{lem}\label{prof26rev3}
Suppose $u : \R \times S^1 \to X$ satisfies (\ref{Fleqref}). We assume
$$
\int_{\R \times S^1}
\left\Vert\frac{\partial u}{\partial \tau}\right\Vert^2 d\tau dt < \infty.
$$
Then
there exist $\tilde{\gamma}^{-} = ({\gamma}^{-},w^{-}) \in R^1_{\alpha_-}$
and
$\tilde{\gamma}^{+} = ({\gamma}^{+},w^{+}) \in R^2_{\alpha_+}$ such that (\ref{bdlatinfrev})
is satisfied.
\end{lem}
The proof is similar to the proof of Proposition \ref{prof263} and is omitted.

\begin{thm}\label{theorem266rev}
Suppose we are in the situation of Definition \ref{def2677}.
\begin{enumerate}
\item
The space
$\mathcal N^{\rm reg}(X,\mathcal J,H^{21};\alpha_-,\alpha_+)$ has a compactification
$$\mathcal N(X,\mathcal J,H^{21};\alpha_-,\alpha_+).$$
\item
The compact space $\mathcal N(X,\mathcal J,H^{21};\alpha_-,\alpha_+)$ has a Kuranishi structure with
corners.
The map ${\rm ev}$ is extended to it as a strongly smooth map.
\item
There exists a morphism of linear K-systems
$\frak N_{21} : \mathcal F_1 \to \mathcal F_2$
whose interpolation space
is $\mathcal N(X,\mathcal J,H^{21};\alpha_-,\alpha_+)^{\boxplus 1}$
given in Definition \ref{3equivrerevl} (2).
\item
The Kuranishi structure on $\mathcal N(X,\mathcal J,H^{21};\alpha_-,\alpha_+)^{\boxplus 1}$
in (3) coincides with the given one in (2) on
 $\mathcal N(X,\mathcal J,H^{21};\alpha_-,\alpha_+)
 \subset \mathcal N(X,\mathcal J,H^{21};\alpha_-,\alpha_+)^{\boxplus 1}$.
\end{enumerate}
\end{thm}
\begin{proof}
The proof of Theorem \ref{theorem266rev} occupies the rest of this section.

\subsection{Proof of Theorem \ref{theorem266rev} (1)(2): Kuranishi structure}
\label{subsec:compatification}
We begin with defining the compactification $\mathcal N_{\ell}(X,\mathcal J,H^{21};\alpha_-,\alpha_+)$.
Let $(\Sigma,(z_-,z_+,\vec z))$
be a genus zero semistable curve with $\ell + 2$ marked points.
We define the notion of mainstream as in Definition \ref{defn41}.
Let $\Sigma_a$ and $\Sigma_{a'}$ be two mainstream components.
\begin{defn}\label{defn2682}
We say $a < a'$ if the connected component of
$\Sigma \setminus \{z_{a',-}\}$ containing $z_-$ contains $\Sigma_{a}$.\footnote{$z_{a',-}$
and $z_{a',+}$ are transit points of $\Sigma_{a'}$ which are defined in 
Definition \ref{defn41}.}
\end{defn}
We observe that one of $a<a'$, $a' < a$ or $a=a'$ holds for any pair of mainstream components $
(\Sigma_a,\Sigma_{a'})$.
\begin{defn}\label{defn210rev}
The set $\widehat{\mathcal N}_{\ell}(X,\mathcal J,H^{21};\alpha_-,\alpha_+)$
consists of triples 
$$((\Sigma,(z_-,z_+,\vec z),a_0),u,\varphi)$$ satisfying
the following conditions: Here $\ell =\# \vec z$.
\begin{enumerate}
\item
$(\Sigma,(z_-,z_+,\vec z))$ is a genus zero semi-stable curve with 
$\ell + 2$ marked points.
\item
$\varphi$ is a parametrization of the mainstream.
\item
$\Sigma_{a_0}$ is one of the mainstream components.
We call it the {\it main component}.
\item
For each extended mainstream component $\widehat{\Sigma}_a$, the map
$u$ induces
$u_a : \widehat{\Sigma}_a \setminus\{z_{a,-},z_{a,+}\} \to X$
which is a continuous map.\footnote{In other words
$u$ is a continuous map from the complement of the set of the transit points.}
\item
If $\Sigma_a$ is a mainstream component and
$\varphi_a : \R \times S^1 \to \Sigma_a$ is as above, then
the composition $u_a \circ \varphi_a$ satisfies the equation
\begin{equation}\label{Fleq26621}
\frac{\partial (u_a \circ \varphi_a)}{\partial \tau} +  J_{a,\tau} \left( \frac{\partial (u_a \circ \varphi_a)}{\partial t} - \frak X_{H^a_{\tau,t}}
\circ (u_a \circ \varphi_a) \right) = 0,
\end{equation}
where
$$
H^a_{\tau,t} =
\begin{cases}
H^1_t  &\text{if $a < a_0$}, \\
H^{21}_{\tau,t}  &\text{if $a = a_0$}, \\
H^2_t  &\text{if $a > a_0$},
\end{cases}
$$
and
$$
J_{a,\tau} =
\begin{cases}
J_1 &\text{if $a < a_0$}, \\
J_{\tau}^{21}  &\text{if $a = a_0$}, \\
J_2 &\text{if $a > a_0$}.
\end{cases}
$$
\item
$$
\int_{\R \times S^1}
\left\Vert\frac{\partial (u \circ \varphi_a)}{\partial \tau}\right\Vert^2 d\tau dt < \infty.
\footnote{Condition (6) follows from the rest of the conditions in Definition \ref{defn210rev} using 
e.g., \cite[(2.14)]{On}.  
The same remark holds for 
Definitions \ref{defn210revrev}, \ref{defn210revto0}, \ref{defn210revto0rev}, \ref{defn210revto01}, \ref{defn210revto022}, \ref{defn210revto044}.}
$$
\item
Suppose $\Sigma_{\rm v}$  is a bubble component in $\widehat{\Sigma}_a$.
Let   $\varphi_a(\tau,t)$ be the root of the tree of sphere bubbles
containing $\Sigma_{\rm v}$.  Then $u$ is $J$-holomorphic on  $\Sigma_{\rm v}$ where
$$
J =
\begin{cases}
J_1 &\text{if $a < a_0$}, \\
J_{\tau}^{21}  &\text{if $a = a_0$}, \\
J_2 &\text{if $a > a_0$}.
\end{cases}
$$
\item
If
${\Sigma}_{a_1}$ and ${\Sigma}_{a_2}$ are
mainstream components and
$z_{a_1,+} = z_{a_2,-}$, then
$$
\lim_{\tau\to+\infty} (u_{a_1} \circ \varphi_{a_1})(\tau,t)
=
\lim_{\tau\to-\infty} (u_{a_2} \circ \varphi_{a_2})(\tau,t)
$$
holds for each $t \in S^1$. ((6) and Lemma \ref{prof26rev3} imply that the 
left and right hand sides both converge.)
\item
If
${\Sigma}_{a}$, ${\Sigma}_{a'}$
are
mainstream components and $z_{a,-} = z_-$, $z_{a',+}
= z_+$, then there exist $(\gamma_{\pm},w_{\pm})
\in R_{\alpha_{\pm}}$ such that
$$
\aligned
\lim_{\tau\to-\infty} (u_{a} \circ \varphi_{a})(\tau,t)
&= \gamma_-(t), \\
\lim_{\tau\to +\infty} (u_{a'} \circ \varphi_{a'})(\tau,t)
&= \gamma_+(t).
\endaligned
$$
Moreover
$$
[u_*[\Sigma]] \# w_- = w_+,
$$
where $\#$ is the obvious concatenation.
\item
We assume that $((\Sigma,(z_-,z_+,\vec z),a_0),u,\varphi)$
is stable in the sense of Definition \ref{stable26rev} below.
\end{enumerate}
\end{defn}
To define stability we first define the group of automorphisms.
\begin{defn}\label{defn2615rev}
Assume that $((\Sigma,(z_-,z_+,\vec z),a_0),u,\varphi)$
satisfies (1)-(9) above.
The {\it extended automorphism group}
${\rm Aut}^+((\Sigma,(z_-,z_+,\vec z),a_0),u,\varphi)$
of $((\Sigma,(z_-,z_+,\vec z),a_0),u,\varphi)$
consists of maps $v : \Sigma \to \Sigma$
with the following properties:
\begin{enumerate}
\item
$v(z_{-}) = z_{-}$ and $v(z_{+}) = z_{+}$.
In particular, $v$ preserves each of the mainstream component $\Sigma_a$
of $\Sigma$.
Moreover $v$ fixes each of the transit points.
\item
$u = u\circ v$ holds outside the set of the transit points.
\item
If $\Sigma_a$ is a mainstream component of $\Sigma$, 
there exists $\tau_a \in \R$ such that
\begin{equation}\label{preserveparame22}
(v \circ \varphi_a)(\tau,t) = \varphi_a(\tau+\tau_a,t)
\end{equation}
on $\R \times S^1$.
\item We require $\tau_{a_0} = 0$.
\item
There exists $\sigma \in {\rm Perm}(\ell)$ such that
$v(z_i) = z_{\sigma(i)}$.
\end{enumerate}
\par
The {\it automorphism group} denoted by 
${\rm Aut}((\Sigma,(z_-,z_+,\vec z),a_0),u,\varphi)$
of $((\Sigma,(z_-,z_+,\vec z),a_0),u,\varphi)$
consists of the elements of ${\rm Aut}^+((\Sigma,(z_-,z_+,\vec z),a_0),u,\varphi)$
such that $\sigma$ in Item (5) above is the identity.
\end{defn}
\begin{rem}
Definition \ref{defn2615rev} is mostly the same as
Definition \ref{defn2615}. The most important difference is Item (4)
where we assume $\tau_{a_0} = 0$. Note the equation (\ref{Fleq26621})
is {\it not} invariant under the translation of the $\R$ direction on the main component.
\end{rem}
\begin{defn}\label{stable26rev}
We say $((\Sigma,(z_-,z_+,\vec z),a_0),u,\varphi)$ is {\it stable}
if ${\rm Aut}((\Sigma,(z_-,z_+,\vec z),a_0),u,\varphi)$ is a finite group.
(This is equivalent to the finiteness of
${\rm Aut}^+((\Sigma,(z_-,z_+,\vec z),a_0),u,\varphi)$.)
\end{defn}
\begin{defn}\label{3equivrerevl}
\begin{enumerate}
\item 
On the set $\widehat{\mathcal N}_{\ell}(X,\mathcal J,H^{21};\alpha_-,\alpha_+)$  we define two
equivalence relations $\sim_1$, $\sim_2$.
The definition of $\sim_1$ is the same as Definition \ref{3equivrel}.
The definition of $\sim_2$ is the same as Definition \ref{3equivrel} except
we require $\tau_{a_0} = 0$, in addition.
We put
$$\aligned
\widetilde{\mathcal N}_{\ell}(X,\mathcal J,H^{21};\alpha_-,\alpha_+) &=
\widehat{\mathcal N}_{\ell}(X,\mathcal J,H^{21};\alpha_-,\alpha_+)/\sim_1, \\
{\mathcal N}_{\ell}(X,\mathcal J,H^{21};\alpha_-,\alpha_+) &=
\widehat{\mathcal N}_{\ell}(X,\mathcal J,H^{21};\alpha_-,\alpha_+)/\sim_2.
\endaligned$$
In the case $\ell = 0$ we write
${\mathcal N}(X,H;\alpha_-,\alpha_+)$ etc.
\item
We define 
$\mathcal N(X,\mathcal J,H^{21};\alpha_-,\alpha_+)^{\boxplus 1}$
in the same way as in Definition \ref{def:outcollar}.
Namely, it is the set of equivalence classes of 
objects 
$(((\Sigma,(z_-,z_+,\vec z),a_0),u,\varphi),\vec t)$
where 
$((\Sigma,(z_-,z_+,\vec z),a_0),u,\varphi) \in 
\mathcal N(X,\mathcal J,H^{21};\alpha_-,\alpha_+)$ 
and $\vec t$ assigns numbers 
$t_p \in [-1,0]$ to each transit points.\footnote{
This space is the outer collaring of $\mathcal N(X,\mathcal J,H^{21};\alpha_-,\alpha_+)$ in the sense of \cite[Definition 17.29]{fooonewbook}.}
\end{enumerate}
\end{defn}
\begin{defn}\label{defn2688Nsource}
We put $X = \text{one point}$ and $H^{21}\equiv 0$.
Then we obtain the space ${\mathcal N}_{\ell}(\text{one point},\mathcal J,0;\alpha_0,\alpha_0)$.
Here $\alpha_0$ is the unique point in ${\rm Per}(0)$.
We denote this space by ${\mathcal N}_{\ell}(\text{source})$.
\end{defn}
We remark that ${\mathcal N}_{\ell}(\text{source})$
is similar to but is different
from ${\mathcal M}_{\ell}(\text{source})$.
In fact, ${\mathcal N}_{\ell}(\text{source})$ includes the data
that specify which mainstream component is the main component
and also the isomorphism between two elements of ${\mathcal N}_{\ell}(\text{source})$ is required to be strictly compatible with
the parametrization of the main component.
\begin{exm}
${\mathcal N}_{0}(\text{source})$ is one point.
${\mathcal N}_{1}(\text{source})$ is $S^1 \times [0,1]$.
In fact, if there is only one mainstream component and the marked point is
$\varphi(\tau,t)$, then the coordinates $(\tau,t)$ determine an element of ${\mathcal N}_{1}(\text{source})$.
We compactify it by including the case when there are two mainstream components.
In such a case the marked point can not be on the main component.
So the $S^1$ factor of the coordinates of the marked points determine an element of ${\mathcal N}_{1}(\text{source})$.
There are two cases: $a < a_0$ or $ a > a_0$. 
(Here $a$ is the mainstream component which is not the
main component.) 
Thus ${\mathcal N}_{1}(\text{source})$ is a union of $\R \times S^1$ and two
copies of $S^1$.
\end{exm}
We can define a topology on
${\mathcal N}_{\ell}(X,\mathcal J,H^{21};\alpha_-,\alpha_+)$
in the same way as in Definition \ref{def2626} and can prove the following:
\begin{lem}\label{cpthausdorff}
The space ${\mathcal N}_{\ell}(X,\mathcal J,H^{21};\alpha_-,\alpha_+)$ is compact and
Hausdorff.
\end{lem}
Next, we define a Kuranishi structure on the compactification ${\mathcal N}_{\ell}(X,\mathcal J,H^{21};\alpha_-,\alpha_+)$.
We define the notion of symmetric stabilization of an element $[((\Sigma,(z_-,z_+,\vec z),a_0),u,\varphi)]$
of ${\mathcal N}_{\ell}(X,\mathcal J,H^{21};\alpha_-,\alpha_+)$ in exactly the same way as in Definition \ref{symstab}.
\par
We define the notion of {\it canonical marked point} $w_{a,{\rm can}}$ of a mainstream component $\Sigma_a$
of $\Sigma$ such that there is no marked or singular points on $\Sigma_a$  other than transit point, as follows.
If $a\ne a_0$, the definition of $w_{a,{\rm can}}$ is exactly the same as
Definition \ref{defncanmark}.
For $a = a_0$, we do not define canonical marked points.
\begin{rem}\label{rem269494}
Note that $\mathcal M_0({\rm source}) = \emptyset$. This is the
reason we need to introduce the canonical marked points.
On the other hand, $\mathcal N_0({\rm source})$ consists of one point and is not an empty set.
Namely the unique element in it is represented by a stable object.
This is the reason we do not need to introduce the canonical marked points
on the main component.
\end{rem}
Let $\vec w_{\rm can}$ be the totality of all the canonical marked points.
In the same way as in Lemma \ref{lemma2639},  
we can prove that
$((\Sigma,(z_-,z_+,\vec z),a_0) \cup \vec w \cup \vec w_{\rm can},\varphi)$
is stable\footnote{We say 
$((\Sigma,(z_-,z_+,\vec z),a_0)\cup \vec w\cup \vec w_{\rm can},\varphi)$ is {\it stable} if the automorphism group 
${\rm Aut}^+((\Sigma,(z_-,z_+,\vec z),a_0),\varphi)$ defined as in Definition \ref{defn2615rev} by removing $u$ (the condition (2)) is finite.}, where
$[((\Sigma,(z_-,z_+,\vec z),a_0),u,\varphi)] \in
{\mathcal N}_{\ell}(X,\mathcal J,H^{21};\alpha_-,\alpha_+)$ and
$\vec w$ is a symmetric stabilization.
\par
We then define the notion of obstruction bundle data $\frak C_{\bf p}$
for each element
${\bf p} \in {\mathcal N}_{\ell}(X,\mathcal J,H^{21};\alpha_-,\alpha_+)$
in the same way as in Definition \ref{obbundeldata1}.
Its existence can be proved easily.
(See for example \cite[Lemma 17.11]{foootech}.)
\par
We next explain how we use the obstruction bundle data $\frak C_{\bf p}$
to define an obstruction space
for each object close to ${\bf p} \in {\mathcal N}_{\ell}(X,\mathcal J,H^{21};\alpha_-,\alpha_+)$.
\par
Let ${\bf p} = [((\Sigma,(z_-,z_+,\vec z),a_0),u,\varphi)]$.
We assume that $\Sigma$ has exactly $k$ transit points.
Taking the condition that the main component is equipped with 
the parametrization $\varphi_{a_0}$ into account, we obtain a map 
\begin{equation}\label{form418rev}
\aligned
{\Phi}_{\bf p} : \prod_{\rm v} \mathcal V(\frak x_{\rm v} \cup \vec w_{\rm v}\cup \vec w_{{\rm can},{\rm v}}) \times (T_0,\infty]^k
\times &\prod_{j=1}^m\left( ((T_{0,j},\infty] \times S^1)/\sim \right)
\\
&\to {\mathcal N}_{\ell+\ell'+\ell''}(\text{source})
\endaligned\end{equation}
in the same way as in \eqref{form418}.
Here $\mathcal V(\frak x_{\rm v} \cup \vec w_{\rm v}\cup \vec w_{{\rm can},{\rm v}})$
is an open subset of 
$\overset{\circ}{\mathcal M_*^{\rm cl}}$
or
$\overset{\circ}{\mathcal M}_*({\rm source})$. See Definition \ref{obbundeldata1} (2) (a)(b).
\par
The factor $(T_0,\infty]^k$ parametrizes the way
how we smooth the singular points that are the transit points.
In (\ref{form418}) the similar factor is ${D}(k;\vec T_0)$.
The difference is that in the situation of (\ref{form418}) the isomorphism
$v : \Sigma \to \Sigma$ includes the translation,
which is a map $v$ such that $v\circ \varphi_a(\tau,t) =
\varphi_a(\tau+\tau_0,t)$. This shifts $T_i \in (T_{0,j},\infty]$
by $\tau_0$.  (Here $\tau_0$ is independent of $a$.)
On the other hand, the isomorphism $v$ here is required
to  commute strictly with $\varphi_{a_0}$. So
there is not such a shift.
Taking this point into account, the map (\ref{form418rev})
is defined in the same way as in \eqref{form418}.
\par
Now let
\begin{equation}\label{form2663}
({\frak Y} \cup \vec w',\varphi', a_0') = {\Phi}_{\bf p}(\frak y,\vec T,\vec \theta) \in {\mathcal N}_{\ell+\ell'+\ell''}(\text{source})
\end{equation}
where $\frak y = (\frak y_{\rm v})$. ($\rm v$ is in the set of irreducible components of
${\bf p}$.)
Here $\vec w'$ is the set of the additional marked points corresponding to
$\vec w$ and $\vec w_{\rm can}$.
The notation $\frak Y$ includes the marked points corresponding to $\vec z$ and $z_{\pm}$.
$a'_0$ is the datum to specify the main component and $\varphi'$ is the
parametrization of the mainstream.
$(\vec T,\vec \theta) \in (T_0,\infty]^k
\times \prod_{j=1}^m \left(((T_{0,j},\infty] \times S^1)/\sim \right).
$
\par
Let $u' : \Sigma' \setminus \{\text{transit points}\} \to X$.
We assume
that $(\frak Y,u',\varphi',a_0')$ satisfies
Definition \ref{defn210rev} (1)(2)(3)(4)(8)(9).
We can then define the notion
that $(\frak Y\cup \vec w',u',\varphi',a_0')$ is $\epsilon$-close to
${\bf p} \cup \vec w \cup \vec w_{\rm can}$ in the same way as in Definition \ref{orbitecloseness}.
(Note  Definition \ref{orbitecloseness} (2) is the pseudo-holomorphicity
(or pseudo-holomorphicity with Hamiltonian term)
at the neck region. We use the almost complex structure specified in
Definition \ref{defn210rev}
(5).)
\begin{defn}\label{transdef2695}
We define the {\it transversal constraint} for
$(\frak Y\cup \vec w',u',\varphi',a'_0)$ as follows.
Let $w'_i$ be one of the points of
$\vec w'$. If $w'_i$ corresponds to a point in $\vec w$,
we require Definition \ref{constrainttt2} (1).
If $w'_i$ corresponds to $w_{a,{\rm can}}$ with $a \ne a'_0$, we
require Definition \ref{constrainttt2} (2)(3).
\end{defn}\label{defn269955}
Suppose $(\frak Y\cup \vec w',u',\varphi',a'_0)$ is $\epsilon$-close to
${\bf p} \cup \vec w \cup \vec w_{\rm can}$.
In the same way as in (\ref{defn2642}), we define a map
\begin{equation}\label{form2664}
I_{{\bf p},{\rm v};\Sigma',u',\varphi',a'_0} : E_{{\bf p},{\rm v}}(\frak y) \to
C^{\infty}(\Sigma';(u')^*TX \otimes \Lambda^{0,1})
\end{equation}
where $\frak y$ is as in (\ref{form2663}) and $\rm v$ is its
irreducible component, and 
$E_{{\bf p},{\rm v}}(\frak y)$ is defined as a part of the obstruction bundle data 
$\frak C_{\bf p}$ as in \eqref{eq:obbunddata} (see Definition \ref{obbundeldata1} (5)).
\par
Now in the same way as in Choice \ref{choice2650},
we proceed as follows.
We first observe that for any 
${\bf p} \in {\mathcal N}_{\ell}(X,\mathcal J,H^{21};\alpha_-,\alpha_+)$
there exist $\epsilon_{\bf p} >0$ and a closed small neighborhood 
$W({\bf p})$ of $\bf p$ such that 
if ${\bf q}  \in W({\bf p})$
there exists $\vec w^{\bf q}_{{\bf p}}$ {\rm uniquely} with the following
properties:
\begin{enumerate}
\item
${\bf q} \cup \vec w^{\bf q}_{{\bf p}}$ is $\epsilon_{\bf p}$-close to
${\bf p} \cup w_{{\bf p}} \cup \vec w_{{\rm can}}$.
\item
${\bf q} \cup \vec w^{\bf q}_{{\bf p}}$ satisfies the
transversal constraint.
\item
The linearization operator 
$D_{\bf q} \overline{\partial}_{J, H^{21}}$ 
at $\bf q$ as in \eqref{eq:linearized} 
is surjective $\mod \oplus_{\rm v}\operatorname{Im}~ I_{{\bf p},{\rm v};{\bf q}}$,
where $I_{{\bf p},{\rm v};{\bf q}}$ is the map in \eqref{form2664}
for ${\bf q}=(\Sigma', u', \varphi')$.

\end{enumerate}
The proof of this fact is the same as that of Lemma \ref{lem2446}.
Then we have 
$$
\bigcup_{\bf p} 
 {\rm Int}\,\,W({\bf p})
= {\mathcal N}_{\ell}(X,\mathcal J,H^{21};\alpha_-,\alpha_+).
$$
Therefore by compactness of the moduli space 
we can take a finite subset indexed by a finite set 
$\EuScript C_{\ell}(H^{21},\mathcal J;\alpha_-,\alpha_+)$ 
$$\EuScript A_{\ell}(H^{21},\mathcal J;\alpha_-,\alpha_+)
= \{{\bf p}_c \mid c \in \EuScript C_{\ell}(H^{21},\mathcal J;\alpha_-,\alpha_+) \}
\subset {\mathcal N}_{\ell}(X,\mathcal J,H^{21};\alpha_-,\alpha_+)
$$
such that 
for each $c \in \EuScript C_{\ell}(H^{21},\mathcal J;\alpha_-,\alpha_+)$
we take obstruction bundle data $\frak E_{{\bf p}_{c}}$
centered at ${\bf p}_c$,
and
a closed neighborhood $W({\bf p}_c)$ of ${\bf p}_c$
in ${\mathcal N}_{\ell}(X,\mathcal J,H^{21};\alpha_-,\alpha_+)$
with the following property. For each element  ${\bf q}  \in W({\bf p}_c)$
there exists $\vec w^{\bf q}_{{\bf p}_c}$ such that
\begin{enumerate}
\item
${\bf q} \cup \vec w^{\bf q}_{{\bf p}_c}$ is $\epsilon_c$-close
to ${\bf p}_c \cup w_{{\bf p}_c} \cup \vec w_{{\rm can}}$.
\item
${\bf q} \cup \vec w^{\bf q}_{{\bf p}_c}$ satisfies the
transversal constraint.
\item
\begin{equation}\label{coversururev}
\bigcup_{c\in \EuScript C_{\ell}(H^{21},\mathcal J;\alpha_-,\alpha_+)} 
 {\rm Int}\,\,W({\bf p}_c)
= {\mathcal N}_{\ell}(X,\mathcal J,H^{21};\alpha_-,\alpha_+).
\end{equation}
\end{enumerate}
\begin{defn}\label{defn2646rev}
\begin{enumerate}
\item
For each ${\bf q} \in {\mathcal N}_{\ell}(X,\mathcal J,H^{21};\alpha_-,\alpha_+)$
we put
$$
\EuScript E({\bf q}) = \{ c \in \EuScript C_{\ell}(H^{21},\mathcal J;\alpha_-,\alpha_+) \mid {\bf q} \in W({\bf p}_c)\}.
$$
\item
Let $\EuScript B \subset \EuScript E({\bf q})$ be a nonempty subset.
\item
We consider $(\frak Y\cup\bigcup_{c\in \EuScript B}\vec w'_c,u',\varphi',a'_0)$
such that for each $c$,
$(\frak Y\cup \vec w'_c,u',\varphi',a'_0)$ is $\epsilon$-close to ${\bf q} \cup \vec w_{c}^{\bf q}$.
If $\epsilon >0$ is small, then $(\frak Y\cup \vec w'_c,u',\varphi',a'_0)$ is $\epsilon$-close to 
${\bf p}_c \cup \vec w_{c}$
and we can define the map
$$
I_{{\bf p}_c,{\rm v};\Sigma',u',\varphi',a'_0} : 
E_{{\bf p}_c,{\rm v}}(\frak y_{{\bf p}_c}) \to
C^{\infty}(\Sigma';(u')^*TX \otimes \Lambda^{0,1})
$$
as in \eqref{form2664} for each irreducible component ${\rm v}$ of ${\bf p}_c$.
Here
$(\frak Y\cup \vec w'_c,\varphi',a'_0)=  {\Phi}_{{\bf p}_{c}}(\frak y_{{\bf p}_c},\vec T_{{\bf p}_c},\vec \theta_{{\bf p}_c})$
as in Notation \ref{nota:415}.
We now put
\begin{equation}\label{obspacedefHFHFrev}
E((\frak Y\cup\bigcup_{c \in \EuScript B} \vec w'_c,u',\varphi',a'_0);{\bf q};\EuScript B)
=
\bigoplus_{c \in \EuScript B}\bigoplus_{\rm v} {\rm Im} \,\, 
I_{{\bf p}_c,{\rm v};\Sigma',u',\varphi',a'_0}.
\end{equation}
\end{enumerate}
\end{defn}
We can define the notion of the {\it stabilization data} centered at ${\bf q}
\in {\mathcal N}_{\ell}(X,\mathcal J,H^{21};\alpha_-,\alpha_+)$
in the same way as in 
Definition \ref{stabilizationdefdata}.
\par
Using those data we fixed, we will define
a Kuranishi chart of ${\bf q}$ as follows.
\begin{defn}\label{defn2650rev}
We consider the following conditions on
an object
$(\frak Y\cup\bigcup_{c\in \EuScript E({\bf q})}\vec w'_c \cup \vec w'_{\bf q},u',\varphi',a'_0)$:
\begin{enumerate}
\item
If $\Sigma'_a$ is the mainstream component and $\varphi'_{a}$ is a parametrization
of this  mainstream component, the following equation is satisfied on
$\R \times S^1$.
\begin{equation}\label{Fleqobstincl1rev}
\aligned
&\frac{\partial(u'\circ \varphi'_{a})}{\partial \tau} 
+  J_1 \left( \frac{\partial(u'\circ \varphi'_{a})}{\partial t} - \frak X_{H^1_t}
\circ u'  \circ \varphi'_{a}\right) \\
& \equiv 0 
\mod E((\frak Y\cup\bigcup_{c\in \EuScript B} \vec w'_c,u',\varphi',a'_0);{\bf q};\EuScript B)
\endaligned
\end{equation}
for $a < a'_0$, 
\begin{equation}\label{Fleqobstincl2rev}
\aligned
&\frac{\partial(u'\circ \varphi'_{a})}{\partial \tau} 
+  J_{2} \left( \frac{\partial(u'\circ \varphi'_{a})}{\partial t} - \frak X_{H^2_t}
\circ u'  \circ \varphi'_{a}\right) \\
& \equiv 0 
\mod E((\frak Y\cup \bigcup_{c\in \EuScript B}\vec w'_c,u',\varphi',a'_0);{\bf q};\EuScript B)
\endaligned
\end{equation}
for $a > a'_0$,
\begin{equation}\label{Fleqobstincl3rev}
\aligned
&\frac{\partial(u'\circ \varphi'_{a})}{\partial \tau} 
+ J_{\tau}^{21} \left( \frac{\partial(u'\circ \varphi'_{a})}{\partial t} - \frak X_{H^{21}_{(\tau,t)}}
\circ u'  \circ \varphi'_{a}\right)\\
& \equiv 0 \mod E((\frak Y\cup \bigcup_{c\in \EuScript B}\vec w'_c,u',\varphi',a'_0);{\bf q};\EuScript B)
\endaligned
\end{equation}
for $a = a'_0$.
\item
If $\Sigma'_{\rm v}$  is a bubble component, the following equation is satisfied on $\Sigma'_{\rm v}$.
\begin{equation}\label{eq:614}
\overline{\partial}_J u' \equiv 0  \mod E((\frak Y\cup \bigcup_{c\in \EuScript B}\vec w'_c,u',\varphi',a'_0);{\bf q};\EuScript B).
\end{equation}
Here the almost complex structure $J$ is as follows.
Let $\widehat\Sigma'_a$ be the extended mainstream component containing
$\Sigma'_{\rm v}$. If $a < a'_0$, then $J = J_1$. If $a > a'_0$, then
$J = J_2$. If $a = a'_0$ and $\varphi_{a'_0}(\tau,t)$ is the root of the
tree of sphere components containing $\Sigma'_{\rm v}$,
then $J = J_{\tau}^{21}$.
\item For each $c \in \EuScript E({\bf q})$ the additional marked points $\vec w'_c $
satisfy the transversal constraint with respect to ${\bf p}_c$.
\item The additional marked points $\vec w'_{\bf q}$
satisfy the transversal constraint with respect to ${\bf q}$.
\item
$(\frak Y\cup\bigcup_{c\in \EuScript E({\bf q})}\vec w'_c \cup \vec w'_{\bf q},u',
\varphi',a'_0)$ is
$\epsilon_1$-close to
${\bf q} \cup \bigcup_{c\in \EuScript E({\bf q})}\vec w^{\bf q}_c \cup \vec w_{\bf q}$.
\end{enumerate}
\par
The set of isomorphism classes of
$(\frak Y\cup\bigcup_{c\in \EuScript E({\bf q})}\vec w'_c \cup \vec w'_{\bf q},u',\varphi',a'_0)$ satisfying the conditions
(1)-(5) above is denoted by
$$
V({\bf q},\epsilon_1,\EuScript B),
$$
where $(\frak Y'\cup\bigcup_{c\in \EuScript E({\bf q})}\vec w'_c \cup \vec w'_{\bf q},u',
\varphi',a'_0)$
is said to be {\it isomorphic} to $(\frak Y''\cup\bigcup_{c\in \EuScript E({\bf q})}\vec w''_c \cup \vec w''_{\bf q},u'',\varphi'',a''_0)$
if there exists a biholomorphic map $v : \Sigma' \to \Sigma''$ such that
\begin{enumerate}
\item[(a)]
$u'' = u'\circ v$ holds outside the set of the transit points.
\item[(b)]
If $\Sigma'_a$ is a mainstream component of $\Sigma'$
and $v(\Sigma'_a) = \Sigma''_{a'}$, then we have
$
(v \circ \varphi'_a)(\tau+\tau_a,t) = \varphi''_{a'}(\tau,t)
$
for some $\tau_a$.
\item[(c)]
We assume that if $a = a'_0$ and $v(\Sigma'_a) = \Sigma''_{a'}$, then $a' = a''_0$.
Moreover $\tau_{a'_0} = 0$.
\item[(d)]
$v(z'_i) = z''_i$ and $v(w'_i) = w''_i$.
\end{enumerate}
\end{defn}
In the same way as in Lemma \ref{lem2653}, we can prove that
$V({\bf q},\epsilon_1,\EuScript B)$
is a smooth manifold with boundary and corner
if $\epsilon_1 >0$ and $\epsilon_c >0$ are small enough.
\par
We note that the group ${\rm Aut}^+({\bf q})$ acts
on $V({\bf q},\epsilon_1,\EuScript B)$ since
the stabilization data are assumed to preserve it.
In particular, ${\rm Aut}({\bf q})$ acts on it.
By the condition in Definition \ref{obbundeldata1} (6)
this action is effective.
Therefore the quotient space
$V({\bf q},\epsilon_1,\EuScript B)/{\rm Aut}({\bf q})$
is an effective orbifold,
which we denote by 
$$
U({\bf q},\epsilon_1,\EuScript B).
$$
We define a vector bundle on $U({\bf q},\epsilon_1,\EuScript B)$
such that its fiber at
$(\frak Y,u',\bigcup_{c\in \EuScript E({\bf q})}\vec w'_c \cup \vec w'_{\bf q},\varphi',
a'_0)$ is
$E((\frak Y\cup \bigcup_{c\in \EuScript B}\vec w'_c,u',\varphi',
a'_0);{\bf q};\EuScript B)$.
We denote this vector bundle by 
$$
E({\bf q},\epsilon_1,\EuScript B).
$$
We can define its section
$s_{({\bf q},\epsilon_1,\EuScript B)}$ by using the left hand side of
(\ref{Fleqobstincl1rev})-(\ref{eq:614}).
An element of its zero set represents an element of
${\mathcal N}_{\ell}(X,\mathcal J,H^{21};\alpha_-,\alpha_+)$.
Thus we obtain:
$$
\psi_{({\bf q},\epsilon_1,\EuScript B)}
: s_{({\bf q},\epsilon_1,\EuScript B)}^{-1}(0)
\to {\mathcal N}_{\ell}(X,\mathcal J,H^{21};\alpha_-,\alpha_+).
$$
We thus obtain a Kuranishi chart
$$
(U({\bf q},\epsilon_1,\EuScript B),E({\bf q},\epsilon_1,\EuScript B),
s_{({\bf q},\epsilon_1,\EuScript B)},\psi_{({\bf q},\epsilon_1,\EuScript B)}).
$$
In the way similar to the proof of Lemmas \ref{lem2657}, \ref{lem2658},
and also by using the exponential decay estimates in the same way as in
\cite{foooanalysis},
we can define coordinate change among them.
We thus obtain a Kuranishi structure on
${\mathcal N}_{\ell}(X,\mathcal J,H^{21};\alpha_-,\alpha_+)$.
We have proved Theorem \ref{theorem266rev} (1)(2).
\subsection{Proof of Theorem \ref{theorem266rev} (3)(4): Kuranishi structure with outer collar}
\label{subsec:proof}
The strategy of the proof of Theorem \ref{theorem266rev} (3)(4) is similar to that in Section \ref{subsec;KuramodFloercor}.
Namely we first take
the outer collaring
${\mathcal N}_{\ell}(X,\mathcal J,H^{21};\alpha_-,\alpha_+)^{\boxplus 1}$
as in Section \ref{subsec;KuramodFloercor}
and modify them on
$S_k({\mathcal N}_{\ell}(X,\mathcal J,H^{21};\alpha_-,\alpha_+)) \times [-1,0]^k$.
(Note that the union of 
$S_k({\mathcal N}_{\ell}(X,\mathcal J,H^{21};\alpha_-,\alpha_+)) \times [-1,0]^k$
for various $k$ is 
${\mathcal N}_{\ell}(X,\mathcal J,H^{21};\alpha_-,\alpha_+)^{\boxplus 1}$.)
The details are as follows.
\par
Let $\frak A_r$ be the index set of the critical submanifolds of $H^r$ ($r=1,2$).
Let 
$$
\aligned
\alpha_-=\alpha_{1,0},\alpha_{1,1},\dots,\alpha_{1,m_1-1},\alpha_{1,m_1}
& \in \frak A_1, \\
\alpha_{2,1},\dots,\alpha_{2,m_2},\alpha_{2,m_2+1} = \alpha_+
& \in \frak A_2.
\endaligned
$$
We consider the fiber product
\begin{equation}\label{form2641rev}
\aligned
&{\mathcal M}_{\ell_{1,1}}(X,J_1,H^1;\alpha_{1,0},\alpha_{1,1})
\,{}_{{\rm ev}_{+}}\times_{{\rm ev}_-}  \dots \\
&\qquad\qquad\qquad\qquad\dots
\,{}_{{\rm ev}_{+}}\times_{{\rm ev}_-}
{\mathcal M}_{\ell_{1,m_1}}(X,J_1,H^1;\alpha_{1,m_1-1},\alpha_{1,m_1})
\\
&
\,{}_{{\rm ev}_{+}}\times_{{\rm ev}_-}
{\mathcal N}_{\ell'}(X,\mathcal J,H^{21};\alpha_{1,m_1},\alpha_{2,1})
\\
&\,{}_{{\rm ev}_{+}}\times_{{\rm ev}_-}{\mathcal M}_{\ell_{2,2}}(X,J_2,H^2;\alpha_{2,1},\alpha_{2,2})
\,{}_{{\rm ev}_{+}}\times_{{\rm ev}_-}  \dots \\
&\qquad\qquad\qquad\qquad\dots
\,{}_{{\rm ev}_{+}}\times_{{\rm ev}_-}
{\mathcal M}_{\ell_{2,m_2+1}}(X,J_2,H^2;\alpha_{2,m_2},\alpha_{2,m_2+1})
\endaligned
\end{equation}
which we denoted by
$
{\mathcal N}_{\vec \ell_1,\ell',\vec \ell_2}(X,\mathcal J,H^{21};\vec \alpha_1,\vec \alpha_2).
$
We observe that
$$
\aligned
&S_m({\mathcal N}_{\ell}(X,\mathcal J,H^{21};\alpha_-,\alpha_+)) \\
&=
\bigcup_{\vec \ell_1,\ell',\vec \ell_2
\atop \vert\vec \ell_1\vert + \ell' + \vert\vec \ell_2 \vert= \ell}
\bigcup_{\vec \alpha_1,\vec \alpha_2
\atop m_1 + m_2 = m,
\,\,\alpha_{1,0} = \alpha_-, \alpha_{2,m_2+1} = \alpha_+ }
 {\mathcal N}_{\vec \ell_1,\ell',\vec \ell_2}(X,\mathcal J,H^{21};\vec \alpha_1,\vec \alpha_2).
\endaligned
$$
Note the sum is taken for $m_1 = \#\vec \alpha_1-1 \ge 0$ and $m_2 = \#\vec \alpha_2-1 \ge 0$.
We will construct a Kuranishi structure for each of
$$
{\mathcal N}_{\vec \ell_1,\ell',\vec \ell_2}(X,\mathcal J,H^{21};\vec \alpha_1,\vec \alpha_2)^+
=
{\mathcal N}_{\vec \ell_1,\ell',\vec \ell_2}(X,\mathcal J,H^{21};\vec \alpha_1,\vec \alpha_2)
\times  [-1,0]^{m_1} \times  [-1,0]^{m_2}.
$$
Let $A_r \sqcup B_r  \sqcup C_r = \underline{m_r}$ for $r=1,2$.
It induces $
\mathcal I_{A_r,B_r,C_r} : [-1,0]^{b_r} \to [-1,0]^{m_r}
$
by (\ref{form2642}).
\par
We will formulate the compatibility condition below (Condition \ref{conds2660rev}),
which describes the restriction of the
Kuranishi structure
$\widehat{\mathcal U}_{\vec \ell_1,\ell',\vec \ell_2}(X,\mathcal J,H^{21};\vec \alpha_1,\vec \alpha_2)$ of
the product space
$
{\mathcal N}_{\vec \ell_1,\ell',\vec \ell_2}(X,\mathcal J,H^{21};\vec \alpha_1,\vec \alpha_2)^+
$
to the image of the embedding:
\begin{equation}\label{2642formrev}
\aligned
{\rm id} \times \mathcal I_{A_1,B_1,C_1} \times \mathcal I_{A_2,B_2,C_2}:
&{\mathcal N}_{\vec \ell_1,\ell',\vec \ell_2}(X,\mathcal J,H^{21};\vec \alpha_1,\vec \alpha_2) \times [-1,0]^{b_1+b_2} \\
&\to
{\mathcal N}_{\vec \ell_1,\ell',\vec \ell_2}(X,\mathcal J,H^{21};\vec \alpha_1,\vec \alpha_2)^+.
\endaligned
\end{equation}
We put
$A_r = \{ i(A_r,1),\dots, i(A_r,a_r)\}$ with 
$i(A_r,1)<i(A_r,2) < \dots< i(A_r,a_r-1) < i(A_r,a_r)$ and consider the fiber product
\begin{equation}\label{264222rev}
\aligned
&{\mathcal M}_{\vec \ell_{1,A_1,1}}(X,J_1,H^1;\alpha_{1,0},\dots,\alpha_{1, i(A_1,1)})
 \\
&
\,{}_{{\rm ev}_{+}}\times_{{\rm ev}_-} {\mathcal M}_{\vec \ell_{1,A_1,2}}(X,J_1,H^1;\alpha_{1, i(A_1,1)},\dots,\alpha_{ i(A_1,2)})
\,{}_{{\rm ev}_{+}}\times_{{\rm ev}_-} \dots \\
&\,{}_{{\rm ev}_{+}}\times_{{\rm ev}_-} {\mathcal M}_{\vec \ell_{1,A_1,j+1}}(X,J_1,H^1;\alpha_{ i(A_1,j)},\dots,\alpha_{ i(A_1,j+1)})
\,{}_{{\rm ev}_{+}}\times_{{\rm ev}_-}
\dots \\
&
\,{}_{{\rm ev}_{+}}\times_{{\rm ev}_-} {\mathcal M}_{\vec\ell_{1,A_1,a_1}}(X,J_1,H^1;\alpha_{1, i(A_1,a_1-1)},\dots,\alpha_{1,i(A_1,a_1)}) \\
&\,{}_{{\rm ev}_{+}}\times_{{\rm ev}_-}
{\mathcal N}_{\vec\ell_{1,A_1,a_1+1},\ell',\vec\ell_{2,A_2,1}}(X,\mathcal J,H^{21}; \\
&\qquad\qquad\qquad\qquad\qquad\qquad(\alpha_{1, i(A_1,a_1)},\dots,\alpha_{1,m_1}),
(\alpha_{2,1},\dots,\alpha_{2,i(A_2,1)}))
 \\
&
\,{}_{{\rm ev}_{+}}\times_{{\rm ev}_-} {\mathcal M}_{\vec \ell_{2,A_2,2}}(X,J_2,H^2;\alpha_{2, i(A_2,1)},\dots,\alpha_{ i(A_2,2)})
\,{}_{{\rm ev}_{+}}\times_{{\rm ev}_-} \dots \\
&\,{}_{{\rm ev}_{+}}\times_{{\rm ev}_-} {\mathcal M}_{\vec \ell_{2,A_2,j+1}}(X,J_2,H^2;\alpha_{ i(A_2,j)},\dots,\alpha_{ i(A_2,j+1)})
\,{}_{{\rm ev}_{+}}\times_{{\rm ev}_-}
\dots \\
&
\,{}_{{\rm ev}_{+}}\times_{{\rm ev}_-} {\mathcal M}_{\vec\ell_{2,A_2,a_2+1}}(X,J_2,H^2;\alpha_{2, i(A_2,a_2)},\dots,\alpha_{2,m_2+1}).
\endaligned
\end{equation}
Here
$\vec\ell_{r,A_r,j}
=(\ell_{r,i(A_r,j-1)+1},\dots,\ell_{r,i(A_r,j)}).
$
\begin{rem}\label{rem2292}
Here and hereafter ${i(A_1,0)} = 0$ and ${i(A_2,a_2+1)} = m_2+1$ by convention.
\end{rem}
The fiber product (\ref{264222rev}) is nothing but ${\mathcal M}_{\vec \ell_1,\ell',\vec \ell_2}(X,\mathcal J,H^{21};\vec \alpha_1,\vec \alpha_2)$.
Therefore we can use  (\ref{264222rev}) to define a fiber product Kuranishi structure
on the space ${\mathcal N}_{\vec \ell_1,\ell',\vec \ell_2}(X,\mathcal J,H^{21};\vec \alpha_1,\vec \alpha_2)$.
\par
We need more notation.
Although the notation is rather heavy, 
its geometric meaning is simple.
Namely we will consider the moduli space
${\mathcal M}_{\vec\ell_{r,A_r,j,C_r}}(X,J_r,H^r;\vec\alpha_{r,A_r,j,C_r} )$
etc., which is obtained by allowing to smooth the singularities
at the transit points corresponding to the indices belonging to $C_r$.
\begin{notation}\label{not2659rev}
Let $j=1,\dots ,a_r +1$ ($r=1,2$).
\begin{enumerate}
\item
We decompose $C_r$ into
$C'_j(A_r) = [i(A_r,j-1),i(A_r,j)]_{\Z} \cap C_r$
and put $C_j(A_r) = \{ i - i(A_r,j-1) \mid i \in C'_j(A_r)\}$,
$c_j(r,A_r) = \#C_j(A_r)$.
Here for $r=1$
$$
0=i(A_1,0)< i(A_1,1) < \dots < i(A_1,a_1) \le m_1
$$ with $i(A_1,0)=0, i(A_1, a_1 +1)=m_1$ as convention. 
For $r=2$,  
$$
i(A_2,0)=1 \le i(A_2,1) < \dots < i(A_2,a_2) < m_2+1
$$ 
with $i(A_2, 0)=1, i(A_2,a_2+1)=m_2+1$ as convention.
\item
We put $\vec\alpha_{r,A_r,j}
= (\alpha_{r,i(A_r,j-1)},\dots,
\alpha_{r,i(A_r,j)})$.
Note that\footnote{In case 
$i(A_1,a_1)=i(A_1,a_1+1)$ or $i(A_2,0)=i(A_2,1)$, we do not consider 
$\mathcal M_{\vec{\ell}_{r,j}}(X,J_r,H^r;\vec{\alpha}_{r,A_{r},j})$.}
\begin{equation}\label{formulain2696}
\aligned
&{\mathcal M}_{\vec\ell_{A_r,j}}(X,J_r,H^r;\alpha_{r, i(A_r,j-1)},
\dots,\alpha_{r, i(A_r,j)}) \\
&= {\mathcal M}_{\vec\ell_{A_r,j}}(X,J_r,H^r;\vec\alpha_{r,A_r,j} ).
\endaligned
\end{equation}
\item
We remove $\{\alpha_{r,i} \mid i \in C'_j(A_r)\}$ from $\vec\alpha_{r,A_r,j}$ to
obtain $\vec\alpha_{r,A_r,j,C_r}$.
\item
Noticing $i(A_1,a_1+1) = m_1$ for $r=1$, we apply (2)(3) to obtain
$\vec\alpha_{1,A_1,
a_1+1}$, $\vec\alpha_{1,A_1,a_1+1,C_1}$.
Also noticing $i(A_2,0) = 1$ for $r=2$,
we apply (2)(3) to obtain
$\vec\alpha_{2,A_2,1}$, $\vec\alpha_{2,A_2,1,C_2}$.
\par
Note that $\vec\alpha_{A_1,0}$, $\vec\alpha_{1,A_1,0,C_1}$
and
$\vec\alpha_{2,A_2,a_2}$, $\vec\alpha_{2,A_2,a_2,C_2}$
are defined according to Remark \ref{rem2292}.
\item
We put $m_j(r,A_r) = i(A_r,j) - i(A_r,j-1)$,
$$m_j(r,A_r,C_r) = \# (B_r\cap (i(A_r,j-1),i(A_r,j))_{\Z}) + 1.$$
Therefore
\begin{equation}\label{form264555rev}
\sum_{j=1}^{a_r +1} (m_j(r,A_r,C_r)-1)
= \#B_r=b_r.
\end{equation}
\item
We define $\vec \ell_{r,A_r,j,C_r}$ as follows.
Let
$$
\vec{\alpha}_{r,A_r,j,C_r}
= \{\alpha_{r,i(A_r,j-1) + k_s} \mid s =0, \dots, m_j(r,A_r,C_r)\}.
$$
Here 
$0\le k_0 < k_1 < \dots < k_{m_j(r,A_r,C_r)}$.\footnote{
Note that $k_s$ depends on $r,A_r,B_r,C_r$.
We also note that $k_0 = 0$ holds unless $r=2, j=1$ and $i(A_2,1)=1$.
Also $k_{m_j(r,A_r,C_r)} = i(A_r,j) - i(A_r,j-1)$ holds unless $r=1, j=a_1+1$ and $i(A_1,a_1)=m_1$.}
\par
Note if
$i \in (i(A_r,j-1) + k_s,i(A_r,j-1) + k_{s+1})_{\Z} $, then $i \in C'_j(A_r)$.
We put
$$
\ell_{r,A_r,j,C_r,s} = \ell_{r,i(A_r,j-1)+
k_{s-1} +1} + \dots
+ \ell_{r,i(A_r,j-1)+k_{s}}
$$
and
$$
\vec \ell_{r,A_r,j,C_r}
=
( \ell_{r,A_r,j,C_r,1}, \dots,
\ell_{r,A_r,j,C_r,m_j(r,A_r,C_r)})
$$
Finally we put
\begin{equation}
\aligned
\ell''
=
&\ell_{1,i(A_1,a_1)+k_{m_{a_1}(1,A_1,C_1)}+1}  + \dots
+ \ell_{1,m_1} + \\
&+\ell'
+ \ell_{2,2} + \dots
+ \ell_{2,{\rm min} (i(A_2,1), i(B_2,1))}.
\endaligned
\end{equation}
See Figure \ref{Fig3r}.
\begin{figure}[htbp]
\centering
\includegraphics[scale=0.5]{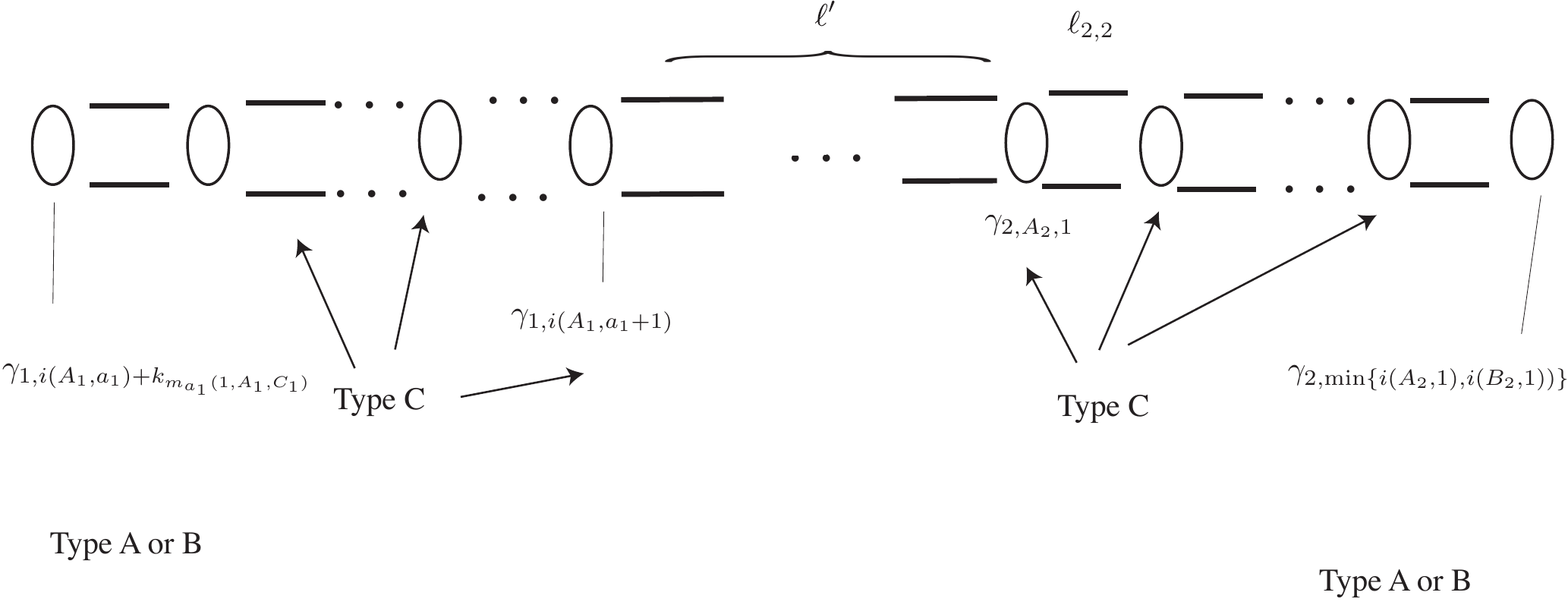}
\caption{The case $m_1 \in C_1$ and $1\in C_2$}
\label{Fig3r}
\end{figure}
\par
It is easy to check
\begin{equation}
\sum_{r,s,j} \ell_{r,A_r,j,C_r,s} + \ell''
=
\sum_{r,i} \ell_{r,i} + \ell'.
\end{equation}
\end{enumerate}
\end{notation}
We note that in Proposition \ref{prop2661}, we determined a Kuranishi structure on
\begin{equation}\label{form2677}
\aligned
&{\mathcal M}_{\vec\ell_{r,A_r,j,C_r}}(X,J_r,H^r;\vec\alpha_{r,A_r,j,C_r} )^+ \\
&=
{\mathcal M}_{\vec\ell_{r,A_r,j,C_r}}(X,J_r,H^r;\vec\alpha_{r,A_r,j,C_r} ) \times [-1,0]^{m_j(r,A_r,C_r)-1}.
\endaligned\end{equation}
We denote it by
$\widehat{\mathcal U}_{\vec\ell_{r,A_r,j,C_r}}(X,J_r,H^r;\vec\alpha_{r,A_r,j,C_r})$.
By construction we have
$$
{\mathcal M}_{\vec\ell_{r,A_r,j}}(X,J_r,H^r;\vec\alpha_{r,A_r,j} )
\subseteq
\widehat S_{c_j(r,A_r)}(
{\mathcal M}_{\vec\ell_{r,A_r,j,C_r}}(X,J_r,H^r;\vec\alpha_{r,A_r,j,C_r})).
$$
\par
By restriction, $\widehat{\mathcal U}_{\vec\ell_{r,A_r,j,C_r}}(X,J_r,H^r;\vec\alpha_{r,A_r,j,C_r})$ determines a Kuranishi structure of
(\ref{264222rev}) times $[-1,0]^*$ except one of the factors
\begin{equation}\label{form2678}
{\mathcal N}_{\vec\ell_{1,A_1,a_1+1},\ell',\vec\ell_{2,A_2,1}}(X,\mathcal J,H^{21}; \vec\alpha_{1, A_1, a_1+1},\vec\alpha_{2, A_2,1}).
\end{equation}
By construction, we can easily show that (\ref{form2678}) is
a component of
$$
\widehat S_{c_{a_1+1}(1,A_1)+c_{1}(2,A_2)}(
{\mathcal N}_{\vec\ell_{1,A_1,a_1+1,C_1},\ell'',\vec\ell_{2,A_2,1,C_2}}(X,\mathcal J,H^{21}; \vec\alpha_{1,A_1, a_1+1,C_1}, \vec\alpha_{2,A_2,1,C_2})).
$$
(See Notation \ref{not2659rev} (1) for the notations.)
Note the sum of the exponent $[-1,0]$ of appearing in (\ref{form2677}) plus 
$m_{a_1+1}(1,A_1,C_1)-1 + m_{1}(2,A_2,C_2)-1$
is $b_1 + b_2$. This is a consequence of (\ref{form264555rev}).
\par
Now the compatibility condition we require is described as follows.
\begin{conds}\label{conds2660rev}
We require the K-system 
$$
\{({\mathcal N}_{\vec \ell_1,\ell',\vec \ell_2}(X,\mathcal J,H^{21};\vec \alpha_1,\vec \alpha_2)^+, ~\widehat{\mathcal U}_{\vec \ell_1,\ell',\vec \ell_2}(X,\mathcal J,H^{21};\vec \alpha_1,\vec \alpha_2)^+ )\}
$$
satisfies the following.
\par
Its restriction
to the image of the embedding
(\ref{2642formrev})
is the fiber product of the following factors.
(Here we use the fiber product description
(\ref{264222rev}).)
\begin{enumerate}
\item
The restriction of the Kuranishi structure 
$\widehat{\mathcal U}_{\vec\ell_{r,A_r,j},C_r}(X,J_r,H^r;\vec\alpha_{r,A_r,j,C_r})$
to ${\mathcal M}_{\vec\ell_{r,A_r,j}}(X,J_r,H^r;\vec\alpha_{r,A_r,j})\times [-1,0]^{m_j(r, A_r, C_r) -1}$.
\item
The restriction of the
Kuranishi structure
$$
\widehat{\mathcal U}_{\vec\ell_{1,A_1,a_1+1,C_1},\ell'',\vec\ell_{2,A_2,1,C_2}}(X,\mathcal J,H^{21};
\vec\alpha_{1,A_1,a_1+1,C_1}, \vec\alpha_{2,A_2,1,C_2})
$$
to \eqref{form2678} 
$\times [-1,0]^{m_{a_1 +1}(1,A_1,C_1)-1} \times 
[-1,0]^{m_1(2,A_2,C_2)-1}$.
\end{enumerate}
\end{conds}
\begin{prop}\label{prop2661rev}
There exists a K-system 
$$
\{ ({\mathcal N}_{\vec \ell_1,\ell',\vec \ell_2}(X,\mathcal J,H^{21};\vec \alpha_1,\vec \alpha_2)^+,
~\widehat{\mathcal U}_{\vec \ell_1,\ell',\vec \ell_2}(X,\mathcal J,H^{21};\vec \alpha_1,\vec \alpha_2) )\}
$$
for various 
$\vec \ell_1,\ell',\vec \ell_2, \vec \alpha_1,\vec \alpha_2$
with the following properties.
\begin{enumerate}
\item
They satisfy Condition \ref{conds2660rev}.
\item
Let $\frak C$ be the union of the components of
$
{\mathcal N}_{\vec \ell_1,\ell',\vec \ell_2}(X,\mathcal J,H^{21};\vec \alpha_1,\vec \alpha_2)^+
$
which are in ${\mathcal N}_{\vec \ell_1,\ell',\vec \ell_2}(X,\mathcal J,H^{21};\vec \alpha_1,\vec \alpha_2)
\times \partial ([-1,0]^{m-1})$.
Then the Kuranishi structure
$\widehat{\mathcal U}_{\vec \ell_1,\ell',\vec \ell_2}(X,\mathcal J,H^{21};\vec \alpha_1,\vec \alpha_2)
$
is $\frak C$-collared 
in the sense of 
Remark \ref{defn1531revrev} (by replacing 
${\mathcal M}_{\vec \ell}(X,H; \vec \alpha)^+$
by
${\mathcal N}_{\vec \ell_1,\ell',\vec \ell_2}(X,\mathcal J,H^{21};\vec \alpha_1,\vec \alpha_2)^+$).
\item
For the case $\vec \alpha_1 = (\alpha_-)$ and $\vec \alpha_2 = (\alpha_+)$,
$\widehat{\mathcal U}_{\emptyset,\ell,\emptyset}(X,\mathcal J,H^{21};\vec \alpha_1,\vec \alpha_2)$
coincides with the Kuranishi structure we
constructed during the proof of
Theorem \ref{theorem266rev} (1)(2).
\end{enumerate}
\end{prop}
\begin{proof}
The proof is entirely the same as the proof of
Proposition \ref{prop2661}.
\end{proof}
Now we replace
the Kuranishi structures
of $S_m({\mathcal N}_{\ell}(X,\mathcal J,H^{21};\alpha_-,\alpha_+)) \times [-1,0]^m$
with ones in Proposition \ref{prop2661rev}
to obtain the Kuranishi structures in
Theorem \ref{theorem266rev} (3)(4).
The proof of Theorem \ref{theorem266rev} is complete.
\end{proof}
\begin{rem}\label{remenergylosss}
We may choose $H^{21}$ so that
$$
H^{21}(x,\tau,t) = (1-\chi(\tau)) H^1(x,t) + \chi(\tau) H^2(x,t),
$$
where $\tau : \R \to [0,1]$ is an increasing function which is $0$ when $\tau < -1$ and
is $1$ when $\tau > 1$.
In this case
the energy loss of the morphism obtained in Theorem \ref{theorem266rev}
is estimated from above by
$$
\int_{t \in S^1} \sup_{x \in X} \vert H^{1}(x,t) - H^{2}(x,t) \vert dt.
$$
This is a well established result. See for example, 
\cite[Section 2, $3^{\circ}$]{C}
\cite[Lemma 4.1]{Sch},
\cite[Lemma 9.3]{fooospectr} for a proof of this
inequality.
\end{rem}

\section{Construction of homotopy}
\label{subsec;homotpy}
\begin{shitu}\label{situ2697}
\begin{enumerate}
\item
Let $H^1$, $H^2$ be two Hamiltonians $X\times S^1 \to \R$
which are Morse-Bott non-degenerate in the sense of Condition \ref{weaknondeg}.
We consider a family of Hamiltonians parametrized by an interval $[0,1]$.
\item
Suppose
$H^{21,[0,1]} : X \times \R \times S^1 \times [0,1] \to \R$
is a smooth function
and $\mathcal J^{[0,1]} = \{J_{\tau,s} \mid \tau \in \R, s \in [0,1]\}$
is an  $\R \times [0,1]$ parametrized smooth family of tame
almost complex structures on $X$.
\item
For each $s \in [0,1]$ we assume the pair
$(H^{21,s},\mathcal J^{s}(=\mathcal J^{21,s}))$ is
as in Situation \ref{situ2676}.
\item
We assume that the families $H^{21,[0,1]}$ and $\mathcal J^{[0,1]}$
are collared in the following sense:
We consider the retraction $\mathcal R : [0,1] \to [\tau,1-\tau]$ such that
$\mathcal R(s) = \tau$ if $s \in [0,\tau]$, $\mathcal R(s) = 1-\tau$ if $s \in [1-\tau,1]$
and $\mathcal R(s) = s$ otherwise.
Then $H^{21,s} = H^{21,\mathcal R(s)}$
and $\mathcal J^{s} = \mathcal J^{\mathcal R(s)}$.
\end{enumerate} 
$\blacksquare$
\end{shitu}
We note that such families always exist.
More precisely, we have the following.
\begin{lem}\label{lem2698}
Suppose we are given $H^{21,s_0}$ 
and $\mathcal J^{s_0}$ for $s_0 = 0,1$.
Then there exist $H^{21,[0,1]}$ and $\mathcal J^{[0,1]}$
as in Situation \ref{situ2697}, whose
restrictions to $s_0$
coincide with $H^{21,s_0}$
and $\mathcal J^{s_0}$ for $s_0 = 0,1$ respectively.
\end{lem}
\begin{proof}
This is an immediate consequence of the facts that the set of
tame almost complex structures is contractible and
the set of smooth functions is contractible.
\end{proof}
\begin{defn}\label{def:NP}
Let $R_{\alpha_-} \in \widetilde{\rm Per}(H^1)$
and $R_{\alpha_+} \in \widetilde{\rm Per}(H^2)$.
We put
$$
\mathcal N_{\ell}(X,\mathcal J^{[0,1]},H^{21,[0,1]};\alpha_-,\alpha_+)
=
\bigcup_{s \in [0,1]}
\mathcal N_{\ell}(X,\mathcal J^{s},H^{21,s};\alpha_-,\alpha_+)
\times \{s\}.
$$
Here $\mathcal N_{\ell}(X,\mathcal J^{s},H^{21,s};\alpha_-,\alpha_+)$
is defined in Definition \ref{3equivrerevl}.
\end{defn}
We can define a topology of
$\mathcal N_{\ell}(X,\mathcal J^{[0,1]},H^{21,{[0,1]}};\alpha_-,\alpha_+)$
in the way similar to one in Definition \ref{def2626}
and show that it is Hausdorff and compact.
\par
We consider the boundary components of
$\mathcal N_{\ell}(X,\mathcal J^{[0,1]},H^{21,{[0,1]}};\alpha_-,\alpha_+)$
consisting of the disjoint union 
$$
\left(\mathcal N_{\ell}(X,\mathcal J^{0},H^{21,0};\alpha_-,\alpha_+)
\times \{0\}\right)
\sqcup \left(\mathcal N_{\ell}(X,\mathcal J^{1},H^{21,1};\alpha_-,\alpha_+)
\times \{1\}\right).
$$
We call it the {\it vertical boundary}
and denote it by $\frak C^{v}$.
We consider the projection 
$\mathcal N_{\ell}(X,\mathcal J^{[0,1]},H^{21,{[0,1]}};\alpha_-,\alpha_+)\to {[0,1]} $.
In the definition of $\frak C$-collared-ness (see Remark \ref{defn1531revrev}) we  replace 
$[-1,0]^{m-1}$
by ${[0,1]}$ to define $\frak C^v$-collared-ness. Then   
$\mathcal N_{\ell}(X,\mathcal J^{[0,1]},H^{21,{[0,1]}};\alpha_-,\alpha_+)\to {[0,1]}$
is $\frak C^v$-collared, by Situation \ref{situ2697} (4).
(This is the collared-ness in the ${[0,1]}$ direction.)
\par
The complement of the vertical boundary is written as $\frak C^{h}$
and we call it the {\it horizontal boundary}.
We denote by
$\mathcal N_{\ell}(X,\mathcal J^{[0,1]},H^{21,{[0,1]}};\alpha_-,\alpha_+)^{\frak C^h\boxplus 1}$
the space
$$
\bigcup_{s \in {[0,1]}}
\mathcal N_{\ell}(X,\mathcal J^{s},H^{21,s};\alpha_-,\alpha_+)^{\boxplus 1}
\times \{s\}.
$$
This space coincides with the `partial outer collaring'
of $\mathcal N_{\ell}(X,\mathcal J^{[0,1]},H^{21,{[0,1]}};\alpha_-,\alpha_+)$ 
in the horizontal direction, which is introduced in
\cite[Chapter 19]{fooonewbook}. (However we do not use 
\cite[Chapter 19]{fooonewbook} in this article.
The symbol ${}^{\frak C^h\boxplus 1}$ here can be regarded just as a notation.)
\begin{thm}\label{thmPparaexist}
We can define a Kuranishi structure on
$$
\mathcal N_{\ell}(X,\mathcal J^{[0,1]},H^{21,{[0,1]}};\alpha_-,\alpha_+)^{\frak C^h\boxplus 1}
$$
so that it will be an interpolation spaces
of a $[0,1]$-parametrized family of morphisms between
linear K-systems associated to $H^1$ and $H^2$ 
obtained by Theorem \ref{theorem266}
in the sense of \cite[Condition 16.21]{fooonewbook}.
\end{thm}
We need certain additional properties on our morphism
for applications. We will state them below.

\begin{prop}\label{collaredPparapara}
We may choose our Kuranishi structure on 
$$
\mathcal N_{\ell}(X,\mathcal J^{[0,1]},H^{21,[0,1]};\alpha_-,\alpha_+)^{\frak C^h\boxplus 1}
$$
so that it is $\frak C^{v}$-collared.
\end{prop}

\begin{proof}[Proof of Theorem \ref{thmPparaexist}
and Proposition \ref{collaredPparapara}]
We are given
morphisms associated to $s_0 = 0, 1$.
We made various choices in Section \ref{subsec;KuramodFloermor}.
Below we will show that we can define a $[0,1]$ parametrized morphism
so that its boundary becomes the union of two morphisms
associated to $s_0 = 0, 1$.
\par
Again the proof is by two steps. In the first step we
construct a Kuranishi structure on
$\mathcal N_{\ell}(X,\mathcal J^{[0,1]},H^{21,{[0,1]}};\alpha_-,\alpha_+)$
which may not be compatible with the fiber product description
at the horizontal boundary.
(We construct our Kuranishi structure on $\mathcal N_{\ell}(X,\mathcal J^{[0,1]},H^{21,{[0,1]}};\alpha_-,\alpha_+)$
so that it is compatible with the given Kuranishi structure
at the vertical boundary $\{0,1\} = \partial {[0,1]}$.)
We then modify it on the collar so that it becomes a
${[0,1]}$-parametrized interpolation space.
The details are as follows.
\begin{rem}
Since the proof is a repetition of the construction in the 
previous sections, the readers may skip it and
go directly to Section \ref{subsec;KuramodFloercom}.
We provide the details of the proof here for the sake of completeness.
\end{rem}
We first define the notion of $\epsilon$-closeness.
Let 
$$
\widehat{\bf q} = ({\bf q},s_{\bf q}) \in \mathcal N_{\ell}(X,\mathcal J^{s},H^{21,s};\alpha_0,\alpha_+)
\times \{s_{\bf q}\}
$$
and we take stabilization data at $\bf q$.
Hereafter we omit the symbol $\,\,\widehat{}\,\,$,  and write 
$$
{\bf q} = ({\bf q},s_{\bf q})
$$
by an abuse of notation.
Then
we consider
$$
({\frak Y} \cup \vec w',\varphi',a_0) = {\Phi}_{\bf q}(\frak y,\vec T,\vec \theta) \in {\mathcal N}_{\ell+\ell'+\ell''}(\text{source})
$$
as in (\ref{form2663}).
We consider $(\frak Y\cup \vec w',u',\varphi',a'_0,s')$
where $u'$ is a map from the curve $\frak Y$ to $X$ and $s' \in {[0,1]}$.
We say that it is {\it $\epsilon$-close} to ${\bf q} \cup \vec w \cup \vec w_{\rm can}$
if Definition \ref{orbitecloseness} (1)-(4) hold and
$\vert s_{\bf q} - s'\vert < \epsilon$.
In (2) we use the Hamiltonian $H^{21,s'}$ and the family of almost complex structures $\mathcal J_{s'}$  as in
Definition \ref{defn210rev} (5).
In (3)  we also use $H^{21,s'}$ to define the redefined connecting orbit map.
\par
Next, for a point 
${\bf p} \in {\mathcal N}_{\ell}(X,\mathcal J,H^{21};\alpha_-,\alpha_+)$
we suppose to be given obstruction bundle data $\frak E_{{\bf p}}$
and  $(\frak Y\cup \vec w',u',\varphi',a'_0,s')$ is
$\epsilon_{\bf p}$-close to ${\bf p} \cup \vec w \cup \vec w_{\rm can}$
for a sufficiently small $\epsilon_{\bf p}>0$.
Then we can define a complex linear map
\begin{equation}\label{eq:71}
I_{{\bf p},{\rm v};\Sigma',u',\varphi',a'_0} : 
E_{{\bf p},{\rm v}}(\frak y_{{\bf p}}) \to
C^{\infty}(\Sigma';(u')^*TX \otimes \Lambda^{0,1})
\end{equation}
in the same way as in the parametrized version of Definition \ref{defn2642}.
\par
We recall from the paragraph around \eqref{coversururev}
that we took a finite set indexed by the set 
$\EuScript C_{\ell}(H^{21},\mathcal J;\alpha_-,\alpha_+)$
$$
\EuScript A_{\ell}(H^{21},\mathcal J;\alpha_-,\alpha_+)
= \{{\bf p}_c \mid c \in \EuScript C_{\ell}(H^{21},\mathcal J;\alpha_-,\alpha_+) \}
\subset {\mathcal N}_{\ell}(X,\mathcal J,H^{21};\alpha_-,\alpha_+)
$$
and for each $c \in \EuScript C_{\ell}(H^{21},\mathcal J;\alpha_-,\alpha_+)$
we took obstruction bundle data $\frak E_{{\bf p}_{c}}$
centered at ${\bf p}_c$.
It corresponds to the case $s_0 = 0,1$ in the current circumstances.
We write
$\EuScript A_{\ell}(H^{21,s_0},\mathcal J_{s_0};\alpha_-,\alpha_+)$,
$\EuScript C_{\ell}(H^{21,s_0},\mathcal J_{s_0};\alpha_-,\alpha_+)$,
$\frak E_{{\bf p}_{c},s_0}$ to specify $s_0$. 
We also took $W({\bf p}_c)$ such that
(\ref{coversururev}) is satisfied.
We denote them by $\overline W({\bf p}_c,s_0)$ here.
Note it is a neighborhood of ${\bf p}_c$ in
$\mathcal N_{\ell}(X,\mathcal J^{s_0},H^{21,s_0};\alpha_-,\alpha_+)$.
\par
We take a closed neighborhood $W({\bf p}_c)$ of ${\bf p}_c$
in
$\mathcal N_{\ell}(X,\mathcal J^{[0,1]},H^{21,[0,1]};\alpha_-,\alpha_+)$
such that
$$
W({\bf p}_c)
\cap \mathcal N_{\ell}(X,\mathcal J^{s_0},H^{21,s_0};\alpha_-,\alpha_+)
=
\overline W({\bf p}_c,s_0).
$$
Moreover using Situation \ref{situ2697} (4) we may assume
\begin{equation}\label{form268181}
\aligned
W({\bf p}_c) = \overline W({\bf p}_c,s_0) \times [0,\epsilon)
\qquad & \text{for $s_0 =0$}, \\
W({\bf p}_c) = \overline W({\bf p}_c,s_0) \times (1-\epsilon,1]
\qquad & \text{for $s_0 =1$}.
\endaligned
\end{equation}
Now we similarly take a finite set 
$$
\aligned
&\EuScript A_{\ell}(H^{21,[0,1]},\mathcal J^{[0,1]};\alpha_-,\alpha_+)
= \{{\bf p}_c \mid c \in \EuScript C_{\ell}(H^{21,{[0,1]}},\mathcal J^{[0,1]};\alpha_-,\alpha_+) \}
\\
&\subset {\mathcal N}_{\ell}(X,\mathcal J^{[0,1]},H^{21,{[0,1]}};\alpha_-,\alpha_+)
\setminus
\partial_{\frak C^v}({\mathcal N}_{\ell}(X,\mathcal J^{[0,1]},H^{21,{[0,1]}};\alpha_-,\alpha_+))
\endaligned
$$
and a closed neighborhood $W({\bf p}_c)$ of ${\bf p}_c$ in
$\mathcal N_{\ell}(X,\mathcal J^{[0,1]},H^{21,{[0,1]}};\alpha_-,\alpha_+)$
for each 
$c \in \EuScript C_{\ell}(H^{21,{[0,1]}},\mathcal J^{[0,1]};\alpha_-,\alpha_+) $
such that the following conditions are satisfied.
\begin{conds}\label{conds26103}
\begin{enumerate}
\item
$$
\aligned
&\mathcal N_{\ell}(X,\mathcal J^{[0,1]},H^{21,{[0,1]}};\alpha_-,\alpha_+)\\
&=
\bigcup_{c \in \EuScript C_{\ell}(H^{21,{[0,1]}},\mathcal J^{[0,1]};\alpha_-,\alpha_+)}
{\rm Int}\,\, W({\bf p}_c)
\cup
\bigcup_{s_0 \in \{0,1\}}\bigcup_{c \in \EuScript C_{\ell}(H^{21,s_0},\mathcal J^{s_0};
\alpha_-,\alpha_+)}
{\rm Int}\,\, \overline W({\bf p}_c,s_0).
\endaligned$$
\item
If
$c \in \EuScript C_{\ell}(H^{21,{[0,1]}},\mathcal J^{[0,1]};\alpha_-,\alpha_+)$,
then
$$
W({\bf p}_c) \cap \partial_{\frak C^v}
({\mathcal N}_{\ell}(X,\mathcal J^{[0,1]},H^{21,{[0,1]}};\alpha_-,\alpha_+))
=\emptyset.
$$
\item
Any element of $W({\bf p}_c)$
together with certain marked points is $\epsilon_c$-close to ${\bf p}_c
\cup \vec w\cup \vec w_{\rm can}$.
\end{enumerate}
\end{conds}
Let ${\bf q} \in \mathcal N_{\ell}(X,\mathcal J^{[0,1]},H^{21,{[0,1]}};\alpha_-,\alpha_+)$.
We put
$$
\aligned
\EuScript E({\bf q})  = & \{ {\bf p}_c
\in \EuScript A_{\ell}(H^{21,{[0,1]}},\mathcal J^{[0,1]};\alpha_-,\alpha_+) \\
& \cup
\bigcup_{s_0 = 0,1} \EuScript A_{\ell}(H^{21,s_0},\mathcal J_{s_0};\alpha_-,\alpha_+)
\mid {\bf q} \in \overline W({\bf p}_c,s_0)
\}.
\endaligned
$$
By taking $\EuScript A_{\ell}(H^{21,{[0,1]}},\mathcal J^{[0,1]};\alpha_-,\alpha_+)$ suitably, 
we may take $\epsilon_c >0$ in Condition \ref{conds26103} (3)
small enough as we wish. 
Therefore for each 
${\bf p}_c \in \EuScript E({\bf q})$
we can uniquely find $\vec w_{c}^{\bf q}$ for any ${\bf q} \in W({\bf p}_c)$ such that
${\bf q} \cup \vec w_{c}^{\bf q}$
is $\epsilon'_c$-close to ${\bf p}_c \cup \vec w_{{\bf p}_c}
\cup \vec w_{{\bf p}_c,{\rm can}}$
and $\vec w_{c}^{\bf q}$ satisfies the transversal constraint, 
and moreover, the linearization operator 
$D_{\bf q} \overline{\partial}_{J,H}$ at $\bf q$ in \eqref{eq:linearized} 
is surjective $\mod \oplus_{\rm v}\operatorname{Im}~ I_{{\bf p}_c,{\rm v};{\bf q}}$,
where $I_{{\bf p}_c,{\rm v};{\bf q}}$ is the map in \eqref{eq:71}
for ${\bf p}={\bf p}_c$, ${\bf q}=(\Sigma', u', \varphi')$.
\par
Now let us consider
$(\frak Y\cup\bigcup_{{\bf p}_c\in \EuScript E({\bf q})}\vec w'_c,u',\varphi',a'_0,s')$
such that
$(\frak Y\cup \vec w'_c,u',\varphi',a'_0,s')$ is $\epsilon$-close to 
${\bf q} \cup \vec w_{c}^{\bf q}$ for each ${\bf p}_c \in \EuScript E({\bf q})$.
We may choose $\epsilon >0$ and $\epsilon_c >0$ small so that we obtain a
complex linear map
$$
I_{{\bf p}_c,{\rm v};\Sigma',u',\varphi',a'_0,s'} : E_{{\bf p}_c,{\rm v}}(\frak y_{{\bf p}_c}) \to
C^{\infty}(\Sigma';(u')^*TX \otimes \Lambda^{0,1})
$$
as in \eqref{eq:71}.
We now put
\begin{equation}\label{obspacedefHFHFrev}
E((\frak Y\cup \bigcup_{c\in \EuScript B}\vec w'_c,u',\varphi',a'_0,s');{\bf q};\EuScript B)
=
\bigoplus_{c \in \EuScript B}\bigoplus_{\rm v} {\rm Im} \,\, 
I_{{\bf p}_c,{\rm v};\Sigma',u',\varphi',a'_0,s'},
\end{equation}
where $\EuScript B$ is a subset of $\{c \mid {\bf p}_c \in \EuScript E({\bf q})\}$.
By perturbing the bundle part of the obstruction bundle data $\frak E_{{\bf p}_{c}}$
for ${\bf p}_c
\in \EuScript A_{\ell}(H^{21,{[0,1]}},\mathcal J^{[0,1]};\alpha_-,\alpha_+)$
slightly we may assume that the right hand side of
(\ref{obspacedefHFHFrev}) is a direct sum.\footnote{See \cite[Subsection 11.4]{fooo:const1}.}
(We can do so without changing the obstruction bundle data which
was already fixed at $s_0 = 0,1$.)
\par
Now we define $V({\bf q},\epsilon_1,\EuScript B)$
to be the set of the isomorphism classes of
$(\frak Y\cup \vec w'_c,u',\varphi',a'_0,s')$
satisfying Definition \ref{defn2650rev} (1)-(5) and $\vert s'-s_{\bf q}\vert
< \epsilon_1$.
Note that we use $s'$ to parametrize a Hamiltonian and a family of almost
complex structures appearing in Definition \ref{defn2650rev} (1),(2)
and $s_{\bf q}$ is the $[0,1]$ component of ${\bf q}$.
\par
In the same way as in Lemma \ref{lem2653}, we can show that
$V({\bf q},\epsilon_1,\EuScript B)$ is a smooth manifold with corners
by choosing various constants sufficiently
small.
Then in the same way as the proof of Theorem \ref{theorem266rev} (2), we can find other objects
so that $V({\bf q},\epsilon_1,\EuScript B)$ together with them
is a Kuranishi chart of ${\bf q}$.
We can also show the existence of coordinate changes in the same way as in
Lemmas \ref{lem2657} and \ref{lem2658}.
We shrink the Kuranishi neighborhood and
discuss in the same way as in 
Lemmas \ref{lem2657}, \ref{lem2658} and 
\cite[Chapter 8]{foooanalysis}
to obtain a Kuranishi structure on
$\mathcal N_{\ell}(X,\mathcal J^{[0,1]},H^{21,{[0,1]}};\alpha_-,\alpha_+)$.
By using Condition \ref{conds26103} (2) and (\ref{form268181}) 
we can show that this Kuranishi structure
is $\frak C^v$-collared and its restriction to
$\partial{[0,1]} = \{0,1\}$ coincides with the given one.
See the paragraphs after Definition \ref{def:NP} for the $\frak C^v$-collared-ness.
\par
We next define the Kuranishi structure on 
$\mathcal N_{\ell}(X,\mathcal J^{[0,1]},H^{21,{[0,1]}};\alpha_-,\alpha_+)^{\boxplus 1}$
extending the one on 
$\mathcal N_{\ell}(X,\mathcal J^{[0,1]},H^{21,{[0,1]}};\alpha_-,\alpha_+)$ 
in the way as in the previous sections as follows.
\par
We consider the following fiber product.
(Here and in the next condition we use Notation \ref{not2659rev}.)
\begin{equation}\label{264222revrev0}
\aligned
&{\mathcal M}_{\vec \ell_{1,A_1,1}}
(X,J_1,H^1;\vec\alpha_{1,A_1,1})
 \\
&
\,{}_{{\rm ev}_{+}}\times_{{\rm ev}_-} {\mathcal M}_{\vec \ell_{1,A_1,2}}(X,J_1,H^1;\vec\alpha_{1,A_1,2})
\,{}_{{\rm ev}_{+}}\times_{{\rm ev}_-} \dots \\
&\,{}_{{\rm ev}_{+}}\times_{{\rm ev}_-} {\mathcal M}_{\vec \ell_{1,A_1,j}}(X,J_1,H^1;\vec\alpha_{1,A_1,j})
\,{}_{{\rm ev}_{+}}\times_{{\rm ev}_-}
\dots \\
&
\,{}_{{\rm ev}_{+}}\times_{{\rm ev}_-} {\mathcal M}_{\vec\ell_{1,A_1,a_1+1}}(X,J_1,H^1;\vec\alpha_{1,A_1,a_1+1}) \\
&\,{}_{{\rm ev}_{+}}\times_{{\rm ev}_-}
{\mathcal N}_{\vec\ell_{1,A_1,a_1+1},\ell',\vec\ell_{2,A_2,1}}(X,\mathcal J^{[0,1]},H^{21,{[0,1]}}; \vec\alpha_{1,A_1,a_1+1},\vec\alpha_{2,A_2,1})
\\
&
\,{}_{{\rm ev}_{+}}\times_{{\rm ev}_-} {\mathcal M}_{\vec \ell_{2,A_2,2}}(X,J_2,H^2;\vec{\alpha}_{2, A_2,2})
\,{}_{{\rm ev}_{+}}\times_{{\rm ev}_-}
\dots\\
&\,{}_{{\rm ev}_{+}}\times_{{\rm ev}_-} {\mathcal M}_{\vec \ell_{2,A_2,j}}(X,J_2,H^2;\vec\alpha_{2,A_2,j})
\,{}_{{\rm ev}_{+}}\times_{{\rm ev}_-}
\dots \\
&
\,{}_{{\rm ev}_{+}}\times_{{\rm ev}_-} {\mathcal M}_{\vec\ell_{2,A_2,a_2+1}}(X,J_2,H^2;\vec\alpha_{2, A_2,a_2+1}).
\endaligned
\end{equation}
(Recall $\vec\alpha_{1,A_1,a_1+1} = (\alpha_{1, i(A_1,a_1)},\dots,\alpha_{1,m_1})$
and $\vec\alpha_{2,A_2,a_2+1} = (\alpha_{1, i(A_2,a_2)},\dots,\alpha_{2,m_2+1})$
and other notations in Notation \ref{not2659rev}.)
The fiber product (\ref{264222revrev0}) is the same as in 
(\ref{264222rev}) except one of the factors in
(\ref{264222rev}) is replaced by
\begin{equation}\label{form2678rev}
{\mathcal N}_{\vec\ell_{1,A_1,a_1+1},\ell',\vec\ell_{2,A_2,1}}(X,\mathcal J^{[0,1]},H^{21,{[0,1]}}; (\vec\alpha_{1, A_1,a_1+1}, \vec\alpha_{2,A_2,1})).
\end{equation}
We use the embedding
\begin{equation}\label{2642formrevrev}
\aligned
{\rm id} \times \mathcal I_{A_1,B_1,C_1} \times \mathcal I_{A_2,B_2,C_2} :
&{\mathcal N}_{\vec \ell_1,\ell',\vec \ell_2}(X,\mathcal J^{[0,1]},H^{21,{[0,1]}};\vec \alpha_1,\vec \alpha_2)\times [-1,0]^{b_1+b_2} \\
&\to
{\mathcal N}_{\vec \ell_1,\ell',\vec \ell_2}(X,\mathcal J^{[0,1]},H^{21,{[0,1]}};\vec \alpha_1,\vec \alpha_2)^+
\endaligned
\end{equation}
which is defined in the same way as in (\ref{2642formrev}).
\par
We use the fact that
(\ref{form2678rev}) is
a component of
$$
 \widehat S_{c_{a_1+1}(1,A_1)+c_{1}(2,A_2)}
({\mathcal N}_{\vec\ell_{1,A_1,a_1+1,C_1},\ell'',\vec\ell_{2,A_2,1,C_2}}(X,\mathcal J^{[0,1]},H^{21,{[0,1]}}; \vec\alpha_{1,A_1,a_1+1}, \vec\alpha_{2,A_2,1})).
$$
\par
As we mentioned already, the fiber product factor of
(\ref{264222revrev0}) other than (\ref{form2678rev})
is
\begin{equation}\label{form268686}
{\mathcal M}_{\vec \ell_{r,A_r,j}}(X,J_r,H^r;\vec\alpha_{r,A_r,j}),
\end{equation}
which is a component of
$$
\widehat S_{c_j(r,A_r)}({\mathcal M}_{\vec \ell_{r,A_1,j},C_r}(X,J_r,H^r;\vec\alpha_{r,A_r,j,C_r})).
$$
We put
$$
{\mathcal N}_{\vec \ell_1,\ell',\vec \ell_2}(X,\mathcal J^{[0,1]},H^{21,{[0,1]}};\vec \alpha_1,\vec \alpha_2)^+
= {\mathcal N}_{\vec \ell_1,\ell',\vec \ell_2}(X,\mathcal J^{[0,1]},H^{21,{[0,1]}};\vec \alpha_1,\vec \alpha_2) \times [-1,0]^{m-1},
$$
where $m = \# \vec \alpha_1 + \# \vec \alpha_2 - 2$.
\begin{conds}\label{conds2660revrev}
We require the K-system 
$$
\{ ( {\mathcal M}_{\vec \ell_1,\ell',\vec \ell_2}(X,\mathcal J^{[0,1]},H^{21,{[0,1]}};\vec \alpha_1,\vec \alpha_2)^+, 
~\widehat{\mathcal U}_{\vec \ell_1,\ell',\vec \ell_2}(X,\mathcal J^{[0,1]},H^{21,{[0,1]}};\vec \alpha_1,\vec \alpha_2) )\}
$$
satisfies the following.
\begin{enumerate}
\item
The restriction of
$
\widehat{\mathcal U}_{\vec \ell_1,\ell',\vec \ell_2}(X,\mathcal J^{[0,1]},H^{21,{[0,1]}};\vec \alpha_1,\vec \alpha_2)
$
to the image of the embedding
(\ref{2642formrevrev})
is the fiber product of the  following
Kuranishi strucrures.
(Here we use the fiber product description
(\ref{264222revrev0}).)
\begin{enumerate}
\item
The restrictions of the
Kuranishi structure $$\widehat{\mathcal U}_{\vec\ell_{A_r,j,C_r}}(X,J_r,H^r; \vec\alpha_{r,A_r,j,C_r})$$
on ${\mathcal M}_{\vec \ell_{r,A_{r},j},C_r}(X,J_r,H^r;\vec\alpha_{r,A_r,j,C_r})^+$
to the space which is a direct product
(\ref{form268686}) $\times [-1,0]^{m_j(r,A_r,C_r) -1}$.
\item
The restrictions of the
Kuranishi structure $$\widehat{\mathcal U}_{\vec\ell_{1,A_1,a_1+1,C_1},\ell'',\vec\ell_{2,A_2,1,C_2}}(X,\mathcal J^{[0,1]},H^{21,{[0,1]}}; \vec\alpha_{1,A_1,a_1+1}, \vec\alpha_{2,A_2,1})$$
on
${\mathcal N}_{\vec\ell_{1,A_1,a_1+1,C_1},\ell'',\vec\ell_{2,A_2,1,C_2}}(X,\mathcal J^{[0,1]},H^{21,{[0,1]}}; \vec\alpha_{1,A_1,a_1+1}, \vec\alpha_{2,A_2,1})^+$
to the space which is a direct product
$$
(\ref{form2678rev}) \times [-1,0]^{m_{a_1}(1,A_1,C_1)} \times [-1,0]^{m_{0}(2,A_2,C_2)}.
$$
\end{enumerate}
\item
The restriction of
$\widehat{\mathcal U}_{\vec \ell_1,\ell',\vec \ell_2}(X,\mathcal J^{[0,1]},H^{21,{[0,1]}};\vec \alpha_1,\vec \alpha_2)$ to the vertical boundary
$\partial_{\frak C^v}({\mathcal N}_{\vec \ell_1,\ell',\vec \ell_2}(X,\mathcal J^{[0,1]},H^{21,{[0,1]}};\vec \alpha_1,\vec \alpha_2)^{\frak C^h\boxplus 1})$
is isomorphic to the Kuranishi structure
$
\bigcup_{s_0=0,1}
\widehat{\mathcal U}_{\vec \ell_1,\ell',\vec \ell_2}(X,\mathcal J^{s_0},H^{21,s_0};\vec \alpha_1,\vec \alpha_2)
$, 
which we produced in Proposition \ref{prop2661rev}.
\end{enumerate}
\end{conds}
\begin{prop}\label{prop2661revrev}
There exists a K-system
$$
\{ ({\mathcal N}_{\vec \ell_1,\ell',\vec \ell_2}(X,\mathcal J^{[0,1]},H^{21,{[0,1]}};\vec \alpha_1,\vec \alpha_2)^+, ~ 
\widehat{\mathcal U}_{\vec \ell_1,\ell',\vec \ell_2}(X,\mathcal J^{[0,1]},H^{21,{[0,1]}};\vec \alpha_1,\vec \alpha_2) )\}
$$
satisfying the following properties.
\begin{enumerate}
\item
They satisfy Condition \ref{conds2660revrev}.
\item
Let $\frak C$ be the union of the components of
$
{\mathcal N}_{\vec \ell_1,\ell',\vec \ell_2}(X,\mathcal J^{[0,1]},H^{21,{[0,1]}};\vec \alpha_1,\vec \alpha_2)^+
$
which are in ${\mathcal N}_{\vec \ell_1,\ell',\vec \ell_2}(X,\mathcal J^{[0,1]},H^{21,{[0,1]}};\vec \alpha_1,\vec \alpha_2)
\times \partial ([-1,0]^{m-1})$.
Then the Kuranishi structure
$\widehat{\mathcal U}_{\vec \ell_1,\ell',\vec \ell_2}(X,\mathcal J^{[0,1]},H^{21,{[0,1]}};\vec \alpha_1,\vec \alpha_2)^+
$
is $\frak C$-collared in the sense of 
Remark \ref{defn1531revrev}.
\item
For the case $\vec \alpha_1 = (\alpha_-)$ and $\vec \alpha_2 = (\alpha_+)$, 
$\widehat{\mathcal U}_{\emptyset,\ell,\emptyset}(X,\mathcal J^{[0,1]},H^{21,{[0,1]}};\vec \alpha_1,\vec \alpha_2)^+$
coincides with the Kuranishi structure we
constructed in the first half of the proof of Theorem \ref{thmPparaexist}
and Proposition \ref{collaredPparapara}.
\end{enumerate}
\end{prop}
\begin{proof}
The proof is the same as that of Proposition \ref{prop2661}.
\end{proof}
Now using Proposition \ref{prop2661revrev}
we modify the Kuranishi structures on the spaces
$\mathcal N_{\ell}(X,H^{21,{[0,1]}},\mathcal J^{[0,1]};\alpha_-,\alpha_+)^{\boxplus 1}$
and complete the proof of Theorem \ref{thmPparaexist} and Proposition \ref{collaredPparapara}.
\end{proof}

\begin{rem}
In this section we have constructed a homotopy between two morphisms.
We can continue this process to obtain a homotopy of homotopies etc.
In this paper we do not need such a higher homotopy by the following reason:
First we note that what we obtain in this paper is a linear K-system.
In \cite[Chapter 16]{fooonewbook} we introduced the notion of an 
inductive system of linear K-systems.\footnote{It is a structure which gives 
a sequence of the partial version of the linear K-system for which the moduli spaces of energy $\le E_n$ $(n=0,1,\dots,)$ are used. It 
gives also a homotopy equivalence (modulo $E_n$) between the structure with the moduli spaces of energy $\le E_n$
and that with the moduli spaces of energy $\le E_{n+1}$. }
Sometimes such a structure is easier to obtain since we need to study only a finite 
number of moduli spaces at each stage. 
The proof we provide in this paper defines a system of {\it Kuranishi structures}
on infinitely many moduli spaces at once. This is the reason we do not need to study 
homotopy of homotopies.
On the other hand, to construct a system of {\it CF-perturbations} on such infinitely many moduli spaces, 
we need to stop at a certain energy level and use a homotopy inductive limit argument.
Such an argument is given in \cite[Chapter 19]{fooonewbook}. (We can quote the statements 
of \cite{fooonewbook} literary.)  Homotopy of homotopies we used in the homotopy inductive limit argument of \cite[Chapter 19]{fooonewbook} 
is a direct product 
of the homotopy obtained in Theorem \ref{thmPparaexist} and an interval $[0,1]$ in our case.
\end{rem}

\section{Composition of morphisms}
\label{subsec;KuramodFloercom}
\subsection{Statement}\label{subsec:statement2}
\begin{shitu}\label{situ26106}
\begin{enumerate}
\item
Let $H^r$ ($r=1,2,3$) be periodic Hamiltonian functions which are
Morse-Bott non-degenerate in the sense of Condition \ref{weaknondeg}.
\item
For each pair of $r,r'$ with $r\ne r' \in \{1,2,3\}$, 
let
$H^{r' r} : X \times \R \times S^1 \to \R$
be a smooth function
and $\mathcal J^{r' r} = \{J_{\tau;r' r} \mid \tau \in \R\}$
be an $\R$-parametrized smooth family of tame
almost complex structures on $X$.
\item
We assume that
$H^{r'r},$ $\mathcal J^{r' r} $ are
as in Situation \ref{situ2676} with $H^1$,$H^2$, $J_1$, $J_2$, $\mathcal J$
replaced by $H^r$,$H^{r'}$, $J_r$, $J_{r'}$, $\mathcal J^{r'r}$, respectively.
\item
Let $\mathcal F_r:=\mathcal F_X(H^r,J_r)$ be the linear K-system 
constructed by Theorem \ref{theorem266} from $H^r$ and $J_r$.
We made the choices during the constructions.
\item
Let $\frak N_{r' r} : \mathcal F_r \to \mathcal F_{r'}$ be the
morphism defined by Theorem \ref{theorem266rev}  using $H^{r' r}$
and $\mathcal J^{r' r}$ in place of
$H^{21}$
and $\mathcal J$.
We made the choices during the construction.
\end{enumerate}
$\blacksquare$
\end{shitu}
In this section, we prove the following.
\begin{thm}\label{tjm26108}
In Situation \ref{situ26106} the composition
$\frak N_{32} \circ \frak N_{21}$ is homotopic to
$\frak N_{31}$.
\end{thm}
\begin{proof}
By Lemma \ref{lem2698} and Theorem \ref{thmPparaexist},
the homotopy class of
$\frak N_{r' r}$ is independent of $\mathcal J^{r' r}$,
$H^{r' r}$ or other choices we made during the construction.
So it suffices to prove Theorem \ref{tjm26108} for certain fixed choices of them.
We take the choices as follows.
We first take $\mathcal J^{21}$,
$H^{21}$ and $\mathcal J^{32}$,
$H^{32}$. 
{\it We assume that they satisfy Situation \ref{situ2676}
(1)(2)(i)(ii) with $\pm 1$ replaced by $\pm 1/4$.}
We then define $\mathcal J^{31}$,
$H^{31}$ as follows.
\begin{equation}
H^{31}_{\tau}
=
\begin{cases}
H^{21}_{\tau +1/2}  &\text{if $\tau \le 0$}, \\
H^{32}_{\tau -1/2}  &\text{if $\tau \ge 0$}.
\end{cases}
\end{equation}
\begin{equation}
J_{\tau}^{31}
=
\begin{cases}
J_{\tau +1/2}^{21}  &\text{if $\tau \le 0$}, \\
J_{\tau -1/2}^{32}  &\text{if $\tau \ge 0$}.
\end{cases}
\end{equation}
It is easy to see that they satisfy Situation \ref{situ2676}.
\par
We will construct a homotopy between
$\frak N_{32} \circ \frak N_{21}$ and
$\frak N_{31}$ where
$\frak N_{32}$, $\frak N_{21}$,
$\frak N_{31}$ are obtained by  Theorem \ref{theorem266rev}
using those choices.
\par
The interpolation space of the homotopy
is obtained by compactifying the
solution space of the $\tau,t,T$ dependent Hamiltonian
perturbed pseudo-holomorphic curve equation.
We will use the following two parameter family
of Hamiltonians.
Let $T \ge 0$ and $\tau \in \R$. We put
\begin{equation}\label{formula2688}
H^{31,T}_{\tau}
=
\begin{cases}
H^{21}_{\tau +1/2+T}  &\text{if $\tau \le 0$}, \\
H^{32}_{\tau -1/2-T}  &\text{if $\tau \ge 0$},
\end{cases}
\end{equation}
\begin{equation}\label{formula2689}
J_{\tau,T}^{31}
=
\begin{cases}
J_{\tau +1/2+T}^{21}  &\text{if $\tau \le 0$}, \\
J_{\tau -1/2-T}^{32}  &\text{if $\tau \ge 0$}.
\end{cases}
\end{equation}
We put $\mathcal J^{31,T} = \{J_{\tau +1/2+T}^{31}\}$
and consider
\begin{equation}\label{2690formula}
\bigcup_{T \ge 0}
\mathcal N_{\ell}(X,\mathcal J^{31,T},H^{31,T};\alpha_-,\alpha_+)
\times \{T\},
\end{equation}
where
$\mathcal N_{\ell}(X,\mathcal J^{31,T},H^{31,T};\alpha_-,\alpha_+)$
is as in Definition \ref{3equivrerevl}.
We can define a topology on it in the same way as in Definition \ref{def2626}.
This space then becomes Hausdorff. However it is not
compact since the domain $[0,\infty)$ of $T$ is not compact.
We compactify it by adding certain space at $T = \infty$ as follows.
\begin{defn}\label{defn210revrev}
Let $\alpha_- \in \frak A_1$ and $\alpha_+ \in \frak A_3$.
(Here $\frak A_r$ be the index set of the critical submanifolds
of the linear K-system $\mathcal F_r$.)
The set
$\widehat{\mathcal N}_{\ell}(X,\mathcal J^{31,\infty},H^{31,\infty};\alpha_-,\alpha_+)$
consists of 
$((\Sigma,(z_-,z_+,\vec z),a_1,a_2),u,\varphi)$ satisfying
the following conditions:
\begin{enumerate}
\item
$(\Sigma,(z_-,z_+,\vec z))$ is a genus zero semi-stable curve with $\ell + 2$ marked points.
\item
$\varphi$ is a parametrization of the mainstream.
\item
$\Sigma_{a_1}$, $\Sigma_{a_2}$ are two of the mainstream components.
We call them the {\it first main component} and the {\it second main component}.
\item
For each extended main stream component $\widehat{\Sigma}_a$, the map
$u$ induces
$u_a : \widehat{\Sigma}_a \setminus\{z_{a,-},z_{a,+}\} \to X$
which is a continuous map.
\item
We define the relation $<$ on the set of mainstream components as in
Definition \ref{defn2682}.
We require $a_1 < a_2$.
If $\Sigma_a$ is a mainstream component and
$\varphi_a : \R \times S^1 \to \Sigma_a$ is as above, then
the composition $u_a \circ \varphi_a$ satisfies the equation
\begin{equation}\label{Fleq266211}
\frac{\partial (u_a \circ \varphi_a)}{\partial \tau} +  J_{a,\tau} \left( \frac{\partial (u_a \circ \varphi_a)}{\partial t} - \frak X_{H^a_{\tau,t}}
\circ (u_a \circ \varphi_a) \right) = 0,
\end{equation}
where
$$
H^a_{\tau,t} =
\begin{cases}
H^1_t  &\text{if $a < a_1$}, \\
H^{21}_{\tau,t}  &\text{if $a = a_1$}, \\
H^2_t  &\text{if $a_1 < a < a_2$}, \\
H^{32}_{\tau,t}  &\text{if $a = a_2$},\\
H^3_t  &\text{if $a > a_2$},
\end{cases}
$$
and
$$
J_{a,\tau} =
\begin{cases}
J_1 &\text{if $a < a_1$}, \\
J^{21}_{\tau}  &\text{if $a = a_1$}, \\
J_2 &\text{if $a_1 < a < a_2$}, \\
J^{32}_{\tau}  &\text{if $a = a_2$}, \\
J_3 &\text{if $a > a_2$}.
\end{cases}
$$
\item
$$
\int_{\R \times S^1}
\left\Vert\frac{\partial (u \circ \varphi_a)}{\partial \tau}\right\Vert^2 d\tau dt < \infty.
$$
\item
Suppose $\Sigma_{\rm v}$ is a bubble component in $\widehat{\Sigma}_a$.
Let   $\varphi_a(\tau,t)$ be the root of the tree of sphere bubbles
containing $\Sigma_{\rm v}$.  Then $u$ is $J$-holomorphic on  $\Sigma_{\rm v}$ where
$$
J =
\begin{cases}
J_1 &\text{if $a < a_1$}, \\
J^{21}_{\tau}  &\text{if $a = a_1$}, \\
J_2 &\text{if $a_1 < a < a_2$}, \\
J^{32}_{\tau}  &\text{if $a = a_2$}, \\
J_3 &\text{if $a > a_2$}.
\end{cases}
$$
\item
If
${\Sigma}_{a}$ and ${\Sigma}_{a'}$ are
mainstream components and
$z_{a,+} = z_{a',-}$, then
$$
\lim_{\tau\to+\infty} (u_{a} \circ \varphi_{a})(\tau,t)
=
\lim_{\tau\to-\infty} (u_{a'} \circ \varphi_{a'})(\tau,t)
$$
holds for each $t \in S^1$.
((6) and Lemma \ref{prof26rev3} imply that the
left and right hand sides both converge.)
\item
If
${\Sigma}_{a}$, ${\Sigma}_{a'}$
are
mainstream components and $z_{a,-} = z_-$, $z_{a',+}
= z_+$, then there exist $(\gamma_{\pm},w_{\pm})
\in R_{\alpha_{\pm}}$ such that
$$
\aligned
\lim_{\tau\to-\infty} (u_{a} \circ \varphi_{a})(\tau,t)
&= \gamma_-(t) \\
\lim_{\tau\to +\infty} (u_{a'} \circ \varphi_{a'})(\tau,t)
&= \gamma_+(t).
\endaligned
$$
Moreover
$$
[u_*[\Sigma]] \# w_- = w_+
$$
where $\#$ is the obvious concatenation.
\item
We assume $((\Sigma,(z_-,z_+,\vec z),a_1,a_2),u,\varphi)$
is stable in the sense of Definition \ref{stable26revrev} below.
\end{enumerate}
\end{defn}
Assume that $((\Sigma,(z_-,z_+,\vec z),a_1,a_2),u,\varphi)$
satisfies (1)-(9) above.
The {\it extended automorphism group}
${\rm Aut}^+((\Sigma,(z_-,z_+,\vec z),a_1,a_2),u,\varphi)$
of $((\Sigma,(z_-,z_+,\vec z),a_1,a_2)),u,\varphi)$
consists of maps $v : \Sigma \to \Sigma$
satisfying (1)(2)(3)(5) of Definition \ref{defn2615rev}
and $\tau_{a_1} = \tau_{a_2} = 0$.
The {\it automorphism group} denoted by 
${\rm Aut}((\Sigma,(z_-,z_+,\vec z), a_1,a_2),u,\varphi)$
of $((\Sigma,(z_-,z_+,\vec z),a_1,a_2),u,\varphi)$
consists of the elements of ${\rm Aut}^+((\Sigma,(z_-,z_+,\vec z),a_1,a_2),u,\varphi)$
such that $\sigma$ in (5) of Definition \ref{defn2615rev} is the identity.
\begin{defn}\label{stable26revrev}
An element $((\Sigma,(z_-,z_+,\vec z),a_1,a_2),u,\varphi)$ in Definition \ref{defn210revrev} is said to be {\it stable}
if the group ${\rm Aut}((\Sigma,(z_-,z_+,\vec z),a_1,a_2),u,\varphi)$ is a finite group.
\end{defn}
We can define the equivalence relation $\sim_2$ on
$\widehat{\mathcal N}_{\ell}(X,\mathcal J^{31,\infty},H^{31,\infty};\alpha_-,\alpha_+)$
in the same way as in Definition \ref{3equivrel}
except we require $\tau_{a_1}  = \tau_{a_2} = 0$.
We put
\begin{equation}\label{formula2692}
{\mathcal N}_{\ell}(X,\mathcal J^{31,\infty},H^{31,\infty};\alpha_-,\alpha_+) =
\widehat{\mathcal N}_{\ell}(X,\mathcal J^{31,\infty},H^{31,\infty};\alpha_-,\alpha_+)/\sim_2.
\end{equation}
\begin{defn}
The set
$\mathcal N_{\ell}(X,\mathcal J^{31,[0,\infty]},H^{31,[0,\infty]};\alpha_-,\alpha_+)$
is the union of (\ref{2690formula})
and
${\mathcal N}_{\ell}(X,\mathcal J^{31,\infty},H^{31,\infty};\alpha_-,\alpha_+)$.
\end{defn}
We can define a topology on $\mathcal N_{\ell}(X,\mathcal J^{31,[0,\infty]},H^{31,[0,\infty]};\alpha_-,\alpha_+)$ in the same way as in Definition \ref{def2626}
and show that it is Hausdorff and compact.
Theorem \ref{tjm26108} will follow from the next result.
We define the topological space 
$$
\mathcal N_{\ell}(X,\mathcal J^{31,[0,\infty]},H^{31,[0,\infty]};\alpha_-,\alpha_+)^{\boxplus 1}
$$
in the same way as in Definition \ref{def:outcollar}.
\begin{thm}\label{the26112}
\begin{enumerate}
\item
There exists a Kuranishi structure on the compact space
$\mathcal N_{\ell}(X,\mathcal J^{31,[0,\infty]},H^{31,[0,\infty]};\alpha_-,\alpha_+)$.
\item
The Kuranishi structure in (1) extends to a Kuranishi structure on the space
$\mathcal N_{\ell}(X,\mathcal J^{31,[0,\infty]},H^{31,[0,\infty]};\alpha_-,\alpha_+)^{\boxplus 1}$,
which becomes the interpolation space of a homotopy between
$\frak N_{32} \circ \frak N_{21}$ and
$\frak N_{31}$.
\end{enumerate}
\end{thm}
\subsection{Proof of Theorem \ref{the26112} (1): Kuranishi structure}
\label{subsec:Kstrcture}
\begin{proof}
The proof of Theorem \ref{the26112} occupies the rest of this
section.
In this subsection we prove (1). The construction of the Kuranishi structure
on the space $\mathcal N_{\ell}(X,\mathcal J^{31,[0,\infty]},H^{31,[0,\infty]};\alpha_-,\alpha_+)$ is mostly the same as that of the
first half of the proof of Theorem \ref{thmPparaexist},
where we constructed the Kuranishi structure on
$$
\mathcal M_{\ell}(X,\mathcal J^{[0,1]},H^{21,[0,1]};\alpha_-,\alpha_+).
$$
In fact, we are studying $[0,\infty) \times \R \times [0,1]$-parametrized family
of Hamiltonians and almost complex structures and
we are also given the choices which we need for the definition of the Kuranishi structure at
$0 \in \partial [0,\infty)$.
\par
The main difference is that we also include $T = \infty$.
So we here concentrate on constructing a Kuranishi neighborhood
of a point at
${\mathcal N}_{\ell}(X,\mathcal J^{31,\infty},H^{31,\infty};\alpha_-,\alpha_+)$.
\par
Let ${\bf p} = ((\Sigma,(z_{-},z_{+},\vec z),a_{1},a_{2}),u,\varphi)$
be a representative of an element of the moduli space
${\mathcal N}_{\ell}(X,\mathcal J^{31,\infty},H^{31,\infty};\alpha_-,\alpha_+)$.
We assume that
$\Sigma$ has $k_1 + k_2 + k_3 + 2$ mainstream components.
Namely  there exit  $k_1$ mainstream components $\Sigma_a$ with
$a < a_{1}$, $k_2$ mainstream components $\Sigma_{a}$ with
$a_{1} < a < a_{2}$ and
$k_3$ mainstream components $\Sigma_{a}$ with
$a_{2} <a$. 
So there are $k_1+k_2+k_3+1$ transit points.
We consider $\Sigma \setminus \Sigma_{a_1}
\setminus \Sigma_{a_2}$ which 
has three connected components.
Among those transit points $k_1$, $k_2+1$, $k_3$ lie
on the closure of each of those connected components.
We take $T_{1,1},\dots,T_{1,k_1}$,
$T_{2,0},\dots,T_{2,k_2}$, $T_{3,1},\dots,T_{3,k_3}$ 
which are parameters $\in (T_0,\infty]$ to smooth those transit points.
\par
We consider the moduli space $\mathcal N_{\ell}({\rm source})$
as in Definition \ref{defn2688Nsource}.
We also define
$\mathcal N_{\ell}({\rm source},\infty)$ as follows.
In Definition \ref{defn210revrev} we consider the case
when $X$ is one point and $H^1=H^2=H^3 =0$.
The space we obtain as ${\mathcal N}_{\ell}(X,\mathcal J^{31,\infty},H^{31,\infty};\alpha_-,\alpha_+)$ in that case is the space $\mathcal N_{\ell}({\rm source},\infty)$ by definition.
We put
\begin{equation}
\mathcal N_{\ell}({\rm source},(T^0,\infty])
=
\mathcal N_{\ell}({\rm source},\infty)
\cup
\bigcup_{T>T_0} (\mathcal N_{\ell}({\rm source})
\times \{T\}).
\end{equation}
\par
We go back to the situation where we have
${\bf p} = ((\Sigma,(z_{-},z_{+},\vec z),a_{1},a_{2}),u,\varphi) \in {\mathcal N}_{\ell}(X,\mathcal J^{31,\infty},H^{31,\infty};\alpha_-,\alpha_+)$.
We take stabilization data at ${\bf p}$.
Especially we fix a symmetric stabilization $\vec w$ of ${\bf p}$.
We also take the canonical marked point $w_{a,{\rm can}}$ on each mainstream component
$\Sigma_{a}$
where there is at most two special points, (that are $z_{a,-}$ and
$z_{z,+}$.)
The canonical marked point $w_{a,{\rm can}}$ is defined as in Definition \ref{defncanmark}
if $a \ne a_1,a_2$. We do not define canonical marked points in case
$a = a_1,a_2$. (See Remark \ref{rem269494}.)
The totality of the canonical marked points is denoted by
$\vec w_{{\rm can}}$.
We can prove that $(\Sigma,(z_{-},z_{+},\vec z))
\cup \vec w \cup \vec w_{{\rm can}}$ is stable in the same way as in 
Lemma \ref{lemma2639}.
We will define a map
\begin{equation}\label{form418revrev}
\aligned
{\Phi}_{\bf p} : \prod_{\rm v} \mathcal V(\frak x_{\rm v} \cup \vec w_{\rm v}\cup \vec w_{{\rm can},{\rm v}}) 
& \times (T_0,\infty]^{k_1+k_2+k_3+1}
\times \prod_{j=1}^m \left(((T_{0,j},\infty] \times S^1)/\sim \right) \\
& \to {\mathcal N}_{\ell+\ell'+\ell''}(\text{source},(T^{\prime 0},\infty]).
\endaligned\end{equation}
We first explain the notation in \eqref{form418revrev}. $\Sigma_{\rm v}$ is an irreducible components of
$\Sigma$.  $\frak x_{\rm v}$ is $\Sigma_{\rm v}$ together with the
part of marked points $(z_{-},z_{+},\vec z)$ on it.
$\vec w_{\rm v}$ and $\vec w_{{\rm can},{\rm v}}$
are intersection of  $\vec w$ and $\vec w_{{\rm can}}$
with $\Sigma_{\rm v}$, respectively.
$\mathcal V(\frak x_{\rm v} \cup \vec w_{\rm v}\cup \vec w_{{\rm can},{\rm v}})$
is a neighborhood of
$\frak x_{\rm v} \cup \vec w_{\rm v}\cup \vec w_{{\rm can},{\rm v}}$
in the moduli space of pointed curves.
Namely :
\begin{enumerate}
\item If $\Sigma_{\rm v}$ is a bubble component, then
$\mathcal V(\frak x_{\rm v} \cup \vec w_{\rm v}\cup \vec w_{{\rm can},{\rm v}})$
is an open set of $\overset{\circ}{\mathcal M^{\rm cl}_{\ell_{\rm v}}}$
for certain $\ell_{\rm v}$. (It is the number of marked or singular points on $\Sigma_{\rm v}$.)
\item If $\Sigma_{\rm v}$ is a  main stream component $\Sigma_a$ and $a\ne a_1,a_2$,
then $\mathcal V(\frak x_{\rm v} \cup \vec w_{\rm v}\cup \vec w_{{\rm can},{\rm v}})$
is an open set of $\overset{\circ}{\mathcal M}_{\ell_{\rm v}}({\rm Source})$.
(We include the parametrization $\varphi_{\rm v}$ in  $\frak x_{\rm v}$ in this case.)
\item  If $\Sigma_{\rm v}$ is a  main stream component $\Sigma_a$ and $a\in  \{a_1,a_2\}$,
then  $\mathcal V(\frak x_{\rm v} \cup \vec w_{\rm v}\cup \vec w_{{\rm can},{\rm v}})$
is an open set of $\overset{\circ}{\mathcal N}_{\ell_{\rm v}}({\rm Source})$.
(We include the parametrization $\varphi_{\rm v}$ in  $\frak x_{\rm v}$ in this case.)
\end{enumerate}
Here we use Notation \ref{nota2642} and
$m$ is the number of non-transit singular points of $\Sigma$.
\par
Recall that the stabilization data at ${\bf p}$ contain 
the local trivialization data in Definition \ref{obbundeldata1} (3), 
which contain the data of a coordinate at
each singular point. 
(See \cite[Definition 3.8 (1)]{fooo:const1}.)
Using it we can associate a marked Riemann surface
to each element of the domain in (\ref{form418revrev}).
Let $\Sigma'$ be this curve.
Other than $\Sigma'$ and marked points on it
we need  to associate a few more data to obtain
an element of
$ {\mathcal N}_{\ell+\ell'+\ell''}(\text{source},(T^{\prime 0},\infty])$.
We will explain how to associate those data to an element of
the domain in (\ref{form418revrev}).
\par
Recall that we take the parameters $T_{1,1},\dots,T_{1,k_1}$,
$T_{2,0},\dots,T_{2,k_2}$, $T_{3,1},\dots,T_{3,k_3}$ $\in (T_0,\infty]$
to smooth the transit points.
\par\smallskip
\noindent{\bf Case 1:}
We first consider the case when
$T_{2,0} + \dots + T_{2,k_2} = \infty$.
This is equivalent to the condition that
at least one of $T_{2,i}$ is infinity.
We will obtain an element of $\mathcal N_{\ell}({\rm source},\infty)$
in this case as follows.
\par
In this case $\Sigma'$ contains two different components
$\Sigma'_{a'_1}$ and $\Sigma'_{a'_2}$ which are
obtained by gluing $\Sigma_{a_1}$ and $\Sigma_{a_2}$ with
other components, respectively.
We take them as the first and the second main components, respectively.
Using marked points on $\Sigma'$ corresponding to $z_{\pm} \in \Sigma$,
we can define the notion of mainstream components of $\Sigma'$.
It is easy to see that $\Sigma'_{a'_1}$ and $\Sigma'_{a'_2}$
are contained in the mainstream.
\par
As mentioned in Definition \ref{obbundeldata1} (3), we
assume 
the conditions (a) (b) for the local coordinates
at singular points. Since we use them for our gluing, that is, the construction of the map (\ref{form418revrev}),
we can show the next lemma
 about the relationship of parametrization of the
mainstreams of $\Sigma$ and of $\Sigma'$.
\begin{lem}\label{lem26113}
Let $\varphi_{a} : \R \times S^1 \to \Sigma_{a}$ be a parametrization
of a mainstream component.
We take a biholomorphic map $\varphi'_{a'} : \R \times S^1 \to \Sigma'_{a'} \setminus
\{z'_{a',-}, z'_{a',+}\}$ such that
$$
\lim_{\tau \to \pm \infty}\varphi'_{a'}(\tau,t) = z'_{a',\pm}.
$$
We use the diffeomorphism $\frak v$ from
a complement of the neck region of $\Sigma$ to $\Sigma'$,
which we obtain by using 
the local trivialization data contained in  
the stabilization data at ${\bf p}$.
Then if the image of $\frak v \circ \varphi_{a}$ and $\varphi'_{a'}$
intersect, we have
\begin{equation}\label{form26101101}
\varphi'_{a'}(\tau,t) = (\frak v\circ\varphi_{a})(\tau+\tau_0,t+t_0)
\end{equation}
for certain $\tau_0 \in \R$ and $t_ 0 \in S^1$.
\end{lem}
Using this lemma we can take parametrizations of the main components
$\varphi_{a'_i}$ ($i=1,2$) so that $\tau_0, t_0$ become zero for them.
By our choice of the local coordinates
at transit points 
given by Definition \ref{obbundeldata1} (3) (a)
we can choose parametrizations of all the mainstream components
so that the number $t_0$ in (\ref{form26101101}) becomes $0$.
Note other than the first and second main components, the parametrization of the
mainstream component is well-defined only up to a shift on $\R$ direction.
Therefore by Lemma  \ref{lem26113} we obtain
a parametrization of the mainstream.
We thus obtain an element of
$\mathcal N_{\ell}({\rm source},\infty)$ in this case.
\par\smallskip
\noindent{\bf Case 2:}
We next consider the case when
$T_{2,0} + \dots + T_{2,k_2} \ne \infty$.
We will obtain an element of
$\mathcal N_{\ell}({\rm source})
\times \{T\}$ for a certain number 
$T \in (T^{\prime 0}, \infty)$. (The number $T$ is determined by the
data in the domain of the map (\ref{form418revrev})).
\par
In this case $\Sigma_{a_1}$ and $\Sigma_{a_2}$ are glued to
become a part of single mainstream component of $\Sigma'$,
which we write as $\Sigma'_{a'_0}$.
This will be our main component.
\par
We apply Lemma \ref{lem26113} to
$a =a_1$ or $a_2$,  and $a' = a'_{0}$.
By Definition \ref{obbundeldata1} (3) (a)
we may take the same $t_0$
for both. (Note we use the same $\varphi'_{a'_0}$ for both
$a =a_1$ or $a_2$.) We then have
$$
\varphi'_{a'_0}(\tau,t) = (\frak v\circ\varphi_{a_i})(\tau -\tau_i,t).
$$
We put
\begin{equation}\label{formula2Tplus1}
2T + 1 = \tau_2 - \tau_1.
\end{equation}
Then by shifting $\varphi'_{a'_0}$ in $\R$ direction,
we may assume
$$
\aligned
\varphi'_{a'_0}(\tau,t) &= (\frak v\circ\varphi_{a_1})(\tau+T+1/2,t), \\
\varphi'_{a'_0}(\tau,t) &= (\frak v\circ\varphi_{a_2})(\tau-T-1/2,t).
\endaligned$$
By comparing this formula with (\ref{formula2688}), the Hamiltonian term
of our equation (\ref{Fleq266211}) is consistent with $\frak v$ by this choice.
We thus determine the parametrization of the mainstream.
Together with $T$ and $a'_{0}$ we already defined, this parametrization gives
an element of $\mathcal N_{\ell}({\rm source})
\times \{T\}$.
\par
We have thus defined (\ref{form418revrev}).
We find that this map defines a structure
of cornered orbifolds to 
$\mathcal N_{\ell}({\rm source},(T^{\prime 0},\infty])$
for which  (\ref{form418revrev}) becomes a diffeomorphism
onto its image.
We can use this fact to obtain an appropriate coordinate.
Namely we identify
$ (T_0,\infty]^{k_1+k_2+k_3+1}$
with $[0,1/\log T_0)^{k_1+k_2+k_3+1}$
by $T_i \mapsto 1/\log T_i$.
(See Section \ref{subsec;KuraFloer}.)
\begin{rem}
Note that $T$ is not a part of the coordinate function in a neighborhood
of $T=\infty$. It is easy to see that
$T = T_{2,0} + \dots + T_{2,k_2} + $ bounded function.
$1/\log T$ is not a smooth function of  $1/\log T_i$.
(See \cite[page 778]{fooobook2} 
\cite[Chapter 25]{fooonewbook} for a related issue.)
When we are away from $T = \infty$, 
$T$ can be taken as a part of coordinates.
\end{rem}
The construction of the Kuranishi structure of
$\mathcal N_{\ell}(X,\mathcal J^{31,[0,\infty]},H^{31,[0,\infty]};\alpha_-,\alpha_+)$
in a neighborhood of
${\mathcal N}_{\ell}(X,\mathcal J^{31,\infty},H^{31,\infty};\alpha_-,\alpha_+)$,
is similar to the first half of the proof of Theorem \ref{thmPparaexist}
using the construction we gave above and proceed as follows.
\par
We take a finite set
$$
\aligned
\EuScript A_{\ell}(H^{32,\infty},\mathcal J^{32,\infty};\alpha_-,\alpha_+)
&= 
\{{\bf p}_c \mid c \in \EuScript C_{\ell}(H^{32,\infty},\mathcal J^{32,\infty};\alpha_-,\alpha_+) \} \\
&\subset {\mathcal N}_{\ell}(X,\mathcal J^{32,\infty},H^{32,\infty};\alpha_-,\alpha_+).
\endaligned
$$
For each $c \in \EuScript C_{\ell}(H^{32,\infty},\mathcal J^{32,\infty};\alpha_-,\alpha_+)$
we take obstruction bundle data $\frak E_{{\bf p}_{c}}$
centered at ${\bf p}_c$\footnote{The notion of
obstruction bundle data for an element of
$ {\mathcal N}_{\ell}(X,\mathcal J^{32,\infty},H^{32,\infty};\alpha_-,\alpha_+)$
can be defined in the similar way as before.}
and
a closed neighborhood $W({\bf p}_c)$ of ${\bf p}_c$
in our moduli space 
$$
\mathcal N_{\ell}(X,\mathcal J^{31,[0,\infty]},H^{31,[0,\infty]};\alpha_-,\alpha_+)
$$ 
with the following property.
For each element  ${\bf q}  \in W({\bf p}_c)$
there exists $\vec w^{\bf q}_{{\bf p}_c}$ such that
${\bf q} \cup \vec w^{\bf q}_{{\bf p}_c}$ is $\epsilon_c$-close
to ${\bf p}_c \cup w_{{\bf p}_c} \cup \vec w_{{\rm can}}$.
(We can use (\ref{form418revrev})
to define the notion of $\epsilon$-closeness
in the same way as in Definition \ref{orbitecloseness}.)
We also require
\begin{equation}\label{coversururevrev}
\bigcup_{c\in \EuScript C_{\ell}(H^{32,\infty},\mathcal J^{32,\infty};\alpha_-,\alpha_+)}  {\rm Int}\,\, W({\bf p}_c)
\supset \mathcal N_{\ell}(X,\mathcal J^{32,\infty},H^{32,\infty};\alpha_-,\alpha_+).
\end{equation}
\par
We can choose $\epsilon_c >0$ above such that 
for any ${\bf q}  \in W({\bf p})$
there exists $\vec w^{\bf q}_{{\bf p}}$ {\rm uniquely} with the following
properties:
\begin{enumerate}
\item
${\bf q} \cup \vec w^{\bf q}_{{\bf p}}$ is $\epsilon_{\bf p}$-close to
 ${\bf p}_c \cup w_{{\bf p}_c} \cup \vec w_{{\rm can}}$.
\item
${\bf q} \cup \vec w^{\bf q}_{{\bf p}}$ satisfies the
transversal constraint.
(The transversal constraint is defined in the same way as in Definition \ref{transdef2695}.)
\item
For each irreducible component $\Sigma_{\rm v}$ of the source curve of ${\bf p}$
there exist a finite dimensional complex subspace 
${\rm Im}\ I_{{\bf p},{\rm v};{\bf q}}$ of $L^2_{m,\delta}(\Sigma_{\rm v};u^{\ast}TX\otimes \Lambda^{0,1}\Sigma_{\rm v})$
such that the linearization operator of the differential equation at ${\bf q}$ in Definition \ref{defn210revrev} (5) (7) is surjective 
$\mod \oplus_{\rm v}{\rm Im} I_{{\bf p},{\rm v};{\bf q}}.$
\end{enumerate}
The proof of this fact is similar to that of Lemma \ref{lem2446}.
\par
For each ${\bf q} \in
\mathcal N_{\ell}(X,\mathcal J^{31,[0,\infty]},H^{31,[0,\infty]};\alpha_-,\alpha_+)$
which lies
in a neighborhood of ${\mathcal N}_{\ell}(X,\mathcal J^{32,\infty},H^{32,\infty};\alpha_-,\alpha_+)$, 
we put
$$
\EuScript E({\bf q}) = \{ {\bf p}_c \mid {\bf q} \in W({\bf p}_c)\}.
$$
Let $\EuScript B$ be a nonempty subset of $\{ c ~\vert~ {\bf p}_c \in \EuScript E({\bf q})\}$.
We consider $(\frak Y\cup\bigcup_{c\in 
\EuScript B}\vec w'_c,u',\varphi',a'_1,a'_2)$
(resp. $(\frak Y\cup\bigcup_{c\in 
\EuScript B}\vec w'_c,u',\varphi',a'_0,T')$)
such that for each $c$,
$(\frak Y\cup \vec w'_c,u',\varphi',a'_1,a'_2)$ and 
$(\frak Y\cup \vec w'_c,u',\varphi',a'_0,T')$
(resp.
$(\frak Y\cup \vec w'_c,u',\varphi',a'_1,a'_2)$)
are $\epsilon$-close to ${\bf q} \cup \vec w_{c}^{\bf q}$.
If $\epsilon>0$ is small, then $(\frak Y\cup \vec w'_c,u',\varphi',a'_0,T')$
(resp.
$(\frak Y\cup \vec w'_c,u',\varphi',a'_1,a'_2)$)  is $\epsilon$-close to ${\bf p}_c \cup \vec w_{c}$.
Therefore we can define a complex linear map
$$
I_{{\bf p}_c,{\rm v};\Sigma',u',\varphi',a'_0,T'} : E_{{\bf p}_c,{\rm v}}(\frak y_{{\bf p}_c}) \to
C^{\infty}(\Sigma';(u')^*TX \otimes \Lambda^{0,1})
$$
(resp.
$$
I_{{\bf p}_c,{\rm v};\Sigma',u',\varphi',a'_1,a'_2} : E_{{\bf p}_c,{\rm v}}(\frak y_{{\bf p}_c}) \to
C^{\infty}(\Sigma';(u')^*TX \otimes \Lambda^{0,1}))
$$
in the same way as in (\ref{form2664}) for each irreducible component ${\rm v}$ of ${\bf p}_c$.
(We use the map (\ref{form418revrev}) here while 
we used the map (\ref{form418rev}) in (\ref{form2664}).)
We put
\begin{equation}\label{obspacedefHFHFrevrev}
E((\frak Y\cup \bigcup_{c\in \EuScript B}\vec w'_c,u',\varphi',a'_0,T');{\bf q};\EuScript B)
=
\bigoplus_{c \in \EuScript B}\bigoplus_{\rm v} {\rm Im} \,\, I_{{\bf p}_c,{\rm v};\Sigma',u',\varphi',a'_0,T'},
\end{equation}
or
\begin{equation}\label{obspacedefHFHFrevre2v}
E((\frak Y\cup \bigcup_{c\in \EuScript B}\vec w'_c,u',\varphi',a'_1,a'_2);{\bf q};\EuScript B)
=
\bigoplus_{c \in \EuScript B}\bigoplus_{\rm v} {\rm Im} \,\, I_{{\bf p}_c,{\rm v};\Sigma',u',\varphi',a'_1,a'_2}.
\end{equation}
In the next definition we fix stabilization data at ${\bf q}$.
Especially the marked points $\vec w_{\bf q}$ are fixed.
\begin{defn}
We define 
$V({\bf q},\epsilon_1,\EuScript B)$
to be the union of isomorphism classes of
the following (A) and (B).
\begin{enumerate}
\item[(A)]
An object
$(\frak Y\cup\bigcup_{c\in \EuScript B}\vec w'_c \cup \vec w'_{\bf q},u',\varphi',a'_1,a'_2)$
which satisfies the following.
\begin{enumerate}
\item[(1)]
If $\Sigma'_a$ is an irreducible component and
$\varphi_a : \R \times S^1 \to \Sigma'_a$ is as above, then
the composition $u_a \circ \varphi_a$ satisfies the equation
\begin{equation}\label{Fleq2662100}
\aligned
\frac{\partial (u_a \circ \varphi_a)}{\partial \tau} +  J_{a,\tau} &\left( \frac{\partial (u_a \circ \varphi_a)}{\partial t} - 
\frak X_{H^a_{\tau,t}}
\circ (u_a \circ \varphi_a) \right) \equiv 0 \\
&\mod E((\frak Y\cup\bigcup_{c\in \EuScript B} \vec w'_c,u',\varphi',a'_1,a'_2);{\bf q};\EuScript B).
\endaligned
\end{equation}
Here $H^a_{\tau,t}$ and $J_{a,\tau}$ are as in
Definition \ref{defn210revrev} (5).
\item[(2)]
If $\Sigma'_{\rm v}$  is a bubble component, the following equation is satisfied on $\Sigma'_{\rm v}$.
\begin{equation}
\overline{\partial}_J u' \equiv 0  
\mod  E((\frak Y\cup\bigcup_{c\in \EuScript B} \vec w'_c,u',\varphi',a'_1,a'_2);{\bf q};
\EuScript B).
\end{equation}
Here the almost complex structure $J$ is as follows.
Let $\widehat\Sigma'_a$ be the extended mainstream component containing
$\Sigma'_{\rm v}$. If $a < a'_1$, then $J = J_1$. If $a'_1 < a < a'_2$, then
$J = J_2$. If $a'_2 < a $, then
$J = J_3$.
If $a = a'_1$ and $\varphi_{a'_1}(\tau,t)$ is the root of the
tree of sphere components containing $\Sigma'_{\rm v}$
then $J = J^{21}_{\tau}$.
If $a = a'_2$ and $\varphi_{a'_2}(\tau,t)$ is the root of the
tree of sphere components containing $\Sigma'_{\rm v}$,
then $J = J^{32}_{\tau}$.
\item[(3)]
For each $c \in \EuScript E({\bf q})$ the additional marked points $\vec w'_c $
satisfy the transversal constraint with respect to ${\bf p}_c$.
(The transversal constraint is defined in the same way as in 
Definition \ref{transdef2695}.)
\item[(4)]
The additional marked points $\vec w'_{\bf q}$
satisfy the transversal constraint with respect to ${\bf q}$.
\item[(5)]
$(\frak Y\cup\bigcup_{c\in \EuScript E({\bf q})}\vec w'_c \cup \vec w'_{\bf q},u',
\varphi',a'_1.a'_2)$ is
$\epsilon_1$-close to
${\bf q} \cup \bigcup_{c\in \EuScript E({\bf q})}\vec w^{\bf q}_c \cup \vec w_{\bf q}$.
\end{enumerate}
\item[(B)]
An object
$(\frak Y\cup\bigcup_{c\in \EuScript B}\vec w'_c \cup \vec w'_{\bf q},u',\varphi',a'_0,T')$
which satisfies the following.
\begin{enumerate}
\item[(1)]
If $\Sigma'_a$ is an irreducible component and
$\varphi_a : \R \times S^1 \to \Sigma_a$ is as above, then
the composition $u_a \circ \varphi_a$ satisfies the equation
\begin{equation}\label{Fleq2662100}
\aligned
\frac{\partial (u_a \circ \varphi_a)}{\partial \tau} +  J_{a,\tau} &\left( \frac{\partial (u_a \circ \varphi_a)}{\partial t} - 
\frak X_{H^a_{\tau,t}}
\circ (u_a \circ \varphi_a) \right) \equiv 0 \\
&\mod E((\frak Y\cup\bigcup_{c\in \EuScript B} \vec w'_c,u',\varphi',a'_0,T');{\bf q};\EuScript B).
\endaligned
\end{equation}
Here $J_{a,\tau}$ and $H^a_{\tau,t}$ are as follows.
(See (\ref{formula2688}) and (\ref{formula2689}).)
$$
J_{a,\tau} =
\begin{cases}
J_1 &\text{if $a < a'_0$}, \\
J^{31}_{\tau,T'}  &\text{if $a = a_0$}, \\
J_3 &\text{if $a > a'_0$}.
\end{cases}
$$
$$
H^a_{\tau,t} =
\begin{cases}
H^1_t  &\text{if $a < a'_0$}, \\
H^{31,T'}_{\tau,t}  &\text{if $a = a'_0$}, \\
H^3_t  &\text{if $a > a'_0$}.
\end{cases}
$$
\item[(2)]
If $\Sigma'_{\rm v}$ is a bubble component, the following equation is satisfied on $\Sigma'_{\rm v}$:
\begin{equation}
\overline{\partial}_J u' \equiv 0  
\mod E((\frak Y\cup \vec w'_c,u',\varphi',a'_0,T');{\bf q};\EuScript B).
\end{equation}
Here the almost complex structure $J$ is as follows.
Let $\widehat\Sigma'_a$ be the extended mainstream component containing
$\Sigma'_{\rm v}$. If $a < a'_0$, then $J = J_1$. If $a'_0 < a$, then
$J = J_3$.
If $a = a'_0$ and $\varphi_{a'_0}(\tau,t)$ is the root of the
tree of sphere components containing $\Sigma'_{\rm v}$,
then $J = J^{31}_{\tau,T'}$.
\item[(3)]
The same as (3) of (A).
\item[(4)]
The same as (4) of (A).
\item[(5)]
$(\frak Y\cup\bigcup_{c\in \EuScript E({\bf q})}\vec w'_c \cup \vec w'_{\bf q},u',
\varphi',a'_0,T')$ is
$\epsilon_1$-close to
${\bf q} \cup \bigcup_{c\in \EuScript E({\bf q})}\vec w^{\bf q}_c \cup \vec w_{\bf q}$.
In particular, $T' > 1/\epsilon_1$.
\end{enumerate}
\end{enumerate}
The isomorphism among the objects of (A) is defined in the same way as the equivalence relation $\sim_2$ on
$\widehat{\mathcal N}_{\ell}(X,\mathcal J^{31,\infty},H^{31,\infty};\alpha_-,\alpha_+)$, which is
the same as Definition \ref{3equivrel},
except we require $\tau_{a_1}  = \tau_{a_2} = 0$.
The isomorphism among objects of (B) is
defined in the same way as in Definition \ref{defn2650rev}
(a)(b)(c)(d). (We require $T'$ coincides for two objects to be equivalent.)
An object of (A) is never equivalent to an object of (B).
\end{defn}
In the same way as in Lemma \ref{lem2653}, we can prove that
$V({\bf q},\epsilon_1,\EuScript B)$
is a smooth manifold with boundary and corner
if $\epsilon_1 >0$ and $\epsilon_c >0$ are small enough.
(We use the fact that (\ref{form418revrev}) is an open embedding
here.)
\par
Then in the same way as in Section \ref{subsec;KuraFloer}, we can find other data
so that $V({\bf q},\epsilon_1,\EuScript B)$ together with them
is a Kuranishi chart of ${\bf q}$.
We can also show the existence of coordinate changes.
We shrink the Kuranishi neighborhood and
discuss in the same way as in 
Lemmas \ref{lem2657}, \ref{lem2658} and 
use the exponential decay estimates in the same way as in \cite[Chapter 8]{foooanalysis}
to obtain a Kuranishi structure on
$\mathcal N_{\ell}(X,\mathcal J^{31,[0,\infty]},H^{31,[0,\infty]};\alpha_-,\alpha_+)$ on a neighborhood of
${\mathcal N}_{\ell}(X,\mathcal J^{31,\infty},H^{31,\infty};\alpha_-,\alpha_+)$.
We can extend it to the whole
$\mathcal N_{\ell}(X,\mathcal J^{31,[0,\infty]},H^{31,[0,\infty]};\alpha_-,\alpha_+)$
in the same way as in the proof of Theorem \ref{thmPparaexist}.
We have thus proved Theorem \ref{the26112} (1).
\subsection{Proof of Theorem \ref{the26112} (2): 
Kuranishi structure with outer collar}
\label{subsec:KstrctureCollar}
In this subsection we prove Theorem \ref{the26112} (2).
\begin{rem}
Since the rest of the proof is again a repetition of the construction of
previous sections, the readers may safely skip it and
go directly to Section \ref{sec;welldefinedness}.
We provide the details of the proof here for the sake of completeness.
The formulas appearing during the proof are lengthy but the argument is a
straightforward analogue.
\end{rem}
We will modify the Kuranishi structure of
$\mathcal N_{\ell}(X,\mathcal J^{31,[0,\infty]},H^{31,[0,\infty]};\alpha_-,\alpha_+)^{\boxplus 1}$
in the collar so that it will be compatible with the fiber product description of
its boundary and corners.
The way to modify our Kuranishi structure
is entirely the same as the proof of Theorem \ref{theorem266rev} (3)(4)
except at the boundary corresponding to $T=\infty$.
(Namely the subset described in Definition \ref{defn210revrev}.) 
We denote this boundary component
$\frak C_{\infty}$ and discuss our construction
only on 
$$
\widehat S_m^{\frak C_{\infty}}
(\mathcal N_{\ell}(X,\mathcal J^{31,[0,\infty]},H^{31,[0,\infty]};\alpha_-,\alpha_+)^{\boxplus 1}).
$$
Let $\frak A_r$ be the index set of the critical submanifolds of $H^r$ ($r=1,2,3$).
Let 
$$
\aligned
\alpha_- =\alpha_{1,0}, \alpha_{1,1},\dots,\alpha_{1,m_1-1},\alpha_{1,m_1}
& \in \frak A_1,\\
\alpha_{2,1},\dots,\alpha_{2,m_2-1},\alpha_{2,m_2}
& \in \frak A_2, \\ 
\alpha_{3,1},\dots,\alpha_{3,m_3},\alpha_{3,m_3+1}
=\alpha_+
& \in \frak A_3.
\endaligned
$$
We put
$\vec \alpha_1 =
(\alpha_{1,0},\alpha_{1,1},\dots,\alpha_{1,m_1-1},\alpha_{1,m_1} )$,
$\vec \alpha_2 = (\alpha_{2,1},\dots,\alpha_{2,m_2-1},\alpha_{2,m_2})$ and
$\vec \alpha_3 = (\alpha_{3,1},\dots,\alpha_{3,m_3},\alpha_{3,m_3+1})$.
\par
We consider the fiber product
\begin{equation}\label{form2641revrev}
\aligned
&{\mathcal M}_{\ell_{1,1}}(X,J_1,H^1;\alpha_{1,0},\alpha_{1,1})
\,{}_{{\rm ev}_{+}}\times_{{\rm ev}_-}  \dots \\
&\qquad\qquad\qquad\qquad\dots
\,{}_{{\rm ev}_{+}}\times_{{\rm ev}_-}
{\mathcal M}_{\ell_{1,m_1}}(X,J_1,H^1;\alpha_{1,m_1-1},\alpha_{1,m_1})
\\
&
\,{}_{{\rm ev}_{+}}\times_{{\rm ev}_-}
{\mathcal N}_{\ell_{(12)}}(X,\mathcal J^{21},H^{21};\alpha_{1,m_1},\alpha_{2,1})
\\
&\,{}_{{\rm ev}_{+}}\times_{{\rm ev}_-}{\mathcal M}_{\ell_{2,2}}(X,J_2,H^2;\alpha_{2,1},\alpha_{2,2})
\,{}_{{\rm ev}_{+}}\times_{{\rm ev}_-}  \dots \\
&\qquad\qquad\qquad\qquad\dots
\,{}_{{\rm ev}_{+}}\times_{{\rm ev}_-}
{\mathcal M}_{\ell_{2,m_2}}(X,J_2,H^2;\alpha_{2,m_2-1},\alpha_{2,m_2})
\\
&
\,{}_{{\rm ev}_{+}}\times_{{\rm ev}_-}
{\mathcal N}_{\ell_{(23)}}(X,\mathcal J^{32},H^{32};\alpha_{2,m_2},\alpha_{3,1})
\\
&\,{}_{{\rm ev}_{+}}\times_{{\rm ev}_-}{\mathcal M}_{\ell_{3,2}}(X,J_3,H^3;\alpha_{3,1},\alpha_{3,2})
\,{}_{{\rm ev}_{+}}\times_{{\rm ev}_-}  \dots \\
&\qquad\qquad\qquad\qquad\dots
\,{}_{{\rm ev}_{+}}\times_{{\rm ev}_-}
{\mathcal M}_{\ell_{3,m_3+1}}(X,J_3,H^3;\alpha_{3,m_3},\alpha_{3,m_3+1}),
\endaligned
\end{equation}
which we denote by
$
{\mathcal N}_{\vec \ell_1,\ell_{(12)},\vec \ell_2
,\ell_{(23)},\vec \ell_3}(X,\mathcal J^{21},\mathcal J^{32},H^{21},H^{32};\vec \alpha_1,\vec \alpha_2,\vec \alpha_3).
$
\par
We note that a neighborhood of $\widehat S_m^{\frak C_{\infty}}
(\mathcal N_{\ell}(X,\mathcal J^{31,[0,\infty]},H^{31,[0,\infty]};\alpha_-,\alpha_+)^{\boxplus 1})$
is a union of the direct products
$$
{\mathcal N}_{\vec \ell_1,\ell_{(12)},\vec \ell_2
,\ell_{(23)},\vec \ell_3}(X,\mathcal J^{21},\mathcal J^{32},H^{21},H^{32};\vec \alpha_1,\vec \alpha_2,\vec \alpha_3)
\times [-1,0]^{m-2}
$$
for
various
$\vec \ell_1,\ell_{(12)},\vec \ell_2
,\ell_{(23)},\vec \ell_3$,
$\vec \alpha_1,\vec \alpha_2,\vec \alpha_3$
such that
$$
\aligned
&\vert \vec \ell_1\vert + \vert \vec \ell_2\vert +
\vert \vec \ell_3 \vert +\ell_{(12)} + \ell_{(23)} = \ell, \\
& \# \vec\alpha_1 + \# \vec\alpha_2 + \# \vec\alpha_3 = m.
\endaligned
$$
Note $m_r + 1 =  \# \vec\alpha_r$ for $r=1,3$ and
 $m_2 =  \# \vec\alpha_2$.
\par
We will construct a Kuranishi structure for each of
$$
\aligned
&{\mathcal N}_{\vec \ell_1,\ell_{(12)},\vec \ell_2
,\ell_{(23)},\vec \ell_3}(X,\mathcal J^{21},\mathcal J^{32},H^{21},H^{32};\vec \alpha_1,\vec \alpha_2,\vec \alpha_3)^+ \\
&=
{\mathcal N}_{\vec \ell_1,\ell_{(12)},\vec \ell_2
,\ell_{(23)},\vec \ell_3}(X,\mathcal J^{21},\mathcal J^{32},H^{21},H^{32};\vec \alpha_1,\vec \alpha_2,\vec \alpha_3)
\\
&\qquad\times  [-1,0]^{m_1}  \times  [-1,0]^{m_2} \times  [-1,0]^{m_3}.
\endaligned
$$
Let $A_r \sqcup B_r  \sqcup C_r = \underline{m_r}$ for $r=1,2,3$.
It induces 
$\mathcal I_{A_r,B_r,C_r} : [-1,0]^{b_r} \to [-1,0]^{m_r}$ 
by (\ref{form2642}).
We will formulate compatibility conditions below (Conditions \ref{conds26117} - \ref{conds26117223}),
which describe the restriction of the
Kuranishi structure
\begin{equation}\label{form2610500}
\widehat{\mathcal U}_{\vec \ell_1,\ell_{(12)},\vec \ell_2
,\ell_{(23)},\vec \ell_3}(X,\mathcal J^{21},\mathcal J^{32},H^{21},H^{32};\vec \alpha_1,\vec \alpha_2,\vec \alpha_3)
\end{equation}
of
the product space
$
{\mathcal N}_{\vec \ell_1,\ell_{(12)},\vec \ell_2
,\ell_{(23)},\vec \ell_3}(X,\mathcal J^{21},\mathcal J^{32},H^{21},H^{32};\vec \alpha_1,\vec \alpha_2,\vec \alpha_3)^+$
to the image of the embedding
${\rm id} \times \mathcal I_{A_1,B_1,C_1} \times \mathcal I_{A_2,B_2,C_2}
\times \mathcal I_{A_3,B_3,C_3}$.
\par
We put
$A_r = \{ i(A_r,1),\dots, i(A_r,a_r)\}$ with 
$$i(A_r,1)<i(A_r,2) < \dots< i(A_r,a_r-1) < i(A_r,a_r).
$$
We define
$C'_j(A_r)$, $C_j(A_r)$, $\vec\alpha_{r,A_r,j}$,
$\vec\alpha_{r,A_r,j,C_r}$
in the same way as in Notation \ref{not2659rev} (1)(2)(3).
Here 
we put $i(A_r,a_r+1) = m_r$,
$i(r,A_r,0) = 1$, except
$i(A_1,0) = 0$, $i(A_3,a_3+1) = m_3+1$.
(Compare Remark \ref{rem2292} and  Notation \ref{not2659rev} (1).)
\par\medskip
\noindent{\bf Case 1:}
We first consider the case $A_2 \ne \emptyset$.
We consider the fiber product
\begin{equation}\label{264222revrev}
\aligned
&{\mathcal M}_{\vec \ell_{1,A_1,1}}(X,J_1,H^1;\vec \alpha_{1,A_1,1})
 \\
&
\,{}_{{\rm ev}_{+}}\times_{{\rm ev}_-} {\mathcal M}_{\vec \ell_{1,A_1,2}}(X,J_1,H^1;
\vec\alpha_{1, A_1,2})
\,{}_{{\rm ev}_{+}}\times_{{\rm ev}_-} \dots \\
&\,{}_{{\rm ev}_{+}}\times_{{\rm ev}_-} {\mathcal M}_{\vec \ell_{1,A_1,j}}(X,J_1,H^1;
\vec\alpha_{1,A_1,j})
\,{}_{{\rm ev}_{+}}\times_{{\rm ev}_-}
\dots \\
&
\,{}_{{\rm ev}_{+}}\times_{{\rm ev}_-} {\mathcal M}_{\vec\ell_{1,A_1,a_1}}(X,J_1,H^1;
\vec\alpha_{1, A_1,a_1}) \\
&\,{}_{{\rm ev}_{+}}\times_{{\rm ev}_-}
{\mathcal N}_{\vec\ell_{1,A_1,a_1+1},\ell_{(12)},\vec\ell_{2,A_2,1}}(X,\mathcal J^{21},H^{21};
\vec\alpha_{1, A_1,a_1+1}.
\vec\alpha_{2,A_2,1})
 \\
&
\,{}_{{\rm ev}_{+}}\times_{{\rm ev}_-} {\mathcal M}_{\vec \ell_{2,A_2,2}}(X,J_2,H^2;
\vec\alpha_{2,A_2,2})
\,{}_{{\rm ev}_{+}}\times_{{\rm ev}_-} \dots \\
&\,{}_{{\rm ev}_{+}}\times_{{\rm ev}_-} {\mathcal M}_{\vec \ell_{2,A_2,j}}(X,J_2,H^2;
\vec\alpha_{2,A_2,j})
\,{}_{{\rm ev}_{+}}\times_{{\rm ev}_-}
\dots \\
&
\,{}_{{\rm ev}_{+}}\times_{{\rm ev}_-} {\mathcal M}_{\vec\ell_{2,A_2,a_2}}(X,J_2,H^2;
\vec\alpha_{2, A_2,a_2})\\
&\,{}_{{\rm ev}_{+}}\times_{{\rm ev}_-}
{\mathcal N}_{\vec\ell_{2,A_2,a_2+1},\ell_{(23)},\vec\ell_{3,A_3,1}}(X,\mathcal J^{32},H^{32};
\vec\alpha_{2, A_2,a_2+1},
\vec\alpha_{3,A_3,1})
 \\
&
\,{}_{{\rm ev}_{+}}\times_{{\rm ev}_-} {\mathcal M}_{\vec \ell_{3,A_3,2}}(X,J_3,H^3;
\vec\alpha_{3,A_3,2})
\,{}_{{\rm ev}_{+}}\times_{{\rm ev}_-} \dots \\
&\,{}_{{\rm ev}_{+}}\times_{{\rm ev}_-} {\mathcal M}_{\vec \ell_{3,A_3,j}}(X,J_3,H^3;
\vec\alpha_{3,A_3,j})
\,{}_{{\rm ev}_{+}}\times_{{\rm ev}_-}
\dots \\
&
\,{}_{{\rm ev}_{+}}\times_{{\rm ev}_-} {\mathcal M}_{\vec\ell_{3,A_3,a_3+1}}(X,J_3,H^3;
\vec\alpha_{3,A_3,a_3+1}).
\endaligned
\end{equation}
Here
$\vec\ell_{r,A_r,j}
=(\ell_{r,i(A_r,j-1)+1},\dots,\ell_{r,i(A_r,j)}).
$
\par
We define $m_j(r,A_r)$ and $m_j(r,A_r,C_r)$ as in
Notation \ref{not2659rev} (5).
Then (\ref{form264555rev}) holds.
We define $\vec \ell_{r,A_r,j,C_r}$
in the same way as in                                                                                                                                                                                                                                                                                                                                                                                                                                                                                                                                                                                                                                                                                                                                                                                                                                                                                                                                                                                                                                                                                                                                                                                                                                                                                                                             Notation \ref{not2659rev} (6).
We put
\begin{equation}
\aligned
\ell'_{(r r+1)}
=
&\ell_{r,i(A_r,a_r)+k_{m_{a_r}(r,A_r,C_r)}+1}  \dots
+ \ell_{r,m_r} + \\
&+\ell_{(r r+1)}
+ \ell_{r+1,2} + \dots
+ \ell_{r+1,{\rm min} (i(A_{r+1},1), i(B_{r+1},1))}.
\endaligned
\end{equation}
We remark that in Proposition \ref{prop2661} we determined 
the Kuranishi structure on
\begin{equation}\label{form2677rev}
\aligned
&{\mathcal M}_{\vec\ell_{r,A_r,j},C_r}(X,J_r,H^r;\vec\alpha_{r,A_r,j,C_r} )^+ \\
&=
{\mathcal M}_{\vec\ell_{r,A_r,j},C_r}(X,J_r,H^r;\vec\alpha_{r,A_r,j,C_r} ) \times [-1,0]^{m_j(r,A_r,C_r)-1}.
\endaligned\end{equation}
By construction we have
$$
{\mathcal M}_{\vec\ell_{r,A_r,j}}(X,J_r,H^r;\vec\alpha_{r,A_r,j} )
\subseteq
\widehat S_{c_j(r,A_r)}(
{\mathcal M}_{\vec\ell_{r,A_r,j},C_r}(X,J_r,H^r;\vec\alpha_{r,A_r,j,C_r}))
$$
and the left hand side is a factor in (\ref{264222revrev}).
\par
By restriction it determines a Kuranishi structure of
(\ref{264222revrev}) times $[-1,0]^*$ except two of the factors
\begin{equation}\label{form2678rev00}
\aligned
{\mathcal N}_{\vec\ell_{r,A_r,a_r+1},\ell_{(r r+1)},\vec\ell_{r+1,A_{r+1},1}}(X,
\mathcal J^{r+1 r},H^{r+1 r}; &\vec\alpha_{r,A_r,a_r+1},
\vec\alpha_{r+1,A_{r+1},1}).
\endaligned
\end{equation}
By construction, we can easily show that (\ref{form2678rev00}) is
a component of
\begin{equation}
\aligned
 &\widehat S_{c_{a_r+1}(r,A_r)+c_{1}(r+1,A_{r+1})}(
{\mathcal N}_{\vec\ell_{r,A_r,a_r+1,C_r},\ell'_{(r r+1)},\vec\ell_{r+1,A_{r+1},1,C_{r+1}}}
\\
&
\qquad\qquad
(X,\mathcal J^{r+1 r},H^{{r+1}r}; \vec\alpha_{r,A_r,a_r+1,C_r}, \vec\alpha_{r+1,A_{r+1},1,C_{r+1}}).
\endaligned
\end{equation}
Note that we defined the Kuranishi structure on
$$
{\mathcal N}_{\vec\ell_{r,A_r,a_r+1,C_r},\ell'_{(r r+1)},\vec\ell_{r+1,A_{r+1},1,C_{r+1}}}
(X,\mathcal J^{r+1 r},H^{{r+1}r}; \vec\alpha_{r,A_r,a_r+1,C_r}, \vec\alpha_{r+1,A_{r+1},1,C_{r+1}})^+
$$
during the construction of the morphism $\frak N_{r+1 r}$ in 
Proposition \ref{prop2661rev}.
We write it as
\begin{equation}\label{form2611111}
\aligned
&\widehat{\mathcal U}_{\vec\ell_{r,A_r,a_r+1,C_r},\ell'_{(r r+1)},\vec\ell_{r+1,A_{r+1},1,C_{r+1}}}  \\
&\qquad(X,\mathcal J^{r+1 r},H^{{r+1}r}; \vec\alpha_{r,A_r,a_r+1,C_r}, \vec\alpha_{r+1,A_{r+1},1,C_{r+1}}).
\endaligned
\end{equation}
\begin{conds}\label{conds26117}
In the case $A_2 \ne \emptyset$ we require the restriction of the
Kuranishi structure (\ref{form2610500}) to the image of
${\rm id} \times \mathcal I_{A_1,B_1,C_1} \times \mathcal I_{A_2,B_2,C_2}
\times \mathcal I_{A_3,B_3,C_3}$
is the fiber product of the following Kuranishi structures.
(Here we use the fiber bundle description (\ref{264222revrev}).)
\begin{enumerate}
\item
The Kuranishi structure on
${\mathcal M}_{\vec\ell_{r,A_r,j}}(X,J_r,H^r;\vec\alpha_{r,A_r,j})
\times [-1,0]^{m_j(r, A_r, C_r) -1}$
which is a restriction of those on (\ref{form2677rev}) produced in Proposition \ref{prop2661}.
\item The Kuranishi structure on
(\ref{form2678rev00})$\times [-1,0]^{m_{a_r +1}(r, A_r, C_r)+ m_1(r+1, A_{r+1}, C_{r+1}) -2}$ 
which is a restriction of those on
 (\ref{form2611111}) produced in Proposition \ref{prop2661rev}.
\end{enumerate}
\end{conds}
\par\medskip
\noindent{\bf Case 2:}
We next consider the case $A_2 = \emptyset$ but $B_2 \ne \emptyset$.
\par
We consider the fiber product
\begin{equation}\label{264222revrev222}
\aligned
&{\mathcal M}_{\vec \ell_{1,A_1,1}}(X,J_1,H^1;\vec\alpha_{1,A_1,1})
 \\
&
\,{}_{{\rm ev}_{+}}\times_{{\rm ev}_-} {\mathcal M}_{\vec \ell_{1,A_1,2}}(X,J_1,H^1;
\vec\alpha_{1,A_1,2})
\,{}_{{\rm ev}_{+}}\times_{{\rm ev}_-} \dots \\
&\,{}_{{\rm ev}_{+}}\times_{{\rm ev}_-} {\mathcal M}_{\vec \ell_{1,A_1,j}}(X,J_1,H^1;
\vec\alpha_{1,A_1,j})
\,{}_{{\rm ev}_{+}}\times_{{\rm ev}_-}
\dots \\
&
\,{}_{{\rm ev}_{+}}\times_{{\rm ev}_-} {\mathcal M}_{\vec\ell_{1,A_1,a_1}}(X,J_1,H^1;
\vec\alpha_{1,A_1,a_1}) \\
&\,{}_{{\rm ev}_{+}}\times_{{\rm ev}_-}
{\mathcal N}_{\vec \ell_{1,A_1,a_1+1},\ell_{(12)},\vec \ell_2
,\ell_{(23)},\vec \ell_{3,A_3,1}}(X,\mathcal J^{21},\mathcal J^{32},H^{21},H^{32} \\
&\qquad\qquad\qquad\qquad\qquad\qquad\qquad\qquad\qquad
;\vec\alpha_{1, A_1,a_1+1},
\vec \alpha_2,
\vec\alpha_{3,A_3,1})\\
&
\,{}_{{\rm ev}_{+}}\times_{{\rm ev}_-} {\mathcal M}_{\vec \ell_{3,A_3,2}}(X,J_3,H^3;
\vec\alpha_{3,A_3,2})
\,{}_{{\rm ev}_{+}}\times_{{\rm ev}_-} \dots \\
&\,{}_{{\rm ev}_{+}}\times_{{\rm ev}_-} {\mathcal M}_{\vec \ell_{3,A_3,j}}(X,J_3,H^3;
\vec\alpha_{3,A_3,j})
\,{}_{{\rm ev}_{+}}\times_{{\rm ev}_-}
\dots \\
&
\,{}_{{\rm ev}_{+}}\times_{{\rm ev}_-} {\mathcal M}_{\vec\ell_{3,A_3,a_3+1}}(X,J_3,H^3;
\vec\alpha_{3,A_3,a_3+1}).
\endaligned
\end{equation}
We define $\vec{\alpha}_{2,C_2}$ by removing $\alpha_{2,i}$, $i\in C_2$ from $\vec\alpha_2$.\par
We define $\vec \ell_{2,C_2}$ as follows.
We put $\vec{\alpha}_{2,C_2}
= \{\alpha_{2, k_s} \mid s =0, \dots, m_{2,C_2}\}$,
$k_0 < k_1 < \dots < k_{m_{2,C_2}}$. (Here $m_{2,C_2} = \#\vec{\alpha}_{2,C_2} -1$.)
Note if
$i \in (k_s,k_{s+1})_{\Z}$, then $i \in C_2$.
We put
$$
\ell_{2,C_2,s} = \ell_{2,k_{s-1} +1} + \dots
+ \ell_{2,k_{s}}
$$
and
$\vec \ell_{2,C_2} = (\ell_{2,C_2,1},\dots,\ell_{2,C_2,m_{2,C_2}})$.
Now we consider the factor 
\begin{equation}\label{26113form}
\aligned
{\mathcal N}_{\vec \ell_{1,A_1,a_1+1},\ell_{(12)},\vec \ell_2
,\ell_{(23)},\vec \ell_{3,A_3,1}}(&X,\mathcal J^{21},\mathcal J^{32},H^{21},H^{32}
\\
&;(\vec\alpha_{1,A_1,a_1+1}, \vec\alpha_2,
\vec\alpha_{3,A_3,1}))
\endaligned
\end{equation}
in \eqref{264222revrev222} and lies in the corner of
\begin{equation}\label{eq:829}
\aligned
{\mathcal N}_{\vec \ell_{1,A_1,a_1+1C_1},\ell'_{(12)},\vec \ell_{2,C_2}
,\ell'_{(23)},\vec \ell_{3,A_3,1,C_3}}(&X,\mathcal J^{21},\mathcal J^{32},
H^{21},H^{32};\\
&\vec\alpha_{1,A_1, a_1+1,C_{1}}, \vec \alpha_{2,C_2},
\vec\alpha_{3,A_3,1,C_{3}}),
\endaligned
\end{equation}
where
$$
\aligned
&\ell'_{(12)} = \ell_{1,i(A_1,a_1) + k_{m_{a_1}(1,A_1,C_1)}+1}
+ \dots + \ell_{1,m_1} + \ell_{(12)} + \ell_{2,1} + \dots + \ell_{2,i(A_2,1)}, \\
&\ell'_{(23)} = \ell_{2,i(A_2,a_2) + k_{m_{a_2}(2,A_2,C_2)}+2}
+ \dots + \ell_{2,m_2} + \ell_{(23)} + \ell_{3,1} + \dots + \ell_{3,i(A_3,1)}.
\endaligned$$
We have a Kuranishi structure
\begin{equation}\label{26114form}
\aligned
\widehat{\mathcal U}_{\vec \ell_{1,A_1,a_1+1,C_1},\ell'_{(12)},\vec \ell_{2,C_2}
,\ell'_{(23)},\vec \ell_{3,A_3,1,C_3}}(&X,\mathcal J^{21},\mathcal J^{32},
H^{21},H^{32};\\
&\vec\alpha_{1,A_1, a_1+1,C_{1}}, \vec \alpha_{2,C_2},
\vec\alpha_{3,A_3, 1,C_{3}})
\endaligned
\end{equation}
on 
\eqref{26113form}$\times [-1,0]^*$.
(More precisely, we are during the process of producing it.
Here $* = \#\vec\alpha_{1,A_1, a_1+1,C_{1}} + \#\vec \alpha_{2,C_2} +
\#\vec\alpha_{3, A_3, 1,C_{3}} -2$.)
\begin{conds}\label{conds2611722}
In the case $A_2 = \emptyset$, $B_2 \ne \emptyset$, we require the restriction of the
Kuranishi structure (\ref{form2610500}) to the image of
${\rm id} \times \mathcal I_{A_1,B_1,C_1} \times \mathcal I_{A_2,B_2,C_2}
\times \mathcal I_{A_3,B_3,C_3}$
is the fiber product of the following Kuranishi structures.
(We use the fiber product description (\ref{264222revrev222}).)
\begin{enumerate}
\item
The Kuranishi structure on
${\mathcal M}_{\vec\ell_{r,A_r,j}}(X,J_r,H^r;\vec\alpha_{r,A_r,j} )\times [-1,0]^{m_j(r, A_r, C_r) -1}$
which is a restriction of those on (\ref{form2677rev}) produced  in Proposition \ref{prop2661}.
\item The Kuranishi structure on
(\ref{26113form})$\times [-1,0]^{*}$ 
($* = \#\vec\alpha_{1,A_1, a_1+1,C_{1}} + \#\vec \alpha_{2,C_2} +
\#\vec\alpha_{3,A_3,1,C_{3}} -2$) which is a restriction of \eqref{26114form}.
\end{enumerate}
\end{conds}
\par\medskip
\noindent{\bf Case 3:}
We finally consider the case when $A_2 = B_2 = \emptyset$.
\par
In this case, using the notation above, we have
$\vec \alpha_{2,C_2} = \emptyset$.
We require the compatibility with the Kuranishi structures on the part
$T < \infty$ in this case, as follows.
\par
We denote the fiber product:
\begin{equation}\label{form2641revrev333}
\aligned
&{\mathcal M}_{\ell_{1,1}}(X,J_1,H^1;\alpha_{1,0},\alpha_{1,1})
\,{}_{{\rm ev}_{+}}\times_{{\rm ev}_-}  \dots \\
&\qquad\qquad\qquad\qquad\dots
\,{}_{{\rm ev}_{+}}\times_{{\rm ev}_-}
{\mathcal M}_{\ell_{1,m_1}}(X,J_1,H^1;\alpha_{1,m_1-1},\alpha_{1,m_1})
\\
&
\,{}_{{\rm ev}_{+}}\times_{{\rm ev}_-}
\mathcal N_{\ell'}(X,\mathcal J^{31,[0,\infty]},H^{31,[0,\infty]};\alpha_{1,m_1},\alpha_{3,1})
\\
&\,{}_{{\rm ev}_{+}}\times_{{\rm ev}_-}{\mathcal M}_{\ell_{3,2}}(X,J_3,H^3;\alpha_{3,1},\alpha_{3,2})
\,{}_{{\rm ev}_{+}}\times_{{\rm ev}_-}  \dots \\
&\qquad\qquad\qquad\qquad\dots
\,{}_{{\rm ev}_{+}}\times_{{\rm ev}_-}
{\mathcal M}_{\ell_{3,m_3+1}}(X,J_3,H^3;\alpha_{3,m_3},\alpha_{3,m_3+1})
\endaligned
\end{equation}
by
$
\mathcal N_{\vec \ell_1,\ell',\vec \ell_3}(X,H^{31,[0,\infty]},\mathcal J^{31,[0,\infty]};\vec\alpha_{1},\vec\alpha_{3}).
$
Note this is also a component of the corner of
$\mathcal N_{\ell}(X,\mathcal J^{31,[0,\infty]},H^{31,[0,\infty]};\alpha_{-},\alpha_{+})$
if $\vert\vec \ell_1\vert + \ell' + \vert\vec \ell_2\vert = \ell$.
So we are during the process of constructing Kuranishi structures on
$\mathcal N_{\ell}(X,\mathcal J^{31,[0,\infty]},H^{31,[0,\infty]};\alpha_{-},\alpha_{+})^+$,
which is the direct product of $\mathcal N_{\ell}(X,\mathcal J^{31,[0,\infty]},H^{31,[0,\infty]};\alpha_{-},\alpha_{+})$
with $[-1,0]^*$. ($* = \#\vec \alpha_1 + \#\vec \alpha_3 - 2$.)
Let us denote by
$\widehat{\mathcal U}_{\vec \ell_1,\ell',\vec \ell_3}(X,\mathcal J^{31,[0,\infty]},H^{31,[0,\infty]};\vec\alpha_{1},\vec\alpha_{3})$
the Kuranishi structure on it.
\par
We put
$$
\ell'_2 = \ell - \vert \vec \ell_{1,C_1}\vert - \vert \vec \ell_{3,C_3}\vert.
$$
We then observe that (\ref{26113form}) lies in the corner of
\begin{equation}\label{form26117}
\mathcal N_{\vec \ell_{1,A_1,a_1+1,C_1},\ell'_2,\vec \ell_{3,A_3,1,C_3}}(X,\mathcal J^{31,[0,\infty]},H^{31,[0,\infty]};\vec\alpha_{1, A_1, a_1+1,C_1}, \vec\alpha_{3, A_3, 1,C_3}).
\end{equation}
\begin{conds}\label{conds26117223}
In the case $A_2 = B_2 = \emptyset$, we require the restriction of the
Kuranishi structure (\ref{form2610500}) to the image of
${\rm id} \times \mathcal I_{A_1,B_1,C_1} \times \mathcal I_{A_2,B_2,C_2}
\times \mathcal I_{A_3,B_3,C_3}$
is the fiber product of the following Kuranishi structures.
(We use the fiber product description (\ref{264222revrev222}).)
\begin{enumerate}
\item
The Kuranishi structure on
${\mathcal M}_{\vec\ell_{r,A_r,j}}(X,J_r,H^r;\vec\alpha_{r,A_r,j} )\times [-1,0]^{m_j(r, A_r, C_r) -1}$
which is a restriction of those on (\ref{form2677rev}) produced  in Proposition \ref{prop2661}.
\item The Kuranishi structure on
\eqref{26113form}$\times [-1,0]^*$ 
which is a restriction of the Kuranishi structure
$$
\widehat{\mathcal U}_{\vec \ell_{1,A_1,a_1+1,C_1},\ell'_2,\vec \ell_{3,A_3,1,C_3}}(X,H^{31,[0,\infty]},\mathcal J^{31,[0,\infty]};\vec\alpha_{1, A_1, a_1+1, C_1}, 
\vec\alpha_{3, A_3,1, C_3}),
$$
on \eqref{form26117}$\times [-1,0]^*$.
Here $* = \#\vec\alpha_{1, A_1, a_1+1, C_1} +
\#\vec\alpha_{3, A_3,1, C_3}-2$.
\end{enumerate}
\end{conds}
We have thus described the conditions we require at $T= \infty$.
\par
There are similar compatibility conditions at $T < \infty$ and $T=0$.
We require such conditions to the Kuranishi structure 
$\widehat{\mathcal U}_{\vec \ell_1,\ell',\vec \ell_3}(X,\mathcal J^{31,[0,\infty]},H^{31,[0,\infty]};\vec\alpha_{1},\vec\alpha_{3})$.
We omit the detailed description of this compatibility condition
since it is the same as that for the case of Condition \ref{conds2660rev} and Proposition \ref{prop2661rev}.
\begin{prop}\label{prop26120}
There exists a K-system 
$$
\aligned
\{ ( & {\mathcal N}_{\vec \ell_1,\ell_{(12)},\vec \ell_2
,\ell_{(23)},\vec \ell_3}(X,\mathcal J^{21},\mathcal J^{32},H^{21},H^{32};\vec \alpha_1,\vec \alpha_2,\vec \alpha_3)^+, \\
~& \widehat{\mathcal U}_{\vec \ell_1,\ell_{(12)},\vec \ell_2
,\ell_{(23)},\vec \ell_3}(X,\mathcal J^{21},\mathcal J^{32},H^{21},H^{32};\vec \alpha_1,\vec \alpha_2,\vec \alpha_3) )\}
\endaligned
$$
whose Kuranishi structure is as in (\ref{form2610500}) 
with the following properties.
\begin{enumerate}
\item
They satisfy Conditions \ref{conds26117} - \ref{conds26117223}.
\item There exists a K-system 
$$
\{(\mathcal N_{\vec \ell_1,\ell',\vec \ell_3}(X,\mathcal J^{31,[0,\infty]},
H^{31,[0,\infty]};\vec\alpha_{1},\vec\alpha_{3}),
~\widehat{\mathcal U}_{\vec \ell_1,\ell',\vec \ell_3}(X,\mathcal J^{31,[0,\infty]},H^{31,[0,\infty]};\vec\alpha_{1},\vec\alpha_{3}) )\}
$$
which satisfies a compatibility condition similar to Condition \ref{conds2660rev}.
\item
A similar compatibility condition is satisfied at $T=0$.
\item
Let $\frak C$ be the union of the boundary components of the underlying topological space $
{\mathcal N}_{\vec \ell_1,\ell_{(12)},\vec \ell_2
,\ell_{(23)},\vec \ell_3}(X,\mathcal J^{21},\mathcal J^{32},H^{21},H^{32};\vec \alpha_1,\vec \alpha_2,\vec \alpha_3)^+$
corresponding to $\partial([0,1]^*)$. Then (\ref{form2610500}) is $\frak C$-collared.
\item
If $\vec{\alpha}_1 = \{\alpha_-\}$ and $\vec{\alpha}_3 = \{\alpha_+\}$, then
$\widehat{\mathcal U}_{\emptyset,\ell,\emptyset}(X,\mathcal J^{31,[0,\infty]},H^{31,[0,\infty]};\alpha_{-},\alpha_{+})$
coincides with the Kuranishi structure we produced in Theorem \ref{the26112} (1).
\end{enumerate}
\end{prop}
\begin{proof}
The proof is entirely similar to the proof of Proposition \ref{prop2661rev} etc.
\end{proof}
We can use Proposition \ref{prop26120} to complete the proof of Theorem \ref{the26112}  (2) in the
same way as before.
(We use smoothing of corners and 
\cite[Lemma 18.40]{fooonewbook}
to show
that at $T=\infty$ we get the Kuranishi structure used to define the composition.)
\end{proof}
The proof of Theorem \ref{tjm26108} is complete.
\end{proof}

\section{Well-definedness of Hamiltonian Floer cohomology}
\label{sec;welldefinedness}

We now use the results of the previous sections to conclude
the well-definedness of the Floer cohomology of a periodic
Hamiltonian system.
Namely we prove the next theorems.

\begin{thm}\label{HFwelldefine}
Let $H : X \times S^1 \to \R$ be a smooth function
such that ${\rm Per}(H)$ is Morse-Bott non-degenerate in the sense of
Condition \ref{weaknondeg}.
Then we can associate the Floer cohomology
$HF(X,H;\Lambda_{0,{\rm nov}})$
which is independent of various choices involved in the definition.
\end{thm}

Recall $\Lambda_{0,{\rm nov}}$ is the Novikov ring defined in \eqref{nov0}.
We define the Novikov field $\Lambda_{\rm nov}$ as its field of fractions 
by allowing $\lambda_i$ to be negative.

\begin{thm}\label{HFwelldefine2}
Let $H^{r} : X \times S^1 \to \R$ $(r=1,2)$ be as in Theorem \ref{HFwelldefine}.
Then the  Floer cohomologies
$
HF(X,H^r;\Lambda_{{\rm nov}})
= HF(X,H^r;\Lambda_{0,{\rm nov}}) \otimes_{\Lambda_{0,{\rm nov}}} \Lambda_{{\rm nov}}
$ $(r=1,2)$
over the Novikov field $\Lambda_{{\rm nov}}$ satisfy
$$
HF(X,H^1;\Lambda_{{\rm nov}})
\cong HF(X,H^2;\Lambda_{{\rm nov}}).
$$
\end{thm}
\begin{rem}
Using Remark \ref{remenergylosss}, we can prove the
Lipschitz continuity of torsion exponent of the Floer cohomology over
$\Lambda_{0,{\rm nov}}$ with respect to the distance 
$$
d(H^1,H^2) =
\int_{t\in S^1} \sup_{x\in X}\vert H^1(t,x) - H^2(t,x)\vert dt
$$
on the set of the Hamiltonians.
We omit the discussion about it since we can
derive it from a similar result on the Floer cohomology
of Lagrangian intersection.
(See \cite {fooodisplace}.)
We can also use Remark \ref{remenergylosss} to derive more precise results
about the filtration of the Floer cohomology of a periodic
Hamiltonian system. We discussed it in detail
in \cite{fooospectr}.
\end{rem}

The proofs of Theorems \ref{HFwelldefine}, \ref{HFwelldefine2}
occupy the main part of this section.

\begin{shitu}\label{situ16127}
\begin{enumerate}
\item
Let $H : X \times S^1 \to \R$ be a smooth function
such that ${\rm Per}(H)$ is Morse-Bott non-degenerate.
We write $H = H^1$.
We define $H^{11} :  X \times \R\times S^1 \to \R$
by $H^{11}(x,\tau,t) = H^{1}(x,t)$.
\item
Let $J$ and $J'$ be two choices of tame almost complex structures.
We define $\mathcal J = \{J_{\tau}\}$ such that
$J_{\tau} = J$ for $\tau < -1$ and
$J_{\tau} = J'$ for $\tau > 1$. $\blacksquare$
\end{enumerate}
\end{shitu}

\begin{cons}\label{const16128}
Suppose we are in Situation \ref{situ16127}.
\begin{enumerate}
\item
We use $H$ and $J$ and apply Theorem \ref{theorem266} (1)
to obtain a linear K-system $\mathcal F_X (H,J)$
whose space of connecting orbits is
${\mathcal M}((X,J),H;\alpha_-,\alpha_+)$.
(We made choices to define it.)
\item
We then apply \cite[Theorem 16.9]{fooonewbook}
to obtain a chain complex
(we made choices to define it) and then use
\cite[Definition 16.12]{fooonewbook}
to obtain
a cochain complex over the universal Novikov ring
$\Lambda_{0,{\rm nov}}$,
which we denote by 
$
CF((X,J),H;\Lambda_{0,{\rm nov}}).
$
By \cite[Theorem 16.9 (2)]{fooonewbook}
the cohomology of this cochain complex
is independent of the choices made in
Item (2).
We denote this cohomology group as
$$
HF((X,J),H;\Lambda_{0,{\rm nov}}).
$$
\item
We start from $H$ and $J'$ and
make choices in Item (2) to obtain
$$
HF((X,J'),H;\Lambda_{0,{\rm nov}}).
$$
\end{enumerate}
\end{cons}
However, \cite[Theorem 16.9 (2)]{fooonewbook}
does not imply that
$
HF((X,J),H;\Lambda_{0,{\rm nov}})
$
is independent of the choices made in Item (1).
The statement of Theorem \ref{HFwelldefine}
is independence of the Floer cohomology (over $\Lambda_{0,{\rm nov}}$) of the choices made in Item (2) as well as
the almost complex structure $J$.
The statement of Theorem \ref{HFwelldefine2}
is independence of the Floer cohomology (over $\Lambda_{{\rm nov}}$) of the choices made in Item (1) as well as Item (2).

\begin{cons}\label{const1612822}
Suppose we are in Situation \ref{situ16127}.
We also assume that we have made all the choices
involved in Construction \ref{const16128}
(1)(2)(3).
\par
We apply Theorem \ref{theorem266rev} to
obtain a morphism from $\mathcal F_X(H,J)$
to $\mathcal F_X(H,J')$.
Here $\mathcal F_X(H,J)$ is
defined in Situation \ref{situ16127} (1) and
$\mathcal F_X(H,J')$ is defined in
Situation \ref{situ16127} (3).
We denote this morphism by $\frak N_{11}(\mathcal J,H^{11})$.
\par
We make choices to define $\frak N_{11}(\mathcal J,H^{11})$.
The interpolation space of $\frak N_{11}(\mathcal J,H^{11})$ is
${\mathcal N}(X,\mathcal J,H^{11};\alpha_-,\alpha_+)^{\boxplus 1}$.
\end{cons}
\begin{lem}\label{lem26130}
We can make the choice in Construction \ref{const1612822}
so that $\frak N_{11}(\mathcal J,H^{11})$ is a morphism of
energy loss $0$.
Namely we have
\begin{enumerate}
\item
${\mathcal N}(X,\mathcal J,H^{11};\alpha_-,\alpha_+)^{\boxplus 1}
= \emptyset$ if $E(\alpha_-) > E(\alpha_+)$
or $E(\alpha_-) = E(\alpha_+)$, $\alpha_- \ne \alpha_+$.
\item
${\mathcal N}(X,\mathcal J,H^{11};\alpha,\alpha)^{\boxplus 1}
= R_{\alpha}$. The evaluation maps on it are the identity maps.
\end{enumerate}
\end{lem}
\begin{proof}
Using the fact that $H^{11}_{\tau,t} = H^1_{t}$ and is $\tau$
independent, (1) is an immediate consequence of Remark \ref{remenergylosss}.
To prove (2) we first observe that
${\mathcal N}(X,\mathcal J,H^{11};\alpha,\alpha) = R_{\alpha}$
set-theoretically.
In fact,
${\mathcal N}(X,\mathcal J,H^{11};\alpha,\alpha)$ consists of
$((\Sigma,\vec z_{\pm}),\varphi,u)$
where $\Sigma = S^2$ (without bubbles)
and $u(\varphi(\tau,t)) = \gamma(t)$ with $\gamma \in R_{\alpha}$.
We also remark that this moduli space is Fredholm regular.
Therefore we can make our choice so that the obstruction bundle is $0$
in this particular case.
(Since $\partial{\mathcal N}(X,\mathcal J,H^{11};\alpha,\alpha) = \emptyset$,
we do not need to study the compatibility with the choices
made in Construction \ref{const16128}
(1)(3).)
Item (2) holds for this choice.
\end{proof}
\begin{proof}[Proof of Theorem \ref{HFwelldefine}]
We now apply \cite[Theorem 16.31]{fooonewbook}
to the morphism obtained by taking the choice made in
Lemma \ref{lem26130}.
We then obtain a chain map
\begin{equation}\label{form26118}
\psi_{\mathcal J,H^{11}} :
CF((X,J),H;\Lambda_{0,{\rm nov}})
\to CF((X,J'),H;\Lambda_{0,{\rm nov}}).
\end{equation}
We note that
$$
CF((X,J),H;\Lambda_{0,{\rm nov}})
= CF((X,J'),H;\Lambda_{0,{\rm nov}})
$$
as $\Lambda_{0,{\rm nov}}$-modules.
(Floer's boundary operator may be different however.)
\par
Using the fact that energy loss of
$\frak N_{11}(\mathcal J,H^{11})$ is zero,
especially Lemma \ref{lem26130} (2), we find that
$$
\psi_{\mathcal J,H^{11}}  \equiv {\rm id}
\mod T^{\epsilon}\Lambda_{0,{\rm nov}}
$$
for some $\epsilon >0$.
Therefore $\psi_{\mathcal J,H^{11}}$
has an inverse, which automatically becomes a chain map.
Therefore we have
$$
HF((X,J),H;\Lambda_{0,{\rm nov}})
\cong
HF((X,J'),H;\Lambda_{0,{\rm nov}}),
$$
as required.
\end{proof}
\begin{proof}[Proof of Theorem \ref{HFwelldefine2}]
\begin{shitu}\label{situ16127000}
\begin{enumerate}
\item
Let $H^r : X \times S^1 \to \R$ $(r=1,2)$ be smooth functions
such that ${\rm Per}(H^r)$ are Morse-Bott non-degenerate.
Let $J^r$ $(r=1,2)$ be tame almost complex structures.
\item
For $r=1,2$ we make choices as in
Construction \ref{const16128} to define chain complexes
$
CF((X,J^r),H^r;\Lambda_{0,{\rm nov}})
$
with $\Lambda_{0,{\rm nov}}$ coefficients,
and their cohomology groups
$
HF((X,J^r),H^r;\Lambda_{0,{\rm nov}}).
$
\item
We take $H^{21} :  X \times \R\times S^1 \to \R$
and $\mathcal J^{21}$ as in Situation \ref{situ2676}.
We exchange the role of $H^1$, $J^1$ and $H^2$, $J^2$
and take  $H^{12}$, $\mathcal J^{12}$. $\blacksquare$
\end{enumerate}
\end{shitu}
We now apply Theorem \ref{theorem266rev}
to $H^{21}, \mathcal J^{21}$ and obtain a
morphism
$$
\frak N_{21} : \mathcal F_X(H^1,J^1)
\to \mathcal F_X(H^2,J^2).
$$
We also apply
Theorem \ref{theorem266rev}
to $H^{12}, \mathcal J^{12}$ and obtain a
morphism
$$
\frak N_{12} : \mathcal F_X(H^2,J^2)
\to \mathcal F_X(H^1,J^1).
$$
We apply \cite[Theorem 16.31 (1)]{fooonewbook}
to obtain
chain maps
$$
\psi_{12} :
CF((X,J^1),H^1;\Lambda_{{\rm nov}})
\to CF((X,J^2),H^2;\Lambda_{{\rm nov}})
$$
and
$$
\psi_{21} :
CF((X,J^2),H^2;\Lambda_{{\rm nov}})
\to
CF((X,J^1),H^1;\Lambda_{{\rm nov}}).
$$

We consider the particular case of $J = J' = J^1$ in
Situation \ref{situ16127000}.
Then we obtain
$$
\frak N_{11} : \mathcal F_X(H^1,J^1)
\to \mathcal F_X(H^1,J^1).
$$
In a similar way we obtain
$$
\frak N_{22} : \mathcal F_X(H^2,J^2)
\to \mathcal F_X(H^2,J^2).
$$
We denote the chain map (\ref{form26118}) in this case
by
\begin{equation}\label{map000FHFH}
\psi_{0,rr} : CF((X,J^r),H^r;\Lambda_{0,{\rm nov}})
\to
CF((X,J^r),H^r;\Lambda_{0,{\rm nov}}).
\end{equation}
This is a chain isomorphism in the proof of Theorem \ref{HFwelldefine} by Lemma \ref{lem26130}.
We change the coefficient ring to
$\Lambda_{{\rm nov}}$ by taking tensor product and obtain
\begin{equation}\label{map000FHFH2}
\psi_{rr} : CF((X,J^r),H^r;\Lambda_{{\rm nov}})
\to
CF((X,J^r),H^r;\Lambda_{{\rm nov}}).
\end{equation}
This is also a chain isomorphism.
\begin{lem}\label{lem26132}
The composition $\frak N_{12} \circ \frak N_{21}$
(resp. $\frak N_{21} \circ \frak N_{12}$)
is homotopic to $\frak N_{11}$ (resp. $\frak N_{22}$).
\end{lem}
\begin{proof}
This is a special case of Theorem \ref{tjm26108}.
\end{proof}
Lemma \ref{lem26132} and \cite[Theorem 16.31 (3)]{fooonewbook}
imply that
$\psi_{21} \circ \psi_{12}$ (resp. $\psi_{12} \circ \psi_{21}$)
is chain homotopic to $\psi_{22}$ (resp. $\psi_{11}$).
Since $\psi_{22}$, $\psi_{11}$ are chain homotopy equivalences,
it implies that $\psi_{21}$ is a chain homotopy equivalence.
Therefore
\begin{equation}\label{isomorphi26119}
\psi_{21 *} : HF((X,J^1),H^1;\Lambda_{{\rm nov}})
\cong
HF((X,J^2),H^2;\Lambda_{{\rm nov}})
\end{equation}
as required.
\end{proof}
We will discuss the well-definedness of the isomorphisms
in Theorems \ref{HFwelldefine} and \ref{HFwelldefine2} of various
choices more precisely below.
\par
In Situation \ref{situ16127000} we obtain a chain map
$\psi_{21}$ (of $\Lambda_{\rm nov}$ coefficients) which induces
an isomorphism of the Floer cohomology (\ref{isomorphi26119}).
\par
In Situation \ref{situ16127} we obtain a chain map
between two chain complexes of $\Lambda_{0,{\rm nov}}$ coefficients, which we denote by
\begin{equation}\label{form26120}
\psi_{0;J'J} : CF((X,J),H;\Lambda_{0,{\rm nov}}) \to CF((X,J'),H;\Lambda_{0,{\rm nov}}).
\end{equation}
This is nothing but the chain map 
(\ref{form26118}).\footnote{The domain and the target of the map (\ref{form26120})
are different not only because we use different almost
complex structures but also we made various choices to define
the linear K-system and
various choices to define a map between spaces of differential forms via smooth correspondence by
K-spaces.}
It induces an isomorphism $\psi_{0;J'J *}$ on cohomology groups.
\par
Finally when we use the same almost complex structure $J$ and the other choices
for the domain and the target, we obtain (\ref{map000FHFH}) and (\ref{map000FHFH2}).
\begin{lem}
The maps $\psi_{21 *}$, $\psi_{0;J'J *}$ and $\psi_{rr *}$ are independent
of various choices involved in the construction.
\end{lem}
\begin{proof}
We consider the two choices to define the morphism
$\frak N_{21}$. 
We denote by $\frak N^a_{21}$ and $\frak N^b_{21}$ the morphism obtained by those two
choices.
By Lemma \ref{lem2698} and Theorem \ref{thmPparaexist}, 
two morphisms
$\frak N^a_{21}$ and $\frak N^b_{21}$ are homotopic each other.
The isomorphism (\ref{isomorphi26119})
is induced from $\frak N_{21}$ by using
\cite[Theorem 16.31 (1)]{fooonewbook}.
We can use 
\cite[Theorem 16.31 (2)]{fooonewbook}
that the homomorphism $\psi_{12}^a$ induced from
$\frak N^a_{21}$ is
chain homotopic to the homomorphism $\psi_{12}^b$ induced from $\frak N^b_{21}$.
This implies the independence of $\psi_{21 *}$ of the choices.
\par
The independence of $\psi_{0;J'J *}$ and $\psi_{rr *}$ are proved in the same way.
We only need to note that in that situation we can take the homotopy to be of energy loss zero.
\end{proof}
This lemma together with the next lemma imply that
the group
$HF(X,H;\Lambda_{{\rm nov}})$ is independent of $H$, $J$
and other choices up to {\it canonical} isomorphism.
\begin{lem}\label{lem26134}
\begin{enumerate}
\item
We consider $H^r$, $J^r$ for $r=1,2,3$ and
isomorphisms  $\psi_{21 *}$, $\psi_{32 *}$, $\psi_{31 *}$
as above. Then
$$
\psi_{31 *} = \psi_{32 *}\circ \psi_{21 *}.
$$
\item
We fix $H$ and take three choices $J_1$, $J_2$ and $J_3$ of almost complex
structures as well as other choices in
Construction \ref{const16128}. Then we have
$$
\psi_{0,J_3J_1 *} = \psi_{0,J_3J_2 *}\circ \psi_{0,J_2J_1 *}.
$$
\item
$\psi_{0,rr}$ induces the identity map
$$
\psi_{0,rr *} : HF((X,J^r),H^r;\Lambda_{0,{\rm nov}})
\to
HF((X,J^r),H^r;\Lambda_{0,{\rm nov}}).
$$
\end{enumerate}
\end{lem}
\begin{proof}
(1) is a consequence of Theorem \ref{tjm26108}
and \cite[Theorem 16.31 (3)]{fooonewbook}.
The proof of (2) is similar. (We again note that in this situation we obtain
a homotopy with energy loss $0$.)
\par
(3) We first observe that $\frak N_{rr}\circ \frak N_{rr}$
is homotopic to $\frak N_{rr}$. This is a consequence of Theorem \ref{tjm26108}.
Furthermore by its proof we can show that the energy loss
of the homotopy between them is $0$.
Therefore by \cite[Theorem 16.31 (2)]{fooonewbook}
we find that $\psi_{0,rr *} \circ \psi_{0,rr *} = \psi_{0,rr *}$.
Since $\psi_{0,rr *}$ is an isomorphism, this implies
that $\psi_{0, rr *}$ is the identity map.
\end{proof}
\begin{rem}\label{rem138}
We did not use the identity morphism in this section.
(In several results such as \cite[Theorem 16.9 (2)]{fooonewbook}, 
we used the identity morphism in their proofs.)
We can actually prove the following.
See also \cite[Section 18.11]{fooonewbook}.
\begin{clm}\label{clain26136}
In the case $\mathcal J$ and $H^{11}$ are $\tau $ independent families,
the morphism $\frak N_{11}(\mathcal J,H^{11})$
is homotopic to the identity morphism.
\end{clm}
This immediately implies Lemma \ref{lem26134} (3) for example.
\par
The strata of $\frak N_{11}(\mathcal J,H^{11})$
is actually the same as the ones
appearing in the definition of the identity morphism.
Let us explain this fact  below.
Recall that we used $\tau \in \R$ independent $H^{11}$
and $J$ to
define our morphism $\frak N_{11}(\mathcal J,H^{11})$.
Therefore an element
$((\Sigma,\vec z_{\pm}),u)$ of the interpolation
space ${\mathcal N}(X,\mathcal J,H^{11};\alpha_-,\alpha_+)$
is the same as an element of
${\mathcal M}(X,J,H;\alpha_-,\alpha_+)$,
except we add the data to specify the main component and
fix a parametrization of the main component.
(Namely the isomorphism between two elements
is required to commutes strictly with the parametrization
of the main component.)
This causes two points where  ${\mathcal N}(X,\mathcal J,H^{11};\alpha_-,\alpha_+)$
is different from ${\mathcal M}(X,J,H;\alpha_-,\alpha_+)$.
\begin{enumerate}
\item
In the case the main component
represents an element of
$\mathcal N^{\rm reg}(X,H;\alpha,\alpha')$
$\alpha \ne \alpha'$,
it has an extra parameter $\in \R$
other than those in $\mathcal M^{\rm reg}(X,H;\alpha,\alpha')$ which specify the
parametrization $\varphi_{a_0}$ of the main component $\Sigma_{a_0}$.
\item
There is a case when the main component
corresponds to an element of $\mathcal N^{\rm reg}(X,H;\alpha,\alpha)$.
\end{enumerate}
In the case (1) the moduli parameter of the main component
is $\mathcal N^{\rm reg}(X,H;\alpha,\alpha')$ which
is isomorphic to
$\mathcal M^{\rm reg}(X,H;\alpha,\alpha') \times \R$.
Thus the strata \cite[Definition 18.55 (1)(a)]{fooonewbook}
appears.
\par
In the case (2) we have $R_{\alpha}$ as the parameter space of the main
component.
In this case $\mathcal M^{\rm reg}(X,H;\alpha,\alpha)$ is an empty set.
The map which is constant in the $\R$ direction corresponds to an
element of  $R_{\alpha}$.
Thus the strata 
\cite[Definition 18.55 (1)(b)]{fooonewbook}
appears.
\par
Thus we find a  K-space $\mathcal N(X,H;\alpha,\alpha')$ which is similar to the interpolation space
of the identity morphism.
We remark however that to prove Claim \ref{clain26136},
we need to show not only the underlying topological
space but also their Kuranishi structures coincide.
In other words, the Kuranishi structure on
${\mathcal N}(X,\mathcal J,H^{11};\alpha_-,\alpha_+)$
should be induced by the forgetful map
$$
{\mathcal N}(X,\mathcal J,H^{11};\alpha_-,\alpha_+)
\to \mathcal M(X,J,H;\alpha_-,\alpha_+).
$$
It is possible to find such a Kuranishi structure on ${\mathcal N}(X,\mathcal J,H^{11};\alpha_-,\alpha_+)$.
However, since we postpone the thorough detail of the
discussion of the forgetful map to \cite{foootech3},
we do not prove Claim \ref{clain26136} here.
For this reason, we organize the proof in this section in a slightly different way.
\end{rem}

\section{Calculation of Hamiltonian Floer cohomology}
\label{sec;calc}
\begin{defn}\label{defn23111}
Let $(X,\omega)$ be a compact symplectic manifold.
We define the {\it trivial linear K-system of $X$}
as follows. 
Here we use the same item numbers and the notation in 
\cite[Condition 16.1]{fooonewbook}.
\begin{enumerate}
\item[(I)]
We define an additive group $\frak G = \pi_2(X)/\sim$,
where $\alpha \sim \alpha'$ if and only if $\omega[\alpha] = \omega[\alpha']$
and $c_1(TX)[\alpha] = c_1(TX)[\alpha']$.
Group homomorphisms $E : \frak G \to \R$ and $\mu : \frak G \to \Z$
are induced by $[\alpha] \mapsto \omega[\alpha]$ and $\mu([\alpha]) = 2c_1(TX)[\alpha]$
respectively.
\item[(II)]
As a set $\frak A = \frak G$ with left multiplication as the $\frak G$ action on $\frak A$ itself.
The maps $E : \frak A \to \R$ and $\mu : \frak A \to \Z$ are as above.
\item[(III)]
For each $\alpha \in \frak A$, $R_{\alpha} = X$.
\item[(IV)]
$\mathcal M(\alpha_-,\alpha_+) = \emptyset$
always.
\item[(VII)]
$o_{R_{\alpha}}$ is the canonical orientation of the symplectic manifold $X$.
The orientation isomorphism ${\rm OI}_{\alpha_-, \alpha_+}$ in \cite[(16.2)]{fooonewbook}
is trivial.
\end{enumerate}
Other items in 
\cite[Condition 16.1]{fooonewbook} 
are satisfied in a trivial way.
We denote the trivial linear K-system of $X$
by $\mathcal F_X^{\rm tri}$.
\end{defn}
The main result of this section is the following.
\begin{thm}\label{therem232}
Suppose we are in Situation \ref{situ16127}.
Let $\mathcal F_X(H,J)$ be as in Construction \ref{const16128} (1).
Then there exist morphisms of linear K-systems 
$\frak N_{* (H,J)} : \mathcal F_X(H,J) \to \mathcal F_X^{\rm tri}$
and $\frak N_{(H,J) *} : \mathcal F_X^{\rm tri} \to \mathcal F_X(H,J)$
with the following properties.
\begin{enumerate}
\item
The composition $\frak N_{* (H,J)}\circ \frak N_{(H,J) *} :
\mathcal F_X^{\rm tri} \to \mathcal F_X^{\rm tri}$ is homotopic to
a morphism of energy loss 0.\footnote{Namely it satisfies Lemma \ref{lem26130} (1)(2).}
\item
The composition
$\frak N_{(H,J) *}\circ  \frak N_{* (H,J)} :
\mathcal F_X(H,J) \to \mathcal F_X(H,J)$ is
homotopic to the
morphism $\frak N_{11}(\mathcal J,H^{11})$ in
Construction \ref{const1612822}, where $\mathcal J$ is the
trivial family and $H^{11} \equiv H$ is constant in the $\R$ factor.
\end{enumerate}
\end{thm}

\begin{cor}\label{cor2333333}
$$
HF((X,H);\Lambda_{\rm nov})
\cong
H(X;\Lambda_{\rm nov}).
$$
\end{cor}
The corollary is an immediate consequence of
Theorem \ref{therem232} , 
\cite[Theorem 16.39 (4)(5)]{fooonewbook},
\cite[Lemma 19.45]{fooonewbook}
and an obvious fact that the Floer cohomology of the trivial linear K-system is
$H(X;\Lambda_{\rm nov})$.
We note that Corollary \ref{cor2333333} implies
(\ref{statementarnold}).

\begin{proof}[Proof of Theorem \ref{therem232}]
We consider a smooth function
$H^{*1} : X \times \R \times S^1  \to \R$ such that:
\begin{enumerate}
\item
If $\tau < -1$ then
$H^{*1}(x,\tau,t) = H(x,t)$.
\item
If $\tau > 1$ then
$H^{*1}(x,\tau,t) = 0$.
\end{enumerate}
Since we already proved $J$ independence of the Floer cohomology,
we fix $J$ in this section and do not include it in the notation.
\par
Let $\alpha_- \in \frak A$ where $\frak A$ is the index set of the space of
contractible periodic orbits of $\mathcal F_X(H,J)$.
Let $\alpha_+ \in \frak G = \pi_2(X)/\sim$ as in Definition \ref{defn23111}.
\begin{defn}\label{defn210revto0}
The set $\widehat{\mathcal N}'_{\ell}(X,H^{*1};\alpha_-,\alpha_+)$
consists of triples 
$$
((\Sigma,(z_-,z_+,\vec z),a_0),u,\varphi)
$$ 
satisfying the following conditions:
\begin{enumerate}
\item
$(\Sigma,(z_-,z_+,\vec z))$ is a genus zero semi-stable curve with $\ell + 2$ marked points.
\item
$\Sigma_{a_0}$ is one of the mainstream components.
We call it the {\it main component}.
\item
$\varphi = (\varphi_a)$ where
$\varphi_a : \R \times S^1 \to \Sigma_a \setminus \{z_{a,-},z_{a +}\}$ is 
a parametrization of
mainstream component $\Sigma_a$ with $a \le a_0$
and  
$\varphi_a$ is a biholomorphic map such that
$$
\lim_{\tau\to \pm} \varphi_{a}(\tau,t) = z_{a,\pm}.
$$
\item
For each extended mainstream component $\widehat{\Sigma}_a$, the map
$u$ induces
$u_a : \widehat{\Sigma}_a \setminus\{z_{a,-},z_{a,+}\} \to X$
which is a continuous map
\item
If $\Sigma_a$ is a mainstream component with $a \le a_0$ and
$\varphi_a : \R \times S^1 \to \Sigma_a$ is as in (3), then
the composition $u_a \circ \varphi_a$ satisfies the equation
\begin{equation}\label{Fleq277621}
\frac{\partial (u_a \circ \varphi_a)}{\partial \tau} +  J \left( \frac{\partial (u_a \circ \varphi_a)}{\partial t} - \frak X_{H^a_{\tau,t}}
\circ (u_a \circ \varphi_a) \right) = 0
\end{equation}
where
$$
H^a_{\tau,t} =
\begin{cases}
H^1_t  &\text{if $a < a_0$}, \\
H^{*1}_{\tau,t}  &\text{if $a = a_0$}.
\end{cases}
$$
\item
$$
\int_{\R \times S^1}
\left\Vert\frac{\partial (u \circ \varphi_a)}{\partial \tau}\right\Vert^2 d\tau dt < \infty.
$$
\item
If $\Sigma_{\rm v}$ is a bubble component or
a mainstream component $\Sigma_a$ \ with $a > a_0$,
then $u$ is pseudo-holomorphic on $\Sigma_{\rm v}$.
\item
Let
${\Sigma}_{a_1}$ and ${\Sigma}_{a_2}$ be
mainstream components and
  $z_{a_1,+} = z_{a_2,-}$. Then
$$
\lim_{\tau\to+\infty} (u_{a_1} \circ \varphi_{a_1})(\tau,t)
=
\lim_{\tau\to-\infty} (u_{a_2} \circ \varphi_{a_2})(\tau,t)
$$
holds for each $t \in S^1$ if $a_2 \le a_0$.
((6) and Lemma \ref{prof26rev3} imply that
the left and right hand sides both converge.)
\par
If $a_2 > a_0$ then we require that $u$ is continuous at $z_{a_1,+} = z_{a_2,-}$.
\item
If
${\Sigma}_{a}$
is a
mainstream component and $z_{a,-} = z_-$, then there exists
$(\gamma_{-},w_{-})
\in R_{\alpha_{-}}$ such that
$$
\lim_{\tau\to-\infty} (u_{a} \circ \varphi_{a})(\tau,t)
= \gamma_-(t).
$$
Moreover
$
[u_*[\Sigma]] \# w_-
$
represents the class $[\alpha_+]$,
where $\#$ is the obvious concatenation.
\item
We assume that $((\Sigma,(z_-,z_+,\vec z),a_0),u,\varphi)$
is stable in the sense of Definition \ref{defn2615revrr} below.
\end{enumerate}
\end{defn}
Assume that $((\Sigma,(z_-,z_+,\vec z),a_0),u,\varphi)$
satisfies (1)-(9) above.
The {\it extended automorphism group}
${\rm Aut}^+((\Sigma,(z_-,z_+,\vec z),a_0),u,\varphi)$
of $((\Sigma,(z_-,z_+,\vec z),a_0)),u,\varphi)$
consists of map $v : \Sigma \to \Sigma$ such that
it satisfies (1)(2)(5) of Definition \ref{defn2615rev},
and (3) of Definition \ref{defn2615rev} for $\varphi_a$ with
$a \le a_0$, and $\tau_{a_0} = 0$.

\begin{defn}\label{defn2615revrr} An object
$((\Sigma,(z_-,z_+,\vec z),a_0),u,\varphi)$
satisfies (1)-(9) above is said to be {\it stable}
if ${\rm Aut}^+((\Sigma,(z_-,z_+,\vec z),a_0),u,\varphi)$
is a finite group.
\end{defn}
We can define the equivalence relation $\sim_2$ on
$\widehat{\mathcal N}'_{\ell}(X,H^{*1};\alpha_-,\alpha_+)$
in the same way as in Definition \ref{3equivrel}
except we require $\tau_{a_0} = 0$ and require (3)
only for $a \le a_0$.
We put
\begin{equation}\label{formula26927777}
{\mathcal N}'_{\ell}(X,H^{*1};\alpha_-,\alpha_+) =
\widehat{\mathcal N}'_{\ell}(X,H^{*1};\alpha_-,\alpha_+)/\sim_2.
\end{equation}
This space ${\mathcal N}'_{\ell}(X,H^{*1};\alpha_-,\alpha_+)$
(more precisely, ${\mathcal N}'_{\ell}(X,H^{*1};\alpha_-,\alpha_+)^{\boxplus 1}$)
will be the underlying topological space of the interpolation
space of the morphism
$\frak N_{* (H,J)}$.
\begin{rem}\label{rem:106}
We use the $\R \times S^1$ parametrized family of Hamiltonians
$H^{*1}$ in exactly the same way as in Definition \ref{defn210rev}
etc. and obtain the space
${\mathcal N}_{\ell}(X,\mathcal J,H^{*1};\alpha_-,\alpha_+)$ as in Definition \ref{3equivrerevl}.
(Here $\mathcal J$ is the family $J_{\tau,t} = J_t$ of almost
complex structures.)
\par
The main difference between
${\mathcal N}_{\ell}(X,\mathcal J,H^{*1};\alpha_-,\alpha_+)$ and
${\mathcal N}'_{\ell}(X,H^{*1};\alpha_-,\alpha_+)$ is that
in the latter we do not put parametrizations on the mainstream components $\Sigma_a$
with $a > a_0$.
Note because of the equivalence relation $\sim_2$ the
parametrizations of the mainstream components $\Sigma_a$
(for $a\ne a_0$) is a part of the data of elements of
${\mathcal N}_{\ell}(X,\mathcal J,H^{*1};\alpha_-,\alpha_+)$ up
to the translation of the $\R$ direction.
So for a given element of ${\mathcal N}_{\ell}(X,\mathcal J,H^{*1};\alpha_-,\alpha_+)$
which has exactly $m$ mainstream components $\Sigma_a$
with $a > a_0$, the corresponding element in
${\mathcal N}'_{\ell}(X,H^{*1};\alpha_-,\alpha_+)$ is parametrized by
$(S^1)^{m}$. In other words, there exists a
map
$$
\pi :
{\mathcal N}_{\ell}(X,\mathcal J,H^{*1};\alpha_-,\alpha_+)
\to {\mathcal N}'_{\ell}(X,H^{*1};\alpha_-,\alpha_+)
$$
whose fibers are $(S^1)^{m}$. (Note $\pi$ is not a fiber
bundle and the dimension of the fiber $m$ depends on the
strata.)
In this article, however, we do not define 
an equivariant Kuranishi structure of
${\mathcal N}_{\ell}(X,\mathcal J,H^{*1};\alpha_-,\alpha_+)$
by the strata-wise $(S^1)^{m}$ action, 
but can use the Kuranishi structure on the `quotient space'
of this action, which is nothing but
${\mathcal N}'_{\ell}(X,H^{*1};\alpha_-,\alpha_+)$.
This is the reason we do not need to study an $S^1$ equivariant Kuranishi
structure in this article.
\end{rem}
We will discuss Kuranishi structure on ${\mathcal N}'_{\ell}(X,H^{*1};\alpha_-,\alpha_+)$.
Before doing so we define another moduli space which will be the
interpolation space of $\frak N_{(H,J)*}$.
\par
We put
$$
H^{1*}(x,\tau,t) = H^{*1}(x,-\tau,t).
$$
Let $\alpha_+ \in \frak A$ where $\frak A$ is the index set of
contractible periodic orbits of $\mathcal F_X(H,J)$.
Let $\alpha_- \in \frak G = \pi_2(X)/\sim$ as before.
\begin{defn}\label{defn210revto0rev}
The set 
$\widehat{\mathcal N}'_{\ell}(X,H^{1*};\alpha_-,\alpha_+)$
consists of triples $((\Sigma,(z_-,z_+,\vec z),a_0),u,\varphi)$ satisfying
the following conditions:
\begin{enumerate}
\item
The same as Definition \ref{defn210revto0} (1).
\item
The same as Definition \ref{defn210revto0} (2).
\item
The same as Definition \ref{defn210revto0} (3),
except we replace $a \le a_0$ by $a \ge a_0$.
\item
The same as Definition \ref{defn210revto0} (4).
\item
The same  as Definition \ref{defn210revto0} (5),
except we replace $a \le a_0$ by $a \ge a_0$ and we put
$$
H^a_{\tau,t} =
\begin{cases}
H^1_t  &\text{if $a > a_0$}, \\
H^{1*}_{\tau,t}  &\text{if $a = a_0$}.
\end{cases}
$$
\item
The same as Definition \ref{defn210revto0} (6).
\item
The same as Definition \ref{defn210revto0} (7),
except we replace $a > a_0$ by $a< a_0$.
\item
Let
${\Sigma}_{a_1}$ and ${\Sigma}_{a_2}$ be
mainstream components. 
If  $z_{a_1,+} = z_{a_2,-}$, then
$$
\lim_{\tau\to+\infty} (u_{a_1} \circ \varphi_{a_1})(\tau,t)
=
\lim_{\tau\to-\infty} (u_{a_2} \circ \varphi_{a_2})(\tau,t)
$$
holds for each $t \in S^1$ if $a_1 \ge a_0$.
((6) and Lemma \ref{prof26rev3} imply that the 
left and right hand sides both converge.)
\par
If $a_1 < a_0$, then we require that $u$ is continuous at $z_{a_1,+} = z_{a_2,-}$.
\item
If
${\Sigma}_{a}$
is
mainstream components and $z_{a,+} = z_+$, then there exist
$(\gamma_{+},w_{+})
\in R_{\alpha_{+}}$ such that
$$
\lim_{\tau\to +\infty} (u_{a} \circ \varphi_{a})(\tau,t)
= \gamma_+(t).
$$
Moreover
$
[u_*[\Sigma]] \# [\alpha_-]\cong  [w_+]
$
where $\#$ is the obvious concatenation.
\item
We assume that $((\Sigma,(z_-,z_+,\vec z),a_0),u,\varphi)$
is stable, which can be defined in the same way as in Definition \ref{defn2615revrr} above.
\end{enumerate}
\end{defn}
We can define an equivalence relation $\sim_2$ in the same way and
define
\begin{equation}\label{formularevdddddd}
{\mathcal N}'_{\ell}(X,H^{1*};\alpha_-,\alpha_+) =
\widehat{\mathcal N}'_{\ell}(X,H^{1*};\alpha_-,\alpha_+)/\sim_2.
\end{equation}
\par
When $X$ is a point, 
we denote
${\mathcal N}'_{\ell}(X,H^{*1};\alpha_-,\alpha_+) $
(resp. ${\mathcal N}'_{\ell}(X,H^{1*};\alpha_-,\alpha_+) $)
in this case 
by
${\mathcal N}_{\ell}'({\rm source},{\rm right})$
(resp. ${\mathcal N}'_{\ell}({\rm source},{\rm left})$).

\begin{exm}
The space ${\mathcal N}_{1}'({\rm source},{\rm right})$ is
homeomorphic to the disc $D^2$.
In fact, ${\mathcal N}_{1}'({\rm source},{\rm right})$
consists of three strata.
One is the case when there is one mainstream component.
This stratum is homeomorphic to $\R \times S^1$.
The second is the case when there are two mainstream
components $\Sigma_a$, $\Sigma_{a_0}$ with $a < a_0$.
(Here $\Sigma_{a_0}$ is the main component.)
The marked point is on $\Sigma_a$. If it is $\varphi_a(\tau,t)$,
then $t$ is the well-defined parameter of this stratum,
which is homeomorphic to $S^1$.
The third is the case when there are two mainstream
components $\Sigma_a$, $\Sigma_{a_0}$ with $a > a_0$.
The marked point is on $\Sigma_a$. Since the parametrization of
$\Sigma_a$ is not a part of the data of an element of
 ${\mathcal N}_{1}'({\rm source},{\rm right})$, this stratum is one point.
\par
We remark that ${\mathcal N}_{1}({\rm source})$
is homeomorphic to $S^1 \times [0,1]$. We shrink one of the
boundary components by the $S^1$ action
to obtain ${\mathcal N}_{1}'({\rm source},{\rm right})$.
\end{exm}
The next proposition proves a part of Theorem \ref{therem232}.
\begin{prop}\label{prop23999}
\begin{enumerate}
\item
We can define topologies on ${\mathcal N}'_{\ell}(X,H^{*1};\alpha_-,\alpha_+)$
and ${\mathcal N}'_{\ell}(X,H^{1*};\alpha_-,\alpha_+)$
so that they are compact and Hausdorff.
\item
We can define Kuranishi structures on them.
\item
We can define Kuranishi structures on the spaces\footnote{This is defined from 
${\mathcal N}'_{\ell}(X,H^{*1};\alpha_-,\alpha_+)$
in the same way as before by including 
$t_z \in [-1,0]$ for any transit point $z$. See also Remark \ref{defn1531revrev}.}
$\mathcal N'_{\ell}(X,H^{*1};\alpha_-,\alpha_+)^{\boxplus 1}$
and  ${\mathcal N}'_{\ell}(X,H^{1*};\alpha_-,\alpha_+)^{\boxplus 1}$.
\item
Together with other objects
${\mathcal N}'_{\ell}(X,H^{*1};\alpha_-,\alpha_+)^{\boxplus 1}$
defines a morphism 
$\frak N_{* (H,J)} : \mathcal F_X(H,J) \to \mathcal F_X^{\rm tri}$,
of which it will be an interpolation space.
\item Together with other objects
${\mathcal N}'_{\ell}(X,H^{1*};\alpha_-,\alpha_+)^{\boxplus 1}$
defines a morphism
$\frak N_{(H,J) *} : \mathcal F_X^{\rm tri} \to \mathcal F_X(H,J)$,
of which it will be an interpolation space.
\end{enumerate}
\end{prop}
\begin{proof}
We prove the case of ${\mathcal N}'_{\ell}(X,H^{*1};\alpha_-,\alpha_+)$ since 
the case of ${\mathcal N}'_{\ell}(X,H^{1*};\alpha_-,\alpha_+)$
is entirely similar.
\par
The proof is classical and similar to that of Theorem \ref{theorem266rev}.
Indeed Proposition \ref{prop23999} (1) can be proved in 
the same way as the proof of
Lemma \ref{cpthausdorff}.
(See Definitions \ref{stableconvergence}, \ref{def2626} and
\cite[Lemma 10.4 and Theorem 11.1]{FO} etc.)
Thus we will describe the point where the proof is different below.
\par
The notion of symmetric stabilization $\vec w$ is slightly different.
We put marked points not only on bubble components but also on the
mainstream components $\Sigma_a$ with $a > a_0$
such that $\Sigma_a$ contains no singular or marked points other
than transit points.
\par
We put canonical marked points only on the mainstream component $\Sigma_a$
with $a < a_0$ (such that $\Sigma_a$ contains no singular or marked points other
than transit points.)
\par
The notation of obstruction bundle data is modified as follows.
Definition \ref{obbundeldata1} (1) (symmetric stabilization)
is modified as above.
Definition \ref{obbundeldata1} (2),(a) is changed as follows.
Let $\Sigma_{\rm v} = \Sigma_{a}$ be a mainstream component.
\begin{enumerate}
\item[(i)]
$\mathcal V(\frak x_{a}\cup \vec w_{{\rm can},a})$
is an open subset of
$\overset{\circ}{\mathcal M}_{\ell_{a}+\ell_{a}'+\ell_{a}''}({\rm source})
$
if $a < a_0$.
\item[(ii)]
$\mathcal V(\frak x_{a}\cup \vec w_{a})$
is an open subset of
$\overset{\circ}{\mathcal N}_{\ell_{a}+\ell_{a}'+\ell_{a}''}({\rm source})
$
if $a = a_0$.
\item[(iii)]
$\mathcal V(\frak x_{a} \cup \vec w_{a})$
is an open subset of
$\overset{\circ}{\mathcal M^{\rm cl}}
_{\ell_{a}+\ell_{a}'+\ell_{a}''}$
if $a > a_0$.
\end{enumerate}
\par
Definition \ref{obbundeldata1} (3) is mostly the same 
but Definition \ref{obbundeldata1} (3) (a)
for transit points
$z_{a,-}$ with $a > a_0$
 (or $z_{a,+}$ with $a> a_0$, $z_{a,+} \ne z_+$) is slightly modified as follows.
As we mentioned in Remark \ref{rem:106},  
we did not fix a parametrization $\varphi_{a}$ for each mainstream component $\Sigma_a$ 
with $a> a_0$. 
Instead
we take any biholomorphic map 
$\varphi_a : \R \times S^1 \to \Sigma_a \setminus \{z_{a,-},z_{a,+}\}$
such that $\lim_{\tau\to \pm\infty}\varphi_a (\tau,t) = z_{a,\pm}$.
It is determined up to the $\R \times S^1$ action on the source.
\par
For transit points with extra $S^1$ factor, we perform the process of outer collaring for the factor $(T_{0,j}, \infty] \cong 
[0, s_{0,j})$, where $s_{0,j}= 1/\log T_{0,j}$.  
Namely, we change $\prod ([0,s_{0,j}) \times S^1)/\sim$ to 
$\prod ([-1,s_{0,j}) \times S^1)/ \sim'$, where $\sim'$ is the equivalence relation given by 
$$(s_1,t_1) \sim' (s_2, t_2) \text{ if and only if either } (s_1,t_1) = (s_2, t_2) \text{ or } s_1=s_2 = -1.$$
Write $D_j =  ([-1,s_{0,j}) \times S^1)/ \sim'$.  Namely, we fill a smaller disk around the origin of $([0, s_{0,j}) \times S^1)/\sim$.  
We call this procedure the {\it fattening} of the origin of the complex smoothing parameter.  
See Figure \ref{FigFlatter}.
\begin{figure}[htbp]
\centering
\includegraphics[scale=0.5]{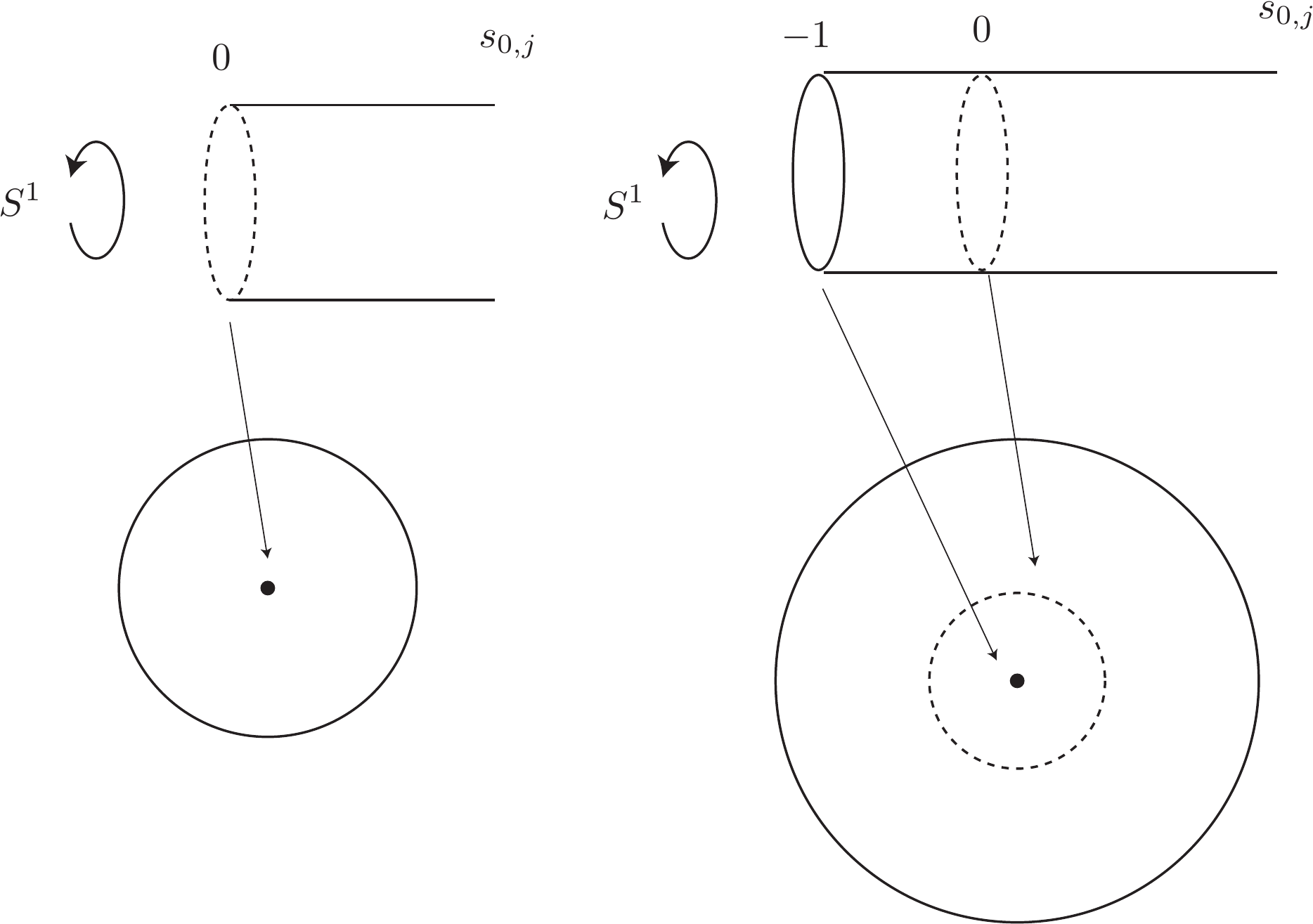}
\caption{Flattering of the origin of the complex smoothing parameter}
\label{FigFlatter}
\end{figure}
Note that the boundary of $\prod D_j$ is empty and does not contribute to the coodimension one boundary of the moduli spaces, 
which we are dealing with in this proof.  
Recall that the exponential decay estimates \cite[Chapter 8]{foooanalysis} were used during the outer collar construction in Sections \ref{subsec;KuramodFloercor}, \ref{subsec;KuramodFloermor}.  
We also have the exponential decay estimates for derivatives involving the $S^1$ directions in the smoothing parameters \cite[Chapter 8]{foooanalysis}.  
Therefore the fattening construction goes in a similar way to the outer collar construction.  The same is also for the interpolation of Kuranishi 
structures among various Kuranishi structures.  Hence we have (3).
\par
Definition \ref{obbundeldata1} (4)-(6) are the same.
\par
Definition \ref{obbundeldata1} (7) is slightly changed as follows.
We also take codimension $2$ submanifolds $\mathcal D_{i}$
for the additional marked points $w_i$ lying
on the mainstream components $\Sigma_a$ with $a > a_0$.
\par
We have thus defined the notion of obstruction bundle data.
We define the notion of stabilization data 
in Definition \ref{stabilizationdefdata} as a part of the obstruction bundle data modified above.
\par
Let ${\bf p} = [((\Sigma,(z_-,z_+,\vec z),a_0),u,\varphi)] \in
{\mathcal N}'_{\ell}(X,H^{1*};\alpha_-,\alpha_+)$
and we take a symmetric stabilization $\vec w$ and the
canonical marked points $\vec w_{\rm can}$ in the above sense.
Suppose we have obstruction bundle data
or stabilization data on {\bf p}.
We will define a map
\begin{equation}\label{form668rev}
\aligned
{\Phi}_{\bf p} : \prod \mathcal V(\frak x_{\rm v} \cup \vec w_{\rm v}\cup \vec w_{{\rm can},{\rm v}}) \times (T_0,\infty]^{k_1}
\times &\prod_{j=1}^{m+k_2}\left(((T_{0,j},\infty] \times S^1)/\sim \right)
\\
&\to {\mathcal N}^{\prime}_{\ell+\ell'+\ell''}({\rm source},{\rm left}),
\endaligned\end{equation}
which is similar to but slightly different from (\ref{form418rev}).
Here $k_1$ is the number of transit points $z_{a,+}$ with $a < a_0$,
$k_2$ is the number of transit points $z_{a,-}$ with $a > a_0$.
(Note $k_1 + k_2$ is the number of all transit points.
$m$ is the number of non-transit singular points.)
\par
The definition of (\ref{form668rev}) is mostly the same as (\ref{form418rev}).
The main difference is that we have an extra $S^1$ factor for each 
transit point $z_{a,-}$ with $a > a_0$.
We use this parameter as follows.
In the case when the components will not be glued with
the main component,
the role of the $S^1$ factor is the same as that for the case of
non-transit singular points 
in the definition of the map \eqref{form418}.
So we consider the case when all the parameters $T_{j}$ corresponding to
transit points   $z_{a,-}$ with $a > a_0$  are finite.
In this case we first use the parameters $(T_{j},\theta_{j})
\in ((T_{0,j},\infty] \times S^1)/\sim$ (which is
associated to those transit points) with local coordinates at those transit points
to glue the spaces to obtain $\Sigma'$.
We define the parametrization of the main component $\Sigma'_{a'_0}$
as follows.
We identify $\Sigma_{a_0} \cong \R \times S^1$. Then
we can embed $[-T_0,T_0] \times S^1 \subset \Sigma_{a_0}$ into
$\Sigma'_{a'_0}$
for sufficiently large number $T_0$ in a canonical way. 
Let $\frak v$ be this embedding.
We require that $\varphi'_{a'_0} = \varphi_{a_0} \circ \frak v$
on  $[-T_0,T_0] \times S^1$.
This condition determines $\varphi'_{a'_0}$ uniquely.
\par
The definition of (\ref{form668rev}) is the same as (\ref{form418rev})
in the other points.
\par
Using  (\ref{form668rev})  we can define the notion of $\epsilon$-closeness
in the same way as in (\ref{orbitecloseness}).
\par
The definition of the notion of transversal constraint is the same as Definition \ref{constrainttt2},
except we apply  Definition \ref{constrainttt2} (1) to the marked points $w'_i$
corresponding to ones on the mainstream component $\Sigma_a$ with
$a > a_0$ also.
\par
We can then define an embedding
$$
I_{{\bf p},{\rm v};\Sigma',u',\varphi'} : E_{{\bf p},{\rm v}}(\frak y) \to
C^{\infty}(\Sigma';(u')^*TX \otimes \Lambda^{0,1})
$$
of obstruction spaces in the same way as in Definition \ref{defn2642}.
\par
In the rest of the construction of the Kuranishi structure, there is nothing to change
the proof of Theorem \ref{kuraexists} (1)
and we obtain a Kuranishi structure on
${\mathcal N}'_{\ell}(X,H^{*1};\alpha_-,\alpha_+)$ in the same way.
\par
Note that the normalized corner
$\widehat S_m({\mathcal N}'_{\ell}(X,H^{*1};\alpha_-,\alpha_+))$
of this K-space is
the disjoint union of the following fiber products.
\begin{equation}\label{form26777741rev1}
\aligned
&{\mathcal M}_{\ell_{1,1}}(X,J_1,H^1;\alpha_{1,0},\alpha_{1,1})
\,{}_{{\rm ev}_{+}}\times_{{\rm ev}_-}  \dots \\
&\dots
\,{}_{{\rm ev}_{+}}\times_{{\rm ev}_-}
{\mathcal M}_{\ell_{1,m}}(X,J_1,H^1;\alpha_{1,m-1},\alpha_{1,m})
\\
&
\,{}_{{\rm ev}_{+}}\times_{{\rm ev}_-}
{\mathcal N}'_{\ell'}(X, H^{*1};\alpha_{1,m},\alpha_{+}),
\endaligned
\end{equation}
where $\alpha_-=\alpha_{1,0},\alpha_{1,1},\dots,\alpha_{1,m-1},\alpha_{1,m}
\in \frak A_1$ and
$\ell_{1,1} +\dots + \ell_{1,m} + \ell' = \ell$.
\par
The important remark here is that there is no fiber product factor
appearing to the `right' of
${\mathcal N}'_{\ell'}(X, H^{*1};\alpha_{1,m},\alpha_{+})$.
Note the fiber product factors such as
${\mathcal M}_{\ell_{1,j}}(X,J_1,H^1;\alpha_{1,j-1},\alpha_{1,j})$
are attached to the main component at the part where $\tau \to -\infty$.
One of the factors of  $(T_0,\infty]^{k_1}$ in (\ref{form668rev}) parametrizes the way we glue it
on the transit points (which are $z_{a,+}$ with $a < a_0$).
On the other hand, the way we glue  mainstream components
$\Sigma_a$ with $a > a_0$ at  the transit points (which are $z_{a,-}$ with $a > a_0$)
are parametrized by
one of the factors in
$\prod_{j=1}^{m+k_2}\left( ((T_{0,j},\infty] \times S^1)/\sim \right)$
in (\ref{form668rev}).
Because of the $S^1$ factors, this parameter does {\it not} correspond
to the boundary or the corner of the K-space.
This is the reason there is no fiber product factor
appearing to the `right'  in (\ref{form26777741rev1}).
\par
We have thus completed the construction of the Kuranishi structure
on 
$$
{\mathcal N}'_{\ell}(X,H^{*1};\alpha_-,\alpha_+).
$$
So we have proved Proposition \ref{prop23999} (2).
\par
We next modify the Kuranishi structure of
${\mathcal N}'_{\ell}(X,H^{*1};\alpha_-,\alpha_+)^{\boxplus 1}$
to obtain an interpolation space of the morphism
$\frak N_{* (H,J)} : \mathcal F_X(H,J) \to \mathcal F_X^{\rm tri}$,
(which is the proof of Proposition \ref{prop23999} (3) (4).)
For this purpose we modify the
Kuranishi structure of
${\mathcal N}'_{\ell}(X,H^{*1};\alpha_-,\alpha_+)^{\boxplus 1}$
obtained by 
\cite[Lemma-Definition 17.38]{fooonewbook}
on
$
{\rm \eqref{form26777741rev1}} \times [-1,0]^m
$.
The way to modify it is the same as the proof of Theorem \ref{theorem266rev} (3)(4).
So we do not repeat it here.
We remark again that  the boundary and corner in (\ref{form26777741rev1})
appear only to the `left' from the main component.
This implies that the spaces ${\mathcal N}'_{\ell}(X,H^{*1};\alpha_-,\alpha_+)^{\boxplus 1}$
(after adjusting the Kuranishi structure so that it is compatible
with the fiber product description (\ref{form26777741rev1}))
will become the interpolation space of the morphism $: \mathcal F_X(H,J) \to \mathcal F_X^{\rm tri}$.
In fact, in the trivial linear K-system $\mathcal F_X^{\rm tri}$ all the spaces
of connecting orbits $\mathcal M(\alpha_-,\alpha_+)$ are empty sets.
\par
The proof of Proposition \ref{prop23999} is now complete.
\end{proof}
\begin{rem}
The constructed Kuranishi structures on ${\mathcal N}'_{\ell}(X,H^{*1};\alpha_-,\alpha_+)^{\boxplus 1}$ and 
${\mathcal N}'_{\ell}(X,H^{*1};\alpha_-,\alpha_+)^{\boxplus 1}$ are 
${\mathfrak C}$-collared in the sense of Remark \ref{defn1531revrev}.  
The Kuranishi structures on ${\mathcal N}'_{\ell}(X,H^{*1};\alpha_-,\alpha_+)^{\boxplus 1}$ and 
${\mathcal N}'_{\ell}(X,H^{1*};\alpha_-,\alpha_+)^{\boxplus 1}$ are also compatible with the fattening in the following sense.  
\end{rem}
\begin{defn}\label{compwithfattening}
We call a Kuranishi structure $\widehat{\mathcal U}$ on $Z \times D^2$ 
{\it compatible with the fattening} if 
$\widehat{\mathcal U}$ is the product of a Kuranishi structure on $Z$ and the trivial Kuranishi structure on $D^2$.  
\end{defn}

We have thus constructed morphisms in Theorem \ref{therem232}.
We will prove their properties (1)(2) of Theorem \ref{therem232}.
The proof is similar to that of Theorem \ref{tjm26108}.
We will define the interpolation space of the
homotopy.
\par
We use smooth functions
$H^{*1*,\cdot} : X \times [0,\infty) \times \R \times S^1  \to \R$,
$H^{1*1,\cdot} : X \times [0,\infty) \times \R \times S^1  \to \R$
as follows.
(Here $H^{*1*,\cdot}(x,T,\tau,t)=H^{*1*,T}(x,\tau,t)$
and $H^{1*1,\cdot}(x,T,\tau,t)=H^{1*1,T}(x,\tau,t)$.
)
\begin{enumerate}
\item
$
H^{*1*,0}(x,\tau,t) \equiv 0
$ and
$
H^{1*1,0}(x,\tau,t) \equiv H(x,t).
$
\item
$$
H^{1*1,T}(x,\tau,t)
=
\begin{cases}
H^{*1}(x,\tau+T,t)
& \text{if $\tau \le 0$}, \\
H^{1*}(x,\tau-T,t)
& \text{if $\tau \ge 0$},
\end{cases}
$$
for sufficiently large $T$.
\item
$$
H^{*1*,T}(x,\tau,t)
=
\begin{cases}
H^{1*}(x,\tau+T,t)
& \text{if $\tau \le 0$}, \\
H^{*1}(x,\tau-T,t)
& \text{if $\tau \ge 0$},
\end{cases}
$$
for sufficiently large $T$.
\end{enumerate}
We define the moduli space
$\mathcal N'_{\ell}(X,H^{1*1,T};\alpha_-,\alpha_+)$
for each $T$ by Definition \ref{3equivrerevl}.
We next define $\mathcal N'_{\ell}(X,H^{*1*,T};\alpha_-,\alpha_+)$.
Let $\alpha_{\pm} \in \pi_2(X)/\sim$, where $\pi_2(X)/\sim$ is
as in Definition \ref{defn23111}.
\begin{defn}\label{defn210revto01}
The set $\widehat{\mathcal N}'_{\ell}(X,H^{*1*,T};\alpha_-,\alpha_+)$
consists of triples 
$$
((\Sigma,(z_-,z_+,\vec z),a_0),u,\varphi)
$$ 
satisfying the following conditions:
\begin{enumerate}
\item
$(\Sigma,(z_-,z_+,\vec z))$ is a genus zero semi-stable curve with $\ell + 2$ marked points.
\item
$\Sigma_{a_0}$ is one of the mainstream components.
We call it the {\it main component}.
\item
$\varphi = \varphi_{a_0}$ where
$\varphi_{a_0} : \R \times S^1 \to \Sigma_{a_0} \setminus \{z_{{a_0},-},z_{{a_0} +}\}$ is 
a parametrization of 
the main component $\Sigma_{a_0}$
and
$\varphi_{a_0}$ is a biholomorphic map such that
$$
\lim_{\tau\to \pm} \varphi_{{a_0}}(\tau,t) = z_{{a_0},\pm}.
$$
\item
For each extended mainstream component $\widehat{\Sigma}_a$, the map
$u$ induces
$u_a : \widehat{\Sigma}_a \setminus\{z_{a,-},z_{a,+}\} \to X$
which is a continuous map.
\item
If $\Sigma_{a_0}$ is the main component and
$\varphi_{a_0} : \R \times S^1 \to \Sigma_{a_0}$ is as in (3), then
the composition $u_{a_0} \circ \varphi_{a_0}$ satisfies the equation
\begin{equation}\label{Fleq27762122}
\frac{\partial (u_{a_0} \circ \varphi_{a_0})}{\partial \tau} +  J \left( \frac{\partial (u_{a_0} \circ \varphi_{a_0})}{\partial t} - \frak X_{H^{*1*,T}_{\tau,t}}
\circ (u_{a_0} \circ \varphi_{a_0}) \right) = 0.
\end{equation}
\item
Void. (In our situation, the finiteness of the energy is a consequence of the continuity of
$u$ in Item (8) below. See \cite[(2.14)]{On}.)
\item
If $\Sigma_{\rm v}$  is a bubble component or
a mainstream component $\Sigma_a$ \ with $a \ne a_0$,
then $u$ is pseudo-holomorphic on $\Sigma_{\rm v}$.
\item
$u$ defines a continuous map on $\Sigma$.
\item
$
[u_*[\Sigma]] \# [\alpha_-] = [\alpha_+].
$
\item
We assume that $((\Sigma,(z_-,z_+,\vec z),a_0),u,\varphi)$
is stable in the sense of Definition \ref{defn2615revrr22} below.
\end{enumerate}
\end{defn}
Assume that $((\Sigma,(z_-,z_+,\vec z),a_0),u,\varphi)$
satisfies (1)-(9) above.
The {\it extended automorphism group}
${\rm Aut}^+((\Sigma,(z_-,z_+,\vec z),a_0),u,\varphi)$
of $((\Sigma,(z_-,z_+,\vec z),a_0)),u,\varphi)$
consists of map $v : \Sigma \to \Sigma$ such that 
it satisfies (1)(2)(5) of Definition \ref{defn2615rev},
and (3) of Definition \ref{defn2615rev} for $\varphi_{a_0}$, 
and $\tau_{a_0} = 0$.
\begin{defn}\label{defn2615revrr22} An object
$((\Sigma,(z_-,z_+,\vec z),a_0),u,\varphi)$
satisfies (1)-(9) above is said to be {\it stable}
if ${\rm Aut}^+((\Sigma,(z_-,z_+,\vec z),a_0),u,\varphi)$
is a finite group.
\end{defn}
We can define the equivalence relation $\sim_2$ on
$\widehat{\mathcal N}_{\ell}(X,H^{*1*};\alpha_-,\alpha_+)$
in the same way as in Definition \ref{3equivrel}
except we require $\tau_{a_0} = 0$ and require (3)
only for $a = a_0$.
We put
\begin{equation}\label{formula269277772222}
{\mathcal N}'_{\ell}(X,H^{*1*{, T}};\alpha_-,\alpha_+) =
\widehat{\mathcal N}'_{\ell}(X,H^{*1*,T};\alpha_-,\alpha_+)/\sim_2.
\end{equation}
When $X$ is a point, we denote 
the space ${\mathcal N}'_{\ell}(X,H^{*1*{, T}};\alpha_-,\alpha_+)$
by
${\mathcal N}'_{\ell}({\rm source};*1*,{\rm finite})$.
\par
For a later purpose, we introduce the following moduli space, which is one-dimensional higher than 
${\mathcal N}'_{\ell}(X,H^{1*1, T};\alpha_-,\alpha_+)$.  
By the definition of $H^{1*1, T}$, when $T > 1$, we have $H^{1*1, T} (x, \tau, t) = 0$ for $|\tau | < T -1$.  
So we can consider the following condition in place of  \eqref{Fleq27762122} in Definition \ref{defn210revto01} (5) 
\begin{equation}\nonumber
\frac{\partial (u_{a_0} \circ \varphi_{a_0})}{\partial \tau} +  J \left( \frac{\partial (u_{a_0} \circ \varphi_{a_0})}{\partial t} - \frak X_{H^{1*1,T}_{\tau,t}}
\circ (u_{a_0} \circ \varphi_{a_0}) \right) = 0  \text{ for } \tau \leq 0, 
\end{equation}
and 
\begin{equation}\nonumber
\frac{\partial (u_{a_0} \circ \varphi_{a_0})}{\partial \tau} +  J \left( \frac{\partial (u_{a_0} \circ \varphi_{a_0})}{\partial t} - \frak X_{H^{1*1,T}_{\tau,t + t_0}}
\circ (u_{a_0} \circ \varphi_{a_0}) \right) = 0  \text{ for } \tau \geq 0, 
\end{equation}
where $t_0 \in {\mathbb R}/{\mathbb Z}$.  
Then we consider the space of 
$$
((\Sigma,(z_-,z_+,\vec z),a_0),u,\varphi, t_0) 
$$ 
satisfying Conditions in Definition \ref{defn210revto01} with (5) replaced by the condition above and obtain the moduli space 
$\mathcal N^{\bullet}_{\ell}(X,H^{1*1,T};\alpha_-,\alpha_+)$.  
(In a later argument, we only need such a one-dimensional higher moduli space for $\mathcal N'_{\ell}(X,H^{1*1,T};\alpha_-,\alpha_+)$ but not for 
$\mathcal N'_{\ell}(X,H^{*1*,T};\alpha_-,\alpha_+)$.)
\par
We consider
\begin{equation}\label{form23999}
\bigcup_{T \in [0,\infty)}\mathcal N'_{\ell}(X,H^{*1*,T};\alpha_-,\alpha_+)
\times \{T\},
\end{equation}
\begin{equation}\label{form23888}
\bigcup_{T \in [0,\infty)}\mathcal N^{\prime}_{\ell}(X,H^{1*1,T};\alpha_-,\alpha_+)
\times \{T\},
\end{equation}
\begin{equation}\label{form23888bullet}
\bigcup_{T \in [10,\infty)}\mathcal N^{\bullet}_{\ell}(X,H^{1*1,T};\alpha_-,\alpha_+)
\times \{T\},
\end{equation}
and will compactify them
by adding certain spaces at $T=\infty$ as follows.
\begin{defn}\label{defn210revto022}
The set $\widehat{\mathcal N}'_{\ell}(X,H^{*1*,\infty};\alpha_-,\alpha_+)$
consists of triples 
$$
((\Sigma,(z_-,z_+,\vec z),a_1,a_2),u,\varphi)
$$ satisfying the following conditions:
\begin{enumerate}
\item
$(\Sigma,(z_-,z_+,\vec z))$ is a genus zero semi-stable curve with $\ell + 2$ marked points.
\item
$\Sigma_{a_1}, \Sigma_{a_2}$ are mainstream components
such that $a_1 < a_2$.
We call them  {\it the first main component} and
{\it the second main component}.
\item
$\varphi = (\varphi_{a})$ where
$\varphi_{a} : \R \times S^1 \to \Sigma_{a} \setminus \{z_{{a},-},z_{{a} +}\}$ is 
a parametrization of
the mainstream component $\Sigma_{a}$ with $a_1 \le a \le a_2$
and 
$\varphi_{a}$ is a biholomorphic map such that
$$
\lim_{\tau\to \pm} \varphi_{{a}}(\tau,t) = z_{{a},\pm}.
$$
\item
For each extended mainstream component $\widehat{\Sigma}_a$, the map
$u$ induces
$u_a : \widehat{\Sigma}_a \setminus\{z_{a,-},z_{a,+}\} \to X$
which is a continuous map.
\item
If $\Sigma_{a}$ is a mainstream component with $a_1 \le a \le a_2$ and
$\varphi_{a} : \R \times S^1 \to \Sigma_{a}$ is as in (3), then
the composition $u_{a} \circ \varphi_{a}$ satisfies the equation
\begin{equation}\label{Fleq2776212233}
\frac{\partial (u_{a} \circ \varphi_{a})}{\partial \tau} +  J \left( \frac{\partial (u_{a} \circ \varphi_{a})}{\partial t} - \frak X_{H^{a,\infty}_{\tau,t}}
\circ (u_{a} \circ \varphi_{a}) \right) = 0.
\end{equation}
Here $H^{a,\infty} = H^{*1}$ if $a = a_1$,
$H^{a,\infty} = H^{1}$ if $a_1 < a < a_2$,
and $H^{a,\infty} = H^{1*}$ if $a = a_2$.
\item
$$
\int_{\R \times S^1}
\left\Vert\frac{\partial (u \circ \varphi_a)}{\partial \tau}\right\Vert^2 d\tau dt < \infty.
$$
\item
If $\Sigma_{\rm v}$  is a bubble component or
a mainstream component $\Sigma_a$ \ with $a < a_1$
or $a > a_2$,
then $u$ is pseudo-holomorphic on $\Sigma_{\rm v}$.
\item
Let
${\Sigma}_{a}$ and ${\Sigma}_{a'}$ be
mainstream components.
If  $z_{a,+} = z_{a',-}$ and
$a_1 \le a < a' \le a_2$, then
$$
\lim_{\tau\to+\infty} (u_{a} \circ \varphi_{a})(\tau,t)
=
\lim_{\tau\to-\infty} (u_{a'} \circ \varphi_{a'})(\tau,t)
$$
holds for each $t \in S^1$.
((6) and Lemma \ref{prof26rev3} imply that the
left and right hand sides both converge.)
\par
If $a < a_1$, then we require that $u$ is continuous at $z_{a,+}$.
If $a > a_2$, then we require that $u$ is continuous at $z_{a,-}$.
\item
$
[u_*[\Sigma]] \# [\alpha_-] = [\alpha_+].
$
\item
We assume that $((\Sigma,(z_-,z_+,\vec z),a_1,a_2),u,\varphi)$
is stable in the sense of Definition \ref{defn2615revrr33} below.
\end{enumerate}
\end{defn}
Assume that $((\Sigma,(z_-,z_+,\vec z),a_1,a_2),u,\varphi)$
satisfies (1)-(9) above.
The {\it extended automorphism group}
${\rm Aut}^+((\Sigma,(z_-,z_+,\vec z),a_1,a_2),u,\varphi)$
of $((\Sigma,(z_-,z_+,\vec z),a_1,a_2)),u,\varphi)$
consists of map $v : \Sigma \to \Sigma$ such that
it satisfies (1)(2)(5) of Definition \ref{defn2615rev},
and (3) of Definition \ref{defn2615rev} for $\varphi_{a}$
with $a_1 \le a \le a_2$, and $\tau_{a} = 0$ if $a = a_1$ or
$a = a_2$.
\begin{defn}\label{defn2615revrr33} An object
$((\Sigma,(z_-,z_+,\vec z),a_1,a_2),u,\varphi)$
satisfies (1)-(9) above is said to be {\it stable}
if ${\rm Aut}^+((\Sigma,(z_-,z_+,\vec z),a_1,a_2),u,\varphi)$
is a finite group.
\end{defn}
\begin{defn}\label{defn210revto044}
The set $\widehat{\mathcal N}'_{\ell}(X,H^{1*1,\infty};\alpha_-,\alpha_+)$
consists of triples 
$$
((\Sigma,(z_-,z_+,\vec z),a_1,a_2),u,\varphi)
$$ 
satisfying the following conditions:
\begin{enumerate}
\item
$(\Sigma,(z_-,z_+,\vec z))$ is a genus zero semistable curve with $\ell + 2$ marked points.
\item
$\Sigma_{a_1}, \Sigma_{a_2}$ are mainstream components
such that $a_1 < a_2$.
We call them  {\it the first main component} and
{\it the second main component}.
\item
$\varphi = (\varphi_{a})$ where
$\varphi_{a} : \R \times S^1 \to \Sigma_{a} \setminus \{z_{{a},-},z_{{a} +}\}$ is 
a parametrization of 
mainstream component $\Sigma_{a}$ with $a \le a_1$ or $a \ge a_2$,
and 
$\varphi_{a}$ is a biholomorphic map such that
$$
\lim_{\tau\to \pm} \varphi_{{a}}(\tau,t) = z_{{a},\pm}.
$$
\item
For each extended mainstream component $\widehat{\Sigma}_a$, the map
$u$ induces
$u_a : \widehat{\Sigma}_a \setminus\{z_{a,-},z_{a,+}\} \to X$
which is a continuous map
\item
If $\Sigma_{a}$ is a mainstream component with $a \le a_1$ or $a \ge a_2$ and
$\varphi_{a} : \R \times S^1 \to \Sigma_{a}$ is as in (3), then
the composition $u_{a} \circ \varphi_{a}$ satisfies the equation
\begin{equation}\label{Fleq2776244444}
\frac{\partial (u_a \circ \varphi_a)}{\partial \tau} +  J \left( \frac{\partial (u_a \circ \varphi_a)}{\partial t} - \frak X_{H^{a,\infty}_{\tau,t}}
\circ (u_a \circ \varphi_a) \right) = 0.
\end{equation}
Here $H^{a,\infty} = H^{1*}$ if $a = a_1$,
$H^{a,\infty} = H^{1}$ if $a < a_1$ or $a > a_2$,
and $H^{a,\infty} = H^{*1}$ if $a = a_2$.
\item
$$
\int_{\R \times S^1}
\left\Vert\frac{\partial (u \circ \varphi_a)}{\partial \tau}\right\Vert^2 d\tau dt < \infty.
$$
\item
If $\Sigma_{\rm v}$  is a bubble component or
a mainstream component $\Sigma_a$ \ with $a_1 < a < a_2$,
then $u$ is pseudo-holomorphic on $\Sigma_{\rm v}$.
\item
Let
${\Sigma}_{a}$ and ${\Sigma}_{a'}$ be
mainstream components. 
If  $z_{a,+} = z_{a',-}$ and
$a  \le a' \le a_1$ or $a_2 \le a \le a'$, then
$$
\lim_{\tau\to+\infty} (u_{a} \circ \varphi_{a})(\tau,t)
=
\lim_{\tau\to-\infty} (u_{a'} \circ \varphi_{a'})(\tau,t)
$$
holds for each $t \in S^1$.
((6) and Lemma \ref{prof26rev3} imply that the 
left and right hand sides both converge.)
\par
If $a_1 < a \le a_2$, then we require that $u$ is continuous at $z_{a,-}$.
\item
If
${\Sigma}_{a}$
is
mainstream components and $z_{a,-} = z_-$ (resp.  $z_{a,+} = z_+$),
 then there exist
$(\gamma_{-},w_{-})
\in R_{\alpha_{-}}$ (resp. $(\gamma_{+},w_{+})
\in R_{\alpha_{+}}$) such that
$$
\lim_{\tau\to \pm\infty} (u_{a} \circ \varphi_{a})(\tau,t)
= \gamma_{\pm}(t)
$$
Moreover
$
[u_*[\Sigma]] \# w_- \sim w_+
$.
\item
We assume that $((\Sigma,(z_-,z_+,\vec z),a_1,a_2),u,\varphi)$
is stable in the sense of Definition \ref{defn2615revrr44} below.
\end{enumerate}
\end{defn}
Assume that $((\Sigma,(z_-,z_+,\vec z),a_1,a_2),u,\varphi)$
satisfies (1)-(9) above.
The {\it extended automorphism group}
${\rm Aut}^+((\Sigma,(z_-,z_+,\vec z),a_1,a_2),u,\varphi)$
of $((\Sigma,(z_-,z_+,\vec z),a_1,a_2)),u,\varphi)$
consists of map $v : \Sigma \to \Sigma$ such that
it satisfies (1)(2)(5) of Definition \ref{defn2615rev},
and (3) of Definition \ref{defn2615rev} for $\varphi_{a}$
with $a \le a_1$ or $a_2 \le a$, and $\tau_{a} = 0$ if $a = a_1$ or
$a = a_2$.
\begin{defn}\label{defn2615revrr44} An object
$((\Sigma,(z_-,z_+,\vec z),a_1,a_2),u,\varphi)$
satisfies (1)-(9) above is said to be {\it stable}
if ${\rm Aut}^+((\Sigma,(z_-,z_+,\vec z),a_1,a_2),u,\varphi)$
is a finite group.
\end{defn}
We can define the equivalence relation $\sim_2$ on the spaces
$\widehat{\mathcal N}'_{\ell}(X,H^{*1*,\infty};\alpha_-,\alpha_+)$
and
$\widehat{\mathcal N}'_{\ell}(X,H^{1*1,\infty};\alpha_-,\alpha_+)$
in the same way as in Definition \ref{3equivrel}
except we require $\tau_{a_1} = \tau_{a_1}= 0$ and require (3)
only for $a$ for which $\varphi_a$ is defined.
We put
\begin{equation}\label{formula2692777tugi}
\aligned
{\mathcal N}'_{\ell}(X,H^{*1*,\infty};\alpha_-,\alpha_+) &=
\widehat{\mathcal N}'_{\ell}(X,H^{*1*,\infty};\alpha_-,\alpha_+)/\sim_2,
\\
{\mathcal N}'_{\ell}(X,H^{1*1,\infty};\alpha_-,\alpha_+) &=
\widehat{\mathcal N}'_{\ell}(X,H^{1*1,\infty};\alpha_-,\alpha_+)/\sim_2.
\endaligned
\end{equation}
As in the case of $T < \infty$, we have the moduli space  
$$
{\mathcal N}^{\bullet}_{\ell}(X,H^{1*1,\infty};\alpha_-,\alpha_+)
$$ 
by replacing Condition (5) in 
Definition  \ref{defn210revto044} by the following condition: 
\par
\smallskip
For $a \leq a_1$, $u \circ \varphi_a$ satisfies 
\begin{equation} \nonumber 
\frac{\partial (u_a \circ \varphi_a)}{\partial \tau} +  J \left( \frac{\partial (u_a \circ \varphi_a)}{\partial t} - \frak X_{H^{a,\infty}_{\tau,t}}
\circ (u_a \circ \varphi_a) \right) = 0.  
\end{equation}
\par
For $a \geq a_2$, $u \circ \varphi_a$ satisfies 
\begin{equation} \nonumber 
\frac{\partial (u_a \circ \varphi_a)}{\partial \tau} +  J \left( \frac{\partial (u_a \circ \varphi_a)}{\partial t} - \frak X_{H^{a,\infty}_{\tau,t +t_0}}
\circ (u_a \circ \varphi_a) \right) = 0, 
\end{equation}
for some $t_0 \in {\mathbb R}/{\mathbb Z}$.  
\par
When $X$ is a point, we denote by
$$
{\mathcal N}'_{\ell}({\rm source};*1*,\infty),
\quad {\mathcal N}'_{\ell}({\rm source};1*1,\infty),
\quad {\mathcal N}^{\bullet}_{\ell}({\rm source};1*1,\infty)
$$
the spaces
${\mathcal N}'_{\ell}(\text{\rm one point},H^{*1*,\infty};\alpha_-,\alpha_+)$,
${\mathcal N}'_{\ell}(\text{\rm one point},H^{1*1,\infty};\alpha_-,\alpha_+)$
and 
\linebreak
${\mathcal N}^{\bullet}_{\ell}(\text{\rm one point},H^{1*1,\infty};\alpha_-,\alpha_+)$
respectively.
\begin{defn}
We put
\begin{equation}
({\rm\ref{form23999}}) \cup {\mathcal N}'_{\ell}(X,H^{*1*,\infty};\alpha_-,\alpha_+)
={\mathcal N}'_{\ell}(X,H^{*1*};\alpha_-,\alpha_+),
\end{equation}
\begin{equation}
({\rm\ref{form23888}}) \cup {\mathcal N}'_{\ell}(X,H^{1*1,\infty};\alpha_-,\alpha_+)
={\mathcal N}'_{\ell}(X,H^{1*1};\alpha_-,\alpha_+)
\end{equation}
and
\begin{equation}
\eqref{form23888bullet} \cup {\mathcal N}^{\bullet}_{\ell}(X,H^{1*1,\infty};\alpha_-,\alpha_+)
={\mathcal N}^{\bullet}_{\ell}(X,H^{1*1};\alpha_-,\alpha_+).
\end{equation}
\par
When $X$ is a point, they are written as
${\mathcal N}'_{\ell}({\rm source};*1*)$, ${\mathcal N}'_{\ell}({\rm source};1*1)$ and 
${\mathcal N}^{\bullet}_{\ell}({\rm source};1*1)$
respectively.
\end{defn}
\begin{prop}\label{prop23999222}
\begin{enumerate}
\item
We can define topologies on ${\mathcal N}'_{\ell}(X,H^{*1*};\alpha_-,\alpha_+)$
and  ${\mathcal N}'_{\ell}(X,H^{1*1};\alpha_-,\alpha_+)$
so that they are compact and Hausdorff.
\item
We can define Kuranishi structures on them.
\item
We can define Kuranishi structures on the spaces
${\mathcal N}'_{\ell}(X,H^{*1*};\alpha_-,\alpha_+)^{\boxplus 1}$
and  ${\mathcal N}'_{\ell}(X,H^{1*1};\alpha_-,\alpha_+)^{\boxplus 1}$.
\item
Together with other objects
${\mathcal N}'_{\ell}(X,H^{*1*};\alpha_-,\alpha_+)^{\boxplus 1}$
defines a homotopy between
$\frak N_{* (H,J)} \circ \frak N_{(H,J) *} :  \mathcal F_X^{\rm tri} \to \mathcal F_X^{\rm tri}$ and a morphism of
energy loss $0$,
of which it will be an interpolation space.
\item Together with other objects
${\mathcal N}'_{\ell}(X,H^{1*1};\alpha_-,\alpha_+)^{\boxplus 1}$
defines a homotopy between
$ \frak N_{(H,J) *} \circ \frak N_{* (H,J)} : \mathcal F_X(H,J) \to \mathcal F_X(H,J)$
and $\frak N_{11}(\mathcal J,H^{11})$,
of which it will be an interpolation space.
\end{enumerate}
\end{prop}
\begin{proof}
The proof is a straight forward generalization of that of
Theorem \ref{the26112} etc. by taking the points
we discussed in the proof of
Proposition \ref{prop23999} into account.
So we only explain the point where the proof is different from that of
Theorems \ref{the26112} etc..
\par
When we define the notion of symmetric stabilization,
we put marked points to each unstable component\footnote
{We call an irreducible component of the source curve {\it unstable} if the
set of biholomorphic map of this component preserving all the marked and
singular points on it is of infinite order.}
of the source curve which is either a bubble component or
a mainstream component where we do not put a parametrization
$\varphi_a$.
We define a canonical marked point on each unstable mainstream
component on which we define a parametrization $\varphi_a$
as a part of the data and
which is not a main component.
\par
When we define obstruction bundle data,
the neighborhood $\mathcal V(\frak x_{\rm v} \cup \vec w_{\rm v}\cup \vec w_{{\rm can},{\rm v}})$
is an open subset of
\begin{enumerate}
\item[(i)] $\overset{\circ}{\mathcal M^{\rm cl}_{\ell_{\rm v}}}$
if $\Sigma_{\rm v}$ is a bubble component or a
mainstream component on which we do not define a parametrization $\varphi$,
\item[(ii)] $\overset{\circ}{\mathcal M}_{\ell_{\rm v}}({\rm Source})$
if $\Sigma_{\rm v}$ is a
mainstream component on which we define
the parametrization $\varphi_{\rm v}$
and which is not a main component,
\item[(iii)]
$\overset{\circ}{\mathcal N}_{\ell_{\rm v}}({\rm Source})$ if
$\Sigma_{\rm v}$ is a main component.
\end{enumerate}
We also include a codimension 2 submanifold $\mathcal D_i$
as in Definition \ref{obbundeldata1} (7) for each $i$
if $w_i$ corresponds to an additional marked point 
which is not a canonical marked point.
\par
We will define a map ${\Phi}_{\bf p}$ similar
to (\ref{form668rev}) as follows.
\par\noindent
{\bf Case 1:}
${\bf p}
= [((\Sigma,(z_-,z_+,\vec z),a_1,a_2),u,\varphi)] \in {\mathcal N}'_{\ell}(X,H^{*1*,\infty};\alpha_-,\alpha_+)$.
Let $k_1$ be the number of mainstream components $\Sigma_a$ with
$a_1 < a < a_2$. Then the map ${\Phi}_{\bf p}$ is
\begin{equation}\label{form668rev222}
\aligned
{\Phi}_{\bf p} : \prod \mathcal V(\frak x_{\rm v} \cup \vec w_{\rm v}\cup \vec w_{{\rm can},{\rm v}}) \times (T_0,\infty]^{k_1+1}
\times &\prod_{j=1}^{m+k_2}\left(((T_{0,j},\infty] \times S^1)/\sim \right)
\\
&\to {\mathcal N}'_{\ell}({\rm source};*1*,\infty).
\endaligned\end{equation}
We note that $k_1 + 1$ is the number of transit points which lie 
between $\Sigma_{a_1}$ and $\Sigma_{a_2}$.
Moreover $k_2$ is the number of other transits points  and $m$ is the
number of non-transit singular points.
The reason we have the $S^1$ factors for the transits points which do not
lie between $\Sigma_{a_1}$ and $\Sigma_{a_2}$ is the same as
the case of (\ref{form668rev}).
\par\noindent
{\bf Case 2:}
${\bf p}
= [((\Sigma,(z_-,z_+,\vec z),a_0),u,\varphi),T] \in\mathcal N_{\ell}(X,H^{*1*,T};\alpha_-,\alpha_+)
\times \{T_{\bf p}\}
$
with $T_{\bf p} > 0$. 
The map ${\Phi}_{\bf p}$ is
\begin{equation}\label{form668rev233}
\aligned
{\Phi}_{\bf p} : \prod \mathcal V(\frak x_{\rm v} \cup \vec w_{\rm v}\cup \vec w_{{\rm can},{\rm v}})
\times &\prod_{j=1}^{m+k}\left(((T_{0,j},\infty] \times S^1)/\sim \right)
\times (T_{\bf p} - \epsilon, T_{\bf p}+\epsilon)
\\
&\to {\mathcal N}'_{\ell}({\rm source};*1*,\infty).
\endaligned\end{equation}
Here $k$ is the number of transit points and $m$ is the number of
non-transit singular points.
We defined the parametrization of the mainstream only on the component $\Sigma_{a_0}$. This is the reason all the factors $(T_{0,j},\infty]$
come with the $S^1$ factors. Note in this case ${\bf p}$ is an interior
point of our moduli space ${\mathcal N}'_{\ell}({\rm source};*1*,\infty)$.
\par
\noindent
{\bf Case 3:}
${\bf p}
= [((\Sigma,(z_-,z_+,\vec z),a_1,a_2),u,\varphi)] \in {\mathcal N}^{\bullet}_{\ell}(X,H^{1*1,\infty};\alpha_-,\alpha_+)$.
The map ${\Phi}^{\bullet}_{\bf p}$ is
\begin{equation}\label{form668rev444bullet}
\aligned
{\Phi}^{\bullet}_{\bf p}: \prod \mathcal V(\frak x_{\rm v} \cup &\vec w_{\rm v}\cup \vec w_{{\rm can},{\rm v}}) \times (T_0,\infty]^{m_1+m_2}
\times \prod_{j=1}^{m}\left(((T_{0,j},\infty] \times S^1)/\sim \right)
\\
&\times \left(\prod_{j=1}^{m_*}\left(((T_{0,j},\infty] \times S^1)/\sim \right)\right)
\to {\mathcal N}^{\bullet}_{\ell}({\rm source};1*1,\infty).
\endaligned\end{equation}
Here $m_*$ is the number of transit points which lie between
$\Sigma_{a_1}$ and $\Sigma_{a_2}$ and 
$m_1$ (resp. $m_3$) is the number of transit points which lie
`left' (resp. `right') from $\Sigma_{a_1}$ (resp. $\Sigma_{a_2}$).
Also $m$ is the number of non-transit singular points.  

For ${\bf p}
= [((\Sigma,(z_-,z_+,\vec z),a_1,a_2),u,\varphi)] \in {\mathcal N}'_{\ell}(X,H^{1*1,\infty};\alpha_-,\alpha_+)$, we have 

\begin{equation}\label{form668rev444}
\aligned
{\Phi}_{\bf p} : \prod \mathcal V(\frak x_{\rm v} \cup &\vec w_{\rm v}\cup \vec w_{{\rm can},{\rm v}}) \times (T_0,\infty]^{m_1+m_2}
\times \prod_{j=1}^{m}\left(((T_{0,j},\infty] \times S^1)/\sim \right)
\\
&\times \left(\prod_{j=1}^{m_*}\left(((T_{0,j},\infty] \times S^1)/\sim \right)\right)' 
\to {\mathcal N}'_{\ell}({\rm source};1*1,\infty), 
\endaligned\end{equation}
where 
\begin{equation}\label{2318factor}
\left(\prod_{j=1}^{m_*}\left(((T_{0,j},\infty] \times S^1)/\sim\right)\right)'
\end{equation}
is the subset of $\prod_{j=1}^{m_*}\left(((T_{0,j},\infty] \times S^1)/\sim\right)$ defined by  
$$\theta_1 + \cdots + \theta_{m_*} = 0 \text{ in } S^1={\mathbb R}/{\mathbb Z}.$$ 
Here $\theta_j$,  $j=1, \dots, m_*$ are the coordinates of the $S^1$ factors.  
The way the factors appear in the left hand side of
(\ref{form668rev444}) is similar to other cases except the factor \eqref{2318factor}.  
We explain the way this factor appears.  
\par
Suppose all the $(T_{0,j},\infty]$ components in this factor are finite.
In this case we use this parameter together with
$\theta_j \in S^1$ to glue $\Sigma_{a_1}$, $\Sigma_{a_2}$
and the components $\Sigma_a$ with $a_1 < a < a_2$.
We then obtain a component which we write as $\Sigma_{a_0}$.
This will be the main component of the resulting element of
${\mathcal N}'_{\ell}({\rm source};*1*,\infty)$.
(More precisely, it may be glued with other mainstream components.)
\par
Note we defined the parametrizations $\varphi_{a_1}$,
$\varphi_{a_2}$ but not defined parametrizations for other
$\Sigma_a$ with $a_1 < a < a_2$.
We consider a parametrization $\varphi_{a_0} : \R \times S^1 \to
\Sigma_{a_0} \setminus \{z_{a_0,-},z_{a_0,+}\}$
such that $\lim_{\tau\to \pm\infty}\varphi_{a_0}(\tau,t)
= z_{a_0,\pm}$.
\par
We can identify $[-T_0,T_0] \times S^1 \subset \Sigma_{a_1}
\subset \Sigma_{a_0}$. Let $\frak v_1$ be this
map.
We also take an embedding $\frak v_2$
from  $[-T_0,T_0] \times S^1 \subset \Sigma_{a_2}$
to $\Sigma_{a_0}$.
There exist $t_1$, $t_2$ such that
$$
(\frak v_j\circ \varphi_{a_j}) (\tau,t) = \varphi_{a_0}(\tau+\tau_j,t+t_j)
$$
for $j=1,2$. Note the choice of $\varphi_{a_0}$ is not unique.
Namely we may change it by the $\R$ action.
We may choose $\varphi_{a_0}$ so that $\tau_1 = - \tau_2$.
Then $\tau_2 - \tau_1 +1 = T$.
(See (\ref{formula2Tplus1}).)
Here $T$ is the second factor in (\ref{form23888}).
\par
We consider $t_j$. We may choose the representative $\varphi_{a_0}$
so that one of them, say $t_1$ to be $0$. However it is
impossible to take the representative $\varphi_{a_0}$
for which both of $t_j$ are zero.
The notation $'$ stands for the constraint that $t_1 = t_2$.
Under this assumption we may choose both of $t_1$ and $t_2$ are zero.
Note if we change the $S^1$ factor in
$((T_{0,j},\infty] \times S^1)/\sim$ by $\theta_j$, then
$t_2 - t_1$ changes by $\sum_{j=1}^{m_*} \theta_j$.
\par
We remark that the point $[\infty,\dots,\infty]$
in (\ref{2318factor}) is a boundary point of  (\ref{2318factor}).
In fact, if we remove $'$ then it will be an interior point.
Because of the constraint  $t_1 = t_2$ this point becomes
the boundary points.
This is consistent to the fact that
${\mathcal N}'_{\ell}({\rm source};1*1,\infty)$
lies at the part $T = \infty$ of
${\mathcal N}'_{\ell}({\rm source};1*1)$.
\par
Note that there exists ${\mathfrak T}_0 > 0$ such that the closed neighborhoods $W({\mathbf p}_c)$ 
of  ${\mathbf p}_c \in {\mathcal N}'_{\ell}(X,H^{1*1,\infty};\alpha_-,\alpha_+)$ cover 
$\cup_{T \geq {\mathfrak T}_0}  \mathcal N_{\ell}(X,H^{1*1,T};\alpha_-,\alpha_+) \times \{T\}$.  
For a sufficiently large $\mathfrak T > {\mathfrak T}_0$, the obstruction space for 
${\mathbf q} \in \cup_{T \geq {\mathfrak T}_0}  \mathcal N_{\ell}(X,H^{1*1,T};\alpha_-,\alpha_+) \times \{T\}$ 
is the direct sum of $I_{{\bf p_c},{\rm v}; {\mathbf q}}(E_{{\bf p_c},{\rm v}}(\frak y))$ with ${\mathbf q} \in W({\mathbf p}_c)$.  
\par\noindent
{\bf Case 4:}
${\bf p}  \in\mathcal N_{\ell}(X,H^{1*1,T};\alpha_-,\alpha_+)
\times \{T_{\bf p}\}.$
This case 
is entirely the same as the case of (\ref{form418revrev}).
Using this map ${\Phi}_{\bf p}$, we define the notion of $\epsilon$-closeness.
Then we can define an obstruction space using them.
The definition of transversal constraint is similar.
We then take an appropriate closed finite covering of
our moduli space by $\{ W({\bf p}_c)\}_c$
and use it to define a Kuranishi structure in the same way
as in the proof of
Theorems \ref{the26112}.
\par
We note that the boundary of it has the form required
by Proposition \ref{prop23999222} (4)(5).
In fact, the boundary of the space
${\mathcal N}'_{\ell}(X,H^{*1*};\alpha_-,\alpha_+)$ is
only at $T=0$ and $T=\infty$.
This is the consequence of the description of the
domain of the maps (\ref{form668rev222}),
(\ref{form668rev233}), especially of the fact
that the domain of (\ref{form668rev233}) has no boundary.
\par

The boundary of the space
${\mathcal N}'_{\ell}(X,H^{1*1};\alpha_-,\alpha_+)$
other than those which are at $T=0$ and $T=\infty$
can be described by using the space of
connecting orbits of $\mathcal F_X(H,J)$.
Moreover the part of 
${\mathcal N}'_{\ell}(X,H^{1*1};\alpha_-,\alpha_+)$ which
appears at the part $T=\infty$ is the union of various spaces  
\begin{equation}\label{form26777741rev}
\aligned
&{\mathcal M}_{\ell_{1,1}}(X,H^1;\alpha_{1,0},\alpha_{1,1})
\,{}_{{\rm ev}_{+}}\times_{{\rm ev}_-}  \dots \\
&\dots
\,{}_{{\rm ev}_{+}}\times_{{\rm ev}_-}
{\mathcal M}_{\ell_{1,m_1}}(X,H^1;\alpha_{1,m_1-1},\alpha_{1,m_1})
\\
&
\,{}_{{\rm ev}_{+}}\times_{{\rm ev}_-}
{\mathcal N}'_{\ell'}(X,H^{1*};\alpha_{1,m},\alpha_{*,1}) \\
&\,{}_{{\rm ev}_{+}}\times_{{\rm ev}_-}{\mathcal M}^{\rm cl}_{\ell_{*,1}+2}(X;\alpha_{*,2}-\alpha_{*,1})
\,{}_{{\rm ev}_{+}}\times_{{\rm ev}_-} \dots\\
&\,{}_{{\rm ev}_{+}}\times_{{\rm ev}_-}
{\mathcal M}^{\rm cl}_{\ell_{*,m_*}+2}(X;\alpha_{*,m_*}-\alpha_{*,m_{*-1}})
\\
&
\,{}_{{\rm ev}_{+}}\times_{{\rm ev}_-}
{\mathcal N}'_{\ell''}(X,H^{*1};\alpha_{*,m},\alpha_{2,1}) 
\\
&{\mathcal M}_{\ell_{2,1}}(X,H^1;\alpha_{2,1},\alpha_{2,2})
\,{}_{{\rm ev}_{+}}\times_{{\rm ev}_-}  \dots
\\
&\dots
\,{}_{{\rm ev}_{+}}\times_{{\rm ev}_-}
{\mathcal M}_{\ell_{2,m_2+1}}(X,H^1;\alpha_{2,m_2},\alpha_{2,m_2+1}).
\endaligned
\end{equation}
Here ${\mathcal M}^{\rm cl}_{\ell}(X,\alpha)$
is the moduli space of stable maps 
of genus $0$ with $\ell$ marked points
and of homology class $\alpha$.
\par
The situation at $T = \infty$ is analogous to the space ${\frak{forget}}^{-1}([\Sigma_0])$ in \cite[Section 2.6, Lemmas 2.6.3, 2.6.27]{foooAst}, 
which is examined in detail in \cite[Section 4.6]{foooAst}.  
We remark that (\ref{form26777741rev}) lies in a codimension $m_1+m_2 + 1$ locus.  
This is a consequence of the discussion on (\ref{form668rev444}).
\par
For ${\mathcal N}^{\bullet}_{\ell}(X,H^{1*1};\alpha_-,\alpha_+)$, 
we perform the collar construction as well as the fattening of origins of complex smoothing parameters 
at transit points between $\Sigma_{a_1}$ and $\Sigma_{a_2}$ as in the proof of Proposition \ref{prop23999} 
to obtain the desired Kuranishi structure on ${\mathcal N}^{\bullet}_{\ell}(X,H^{1*1};\alpha_-,\alpha_+)^{\boxplus 1}$.  
\par
We regard the part of ${\mathcal N}'_{\ell}(X,H^{1*1};\alpha_-,\alpha_+)$ with 
$T \in [\mathfrak T, \infty]$ as a subspace of 
${\mathcal N}^{\bullet}_{\ell}(X,H^{1*1};\alpha_-,\alpha_+)$ defined by $\theta_1 + \cdots + \theta_{m_*} = 0$.  
For the construction of Kuranishi structures, we take a finite subset 
$\{{\mathbf p}_c\} \subset {\mathcal N}'_{\ell}(X,H^{1*1,T};\alpha_-,\alpha_+)$ such that 
the union of ${\rm Int} \ W({\mathbf p}_c)$ contains ${\mathcal N}'_{\ell}(X,H^{1*1,T};\alpha_-,\alpha_+)$, 
and a finite subset $\{{\mathbf p}_{c'}\} \subset {\mathcal N}^{\bullet}_{\ell}(X,H^{1*1};\alpha_-,\alpha_+) \setminus {\mathcal N}'_{\ell}(X,H^{1*1};\alpha_-,\alpha_+)$ 
such that $W({\mathbf p}_{c'})$'s are disjoint from ${\mathcal N}'_{\ell}(X,H^{1*1};\alpha_-,\alpha_+)$ 
and ${\mathcal N}^{\bullet}_{\ell}(X,H^{1*1};\alpha_-,\alpha_+)$ is covered by 
${\rm Int} \ W({\mathbf p}_c)$'s and ${\rm Int} \ W({\mathbf p}_{c'})$'s.  
By the construction, the Kuranishi structure on ${\mathcal N}^{\bullet}_{\ell}(X,H^{1*1,T};\alpha_-,\alpha_+)$ is compatible with 
the fattening on $D_j(\delta)$ for a sufficiently small $\delta > 0$ (Definition \ref{compwithfattening}). 
\par
The Kuranishi structure on ${\mathcal N}^{\bullet}_{\ell}(X,H^{1*1};\alpha_-,\alpha_+)^{\boxplus 1}$ restricts to the one on 
the subspace defined by $\sum_{j=1}^{k_2} \theta_j = 0$ and $T \in [\mathfrak T, \infty]$.   
Our choice of ${\mathbf p}_c$, we see that the Kuranishi structure on ${\mathcal N}'_{\ell}(X,H^{1*1};\alpha_-,\alpha_+)$ with $T \in [\mathfrak T, \infty]$ 
can be arranged to match with the part $T \in [0, \mathfrak T +1]$.  Thus we obtain 
${\mathcal N}'_{\ell}(X,H^{1*1};\alpha_-,\alpha_+)^{\boxplus 1}$ with the desired Kuranishi structure.  
\par
We approximate ${\mathcal N}_{\ell}(X,H^{1*1, \infty};\alpha_-,\alpha_+)^{\boxplus 1}$ in ${\mathcal N}^{\bullet}_{\ell}(X,H^{1*1};\alpha_-,\alpha_+)^{\boxplus 1}$ 
using a family of smooth submanifolds $S^{(k)}_{\epsilon}$ of codimension $2$
in $\prod_{j=1}^{k_2} D^2_j$ given below.  
\par
Let $\rho:[-1,0] \to {\mathbb R}$ be a non-decreasing smooth function such that $\rho (t) = 3(t+1)/2$ for $-1 \leq t \leq-1 + \delta/2$ and $\rho (t) = \delta$ for 
$t \geq -1 + \delta$.  
Choose $\epsilon < \delta/10$ and define $S^{(k)}_{\epsilon}$ by $\prod_{j=1}^{k} \rho(t_j) = \epsilon \delta^{k-1}$ and $\sum_{j=1}^{k} \theta_j = 0$ in ${\mathbb R}/{\mathbb Z}$.  
Then $S^{(k)}_{\epsilon}$ approachs to $\cup_{j=1}^{k}  D^2_1 \times \cdots \times \check{D^2_j}\times  \cdots \times D^2_k$.  
\par
The space ${\mathcal N}^{\bullet}_{\ell}(X,H^{1*1};\alpha_-,\alpha_+)^{\boxplus 1}$ is obtained by gluing various 
\begin{equation} \label{variouspieces}
\eqref{form26777741rev} \times [-1,0]^{m_1 + m_2} \times \prod_{j=1}^{m_*} D^2_j
\end{equation}
to ${\mathcal N}^{\bullet}_{\ell}(X,H^{1*1};\alpha_-,\alpha_+)$.  
\par
We define $({\mathcal N}^{\bullet}_{\ell}(X,H^{1*1};\alpha_-,\alpha_+)^{\boxplus 1})_{\epsilon}$ by 
the union of $\eqref{form26777741rev} \times \prod_{j=1}^{m_1 + m_2} \{-1\} \times S^{(m_*)}_{\epsilon}$.  
In order to define the smooth correspondence (see \cite[Section 9.5, (9.17)]{fooonewbook}), we need to take a CF-perturbation (see 
\cite[Chapter 7]{fooonewbook}).  Since the Kuranishi structure is ${\mathfrak C}$-collared
as in Proposition \ref{prop2661} and Remark \ref{defn1531revrev} 
and compatible with the fattening on $D^2_j(\delta)$, we can 
take a CF-perturbation, which is $\tau$-collared for a small $\tau>0$ with respect to $\mathfrak C$ (see \cite[Lemma 17.40 (2)]{fooonewbook}).  
In the same way, we arrange the CF-perturbation such that it is compatible with the fattening, i.e., the product type, on $D^2_j(\delta/2)$.  
Then the contribution to the smooth correspondence from the part with $t_j \in [-1, -1+\delta/2]$ is zero and 
the smooth correspondence given by $({\mathcal N}^{\bullet}_{\ell}(X,H^{1*1};\alpha_-,\alpha_+)^{\boxplus 1})_{\epsilon}$ equipped with the CF-perturbation 
is equal to the smooth correspondence given by \eqref{form26777741rev}.  

In fact, the composition of $ \frak N_{(H,J) *} \circ \frak N_{* (H,J)}$ is given by 
$
((\Sigma,(z_-,z_+,\vec z),a_1,a_2),u,\varphi)
$
in Definition \eqref{defn210revto022} with a choice of one of $z_{a,+}=z_{a+1,-}$, $a_1 \leq a \leq a_2-1$, i.e., a transit point between 
$\Sigma_{a_1}$ and $\Sigma_{a_2}$.  (This choice determines the output of $\frak N_{* (H,J)}$ and the input of $ \frak N_{(H,J) *}$.  
Adding this choice in the data to the object, $((\Sigma,(z_-,z_+,\vec z),a_1,a_2),u,\varphi, a), a \in \{a_1, \dots, a_2 -1\}$, the intersection of 
$\cup_{j=1}^{k}  D^2_1 \times \cdots \times \check{D^2_j} \times \cdots \times D^2_k$ in the fattening factor is resolved (normalization) and 
$({\mathcal N}^{\bullet}_{\ell}(X,H^{1*1};\alpha_-,\alpha_+)^{\boxplus 1})_{\epsilon}$ becomes the K-space giving the smooth correspondence 
for $ \frak N_{(H,J) *} \circ \frak N_{* (H,J)}$.  Since our choice of the CF-perturbation makes such intersections do not contribute to 
the smooth correspondence, we can also use $({\mathcal N}^{\bullet}_{\ell}(X,H^{1*1};\alpha_-,\alpha_+)^{\boxplus 1})_{\epsilon}$. 
\par
The proof of Proposition \ref{prop23999222} is complete.
\end{proof}
The proof of Theorem \ref{therem232} is now complete.
\end{proof}
\begin{rem}
In Theorem \ref{therem232} we did not claim
that the composition
$\frak N_{* (H,J)}\circ \frak N_{(H,J) *} :
\mathcal F_X^{\rm tri} \to \mathcal F_X^{\rm tri}$ is homotopic to the
identity morphism.
We can prove it.
In fact, it follows from Claim \ref{clain26136}, but
we do not provide the technical details of the
proof of Claim \ref{clain26136} in this article.
\end{rem}

\bibliographystyle{amsalpha}

%\include{index}
%\printindex
\end{document}